\documentclass{amsart}
\usepackage[hmargin=1.7in, vmargin=1.7in]{geometry}
\usepackage[dvipsnames]{xcolor}

\usepackage{tikz}
\usetikzlibrary{positioning}
\usetikzlibrary{arrows.meta}
\usetikzlibrary{patterns}
\usetikzlibrary{calc} 
\usepackage[english]{babel}
\usepackage{amssymb, amsmath,mathtools}
\usepackage{mathrsfs,mathabx}
\usepackage{amscd}
\usepackage{verbatim}
\usepackage{stmaryrd}
\usepackage{subfig}

\usepackage{enumerate}
\usepackage{esint}
\newcommand{\nlongto}{\mathrel{\ooalign{$\longrightarrow$\cr\hidewidth$/$\hidewidth}}}

\usepackage[colorlinks,linkcolor={blue},citecolor={blue},urlcolor={purple},]{hyperref}

\usepackage[colorinlistoftodos,prependcaption,textsize=tiny]{todonotes}

\theoremstyle{plain}
\newtheorem{theorem}{Theorem}[section]
\theoremstyle{remark}
\newtheorem{remark}[theorem]{Remark}

\theoremstyle{plain}
\newtheorem{corollary}[theorem]{Corollary}
\newtheorem{lemma}[theorem]{Lemma}
\newtheorem{proposition}[theorem]{Proposition}

\newtheorem{definition}[theorem]{Definition}

\newtheorem{assumption}[theorem]{Assumption}

\numberwithin{equation}{section}

\def\N{{\mathbb N}}
\def\Z{{\mathbb Z}}
\def\R{{\mathbb R}}
\def\C{{\mathbb C}}

\newcommand{\E}{{\mathbf E}}
\renewcommand{\P}{{\mathbf P}}
\newcommand{\F}{{\mathscr F}}

\newcommand{\Progress}{\mathscr{P}}

\newcommand{\g}{\gamma}

\newcommand{\om}{\omega}
\renewcommand{\O}{\Omega}

\newcommand{\loc}{{\rm loc}}
\newcommand{\Tor}{\mathbb{T}}
\newcommand{\A}{\mathcal{A}}

\newcommand{\Sob}{\mathtt{Sob}}

\newcommand{\one}{{{\bf 1}}}
\newcommand{\embed}{\hookrightarrow}

\renewcommand{\emptyset}{\varnothing}
\newcommand{\wh}{\widehat}
\newcommand{\dd}{\mathrm{d}}
\newcommand{\Borel}{\mathscr{B}}

\newcommand{\ud}{u_{{\rm det}}}

\newcommand{\T}{\Tor}
\newcommand{\im}{\mathrm{i}}
\newcommand{\p}{\mathbb{P}}
\newcommand{\q}{\mathbb{P}^{\perp}}
\newcommand{\qq}{\mathbb{Q}}

\newcommand{\wt}{\widetilde}
\renewcommand{\i}{\mathrm{i}}
\newcommand{\Ls}{\mathbb{L}}
\newcommand{\Hs}{\mathbb{H}}
\newcommand{\Bs}{\mathbb{B}}

\newcommand{\D}{\mathscr{D}}
\newcommand{\Do}{\mathsf{D}}

\renewcommand{\H}{\mathcal{H}}

\newcommand{\uh}{\overline{v}}
\newcommand{\vh}{\overline{v}}
\newcommand{\whom}{\overline{w}}
\newcommand{\ph}{\overline{\pi}}

\newcommand{\supp}{\mathrm{supp}\,}

\newcommand{\pb}{\mathrm{p}}
\newcommand{\Cov}{\mathcal{Q}}
\newcommand{\Dir}{\mathrm{Dir}}
\newcommand{\Sto}{\mathrm{Sto}}
\newcommand{\En}{\mathcal{L}}
\newcommand{\Eno}{\mathcal{E}}

\newcommand{\sm}{\eta}

\newcommand{\Errd}{\mathrm{Err}^{{\rm det}}}
\newcommand{\Errs}{\mathrm{Err}^{{\rm sto}}}
\newcommand{\X}{\mathscr{X}}
\newcommand{\Y}{\mathscr{Y}}
\newcommand{\W}{\mathscr{W}}
\newcommand{\Errdloc}{\Errd_{{\rm loc}}}
\newcommand{\Errdglo}{\Errd_{{\rm loc}\to {\rm glo}}}

\newcommand{\qmax}{q_{{\rm max}}}
\newcommand{\qmin}{q_{{\rm min}}}
\newcommand{\pmax}{p_{{\rm max}}}
\newcommand{\pmin}{p_{{\rm min}}}

\newcommand{\ucut}{u_{{\rm cut}}}
\newcommand{\ucutd}{\overline{u}_{{\rm det}}}
\newcommand{\picutd}{\overline{\Pi}_{{\rm det}}}
\newcommand{\vcut}{w_{{\rm cut}}}
\newcommand{\tcut}{\tau_{{\rm cut}}}
\newcommand{\vcutd}{w_{{\rm cut,d}}}

\newcommand{\LL}{\mathcal{L}}
\newcommand{\const}{K}
\newcommand{\LinS}{\mathcal{S}}
\newcommand{\Lin}{\mathcal{R}}
\newcommand{\Tr}{\mathrm{Tr}}

\allowdisplaybreaks

\begin{document}

\thanks{The author is a member of GNAMPA (INdAM) and acknowledges support from INdAM through the GNAMPA 2026 project ``Fluidodinamica stocastica: irregolarità, trasporto e fenomeni di regolarizzazione''}

\author{Antonio Agresti}
\address{Department of Mathematics Guido Castelnuovo\\
	Sapienza University of Rome \\ P.le Aldo Moro 5\\ 00185 Rome\\ Italy.}
\email{antonio.agresti@uniroma1.it}

\date\today

\title[Absence of blow-up in 3D NSE\lowercase{s} with transport noise]{On the absence of blow-up in the 3D Navier-Stokes equations with transport noise}



\begin{abstract}
Establishing the global-in-time smoothness of solutions to the 3D Navier-Stokes equations (NSEs) with large initial data remains a long-standing open problem.
In this paper, we prove that the 3D NSEs driven by a suitably chosen transport noise---a physically motivated stochastic perturbation---admit global-in-time smooth solutions with high probability. 
This regularizing effect holds uniformly for initial data in arbitrarily large balls of subcritical function spaces with positive smoothness. 
\end{abstract}

\maketitle

\tableofcontents

\newpage

\section{Introduction}
\label{s:intro}
In this paper, we establish a regularizing effect of transport noise on the Navier-Stokes equations (NSEs) on the three-dimensional torus $\T^3=\R^3/\Z^3$:
\begin{equation}
\label{eq:navier_stokes_intro}
\left\{
\begin{aligned}
\partial_t u +(u\cdot \nabla)u&=-\nabla P+ \Delta u +\sqrt{\frac{3\mu}{2}}\sum_{k,\alpha}\, [-\nabla \wt{P}_{k,\alpha} +\theta_k (\sigma_{k,\alpha}\cdot\nabla) u]\circ \dot{W}^{k,\alpha}_t,\\
 \nabla \cdot u&=0,\\
 u(0,\cdot)&=u_0,
 \end{aligned}
\right.
\end{equation} 
where the summation is taken over $k\in \Z^3_0=\Z^3\setminus\{0\}$ and $\alpha\in \{1,2\}$, and 
$$
u:[0,\infty)\times\O\times \T^3\to \R^3 \quad \text{ and } \quad P,\ \wt{P}_{k,\alpha}:[0,\infty)\times\O\times \T^3\to \R
$$ 
denote the unknown velocity field, the deterministic pressure, and the stochastic pressures, respectively.
Moreover, $(\sigma_{k,\alpha})_{k,\alpha}$ is a family of smooth divergence-free vector fields, 
 $(W^{k,\alpha})_{k\in \Z^3_0,\alpha\in \{1,2\}}$ is a family of Brownian motions on a filtered probability space; both families are described in Subsection \ref{ss:probabilistic}, and $\circ$ denotes the Stratonovich integration. Finally, $\theta=(\theta_k)_{k\in \Z^3_0}\in \ell^2$ is the sequence of noise coefficients, and $\mu>0$ is the noise intensity.
The physical motivation for such stochastic perturbations is discussed in Subsection \ref{ss:physical_motivation} below.

\smallskip

A long-standing conjecture in stochastic fluid mechanics suggests that \emph{physically motivated} stochastic perturbations may improve the well-posedness theory for the 3D NSEs \cite{HB22_tame,F15_lectures_saint_flour,F15_open_3DNSE,FGP10}. 
Such a regularizing mechanism is consistent with the physical intuition that small-scale turbulence may stabilize the corresponding fluid dynamics (see Subsection \ref{ss:physical_motivation}).

\smallskip

We establish the existence of global-in-time smooth solutions to the stochastic 3D NSEs, provided that the coefficients and the noise intensity are appropriately chosen. An informal version of our main result, Theorem \ref{t:global_NSE}, reads as follows.

\begin{theorem}[Global-in-time regularity of 3D NSEs with transport noise -- Informal statement of Theorem \ref{t:global_NSE}]
\label{t:intro}
For all $\g>1/2$, $M \geq 1$ and $\varepsilon\in (0,1)$, there exist $\mu>0$ and $
\theta\in \ell^2(\Z^3_0) $ such that, for any divergence-free vector field $u_0$ satisfying
$$
\|u_0\|_{H^{\g}(\T^3;\R^3)}\leq M,
$$
there exists a unique smooth solution $u$ to the {\normalfont{3D NSEs}} with transport noise \eqref{eq:navier_stokes_intro} that is global in time with probability at least $1-\varepsilon$.
\end{theorem}

In the deterministic setting ($\mu=0$), proving global smoothness for arbitrarily large initial data is one of the Millennium Prize Problems \cite{F00_NSproblem}; see Remark \ref{r:high_dimensions} for the extension of our results to non-trivial deterministic forcing.

\smallskip

Establishing Theorem \ref{t:intro} is a highly nontrivial task, and it yields the first regulari\-zation-by-noise result for the 3D NSEs driven by a stochastic perturbation arising from physically motivated models of fluid transport. A primary challenge in establishing this regularization phenomenon is that transport noise does not affect the energy balance. Indeed, smooth solutions to \eqref{eq:navier_stokes_intro} satisfy the \emph{pathwise energy equality}:
\begin{equation}
\label{eq:pathwise_preservation}
\frac{1}{2}\int_{\T^3 }|u(t,x)|^2\,\dd x +\int_0^t\int_{\T^3}| \nabla u(s,x)|^2 \,\dd x \,\dd s =
\frac{1}{2}\int_{\T^3 }|u_0(x)|^2\,\dd x
\end{equation}
a.s.\ for all $t<\tau$. Here, $\tau$ is the lifetime of the smooth solution $u$. In particular, one recovers the same energy balance that holds in the absence of noise ($\mu=0$) independently of the noise intensity $\mu>0$, and of the coefficients $\theta=(\theta_k)_{k\in \Z^3_0}$.

\smallskip

The key to overcoming the limitations of standard energy methods lies in capturing the mixing properties of transport noise. This mechanism provides a mathematical framework to extract a macroscopic regularizing effect---reminiscent of the classical Boussinesq hypothesis \cite{Boussinesq77,Fgeophysical}---without altering the fundamental energy balance \eqref{eq:pathwise_preservation}. The resulting \emph{enhanced dissipation} of the Navier-Stokes dynamics is quantified in Corollary \ref{cor:enhanced_dissipation}.

\smallskip

The proof of Theorem \ref{t:intro} is outlined in Figure \ref{fig:scheme}, and the main ingredients are presented in Section \ref{s:proof_strategy}. The underlying analysis extends to subcritical initial data with positive Besov regularity (see Theorem \ref{t:global_NSE}):
\begin{equation}
\label{eq:Besov_regularity_data_intro}
\|u_0\|_{B^{1-2/p}_{q,p}(\T^3;\R^3)}\leq M \qquad \text{ where }\qquad p,q\in (2,\infty)\ \ \ \text{ satisfy }\ \ \  \frac{2}{p}+\frac{3}{q}<2.
\end{equation}
For instance, this accommodates initial data in $W^{\varepsilon,3}(\T^3;\R^3)$ for any $\varepsilon>0$.

\smallskip

In the remaining part of the introduction, we first compare our results with the deterministic case (Subsection \ref{ss:deterministic}), then we briefly discuss physical motivations of transport noise in stochastic fluid mechanics (Subsection \ref{ss:physical_motivation}), and finally, we outline some previous works on regularization by noise (Subsection \ref{ss:regularization_intro}).

\subsection{Comparison with the deterministic 3D NSEs}
\label{ss:deterministic}
The literature on deterministic 3D NSEs is vast; the reader is referred to \cite{LePi} for a comprehensive overview.
While the fundamental problem of global regularity remains open \cite{F00_NSproblem}, several recent breakthroughs have revealed various pathological behaviors of the deterministic 3D NSEs and related models, and, in particular, have shown that energy methods are insufficient to rule out singularities.

\smallskip

In the context of strong solutions, Tao \cite{T16_finite_time_blow_up} constructed an averaged version of the 3D NSEs that exhibits the same energy cancellation as the true 3D NSEs, yet admits strong solutions that blow up in finite time.
Moreover, in \cite[Subsection 1.3]{T16_finite_time_blow_up}, the author discusses how the mechanism driving this singularity formation might also be adapted to prove finite-time blow-up in the true deterministic NSEs.

\smallskip

In the context of weak solutions, several non-uniqueness results have been discovered. Notably, the breakthrough development of convex integration techniques by De Lellis and Sz\'{e}kelyhidi \cite{DS09_euler} has led to major results. Indeed, these techniques laid the groundwork for Isett's proof of Onsager's conjecture for the Euler equations \cite{I18_onsager}, and subsequently enabled Buckmaster and Vicol \cite{BV19_NSE} to prove the flexible non-uniqueness of weak solutions with finite kinetic energy (i.e., $u\in L^\infty_t(L^2_x)$). More recently, non-uniqueness of Leray-Hopf weak solutions (i.e., $u\in L^\infty_t(L^2_x)\cap L^2_t(H^1_x)$) was established for the forced 3D NSEs by Albritton, Bru\'{e} and Colombo \cite{ABC22_annals}, in accordance with the predictions of Jia and \v{S}ver\'{a}k \cite{JS15}. A possible extension to the unforced 3D NSEs has recently appeared as a preprint \cite{hou2025nonuniqueness}.

\smallskip

Finally, non-uniqueness phenomena in the class of finite-kinetic-energy weak solutions have been shown to persist even in the presence of transport noise, as established by Hofmanov\'a, Lange, and Pappalettera \cite{HLP24_convex} in the Euler case, and by Pappalettera \cite{P23_convex_NS} for NSEs via convex integration techniques. 

While these results show that transport noise cannot restore the uniqueness of weak solutions with finite kinetic energy \cite{HB22_tame}, Theorem \ref{t:intro} shows that suitably chosen transport noise prevents finite-time blow-up of strong solutions with high probability.

\subsection{Physical motivation}
\label{ss:physical_motivation}
Transport-type stochastic perturbations of the NSEs arise from several physically motivated modeling approaches: via homogenization of two-scale systems \cite{DP22_two_scale}, by a two-scale ansatz \cite[Subsection 1.2]{FL19} and \cite[Subsection 1.2]{A24_anomalous}, through principles of location uncertainty \cite{DM25_local,M14_derivation}, and from Newton's law \cite{MR01,MR04}. In the latter approach, \eqref{eq:navier_stokes_intro} is derived from the momentum balance, assuming that a fluid particle $\phi_t(x)$ located at $x$ at the initial time (i.e., $\phi_0(x)=x$) evolves according to the following stochastic differential equation:
\begin{equation*}
\frac{\dd}{\dd t}\phi_t(x) = u(t,\phi_t(x)) + \sqrt{\frac{3\mu}{2}}\sum_{k,\alpha}\theta_k\sigma_{k,\alpha}(\phi_t(x))\circ \dot{W}^{k,\alpha}_t  .
\end{equation*}
The physical intuition behind this decomposition is the so-called \emph{separation of scales} (see e.g., \cite{BCF91,FL32_book,MR01}). In a turbulent regime, the fluid is modeled as a two-scale system: a \emph{large regular} component $u$, reflecting our (continuum) observation scale, and a small, \emph{highly time-oscillatory} component $\sqrt{(3\mu)/2}\sum_{k,\alpha}\theta_k\sigma_{k,\alpha}\circ \dot{W}^{k,\alpha}_t$, reflecting the microscopic nature of the fluid.
Although such a strict separation of scales remains a modeling idealization, this reasoning is often used in multiscale fluid modeling \cite{F15_lectures_saint_flour,Fgeophysical,MTVE01}.

\smallskip

Transport noise is not the only stochastic perturbation arising in fluid modeling; indeed, alternative models emerge depending on the specific physical quantities being conserved (see, e.g., \cite{DM25_local,FL32_book,HLN21_annals,H15_SVP} and references therein). 
Nevertheless, transport noise is distinguished by two structural features: it is both energy-preserving (as seen in \eqref{eq:pathwise_preservation}) and scaling-invariant (see Subsection \ref{ss:almost_self}).
Moreover, we note that transport noise serves as a fundamental building block for many stochastic perturbations used in fluid modeling.

\subsection{Related works}
\label{ss:regularization_intro}
The question of whether noise can prevent the formation of singularities in the 3D NSEs, and more generally in fluid equations, has inspired an extensive body of research over the last few decades. While a comprehensive overview is beyond our scope, we highlight a few key results. For stationary weak solutions to the 3D NSEs with additive noise, a stochastic Caffarelli-Kohn-Nirenberg theory was established by Flandoli and Romito \cite{FR02}, proving the almost sure absence of singularities at any fixed time. Still, in the case of 3D NSEs with additive noise, Markov selections with the \emph{strong Feller property} were obtained by the same authors in \cite{FR08}, relying on the analysis of the infinite-dimensional Kolmogorov equation by Da Prato and Debussche \cite{DPD_3DNS} (see \cite{HZZ24} for a result on non-uniqueness in law via convex integration). In the case of multiplicative noise depending linearly on the velocity, global well-posedness with high probability for the Euler equations with large noise intensity was proven by Glatt-Holtz and Vicol \cite{GHV14}. More recently, for the 2D Euler equations with rough Kraichnan noise, Coghi and Maurelli \cite{CM26} established uniqueness beyond the Yudovich class (see also \cite{BGM25,JL25_SQG} for subsequent related works).

\smallskip

The regularizing potential of \emph{transport noise} was first discovered in the seminal work of Flandoli, Gubinelli, and Priola \cite{FGP10} on linear transport equations. In this setting, the authors proved well-posedness for drifts with H\"older regularity below the threshold of the classical DiPerna-Lions theory \cite{DPL89} (see \cite{BCK26} for recent related optimality results). However, as noted in \cite{FGP10}, transport noise with spatially independent coefficients cannot, in general, prevent blow-up in nonlinear fluid equations.
 
\smallskip

A fundamental step in understanding the regularization properties of transport noise was made by Galeati in \cite{G20}. In this work, the author showed that solutions to transport equations with suitably chosen transport noise converge to solutions of a deterministic heat equation (thus, with a non-trivial diffusivity) in a weak sense. 
This methodology is now commonly referred to as the \emph{scaling limit} (see Subsection \ref{ss:scaling_intro}). Flandoli and Luo \cite{FL19} subsequently exploited it to establish, with high probability, the global existence of smooth solutions to the 3D NSEs in vorticity form with transport noise.
However, as acknowledged by the same authors in \cite[Appendix 3]{FL19}, incorporating transport noise directly into the vorticity formulation lacks a clear physical motivation. The key analytic advantage of considering transport noise in the vorticity form is that the energy/enstrophy (i.e., the $L^2$-norm of the corresponding unknown) is a subcritical quantity. 
By contrast, the $L^2$-energy estimate for the velocity is supercritical; see Subsection \ref{ss:scaling_intro}. Therefore, the energy methods used in \cite{FL19} break down for the stochastic 3D NSEs \eqref{eq:navier_stokes_intro}, requiring techniques beyond energy estimates; see Section \ref{s:proof_strategy}.
The scaling limit technique has subsequently been widely employed across a broad range of SPDEs, yielding several results on regularization by noise \cite{A22,AKT26,FGL21,L23_regularization_NS}, mixing and enhanced dissipation \cite{FGL21_quantitative,L23_enhanced,LTZ24}, and anomalous dissipation \cite{A24_anomalous,BFLT26,HPZZ23}.
The mechanism exploited in \cite{FL19} is reminiscent of classical stabilization by noise results \cite{Arn90,ACW83}, as well as the suppression of blow-up by deterministic vector fields with suitable mixing properties \cite{IXZ21,KX16}. The latter line of research finds its roots in the seminal work of Constantin, Kiselev, Ryzhik, and Zlato\v{s} \cite{CKRZ08_mixing}. We refer to \cite{CZC24,MFNsurvery} and the references therein for subsequent related results in the deterministic setting.

Further progress in the context of scaling limits was achieved by the author in \cite{A24_global_small}, where, by combining maximal $L^p$-regularity methods for SPDEs pioneered by Krylov \cite{Kry} and van Neerven, Veraar, and Weis \cite{MaximalLpregularity}, the analogue of Theorem \ref{t:intro} was established when the Laplace operator in the stochastic NSEs \eqref{eq:navier_stokes_intro} is replaced by a hyperviscous operator $-(-\Delta)^{\alpha}$ with $\alpha>1$ (see also \cite{ABL26_Lie}). 
The former provides a significant improvement of the well-posedness theory for the 3D hyperviscous NSEs compared to the deterministic theory, where global regularity results are limited to the case $\alpha\geq \frac{5}{4}$ (so-called J.L.\ Lions exponent \cite{Lio69}), or more recently, $\alpha> \frac{5}{4}-\varepsilon_0$ for some $\varepsilon_0>0$ depending on the initial data, see \cite[Corollary 1.6]{CDLM20} and \cite{CH21}.
However, the presence of hyperviscosity breaks the natural scaling of the NSEs \eqref{eq:navier_stokes_intro} (see Subsection \ref{ss:scaling_intro}), and transport noise becomes a lower-order perturbation along the scaling limit (see \cite[Section 3]{A24_global_small} for details).

\smallskip

For the true stochastic NSEs \eqref{eq:navier_stokes_intro}, the approach in \cite{A24_global_small} breaks down. Indeed, the Laplace operator and the transport noise are of the same differential order (see \cite[Subsection 1.1]{AV21_NS} for the explicit scaling argument), and therefore, the latter cannot be treated via perturbation techniques. Overcoming this difficulty requires the development of new machinery, such as stochastic Caccioppoli inequalities, localized stochastic maximal regularity, quantitative localized scaling limits, and new supremum bounds for SPDEs with highly oscillating coefficients. Moreover, our investigation reveals novel structural connections between scaling limits and homogenization theory (we refer to the recent work of Armstrong and Vicol \cite{AV25} for another application of homogenization techniques to anomalous dissipation in fluid mechanics).

A thorough discussion of the novelty and strategy is given in the following section.

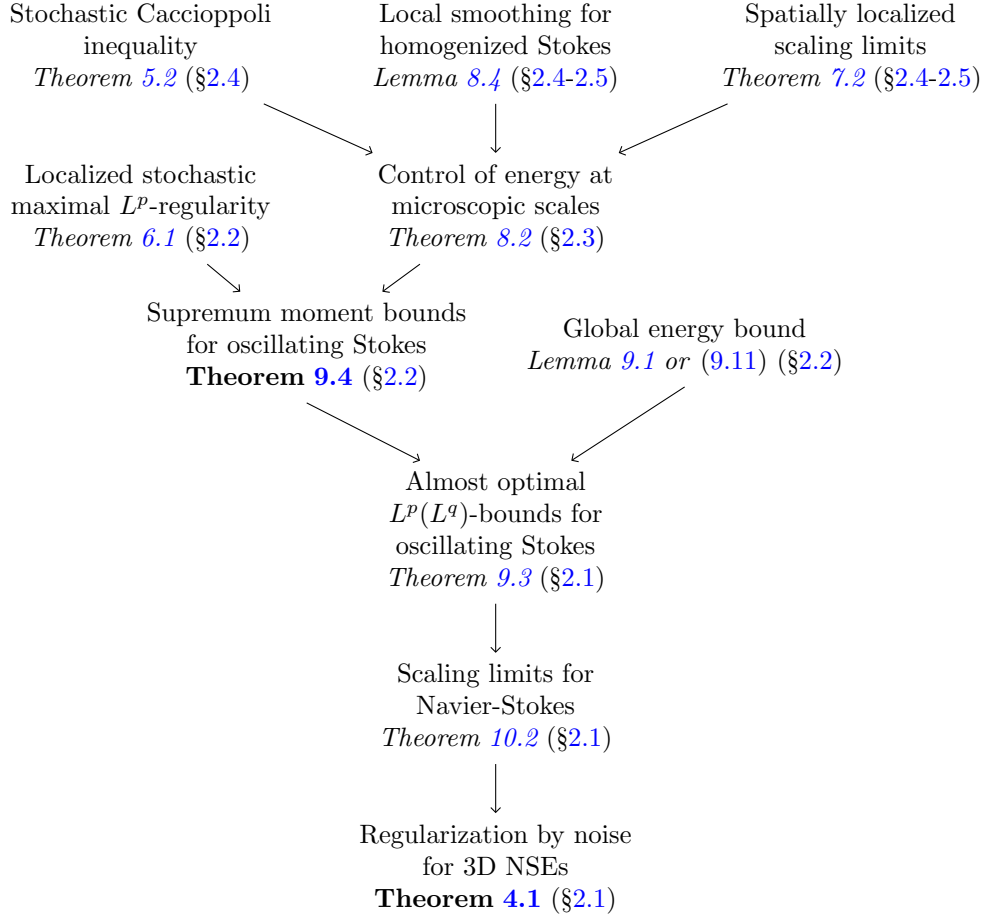
\begin{figure}[htpb]
    \centering
    \begin{tikzpicture}[
        >={Straight Barb[scale=0.8]}, 
        mathnode/.style={
            align=center,
            text width=4.2cm, 
            inner sep=4pt
        }
    ]

    \node[mathnode] (N2) {Local smoothing for\\ homogenized Stokes \\ \emph{Lemma \ref{l:local_smoothing_homogenized}} (\textsection\ref{ss:caccioppoli_intro}-\ref{ss:role_pressure_intro})}; 
    \node[mathnode, left=0.2cm of N2] (N3) {Stochastic Caccioppoli \\ inequality \\ \emph{Theorem \ref{t:caccioppoli}} (\textsection\ref{ss:caccioppoli_intro})};
    \node[mathnode, right=0.2cm of N2] (N1) {Spatially localized\\ scaling limits \\ \emph{Theorem \ref{t:universal_scaling_limit}} (\textsection\ref{ss:caccioppoli_intro}-\ref{ss:role_pressure_intro})};

    \node[mathnode, below=0.65cm of N2] (N4) {Control of energy at \\ microscopic scales \\ \emph{Theorem \ref{t:control_mesoscopic_energy}} (\textsection\ref{ss:control_energy_intro})};
    \node[mathnode] at (N3 |- N4) (N5) {Localized stochastic \\ maximal $L^p$-regularity \\ \emph{Theorem \ref{t:localized_SMR}} (\textsection\ref{ss:almost_self})};
    
    \coordinate[below=1.1cm of N4] (base Tier3);

    \node[mathnode, text width=5.5cm, left=2.5cm of base Tier3, anchor=center] (N6) {Supremum moment bounds  \\ for oscillating Stokes \\ \textbf{Theorem \ref{t:almost_Linfty}} (\textsection\ref{ss:almost_self})};
    
    \node[mathnode, right=2.5cm of base Tier3, anchor=center] (N7)  {Global energy bound \\ \emph{Lemma \ref{l:global_energy_estimate} or \eqref{eq:energy_estimate_vN}} (\textsection\ref{ss:almost_self})};   
    
    \node[mathnode, below=1.5cm of base Tier3] (N8b) {Almost optimal \\ $L^p(L^q)$-bounds for \\  oscillating Stokes \\ \emph{Theorem \ref{t:suboptimal_Lq_Rpositive}} (\textsection\ref{ss:scaling_intro})};   

    \node[mathnode, below=0.65cm of N8b] (N8) {Scaling limits for \\ Navier-Stokes \\ \emph{Theorem \ref{t:scaling_limit_cutoff}} (\textsection\ref{ss:scaling_intro})};

    \node[mathnode, below=0.65cm of N8] (N9) {Regularization by noise  \\  for 3D NSEs \\ \textbf{Theorem \ref{t:global_NSE}} (\textsection\ref{ss:scaling_intro})};
        
    \draw[->] (N3) -- (N4);
    \draw[->] (N2) -- (N4);
    \draw[->] (N1) -- (N4);
    
    \draw[->] (N5) -- (N6);
    \draw[->] (N4) -- (N6);
    
    \draw[->] (N6.south) -- ([xshift=-1.0cm]N8b.north);
    \draw[->] (N7.south) -- ([xshift=1.0cm]N8b.north);
    
    \draw[->] (N8b) -- (N8);
    \draw[->] (N8) -- (N9);

    \end{tikzpicture}
\caption{Proof architecture. The main result (Theorem \ref{t:global_NSE}) and the crucial intermediate step (Theorem \ref{t:almost_Linfty}) are highlighted in bold. Relevant subsections are indicated in parentheses.}    
\label{fig:scheme}
\end{figure}

\section{Proof outline}
\label{s:proof_strategy}
This section provides an overview of the proof strategy for Theorem \ref{t:global_NSE}, which generalizes Theorem \ref{t:intro}. 
The architecture is summarized in Figure \ref{fig:scheme} and proceeds via several steps. 
We present the strategy in reverse order, starting from the main theorem and tracing back to the fundamental estimates, to underline the role of each ingredient.
First, we explain how the regularization by noise for the 3D NSEs can be established via certain $L^p(L^q)$-estimates. Second, we introduce the central mathematical novelty of this paper: \emph{supremum moment bounds} for the stochastic oscillating Stokes system (Subsection \ref{ss:almost_self}). These bounds are derived by combining localized stochastic maximal $L^p$-regularity with \emph{uniform control of the energy at microscopic scales}. 
Finally, we present the main tools to obtain such uniform bounds at microscopic scales, namely, stochastic Caccioppoli inequalities, local smoothing for the homogenized Stokes system, and quantitative local scaling limits.

\smallskip

To establish the supremum moment estimates fundamental to our approach, we cast the classical large-scale regularity and blow-up ideas of Avellaneda and Lin \cite{AL87_compactness} into a novel quantitative stochastic framework (Subsection \ref{ss:control_energy_intro}).
Implementing this approach requires overcoming two fundamental obstructions:
\begin{itemize}
\item
\emph{Breakdown of qualitative methods}. Compactness techniques fail in the stochastic setting, as enforcing strong convergence in the probability space entails a loss of integrability, which destroys the propagation of bounds across spatial scales. Consequently, we establish quantitative error estimates for the oscillating Stokes system relative to its homogenized counterpart (see the end of Subsection \ref{ss:control_energy_intro}).
\item 
\emph{Non-locality and time irregularity of the pressures}. The deterministic and stochastic pressures provide several additional complications, particularly due to their non-local nature and lack of time regularization. To handle them, we follow the spirit of the celebrated work of Caffarelli, Kohn, and Nirenberg \cite{CKN82}, utilizing local solutions, local energy inequalities, and local approximations of the pressures to capture their smoothness (Subsection \ref{ss:role_pressure_intro}).
\end{itemize}

We begin by discussing the relevant background. For any unexplained notation, see Subsection \ref{ss:notation} below.

\subsection{Background: Criticality, scaling limits and $L^p(L^q)$-estimates}
\label{ss:scaling_intro}
As in the deterministic case, solutions to stochastic 3D NSEs \eqref{eq:navier_stokes_intro} are (locally) scaling invariant under the transformation $u\mapsto u_\lambda$ where $\lambda>0$ and 
\begin{align}
\label{eq:NS_scaling_intro}
u_{\lambda}(t,x)&=\lambda u(\lambda^2 t ,\lambda x), 
\end{align}
and the corresponding transformation for the deterministic and stochastic pressures, see \cite[Subsection 1.1]{AV21_NS} and Subsection \ref{ss:role_pressure_intro} below. 
In particular, invariant (or critical) spaces for the stochastic problem \eqref{eq:navier_stokes_intro} coincide with deterministic ones, e.g.,
\begin{equation}
L^3(\T^3;\R^3)\quad \text{ and }\quad B^{3/q-1}_{q,p}(\T^3;\R^3) \ \text{ for }\ q,p\in [1,\infty].
\end{equation}
Scaling invariance uniquely identifies the regularity threshold $-1$ in the sense of Sobolev index,\footnote{The Sobolev index of the Bessel-potential space $H^{s,q}$ on a $d$-dimensional domain is $s-d/q$, and similarly for Besov spaces $B^{s}_{q,p}$. This index rules the scaling of the corresponding homogeneous space: $\|f(\lambda \cdot )\|_{\dot{H}^{s,q}(\R^d)}=\lambda^{s-d/q}\|f\|_{\dot{H}^{s,q}(\R^d)}$. For an interval $I$ and under parabolic scaling, the space-time Sobolev index of $L^p(I;H^{s,q})$ is $-2/p+s-d/q$.} under which well-posedness \emph{cannot} in general be ensured (see e.g., \cite{ABC22_annals,AC23_energy_inequality}). Therefore, to establish Theorem \ref{t:intro} (or Theorem \ref{t:global_NSE}), we need to obtain information in a function space with space-time Sobolev index $\geq -1$. 

\subsubsection*{Scaling limits}
Consider a cutoff version of the NSEs \eqref{eq:navier_stokes_intro} on $\T^3$:
\begin{align}
\label{eq:navier_stokes_intro_cut}
\partial_t \ucut &+\phi_{\X}^R(\cdot,\ucut)(\ucut\cdot \nabla)\ucut=-\nabla \Pi+ \Delta \ucut \\
\nonumber
&\qquad \qquad  +\sqrt{\frac{3\mu}{2}}\sum_{k,\alpha}\, [-\nabla \wt{\Pi}_{k,\alpha} +\theta_k (\sigma_{k,\alpha}\cdot\nabla) \ucut]\circ \dot{W}^{k,\alpha}_t,
\end{align}
together with the conditions $\nabla \cdot \ucut=0$ and $\ucut(0)=u_0$.
In the above, for $R>0$ and $\phi\in C^{\infty}_{{\rm c}}([0,\infty))$ such that $\phi|_{[0,1]}=1$ and $\phi|_{[2,\infty)}=0$, we denote
\begin{equation}
\label{eq:cutoff_nonlinearity_intro}
\phi_{\X}^R(t,\ucut)
= \phi\big(R^{-1}\|\ucut|_{(0,t)}\|_{\X(t)}\big),
\end{equation}
where $\X(t)$ is a space of functions on $(0,t)\times\T^3$. 
The specific space $\X$ depends on the available \emph{a priori} bounds for \eqref{eq:navier_stokes_intro_cut} under \emph{rough} transport noise. However, due to the scaling argument below \eqref{eq:NS_scaling_intro}, the global well-posedness of \eqref{eq:navier_stokes_intro_cut} requires that $\X$ employed in the cutoff \eqref{eq:cutoff_nonlinearity_intro} has a space-time Sobolev index $\geq -1$. 
For instance, in our analysis, we choose 
\begin{equation}
\label{eq:space_chosen_scaling_limit}
\X(t)=L^{2p}(0,t;L^{2q}(\T^3;\R^3)) \quad \text{ where }\quad \frac{2}{p}+\frac{3}{q}<2.
\end{equation}
Following the arguments in \cite{FL19}, it follows that there exists a sequence $(\theta^N)_{N}\subset \ell^2(\Z^3_0)$ satisfying $\supp\theta^N =\{k\in \Z^3_0\,:\, N \leq |k|\leq 2N\}$ and 
\begin{equation}
\label{eq:scaling_limit}
\lim_{N\to \infty} \|\theta^N\|_{\ell^\infty}=0 \qquad \text{ and }\qquad \|\theta^N\|_{\ell^2}=1 \ \text{ for all }N\geq 1,
\end{equation}
such that the unique global solution $\ucut^N$ to \eqref{eq:navier_stokes_intro_cut} with $\theta=\theta^N$ converges in probability within a space $\Y$ (discussed below) to $\ucutd$, the solution of:
\begin{equation}
\label{eq:navier_stokes_intro_cut_lim}
\partial_t \ucutd +\phi_{\X}^R (\cdot,\ucutd)  \nabla\cdot (\ucutd \otimes \ucutd) 
=-\nabla \picutd+ \Big(1 + \frac{3\mu}{5}\Big)\Delta \ucutd ,
\end{equation}
together with the conditions $\nabla \cdot \ucutd=0$ and $\ucutd(0)=u_0$. The above is typically referred to as the \emph{scaling limit} due to its connection to the small-scale interpretation of transport noise; see Subsection \ref{ss:physical_motivation}.
The emergence of this additional dissipation is, to some extent, consistent with the Boussinesq hypothesis \cite{Fgeophysical,FL_boussinsesq}.
We stress that the enhanced viscosity $(3\mu)/5> 0$ appearing in the limit \eqref{eq:navier_stokes_intro_cut_lim} is not present in the original system \eqref{eq:navier_stokes_intro}, which satisfies the pathwise energy balance  \eqref{eq:pathwise_preservation}.

\subsubsection*{$L^p(L^q)$-estimates}
Provided that $\mu\gg 1$ and the space $\X$ has a space-time Sobolev index $\geq -1$, the deterministic limit \eqref{eq:navier_stokes_intro_cut_lim} is globally well-posed even without the cutoff. In this regime, one can select $R\gg 1$ large enough (depending on $u_0$ and $\mu$) so that $\phi_{\X}^R (t,\ucutd(t))=1$ for all $t>0$. Consequently, $\ucutd$ becomes a \emph{true} global solution to the deterministic 3D NSEs with enhanced viscosity.

Therefore, if the convergence $\ucut^N\to \ucutd$ can be established in a space $\Y=\X$ with a space-time Sobolev index $\geq -1$, then this immediately yields a global smooth solution to the \emph{original} stochastic 3D NSEs \eqref{eq:navier_stokes_intro} with high probability, thereby establishing Theorem \ref{t:intro}. 
Thus, by stochastic compactness, it suffices to prove a uniform-in-$N$ bound for the sequence $(\ucut^N)_{N\geq 1}$ in, e.g., 
\begin{equation}
\label{eq:convergence_space}
L^1(\O;\W) \qquad \text{ where }\qquad
\W\stackrel{{\rm c}}{\embed} \X.
\end{equation}
However, deriving such estimates for a sufficiently smooth space $\Y$ presents fundamental challenges due to the enhanced dissipation effect arising in the limiting system \eqref{eq:navier_stokes_intro_cut_lim}. Indeed, as follows from the results in \cite{A24_anomalous} (cf.\ Remark 3.4 therein), for any $t>0$:
\begin{equation*}
\int_0^t \int_{\T^3} |\nabla \ucut^N|^2 \to 
\Big(1+\frac{3}{5}\mu\Big) \int_0^t \int_{\T^3} |\nabla \ucutd|^2\ \  \text{ in probability}.
\end{equation*}
In particular, if $\mu>0$ and $u_0$ is not constant, then for all $t>0$,
\begin{equation}
\label{eq:convergence_lack_gradients}
\nabla \ucut^N\nlongto 
\nabla \ucutd \text{ in probability in }L^2((0,t)\times \T^3).
\end{equation}
Indeed, the limiting procedure \eqref{eq:scaling_limit}-\eqref{eq:navier_stokes_intro_cut_lim} generates large gradients, precluding the strong convergence of $\nabla \ucut^N$ towards $\nabla \ucutd$.
Analytically, this behavior stems from the \emph{lack} of uniform-in-$N$ regularity of the noise coefficients along the sequence as in \eqref{eq:scaling_limit}. More precisely, for any $\gamma>0$, the sequence of noise coefficients satisfies (see e.g., \cite[Proposition 4.1]{Arole}):
\begin{equation}
\label{eq:irregularity_of_coefficients}
\sup_{N} \|(\theta^N_k \sigma_{k,\alpha})_{k,\alpha}\|_{L^\infty(\T^3;\ell^2)}<\infty \quad \text{ and }\quad 
\sup_{N} \|(\theta^N_k \sigma_{k,\alpha})_{k,\alpha}\|_{H^\gamma(\T^3;\ell^2)}=\infty.
\end{equation} 
In practice, the noise coefficient sequence lacks uniform smoothness due to the highly oscillatory nature of the coefficients, since $\supp\theta^N\sim N$. 
As is well known (see, e.g., \cite{FGL21}), standard energy estimates only provide uniform-in-$N$ bounds in the path space
$$
L^\infty(0,T;L^2(\T^3;\R^3))\cap L^2(0,T;H^1(\T^3;\R^3))\cap C^{\alpha_0}(0,T;H^{-\alpha_1}(\T^3;\R^3)) 
$$
for some $\alpha_0>0$ and $\alpha_1\gg 1$. The Sobolev index of the above is at most $-3/2$, with a substantial gap to the critical threshold $-1$. This limitation mirrors the well-known failure of energy methods to rule out blow-up for the deterministic 3D NSEs.

Therefore, as observed in \cite[Section 1]{A24_global_small}, obtaining the necessary uniform-in-$N$ estimates requires a space $\W$ satisfying \eqref{eq:convergence_space} that has spatial regularity strictly less than $1$---to account for the loss of gradient convergence \eqref{eq:convergence_lack_gradients}---but possesses a space-time Sobolev index strictly greater than $-1$.

The space \eqref{eq:space_chosen_scaling_limit} serves as a prototype example for this purpose. The derivation of these uniform-in-$N$ estimates for $\ucut^N$ is outlined in the subsequent subsections. Due to the presence of the cutoff function $\phi_{\X}^R(\cdot,\ucut^N)$, deriving these uniform-in-$N$ estimates reduces to proving a linear bound for the corresponding Stokes system with \emph{highly oscillating coefficients} (in principle, only bounded and measurable) within a solution space having a space-time Sobolev index $>-1$. The rest of this section is devoted to establishing this property.

\subsection{Almost self-similarity, microscopic scale and supremum moment bounds}
\label{ss:almost_self}
Motivated by the previous arguments, we now seek uniform-in-$N$ estimates for the stochastic (linear) Stokes system with \emph{highly oscillating coefficients}:
\begin{equation}
\label{eq:turbulent_Stokes_microscopic_intro}
\left\{
\begin{aligned}
\partial_t v^N &=-\nabla \pi^N+ (1+\mu)\Delta v^N+ \nabla \cdot F - \frac{3\mu}{2}\sum_{k,\alpha} \theta^N_k \nabla \cdot (\nabla \wt{\pi}_{k,\alpha}^N\otimes \sigma_{-k,\alpha})\\ 
&  +\sqrt{\frac{3\mu}{2}} \sum_{k,\alpha} [-\nabla \wt{\pi}^{N}_{k,\alpha} +\theta^{N}_{k}(\sigma_{k,\alpha}\cdot\nabla) v^N+\theta^{N}_{k}\sigma_{k,\alpha}\cdot G]\,\dot{W}^{k,\alpha}_t,\\
\nabla \cdot v^N&=0, \qquad v^N(0)=0,
\end{aligned}
\right.
\end{equation}
where $\theta^N$ are as in \eqref{eq:scaling_limit} and $F$ and $G$ are progressively measurable forcings from the regularity class 
\begin{equation}
\label{eq:forcing_assumption_intro}
L^r(\O;L^p(0,T;L^q(\T^3;\R^{3\times 3}))) \ \  \text{ where } \ p,q\in (2,\infty)\text{ and }r\in (1,\infty).
\end{equation}
The stochastic forcing $G$ is used to treat non-zero initial data (see the proof of Theorem \ref{t:suboptimal_Lq_finite_time} for details). In the above, the coefficients $p,q\in (2,\infty)$ can be chosen to be equal at the expense of a less general regularity class in Theorem \ref{t:global_NSE}. However, it will be essential to choose $r\gg 1$ differently from $p,q$, see \eqref{eq:boundedness_Lr_quenched_supremum_intro} below. 

Compared with \eqref{eq:navier_stokes_intro_cut}, the stochastic Stokes system \eqref{eq:turbulent_Stokes_microscopic_intro} is in the It\^o form. The additional terms $\mu\Delta v^N$ and $ \frac{3\mu}{2}\sum_{k,\alpha} \theta^N_k \nabla \cdot (\nabla \wt{\pi}_{k,\alpha}^N\otimes \sigma_{-k,\alpha})$ appearing in the deterministic part of \eqref{eq:turbulent_Stokes_microscopic_intro} are the so-called It\^o-Stratonovich correctors, and as is well-known, arise when reformulating Stratonovich noise into its It\^o counterpart (Subsection \ref{ss:ito_reformulation}). Crucially, the term $\mu\Delta v^N$ does not provide an additional dissipation; it only compensates for the energy creation of the It\^o noise.  

\smallskip

The deterministic and stochastic pressures contribute non-trivially to \eqref{eq:turbulent_Stokes_microscopic_intro}, from both dynamical and mathematical points of view. For clarity of exposition, we postpone comments on their analysis to Subsection \ref{ss:role_pressure_intro}.

\subsubsection*{Almost self-similarity and microscopic scale}
Unlike periodic and stochastic homogenization \cite{SKM19_book,S18_book}, the system \eqref{eq:turbulent_Stokes_microscopic_intro} does not have a clear self-similar structure, not even for the standard choice (used e.g., in \cite{FL19}):
\begin{equation}
\label{eq:choice_thetaN_intro}
\theta^N_k =\frac{\Theta^N_k}{\|\Theta^N\|_{\ell^2}}\qquad \text{ where }\qquad \Theta_k^N =\frac{1}{|k|^{a}} \one_{\{N\leq |k|\leq 2N\}},
\end{equation}
for some fixed $a>0$. 
However, from a regularity perspective, exact global scaling is not essential; it suffices to identify a characteristic spatial scale at which the system maintains uniform smoothness (or oscillations).
To illustrate this, let us note that the covariance associated with the driving noise in \eqref{eq:turbulent_Stokes_microscopic_intro} is given, for $x,z\in \T^3$, by (cf.\ \cite[Remark 1.8]{FGL21_quantitative})
\begin{equation}
\label{eq:covariance_intro}
\Cov^N (z)=\frac{3\mu}{2} \sum_{k,\alpha}(\theta_k^N)^2 \sigma_{k,\alpha}(x+z)\otimes \sigma_{-k,\alpha}(x)
=\frac{3\mu}{2C_N}\sum_{N\leq |k|\leq 2N} \frac{e^{2\pi \i z\cdot k}}{|k|^{2a}} \Big(\mathrm{Id}-\frac{k\otimes k}{|k|^2}\Big)
\end{equation}
where $C_N=\sum_{N\leq |k|\leq 2N} |k|^{-2a}$. The irregularity of the noise coefficients \eqref{eq:irregularity_of_coefficients} can be read off the covariance by noting that $\Cov^N$ is uniformly-in-$N$ \emph{pointwise} bounded on $\T^3$, but its $C^\gamma(\T^3)$-norm diverges as $N\to \infty$ for any $\gamma>0$. This is due to the highly oscillatory nature of the coefficients, which are concentrated at frequencies $|k|\sim N$.

However, zooming into the spatial domain by a factor $N^{-1}$, one sees roughly the same amount of oscillations (see Figure \ref{fig:microscopic_scale}): For any $z\in\T^3$,
\begin{align*}
\Cov^N (N^{-1}z)
\to \frac{3\mu}{2C_{a}}
\int_{1\leq |\xi|\leq 2} \frac{e^{2\pi \i z\cdot \xi}}{|\xi|^{2a}} \Big(\mathrm{Id}-\frac{\xi \otimes \xi }{|\xi |^2}\Big)\,\dd \xi,
\end{align*}
where $C_{a}=\int_{1\leq |\xi|\leq 2}|\xi|^{-2a}\,\dd \xi$. In particular, the rescaled covariance $\Cov^N(N^{-1}\cdot)$ converges to a smooth covariance kernel, together with its higher-order derivatives. This remains valid at the level of the SPDE, where the transport noise converges (formally) to another transport-type noise with \emph{regular} (but possibly different) coefficients. In this sense, the noise in \eqref{eq:turbulent_Stokes_microscopic_intro} is \emph{almost self-similar}.
 
Analytically, at the scale $N^{-1}$, the rescaled coefficients satisfy (see Lemma \ref{l:uniform_smoothness_microscopic_scale})
\begin{equation}
\label{eq:regularity_of_coefficients_rescaled}
\sup_{N} \|(\theta^N_k \sigma_{k,\alpha}(N^{-1}\cdot ))_{k,\alpha}\|_{C^\gamma(B;\ell^2)}<\infty\  \text{ for all }\gamma>0 ,
\end{equation}
where $B\subseteq \R^d$ is a ball of finite radius, cf.\ \eqref{eq:irregularity_of_coefficients}.
For this reason, and in parallel with homogenization theory, we say that $N^{-1}$ is the \emph{microscopic scale} for the oscillating stochastic Stokes system \eqref{eq:turbulent_Stokes_microscopic_intro}.

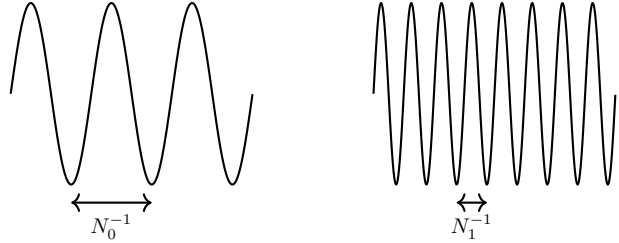
\begin{figure}[htbp]
    \centering
    \begin{tikzpicture}[
        wave/.style={thick, smooth, samples=200},
        scale=0.8,
        transform shape
    ]
        
        \draw[wave, domain=0:4] plot (\x, {1.5*sin(270*\x)});

        \draw[<->, thick] (1, -1.8) -- (2.333, -1.8) node[midway, below=2pt] {$N_0^{-1}$};

        \draw[wave, domain=6:10] plot (\x, {1.5*sin(720*(\x-6))});
        
        \draw[<->, thick] (7.375, -1.8) -- (7.875, -1.8) node[midway, below=2pt] {$N_1^{-1}$};

    \end{tikzpicture}
    \caption{Uniform smoothness at the microscopic scale, where $N_1\gg N_0$.}
    \label{fig:microscopic_scale}
\end{figure}

\subsubsection*{Localized stochastic maximal $L^p$-regularity and supremum moment bounds}
Let $Q$ be a parabolic cylinder with center $(t_0,x_0)$. The just-discussed almost self-similarity implies that the coefficients of the stochastic Stokes system \eqref{eq:turbulent_Stokes_microscopic_intro} at the microscopic scale $N^{-1}Q$ are uniformly-in-$N$ bounded, for example, in H\"older continuous norms, see \eqref{eq:regularity_of_coefficients_rescaled}.
In the latter situation, parabolic smoothing in the form of stochastic maximal $L^p$-regularity methods applies (Section \ref{s:localized_smr}). By localization, it follows that, for any \emph{local solution} $v^N$ to \eqref{eq:turbulent_Stokes_microscopic_intro} on $N^{-1} Q$, the following uniform-in-$N$ pointwise estimate holds (see Corollary \ref{cor:SMR_microscopic_estimate}):
\begin{equation}
\begin{aligned}
\label{eq:corollary_vN_intro}
\big(\E |v^N(t_0,x_0)|^r\big)^{1/r}
&\lesssim \big(\E\|F\|^r_{L^p(I;L^q(B))}\big)^{1/r}+\big(\E\|G\|^r_{L^{p}(I;L^q(B))}\big)^{1/r}\\
& +\big(\E\|v^N\|_{\underline{L}^2(N^{-1}Q)}^r\big)^{1/r}+ \mathscr{C}_Q[\pi^N,(\wt{\pi}^N_{k,\alpha})_{k,\alpha}]
\end{aligned}
\end{equation}
where $\mathscr{C}_Q[\pi^N,(\wt{\pi}^N_{k,\alpha})_{k,\alpha}]$ collects the pressure contributions, $\underline{L}^2$ denotes the averaged Lebesgue spaces 
$\|v^N\|_{\underline{L}^2(N^{-1}Q)}=\big(\fint_{N^{-1}Q}|v^N|^2\big)^{1/2}$, $r\in (1,\infty)$, and 
\begin{equation}
\label{eq:parameters_condition_Linfty_bounds}
\frac{2}{p}+\frac{3}{q}<1.
\end{equation}
As is well-known, the previous condition is necessary even in the absence of noise for the pointwise bound \eqref{eq:corollary_vN_intro} to hold. Targeting the pointwise bound \eqref{eq:corollary_vN_intro} under the assumption \eqref{eq:parameters_condition_Linfty_bounds} is a known strategy in homogenization (cf.\ \cite[Section 4]{S18_book}). This approach exploits the `rigidity' of the $L^\infty$-norm on continuous functions: because pointwise evaluations are well defined, the left-hand side of \eqref{eq:corollary_vN_intro} can be made independent of the microscopic scale $N^{-1}$.

Local smoothing estimates such as \eqref{eq:corollary_vN_intro} are well-known in the context of PDEs, see e.g., \cite{GM12_appunti,S18_book}. 
Note that, to `follow' the scale $N^{-1}Q$, we must work with local solutions to the oscillating Stokes system \eqref{eq:turbulent_Stokes_microscopic_intro}. This creates several difficulties in handling the pressures, see Subsection \ref{ss:role_pressure_intro}.

\smallskip

The control of the seemingly innocent lower-order term
\begin{equation}
\label{eq:energy_microscopic_intro}
\big(\E\|v^N\|_{\underline{L}^2(N^{-1}Q)}^r\big)^{1/r}\quad \text{ (energy at the microscopic scale $N^{-1}Q$)}
\end{equation}  
appearing on the right-hand side of \eqref{eq:corollary_vN_intro} is highly involved and will occupy the rest of the discussion, see Subsections \ref{ss:control_energy_intro} and \ref{ss:caccioppoli_intro}, while the non-trivial pressure contribution is discussed in Subsection \ref{ss:role_pressure_intro}.
To see the difficulty in controlling the energy at microscopic scales, note that the norm $\underline{L}^2(N^{-1}Q)$ as $N\to \infty$ approximates the pointwise evaluation at the center of the cylinder $(t_0,x_0)$ (Lebesgue's differentiation).

\smallskip

For the moment, we first discuss how \eqref{eq:corollary_vN_intro}, estimating the energy at the microscopic scales, actually leads to a key bound for the scaling limit argument in Subsection \ref{ss:scaling_intro}.
More precisely, the former estimate and a standard covering argument imply (see Theorem \ref{t:almost_Linfty} and Figure \ref{fig:cubes_nested}) a corresponding bound for the quantity:
\begin{equation}
\label{eq:supremum_bound_quenched_intro}
\sup_{(t_0,x_0)\in I\times \T^3} \big(\E |v^N(t_0,x_0)|^r\big)^{1/r}\quad \text{ (Supremum moment norm)}
\end{equation}
where $I$ is an interval. 
The above supremum term does not perfectly align with the scaling limit argument of Subsection \ref{ss:scaling_intro}, due to the order of expectation with respect to the supremum over $I\times \T^3$.
To compensate for this, one uses the fact that, trivially, 
\begin{equation}
\label{eq:boundedness_Lr_quenched_supremum_intro}
\big(\E \|v^N\|_{L^r(I\times \T^3)}^r\big)^{1/r} \leq |I|^{1/r} 
\sup_{(t_0,x_0)\in I\times \T^3} \big(\E |v^N(t_0,x_0)|^r\big)^{1/r}.
\end{equation}
The advantage is that the expectation on the right-hand side of \eqref{eq:boundedness_Lr_quenched_supremum_intro} now appears in the right order. However, combining \eqref{eq:boundedness_Lr_quenched_supremum_intro} with \eqref{eq:corollary_vN_intro} comes with two expenses:
\begin{itemize}
\item \emph{Lack of sharpness}. The space-time Sobolev index of the norm on the left-hand side of \eqref{eq:boundedness_Lr_quenched_supremum_intro} is $-\frac{5}{r}<0$ while the optimal one for the solution to \eqref{eq:turbulent_Stokes_microscopic_intro} with forcing in \eqref{eq:forcing_assumption_intro} satisfying \eqref{eq:parameters_condition_Linfty_bounds} is $1-\frac{2}{p}-\frac{3}{q}>0$. 
\item \emph{Need for high moments $r\gg 1$}. The sharpness of the bound obtained by chaining \eqref{eq:corollary_vN_intro} and \eqref{eq:boundedness_Lr_quenched_supremum_intro} increases as $r\to \infty$. Thus, one needs moments for Lebesgue-type norms of $v^N$ that are much higher than the spatial integrabilities.
\end{itemize}
Fortunately, the lack of sharpness does not influence the scaling limit as a loss of Sobolev index is already present due to the use of a compact embedding in \eqref{eq:convergence_space}.

\smallskip

Finally, we remark that, once a (sub-optimal) estimate is established under the conditions \eqref{eq:parameters_condition_Linfty_bounds}, analogous bounds can be deduced for other ranges by interpolation with the energy inequality (see Figure \ref{fig:scheme} and Section \ref{s:uniform_Lp_estimates}).
In this way, one recovers $L^p(L^q)$-estimates with parameters as in the second condition of \eqref{eq:Besov_regularity_data_intro}.
 
\subsection{Large-scale regularity, blow-up and cumulative frequency} 
\label{ss:control_energy_intro}
This subsection outlines the core ideas behind the uniform-in-$N$ energy bounds at the microscopic scale \eqref{eq:energy_microscopic_intro}. We revisit, in this context, the arguments of Avellaneda and Lin \cite{AL87_compactness} used to prove optimal regularity in the homogenization of elliptic PDEs, which are rooted in the study of minimal surfaces \cite{G83_annals}. The reader is referred to \cite[Subsection 4.1]{S18_book} for a streamlined exposition. 
A discussion of the breakdown of qualitative methods and difficulties related to the pressures in \eqref{eq:turbulent_Stokes_microscopic_intro} is postponed to Subsections \ref{ss:caccioppoli_intro} and \ref{ss:role_pressure_intro}, respectively.

\subsubsection*{Homogenization and large-scale regularity}
The core intuition driving this energy control at the cutoff scale $N^{-1}Q$ is as follows.
As $N\to \infty$, the averaged norm $\underline{L}^2(N^{-1}Q)$ shrinks to a pointwise evaluation. Simultaneously, following the discussion in Subsection \ref{ss:scaling_intro}, the solution $v^N$ \emph{homogenizes} to the solution to
\begin{align}
\label{eq:homogenized_stokes_intro}
\partial_t \vh=-\nabla \ph+ \big(1+\tfrac{3}{5} \mu\big) \Delta \vh+  \nabla \cdot \big(F-\tfrac{2}{5} \mu\, G^\top\big)   \quad \text{ on }Q.
\end{align}
From parabolic regularity, and under suitable assumptions on $\ph$ (see Subsection \ref{ss:role_pressure_intro}), $\vh$ is \emph{smooth} on the concentric parabolic cylinder with half the scale $(1/2)Q$. Consequently, pointwise evaluation at the center $(t_0,x_0)$ is well defined.

Therefore, one might hope for a delicate compensation between the emerging smoothness of $v^N$ and the singular nature of the norm $\underline{L}^2(N^{-1}Q)$. This is obtained below by following the ideas of \cite{AL87_compactness}.

\smallskip

From this perspective, achieving uniform control of the energy \eqref{eq:energy_microscopic_intro} down to the microscopic scale naturally translates into a \emph{large-scale regularity} result for the highly oscillatory system \eqref{eq:turbulent_Stokes_microscopic_intro}, in the sense that one can ``access'' the energy of the solution $\big(\E\|v^N\|_{\underline{L}^2(\varrho Q)}^r\big)^{1/r}$ uniformly in $N$ only at a (large) side length $\varrho\gtrsim N^{-1}$.

\subsubsection*{One-step improvement, iteration and blow-up}
Here, we describe how to control the energy at the microscopic scale. 
Recall that $F$ and $G$ are as in \eqref{eq:forcing_assumption_intro} with parameters satisfying \eqref{eq:parameters_condition_Linfty_bounds} and $r\gg 1$. 
Letting  $Q=I\times B$, we decompose \eqref{eq:energy_microscopic_intro} as
\begin{align}
\label{eq:average_decomposition_oscillation_intro}
 &\big(\E\|v^N\|_{\underline{L}^2(N^{-1}Q)}^r\big)^{1/r}\\
 \nonumber
 &\leq 
 \underbrace{\Big(\E\Big\|v^N-\fint_{N^{-1}B} v^N\Big\|_{\underline{L}^2(N^{-1}Q)}^r \Big)^{1/r}}_{\text{Oscillations}}
+ \underbrace{\Big(\E\Big\|\fint_{N^{-1}B} v^N\Big\|_{\underline{L}^2(N^{-1}I)}^r \Big)^{1/r}}_{\text{Average}}.
\end{align}  
The idea behind the above decomposition is as follows. 
Although the condition $\frac{2}{p}+\frac{3}{q}<1$ is necessary for the $L^\infty$-bounds  \eqref{eq:supremum_bound_quenched_intro}, it (typically) also implies H\"older continuity of the solutions (equivalent to decay of averaged $L^2$-oscillations from the local mean by Campanato's lemma). In particular, the second term on the right-hand side of \eqref{eq:average_decomposition_oscillation_intro} can be bounded by the oscillation and a telescoping argument. The latter is non-immediate and requires the stochastic Caccioppoli inequality discussed below (see Step 1 in the proof of Theorem \ref{t:control_mesoscopic_energy} in Subsection \ref{ss:control_mesoscopic_energy_proof}).
However, in the following, we only comment on the bounds of the oscillations. 

We also mention that in the decomposition \eqref{eq:average_decomposition_oscillation_intro}, we used spatial averages rather than space-time averages (as is more common in parabolic PDEs), not only for adaptedness issues, but fundamentally due to the lack of temporal regularity of the homogenized pressure $\ph$ in \eqref{eq:homogenized_stokes_intro}, which limits the regularity of $\vh$ (see Lemma \ref{l:local_smoothing_homogenized}).

\smallskip

A direct bound on the oscillation at the microscopic scale (i.e., the first term on the right-hand side of  \eqref{eq:average_decomposition_oscillation_intro}) does not seem possible. Following Avellaneda and Lin \cite{AL87_compactness}, the argument relies on two main pillars. The first is the \emph{one-step improvement}:

\begin{lemma}[One-step improvement -- Informal version of Lemma \ref{l:one_step_improvement}]
There exists $\g_0>0$, a spatial scale $\sm\in (0,1/4]$ and a frequency $L_0\geq 1$ such that for every $N\geq L_0$, it holds that
\begin{align}
\label{eq:one_step_improvement_intro}
\Big(\E\Big\|v^N-\fint_{\sm B} v^N\Big\|_{\underline{L}^2(\sm Q)}^r \Big)^{1/r}
&
\leq \eta^{\g_0}\Big(\E\Big\|v^N-\fint_{B} v^N\Big\|_{\underline{L}^2( Q)}^r \Big)^{1/r}\\ 
\nonumber
&+ \mathscr{C}_Q[\pi^N,(\wt{\pi}^N_{k,\alpha})_{k,\alpha},F,G],
\end{align}
where the last term collects the contributions of pressures and forcing at the scale $Q$ (see Subsection \ref{ss:role_pressure_intro} for details on the former).
\end{lemma}

For $\sm\in (0,1)$ and $\g_0\in (0,1-\frac{2}{p}-\frac{3}{q})$, the above follows from the convergence $v^N\to \vh$ in the (relatively weak) norm $L^r(\O;L^2((1/4)Q))$ and the fact that one expects
\begin{equation}
\label{eq:regularity_homogenized_intro}
\vh\in L^r(\O;L^p((1/4)I;C^{1-3/q}((1/4)B;\R^3)))
\end{equation} 
via local regularity of the \emph{constant} coefficient PDE \eqref{eq:homogenized_stokes_intro} due to the forcings $F$ and $G$ (see Lemma \ref{l:local_smoothing_homogenized}).
As \eqref{eq:regularity_homogenized_intro} shows, the assertion \eqref{eq:one_step_improvement_intro} is highly suboptimal when $N\gg 1$; in this regime, $v^N\approx \vh$, and the homogenized solution possesses significantly higher regularity than the bound implies. In other words, the estimate \eqref{eq:one_step_improvement_intro} yields non-trivial information only if $N\approx L_0\eqsim 1$. 

\smallskip

These observations naturally lead to the second pillar, namely the \emph{iterative step}:

\begin{lemma}[Iteration via blow-up -- Informal version of Lemma \ref{l:control_mesoscopic}]
Let $L_0$ and $\sm$ be as in the one-step improvement.
Let $m\in \N_{\geq 1}$ be such that 
$$
N^{-1}\eqsim \sm^{m-1} L_0^{-1} \qquad \text{(more precisely,\  $\sm^{m}L_0^{-1}<N^{-1}\leq \sm^{m-1}L_0^{-1}$)}.
$$
Then for all $\ell\in \{1,\dots,m\}$, it holds that:
\begin{align}
\label{eq:iteration_intro}
\Big(\E\Big\|v^N-\fint_{\sm^{\ell} B} v^N\Big\|_{\underline{L}^2(\sm^{\ell} Q)}^r \Big)^{1/r}
&
\leq \eta^{\ell\g_0}\Big(\E\Big\|v^N-\fint_{B} v^N\Big\|_{\underline{L}^2( Q)}^r \Big)^{1/r}\\ 
\nonumber
&+ \mathscr{C}_Q[\pi^N,(\wt{\pi}^N_{k,\alpha})_{k,\alpha},F,G]
\end{align}
where the last term collects the contributions of pressures and forcing at the scale $Q$.
\end{lemma}

The key advantage is that the oscillation term on the right-hand side of \eqref{eq:iteration_intro} for $\ell=m$ is exactly at the microscopic scale $ N^{-1}\eqsim \sm^{m} $. This, together with \eqref{eq:average_decomposition_oscillation_intro}, gives the control of the energy at the microscopic scale \eqref{eq:energy_microscopic_intro}.

\smallskip

The proof of \eqref{eq:iteration_intro} is based on a standard blow-up argument and the one-step improvement. Indeed, by induction, assume that \eqref{eq:iteration_intro} holds for some $\ell\in \{1,\dots,m-1\}$. Thus, if $v^N$ solves \eqref{eq:turbulent_Stokes_microscopic_intro} on $\sm^\ell Q$, then
\begin{equation}
\label{eq:iteration_intro_blow_up}
v^N_{\ell}(t,x)\stackrel{{\rm def}}{=} v^N (\sm^{2\ell} t,\sm^\ell x) \ \text{ for } \ (t,x)\in Q,
\end{equation}
is a solution to the Stokes problem \eqref{eq:turbulent_Stokes_microscopic_intro} on $Q$ (with appropriate rescaled coefficients and forcing) to which we can apply the one-step improvement \eqref{eq:one_step_improvement_intro}. 
The latter and
\begin{equation}
\label{eq:scaling_invariance_smell}
\Big\|v^N_\ell-\fint_{\sm B} v^N_\ell\Big\|_{\underline{L}^2(\sm Q)}
=\Big\|v^N-\fint_{\sm^{1+\ell} B} v^N\Big\|_{\underline{L}^2(\sm^{1+\ell} Q)},
\end{equation}
immediately leads to \eqref{eq:iteration_intro} for $\ell+1$ by the induction assumption.

\subsubsection*{Rescaled coefficients and cumulative oscillation frequency}
The blow-up argument \eqref{eq:iteration_intro_blow_up} behind the iterative claim \eqref{eq:iteration_intro} reveals that in the one-step improvement, one needs to consider the oscillating stochastic Stokes system \eqref{eq:turbulent_Stokes_microscopic_intro} with \emph{rescaled coefficients}:
\begin{equation}
\label{eq:turbulent_Stokes_microscopic_rescaled_intro}
\left\{
\begin{aligned}
\partial_t v^{N,\delta} &=-\nabla \pi^{N,\delta}+ (1+\mu)\Delta v^{N,\delta}+ \nabla \cdot F - \frac{3\mu}{2}\sum_{k,\alpha} \theta^N_k \nabla \cdot (\nabla \wt{\pi}_{k,\alpha}^{N,\delta}\otimes \sigma_{-k,\alpha}^\delta)\\ 
&  +\sqrt{\frac{3\mu}{2}} \sum_{k,\alpha} [-\nabla \wt{\pi}^{N,\delta}_{k,\alpha} +\theta^{N}_{k}(\sigma^{\delta}_{k,\alpha}\cdot\nabla) v^{N,\delta}+\theta^{N}_{k}\sigma_{k,\alpha}^\delta\cdot G]\,\dot{W}^{k,\alpha}_t,\\
\nabla \cdot v^{N,\delta}&=0, \qquad v^{N,\delta}(0)=0,
\end{aligned}
\right.
\end{equation}
where $\sigma_{k,\alpha}^\delta(x)=\sigma_{k,\alpha}(x_0+\delta (x-x_0)) =a_{k,\alpha}e^{2\pi \i x_0\cdot k} e^{2\pi \i \delta (x-x_0)\cdot k}$ for $x\in \R^d$ and $\delta\in (0,1]$ are the rescaled coefficients around the spatial center of $Q$.
Indeed, the process $v^N_{\ell}(t,x)$ in \eqref{eq:iteration_intro_blow_up} solves locally \eqref{eq:turbulent_Stokes_microscopic_rescaled_intro} with $\delta=\sm^\ell$. 
As the blow-up argument leading to \eqref{eq:iteration_intro} shows, the idea is to ``slow down'' the homogenization procedure $N\to \infty$ by spatially zooming out by a factor $\delta$, so that the rescaled process $v^N_\ell$ in \eqref{eq:iteration_intro_blow_up} sees an amount of oscillation $\eqsim L_0\eqsim 1$ on the (large) scale $Q$.

\smallskip

Thus, to obtain a one-step improvement which can be iterated, the homogenization of \eqref{eq:turbulent_Stokes_microscopic_rescaled_intro} must depend jointly on $\delta$ and $N$, and only on the truly present oscillation in \eqref{eq:turbulent_Stokes_microscopic_rescaled_intro}. 
Intuitively, since $\supp\theta^N\sim N$ and $\sigma_{k,\alpha}^\delta$ has frequency $\delta |k|\sim \delta N$, one infers that the oscillations of the rescaled system \eqref{eq:turbulent_Stokes_microscopic_rescaled_intro} are of the order: 
$$
L=\delta N. \qquad \text{(Cumulative oscillation frequency)}
$$ 
From a mathematical point of view, the above scale arises naturally, see Lemma \ref{l:Bessel_property} and Subsection \ref{ss:caccioppoli_intro} below. 

\smallskip

Before diving into the homogenization of \eqref{eq:turbulent_Stokes_microscopic_rescaled_intro}, we turn to the one-step iteration in the above setting. In contrast to \cite{AL87_compactness}, where \emph{strong} convergence in $L^2(Q)$ is proven by the Caccioppoli inequality (see Subsection \ref{ss:caccioppoli_intro} below), strong convergence in $L^r(\O;L^2(Q))$ via compactness arguments leads to a loss of integrability in the probability space. This loss can become arbitrarily high along the iteration \eqref{eq:iteration_intro}. We emphasize that spatial compactness can nonetheless be established via a stochastic version of the Caccioppoli inequality (Subsection \ref{ss:caccioppoli_intro}).

To exploit the expected regularity when increasing the cumulative oscillation frequency in \eqref{eq:turbulent_Stokes_microscopic_rescaled_intro}, we must prove a rate of convergence in $L^r(\O;L^2(Q))$ between $v^{N,\delta}$ and $\vh^{N,\delta}$ with a rate depending on $N,\delta$ only via the cumulative frequency $L=\delta N$, and where $\vh^{N,\delta}$  solves (in the Lions-Magenes transposition sense \cite{LM1_book,LM2_book}; see Subsection \ref{ss:local_well_posedness_homogenized_SPDE}):
\begin{equation}
\label{eq:homogenized_PDE_linear_intro_deltaN}
\left\{
\begin{aligned}
\partial_t \vh^{N,\delta}&=-\nabla \ph^{N,\delta}+ \big(1+\tfrac{3}{5} \mu\big) \Delta \vh^{N,\delta}+  \nabla \cdot \big(F-\tfrac{2}{5} \mu\, G^\top\big)  &  \text{ on }&(1/2)Q,\\
 \nabla \cdot \vh^{N,\delta}&=0 & \text{ on }&(1/2)Q, \\
\vh^{N,\delta}&= v^{N,\delta} -\fint_{(1/2)B} v^{N,\delta}& \text{ on }&\partial_{\pb}(1/2) Q.
\end{aligned}
\right.
\end{equation}  
In the above, $\partial_{\pb}$ denotes the parabolic boundary of the cylinder $(1/2)Q$.
This is proven in Section \ref{s:quantitative_loc_scaling} and outlined in Subsection \ref{ss:caccioppoli_intro}. 

\smallskip

The one-step improvement of \eqref{eq:one_step_improvement_intro} can now be established by writing:
\begin{equation}
\label{eq:error_splitting_vhom_intro}
v^{N,\delta}= \underbrace{\Big(\vh^{N,\delta} +\fint_{(1/2)B} v^{N,\delta}\Big)}_{\text{Locally smooth}}
+  \underbrace{\Big(v^{N,\delta}- \vh^{N,\delta} -\fint_{(1/2)B} v^{N,\delta}\Big),}_{\text{Error}}
\end{equation}
and proceeding as follows (see Lemma \ref{l:one_step_improvement}):
\begin{itemize}
\item First, the parameter $\eta>0$ is fixed by exploiting the local smoothness of $\vh^{N,\delta}$, as guaranteed by \eqref{eq:regularity_homogenized_intro} (see Lemma \ref{l:local_smoothing_homogenized});
\smallskip
\item Second, the parameter $L_0>0$ is chosen via quantitative bounds at the macroscopic scale $Q$, which depend exclusively on the cumulative frequency $L=\delta N$:
\begin{align}
\label{eq:choice_L0_intro}
&\big(\E\|w^{N,\delta}\|_{\underline{L}^2(\sm Q)}^r \big)^{1/r}
 \leq C_\sm
\big(\E\|w^{N,\delta}\|_{\underline{L}^2((1/2)Q)}^r \big)^{1/r}\\
\nonumber
&\quad\leq \frac{C_\sm}{(\delta N)^\g} \Big(\E\Big\|v^N-\fint_{B} v^N\Big\|_{\underline{L}^2( Q)}^r \Big)^{1/r}
+ \mathscr{C}_Q[\pi^N,(\wt{\pi}^N_{k,\alpha})_{k,\alpha},F,G], 
\end{align}
where $
w^{N,\delta}
= v^{N,\delta}- \vh^{N,\delta} -\fint_{(1/2)B} v^{N,\delta}
$
denotes the error term in \eqref{eq:error_splitting_vhom_intro}, and $\mathscr{C}_Q$ collects the contributions of pressures and forcing at the scale $Q$.
\end{itemize}
In the following section, we describe quantitative error estimates in the homogenization of $v^{N,\delta}$ towards $\vh^{N,\delta}$ and the emergence of the cumulative frequency $L=\delta N$.

\subsection{Caccioppoli inequalities and quantitative localized scaling limits} 
\label{ss:caccioppoli_intro}
Closing the one-step improvement \eqref{eq:one_step_improvement_intro} relies on two fundamental analytical ingredients: stochastic Caccioppoli inequalities and quantitative local scaling limits.

\subsubsection*{Stochastic Caccioppoli inequalities}
Broadly speaking, stochastic Caccioppoli estimates are ``reverse Poincar\'e'' inequalities with nested spatial support, and they hold for elliptic and parabolic problems \emph{without any} regularity assumptions beyond boundedness and measurability. For these reasons, they are an indispensable tool in homogenization theory \cite{S18_book}.
For the problem \eqref{eq:turbulent_Stokes_microscopic_rescaled_intro}, the stochastic Caccioppoli inequality takes the form (see Section \ref{s:caccioppoli}): 
\begin{align}
\label{eq:caccioppoli_intro}
\big(\E\|\nabla v^{N,\delta}\|_{\underline{L}^2((1/2) Q)}^r \big)^{1/r}
&\leq C \Big(\E\Big\|v^{N,\delta}-\fint_{B} v^{N,\delta}\Big\|_{\underline{L}^2( Q)}^r \Big)^{1/r} \\
\nonumber
&+ \mathscr{C}_Q[\pi^{N,\delta},(\wt{\pi}^{N,\delta}_{k,\alpha})_{k,\alpha},F,G],
\end{align}
where $\mathscr{C}_Q[\pi^{N,\delta},(\wt{\pi}^{N,\delta}_{k,\alpha})_{k,\alpha},F,G]$ collects the pressure and forcing contributions.
The key point in the above estimate is that the constant $C$ does not depend on the oscillation parameters $N$ and $\delta$. This is essential to bound the error \eqref{eq:choice_L0_intro}. As shown in \eqref{eq:def_error_det_intro}-\eqref{eq:def_error_sto_intro} below, the error explicitly depends on the macroscopic gradient $\nabla v^{N,\delta}$ on $(1/2)Q$. At this large scale, the system remains highly oscillatory, precluding any uniform smoothness (see Figure \ref{fig:microscopic_scale}). Conversely, the one-step improvement \eqref{eq:one_step_improvement_intro} requires an estimate formulated only in terms of local oscillations. 
The Caccioppoli inequality \eqref{eq:caccioppoli_intro} thus serves as a fundamental bridge: it bounds the gradients required by the error estimates entirely in terms of the $L^2$-oscillations demanded by the one-step improvement.

The proof of \eqref{eq:caccioppoli_intro} relies on local energy estimates in the spirit of Scheffer \cite{Scheffer_partial_regularity} and Caffarelli, Kohn, and Nirenberg \cite{CKN82}. Notably, closing these estimates appears intrinsically tied to the Stratonovich formulation of the noise, where essential cancellations occur (see Remark \ref{r:necessity_strotonovich}).

\subsubsection*{Quantitative localized scaling limits in extrapolated spaces}
Recall that $Q=I\times B$ denotes a parabolic cylinder.
Formally, the error $w^{N,\delta}= v^{N,\delta}- \vh^{N,\delta} -\fint_{(1/2)B} v^{N,\delta}$ is a solution to the following SPDE (see Subsection \ref{ss:local_well_posedness_homogenized_SPDE}):
\begin{equation}
\label{eq:turbulent_Stokes_scaling_quantitative_vdifference_formal_intro}
\left\{
\begin{aligned}
\partial_t w^{N,\delta} &=-\nabla R^{N,\delta}+ \Big(1+\mu \frac{3}{5}\Big) \Delta w^{N,\delta}+ \Errd\\ 
&\qquad\ \  + \sqrt{\frac{3\mu}{2}}\sum_{k,\alpha} [-\nabla \wt{R}_{k,\alpha}^{N,\delta}+\Errs_{k,\alpha}]\,\dot{W}^{k,\alpha}_t&\text{ on }&(1/2)Q,\\
 \nabla \cdot w^{N,\delta}&=0&\text{ on }&(1/2)Q,\\
 w^{N,\delta} &=0 &\text{ on }\partial_{\pb }&(1/2)Q,
 \end{aligned}
\right.
\end{equation}
where 
\begin{align}
\label{eq:def_error_det_intro}
\Errd&
=-\frac{2\mu}{5}(\Delta v^{N,\delta}+\nabla \cdot G^\top) + \frac{3\mu}{2} \sum_{k,\alpha} \theta_k^{N}\nabla\cdot  (\nabla \wt{\pi}^{N,\delta}_{k,\alpha}\otimes \sigma^\delta_{-k,\alpha}), \\
\label{eq:def_error_sto_intro}
\Errs_{k,\alpha}
&=
-\theta^{N}_{k}\big[(\sigma^\delta_{k,\alpha}\cdot\nabla) v^{N,\delta}+\sigma^\delta_{k,\alpha}\cdot G\big].
\end{align}
Inspired by the periodic setting \cite{FGL21,G20}, we exploit the oscillations of $(\sigma^\delta_{k,\alpha})_{k,\alpha}$ via a \emph{Bessel-type inequality} (Lemma \ref{l:Bessel_property}): For any ball $B \subset \R^3$ and $f\in L^2(B;\R^3)$,
\begin{equation*}
\sum_{k,\alpha} \Big|\int_{B} \sigma^\delta_{k,\alpha}\cdot f \Big|^2
\leq \frac{1}{\delta^3} \|f\|_{L^2(B)}^2.
\end{equation*}
Note that, for $\theta^N$ as in \eqref{eq:choice_thetaN_intro}, one has $\|\theta^N\|_{\ell^\infty(\Z^3_0)}\eqsim N^{-3/2}$ and thus the above yields
\begin{equation}
\label{eq:Parseval_combined_with_decay_intro}
\sum_{k,\alpha}(\theta_k^N)^2 \Big|\int_{(1/2)B} 
\varphi\cdot[  (\sigma^\delta_{k,\alpha}\cdot\nabla) v^{N,\delta} ] \Big|^2
\leq \frac{C}{(\delta N)^3 }  \|\nabla v^{N,\delta}\|_{L^2((1/2)B)}^2\|\varphi\|_{L^\infty((1/2)B)}^2.
\end{equation}
Crucially, the above decay rate depends only on the cumulative frequency $L=\delta N$.

\smallskip

Building on \eqref{eq:Parseval_combined_with_decay_intro}, we derive a \emph{local mixing estimate} on $(1/2)Q$ for the noise using the \emph{extrapolated spaces} $(\X_{-\sigma}^{\Dir}(B))_{\sigma>0}$ associated with the Dirichlet Laplacian on $B$. These spaces, introduced in a general framework by Amann \cite[Chapter 5]{Am} and recalled in Subsection \ref{ss:Helmholtz_domains}, intuitively act as the counterparts to negative Sobolev spaces on the torus. Crucially, they provide the right mathematical setting in which extensions of the Dirichlet Laplacian remain well-behaved (e.g., continue to generate analytic semigroups).
Letting $(\lambda_j)_{j}$ and $(\varphi_j)_{j}$ denote the eigenvalues and eigenfunctions of the Dirichlet Laplacian on $(1/2)B\subset \R^3$, the estimate \eqref{eq:Parseval_combined_with_decay_intro} implies the mixing-type estimate:
\begin{align*}
\sum_{k,\alpha}(\theta^N_k)^2 \|(\sigma^{\delta}_{k,\alpha}\cdot \nabla )v^{N,\delta}\|_{\X_{-\g}^{\Dir}((1/2)B)}^2
&\eqsim \sum_{j,k,\alpha} \lambda_j^{-2\g} (\theta^N_k)^2
 \Big|\int_{(1/2)B} 
\varphi_j\cdot[  (\sigma^\delta_{k,\alpha}\cdot\nabla) v^{N,\delta} ] \Big|^2\\
&\leq \frac{C}{(\delta N)^3 }  \|\nabla v^{N,\delta}\|_{L^2((1/2)B)}^2,
\end{align*}
where $\g\gg 1$ is chosen so that $\sum_{j} \lambda_j^{-2\g}\|\varphi_j\|_{L^\infty((1/2)B)}^2<\infty$ (which is possible by Weyl's law, see Proposition \ref{prop:error_estimates_operators}).

Together, these extrapolated space techniques and the Caccioppoli inequality \eqref{eq:caccioppoli_intro} provide a robust framework required to bound the stochastic error \eqref{eq:def_error_sto_intro}.

\smallskip

In Subsections \ref{ss:almost_self}-\ref{ss:caccioppoli_intro}, we discussed the strategy to derive a bound on the velocity field $v^N$ solving \eqref{eq:turbulent_Stokes_microscopic_intro}. In the following subsection, we analyze the role of the deterministic and stochastic pressures in the previously discussed argument.

\subsection{Scaling invariance of the pressures and local approximations}
\label{ss:role_pressure_intro}
In this subsection, we address the obstructions introduced by the deterministic and stochastic pressures $\pi^{N,\delta}$ and $\wt{\pi}_{k,\alpha}^{N,\delta}$ within the previously outlined program (see Figure \ref{fig:scheme}).
First, we recall a well-known pathology inherently tied to \emph{local} solutions of incompressible fluid equations, namely the poor temporal regularity of the pressure, as illustrated by Serrin's classical counterexample, cf.\ \cite[Subsection 13.1]{LePi}.
 This obstruction is naturally present when studying the local regularity of the homogenized Stokes system \eqref{eq:homogenized_PDE_linear_intro_deltaN}.
Consequently, we cannot expect any further time regularity in \eqref{eq:regularity_homogenized_intro} for the velocity $\vh^{N,\delta}$.
This deficit directly dictates our strict use of \emph{spatial averages}---rather than standard space-time averages---in the oscillation decomposition \eqref{eq:average_decomposition_oscillation_intro}.

\smallskip

However, even more severe complications arise from the contribution of the global pressures to the iterative scheme \eqref{eq:one_step_improvement_intro}-\eqref{eq:iteration_intro} and the deterministic error \eqref{eq:def_error_det_intro}.
We overcome these difficulties by employing local approximations in the spirit of \cite{CKN82}.
Before detailing this, however, we focus on the scaling invariance, which plays a central role in controlling the energy at the microscopic scale (Subsection \ref{ss:control_energy_intro}).

\subsubsection*{Scaling invariance of the pressures}
Due to the incompressibility $\nabla \cdot v^{N}=0$ in \eqref{eq:turbulent_Stokes_microscopic_intro}, the pressures satisfy the following elliptic problems:
\begin{align}
\label{eq:pressure_PDE_intro1}
\Delta \pi^{N}&= \nabla^2 :F - \frac{3\mu}{2} \,\nabla \cdot \sum_{k,\alpha}\theta_k^N \nabla \cdot(\nabla \wt{\pi}_{k,\alpha}^{N}\otimes \sigma_{-k,\alpha} ) ,\\
\label{eq:pressure_PDE_intro2}
\Delta \wt{\pi}_{k,\alpha}^{N}
&= \theta_k^N\nabla \cdot \big[(\sigma_{k,\alpha}\cdot \nabla)v^{N}+\sigma_{k,\alpha}\cdot G \big].
\end{align}
Observe that the parabolic scaling $v^{N}_\ell (t,x)= v^N (\sm^{2\ell } t,\sm^{\ell} x)$ appearing in the iteration argument \eqref{eq:iteration_intro_blow_up} dictates the corresponding dilation of the pressures:
\begin{align*}
\pi^{N}_\ell (t,x)=\sm^{\ell} \pi^{N}( \sm^{2\ell} t,\sm^{\ell} x)\qquad\text{ and }\qquad 
\wt{\pi}^{N}_{k,\alpha,\ell}(t,x)=\wt{\pi}_{k,\alpha}^{N}( \sm^{2\ell} t,\sm^{\ell} x).
\end{align*}
Moreover, parallel to \eqref{eq:scaling_invariance_smell}, we have 
\begin{align}
\label{eq:scaling_pressure_intro1}
\|\pi^{N}_{\ell}\|_{\underline{L}^2( \sm  I;\underline{H}^{-1}(\sm B))}
&\eqsim \|\pi^{N}\|_{\underline{L}^2(\sm^{1+\ell} I;\underline{H}^{-1}(\sm^{1+\ell} B))}\\
\label{eq:scaling_pressure_intro2}
\|\wt{\pi}_{k,\alpha,\ell}^{N}\|_{\underline{L}^2(Q)}
&\eqsim\|\wt{\pi}_{k,\alpha}^{N}\|_{\underline{L}^2(\sm^{1+\ell}Q)},
\end{align}
where $\underline{H}^{-1}(B)$ denotes the negative Sobolev space endowed with the averaged norm (see \eqref{eq:negative_sob_space_averaged} below).
A similar rescaling also holds for the associated oscillations.

\smallskip

The above scale-invariance precisely determines how the pressures enter both the stochastic Caccioppoli inequality and the large-scale regularity estimates \eqref{eq:one_step_improvement_intro}-\eqref{eq:iteration_intro}. Consequently, controlling the pressure oscillations measured in the norms of \eqref{eq:scaling_pressure_intro1}-\eqref{eq:scaling_pressure_intro2} becomes an intrinsic part of the iterative scheme (see Section \ref{s:energy_control_microscopic}).
Finally, we remark that in the one-step improvement, due to the lack of temporal smoothness, we measure the stochastic pressure in the space $\underline{L}^{p_0}(I;\underline{L}^{2}(B))$ for some $p_0>2$ to ensure the required contraction at the scale $\sm$ analogous to \eqref{eq:one_step_improvement_intro}. In contrast, for the deterministic pressure, the decay of the $\underline{H}^{-1}(\sm B)$-norm as $\sm\to 0$ compensates for the lack of time regularity, thus avoiding additional integrability conditions.

\subsubsection*{Local approximation and deterministic error in the local scaling limit}
Finally, we outline the strategy to bound the deterministic error \eqref{eq:def_error_det_intro} for local solutions to the rescaled system \eqref{eq:turbulent_Stokes_microscopic_rescaled_intro}. Because the global pressures $\pi^{N,\delta}$ and $\wt{\pi}^{N,\delta}_{k,\alpha}$ of the rescaled problem \eqref{eq:turbulent_Stokes_microscopic_rescaled_intro} are inherently non-local, their direct use is precluded. However, the elliptic equations governing these pressures naturally suggest the use of \emph{local approximations} to capture their interior smoothness \cite{CKN82}. For the sake of exposition, we focus here on the stochastic pressure; the deterministic counterpart relies on a similar conceptual localization but is analytically more involved.

A local approximation of the stochastic pressure on $(1/2)Q$ is given by
$$
\wt{\Pi}_{k,\alpha}^{N,\delta}
=\theta^N_k \qq_{\T^3_{B}} \Big[\chi\big((\sigma^{\delta}_{k,\alpha}\cdot \nabla)v^{N,\delta} + \sigma^{\delta}_{k,\alpha} \cdot G\big)\Big],
$$
where $\chi\in C^{\infty}_{{\rm c}}(B)$ is a cutoff function satisfying $\chi|_{(1/2)B}=1$, and $\qq_{\T^3_B}$ is the operator associated with the Helmholtz decomposition on the localized torus $\T^3_B = x_0 + [-1/2, 1/2)^3$ (see Subsection \ref{ss:Helmh_div_free}). In particular, 
$$
\Delta \wt{\Pi}_{k,\alpha}^{N,\delta}= \theta_k^N\nabla \cdot \big[\chi(\sigma_{k,\alpha}^\delta\cdot \nabla)v^{N,\delta}+\chi\,\sigma_{k,\alpha}^{\delta}\cdot G \big],
$$
and from the choice of the cutoff function, it follows that $\wt{\pi}_{k,\alpha}^{N,\delta}-\wt{\Pi}_{k,\alpha}^{N,\delta}$ is \emph{harmonic} on $(1/2)B$, and thus $C^\infty$ on the latter domain.
To handle the deterministic error, we decompose it as $\Errd = \Errdloc+ \Errdglo$, where
\begin{align}
\label{eq:error_loc_intro}
\Errdloc &=-\frac{2\mu}{5} (\Delta v^{N,\delta} +\nabla \cdot G^\top)+\frac{3\mu}{2}\sum_{k,\alpha} \theta_k^N\nabla\cdot  (\nabla \wt{\Pi}^{N,\delta}_{k,\alpha}\otimes \sigma^\delta_{-k,\alpha}),\\
\label{eq:error_glo_intro}
\Errdglo &=\frac{3\mu}{2} \sum_{k,\alpha} \theta_k^N\nabla\cdot  (\nabla[ \wt{\pi}^{N,\delta}_{k,\alpha}-\wt{\Pi}^{N,\delta}_{k,\alpha}]\otimes \sigma^\delta_{-k,\alpha}),
\end{align}
denote the \emph{local} and \emph{local-to-global} errors, respectively.

\smallskip

The mechanism driving the decay of the local-to-global error $\Errdglo$---established rigorously in Proposition \ref{prop:error_estimates_operators}---relies on the interior regularity of the difference $\wt{\pi}_{k,\alpha}^{N,\delta}-\wt{\Pi}_{k,\alpha}^{N,\delta}$. Indeed, as the cumulative frequency $L=\delta N \to \infty$, the highly oscillatory coefficients $\sigma^\delta_{-k,\alpha}$ converge weakly to zero, while $\nabla[\wt{\pi}^{N,\delta}_{k,\alpha}-\wt{\Pi}^{N,\delta}_{k,\alpha}]$ remains smooth and uniformly bounded in $N$ and $\delta$ at large scales.

\smallskip

The local error $\Errdloc$ poses a significantly more difficult challenge.
In the fully periodic setting, this term was treated by Flandoli and Luo \cite{FL19} (see also \cite{L23_enhanced}) relying on Fourier-analytic techniques. 
To accommodate our local setting, we instead develop a purely spatial approach. This framework provides an explicit explanation for a seemingly surprising phenomenon: the inherently non-local It\^o-Stratonovich corrector becomes \emph{local} in the scaling limit, a feature that is already visible in the structure of $\Errdloc$ in \eqref{eq:error_loc_intro}.

As in Subsection \ref{ss:caccioppoli_intro}, we employ extrapolated spaces to capture the error decay as $L=\delta N\to \infty$.
To this end, it suffices to study the bilinear pairing 
$$
(\nabla v^{N,\delta}+G,\Phi)\mapsto \langle \Errdloc, \Phi\rangle
$$ for a smooth test vector field $\Phi$. 
Through standard manipulations, one can isolate the covariance function $\Cov^{N,\delta}$ of the rescaled problem \eqref{eq:turbulent_Stokes_microscopic_rescaled_intro} (defined analogously to \eqref{eq:covariance_intro}) and express the pairing as a local term and a nonlocal term given by:
\begin{equation}
\label{eq:representation_Ito_Strat_intro}
\sum_{1\leq i,j,n,m\leq 3}
{\normalfont{\text{P.V.}}}
\int_{B\times 2B} \Gamma_{i,j}(x-y) \Cov^{N,\delta}_{m,n}(x-y) \chi(y)H_{m,j}(y)\partial_n \Phi_i \,\dd x\,\dd y,
\end{equation}
where $H=[\nabla v^{N,\delta}]^\top+G$, and $\Gamma_{i,j}=\mathcal{F}^{-1}_{\R^3}\big(\xi \mapsto \frac{\xi_i\xi_j}{|\xi|^2}-\frac{\delta_{i,j}}{d}\big)$ is a Calder\'on-Zygmund-type kernel, which roughly corresponds to the second derivatives of the Green's function on $\R^3$. Details can be found in Appendix \ref{app:cancellation_ito_stratonovich}, see \eqref{eq:definition_A_bilinear_appendix} and \eqref{eq:expression_via_covariance_ito_strat}.
Crucially, the representation involving \eqref{eq:representation_Ito_Strat_intro} isolates the role of the covariance operator associated with $\Cov^{N,\delta}$, which exhibits a quantitative decay---governed solely by the cumulative frequency $L=\delta N$---away from the diagonal $x=y$ (see Subsection \ref{ss:decay_covariance}).
As detailed in Appendix \ref{app:cancellation_ito_stratonovich}, this off-diagonal decay not only rigorously justifies the locality of the limiting It\^o-Stratonovich correction but also provides the quantitative rate of convergence necessary to close the one-step improvement scheme introduced in \eqref{eq:choice_L0_intro}.

\subsection{Notation}
\label{ss:notation}
Here, we collect the basic notation used in the manuscript.  
$C$ denotes a constant that might vary from line to line.
For given parameters $p_1,\dots,p_n$ and two quantities $x$ and $y$, we write $x\lesssim y$, if there exists a constant $C$ such that $x\le Cy$. We write $x\eqsim_{p_1,\dots,p_n} y$, whenever $x\lesssim_{p_1,\dots,p_n} y$ and $y\lesssim_{p_1,\dots,p_n}x$. 

The transposition of a matrix is denoted by $\top$. For two matrices $A=(A_{i,j})_{i,j=1}^d$ and $B=(B_{i,j})_{i,j=1}^d$, we set $A:B =\sum_{i,j=1}^d A_{i,j}B_{i,j}$. 
For two vectors $x,y$ we let $x\otimes y=(x_iy_j)_{i,j=1}^d$.
The symbol $\nabla^2 f$ denotes the Hessian of the function $f$ on a $d$-dimensional set, i.e., $\nabla^2 f=(\partial_{i,j} f)_{i,j=1}^d$. Therefore, $\nabla^2 : F = \sum_{i,j=1}^d \partial_{i,j} F_{i,j}$ for a matrix field $F=(F_{i,j})_{i,j=1}^d$.

Let $A\subseteq \R^d$ be a set. We denote by $\Borel(A)$ the associated Borel $\sigma$-algebra. 
For a Banach space $X$ and $q\in (1,\infty)$, $L^q(S;X)$ is the Bochner space of strongly measurable, $q$-integrable $X$-valued functions, see e.g., \cite[Section 1.2b]{Analysis1}. As usual $W^{1,q}(S;X)$ denotes the set of all $f\in L^q(S;X)$ such that $\nabla f\in L^q(S;X)$ endowed with the natural norm. 
We also employ vector-valued Bessel potential spaces 
\begin{equation}
\label{eq:Sobolev_spaces_by_restrictions}
H^{\sigma,q}(S;X)=\big\{f\in \D'(S;X)\,:\, \exists \, F\in H^{\sigma,q}(\R^d;X) \text{ such that }F|_{S}=f\big\},
\end{equation} 
where $\sigma\in \R$. The above space is endowed with the natural induced norm. For details on the space $H^{\sigma,q}(\R^d;X)$, the reader is referred to e.g., \cite[Chapter 14]{Analysis3}. As usual, we let $H^{\sigma}(S;X)=H^{\sigma,2}(S;X)$, and we use that well-known identity $H^{1,q}(S;X)=W^{1,q}(S;X)$ for $q\in (1,\infty)$ if $X$ is a Hilbert space.

If $X=\R^M$ for some $M\geq 1$, we often write $L^q(S)$  instead of $L^q(S;\R^M)$ if no confusion seems likely. 
Finally, $\one_{A_0}$ denotes the indicator function of $A_0\subseteq A$.

\subsubsection*{Geometric set-up and averaged function spaces.} $Q_{\varrho}(t_0,x_0)$ denotes the parabolic cylinder with side length $\varrho>0$ and center $(t_0,x_0)\in \R\times \R^d$ defined as
$$
Q_{\varrho}(t_0,x_0)
\stackrel{{\rm def}}{=}(t_0-\varrho^2,t_0]\times B_\varrho(x_0),
$$
where $I_\varrho(t_0)=(t_0-\varrho^2,t_0]$ and $B_\varrho(x_0)$ is the ball of radius $\varrho$ and center $x_0$.
Often, for notational convenience, we write $Q=Q_{\varrho}(t_0,x_0)$ and $Q=I\times B$ in the case $\varrho>0$ and $(t_0,x_0)\in \R_+\times \R^d$ are clear from the context.
As usual, we denote by $\partial_\pb Q$ the parabolic boundary of $Q$:
$$
\partial_\pb  Q_{\varrho}(t_0,x_0)
\stackrel{{\rm def}}{=} \{t_0-\varrho^2\}\times B_\varrho(x_0)\cup (t_0-\varrho^2,t_0]\times \partial B_\varrho(x_0).
$$
For $\lambda>0$ and a parabolic cylinder $Q=Q_{\varrho}(t_0,x_0)$, we let $\lambda Q$ be the \emph{rescaled} cylinder around the center $(t_0,x_0)$ of $Q$: 
$$
\lambda Q = Q_{\lambda\varrho}(t_0,x_0).
$$
Next, we define the associated averaged function spaces. For $Q=Q_\varrho(t_0,x_0)$, we let 
$$
\fint_{Q}f=\frac{1}{|Q|} \int_{Q}f= \frac{1}{|B_1|\varrho^{d+2}} \int_Qf,
$$ 
where we do not display the variable of integration and the Lebesgue measure $\dd x$ if no confusion seems likely.
A similar notation is employed for balls or intervals.
For a parabolic cylinder $Q$, we let 
\begin{equation*}
\|f\|_{\underline{L}^q(Q)}
=\Big(\fint_{Q}|f|^q \Big)^{1/q},
\end{equation*}
and similarly for balls or intervals. For a ball $B\subseteq \R^d$, for $f\in H^1_0(B)$, we let  
$
\|f\|_{\underline{H}^{1}_0(B)}
\stackrel{{\rm def}}{=}\|\nabla f\|_{\underline{L}^{2}(B)}.
$
As usual, we define $\underline{H}^{-1}(B)=(\underline{H}^{1}_0(B))^*$, and endow it with the norm:
\begin{equation}
\label{eq:negative_sob_space_averaged}
\|f\|_{\underline{H}^{-1}(B)}
=\sup\Big\{\fint_{B}f g\,:\, 
\|g\|_{\underline{H}^{1}_0(B)}\leq 1\Big\},
\end{equation}
where $\fint_{B}f g$ is understood in the distributional sense.

The following scaling properties will often be employed in the manuscript:
\begin{equation}
\label{eq:scaling_negative_sob_space_intro}
\|f(\lambda\cdot)\|_{\underline{H}^{1}_0(B)}
= \lambda
\|f\|_{\underline{H}^{1}_0(\lambda B)}
\quad \text{ and }\quad 
\|f(\lambda\cdot)\|_{\underline{H}^{-1}(B)}
= \lambda^{-1}
\|f\|_{\underline{H}^{-1}(\lambda B)}.
\end{equation}

\subsubsection*{Function spaces}
Standard Sobolev spaces are denoted by $W^{k,q}$ with $k\in \N_{>0}$ and $q\in [1,\infty]$.
Bessel-potential spaces are indicated by $H^{s,q}(\T^d)$ where $s\in \R$ and $q\in (1,\infty)$. It is well known that $H^{k,q}(\T^d)=W^{k,q}(\T^d)$ for $k\in \N_{>0}$.
We also use the usual short-hand notation $H^{s}(\T^d)=H^{s,2}(\T^d)$. Besov spaces are denoted by $B^{s}_{q,p}(\T^d)$ and can be defined as the real interpolation space $(H^{-k,q}(\T^d),H^{k,q}(\T^d))_{(s+k)/(2k),p}$ for $s\in \R$, $\N\ni k>|s|$ and $1<q,p<\infty$, see \cite{BeLo,Analysis1} for details on interpolation and see \cite[Section 3.5.4]{schmeisser1987topics} for details on function spaces over $\T^d$. 

For all $\vartheta_0,\vartheta_1>0$ and $t>0$, we let 
$$
C^{\vartheta_0,\vartheta_1}((0,t)\times \T^d)\stackrel{{\rm def}}{=}C^{\vartheta_0}(0,t;C(\T^d))\cap C([0,t];C^{\vartheta_1}(\T^d))
$$ 
and $C^{\vartheta_0,\vartheta_1}_{\loc}((0,t)\times \T^d)\stackrel{{\rm def}}{=}\cap_{0<\varepsilon<t/2}C^{\vartheta_0,\vartheta_1}((\varepsilon,t-\varepsilon)\times \T^d)$.
Finally, we let $\mathcal{A}(\T^d;\R^k)\stackrel{{\rm def}}{=}(\mathcal{A}(\T^d))^k$ for $\mathcal{A}\in \{L^q,H^{s,q},B^{s}_{q,p},C^{\vartheta_0,\vartheta_1}\}$ and $k\in \N$.

The above spaces can be defined on an arbitrary domain of $\R^d$ by restrictions of corresponding maps as in \eqref{eq:Sobolev_spaces_by_restrictions} endowed with the standard quotient norm.

\subsubsection*{Probabilistic set-up}
We denote by $(\O, \mathscr{A},(\mathscr{F}_t)_{t\geq 0}, \P)$ a filtered probability space carrying a sequence of independent standard Brownian motions which might change depending on the SPDE under consideration. Moreover, $\E[\cdot]=\int_{\O} \cdot \,\dd \P$ for the associated expected value. 
A process $\phi:[0,\infty)\times \O\to X$ is progressively measurable if $\phi|_{[0,t]\times \O}$ is $\Borel([0,t])\otimes \F_t$ measurable for all $t\geq 0$, where $\Borel$ is the Borel $\sigma$-algebra on $[0,t]$ and $X$ is a Banach space. Moreover, a stopping time $\tau$ is a measurable map $\tau:\O\to [0,\infty]$ such that $\{\tau\leq t\}\in \F_t$ for all $t\geq 0$. Finally, a stochastic process $\phi:[0,\tau)\times \O\to X$ is progressively measurable if $\one_{[0,\tau)\times \O}\,\phi$ is progressively measurable where $
[0,\tau)\times \O\stackrel{{\rm def}}{=}\{(t,\om)\in[0,\infty)\times \O\,:\,\,0\leq t<\tau(\om)\}$
and $\one_{[0,\tau)\times \O}$ (or simply $\one_{[0,\tau)}$) stands for the extension by zero outside $[0,\tau)\times \O$. The definition of the stochastic intervals $(0,\tau)\times \O$ and $[0,\tau]\times \O$ is similar.

\section{Preliminaries}
\label{s:preliminaries}

\subsection{The structure of the noise}
\label{ss:probabilistic}
While the main theorems are stated for the physical dimension $d=3$, all the results of this manuscript hold for general dimensions $d\geq 2$ with appropriate modifications, see Remark \ref{r:high_dimensions}. Therefore, we introduce the noise structure directly in this general setting.

\smallskip

Let $\Z_0^d=\Z^d\setminus\{0\}$. We say that the noise coefficient $\theta=(\theta_k)_{k\in \Z^d_0}\in \ell^2$ is \emph{normalized} if $\|\theta\|_{\ell^2(\Z^d_0)}=1$ and \emph{radially symmetric} if 
\begin{equation*}
 \theta_{j}=\theta_k \  \text{ for all $j,k\in\Z_0^d$ \  such that }|j|=|k|.
\end{equation*}
For some fixed $a>0$, a typical example of normalized radially symmetric coefficients is given by the elements of the following sequence:
\begin{equation}
\label{eq:choice_thetaN}
\theta^N_k =\frac{\Theta^N_k}{\|\Theta^N\|_{\ell^2}}\qquad \text{ where }\qquad \Theta_k^N =\frac{1}{|k|^{a}} \one_{\{N\leq |k|\leq 2N\}}
\end{equation}
which was already mentioned in Subsection \ref{ss:almost_self}.

\smallskip

Let $\Z_{+}^d$ and $\Z_-^d$ be a  partition of $\Z_0^d$ such that $-\Z_+^d=\Z_-^d$. For any $k\in \Z_+^d$, let $\{a_{k,1},\dots,a_{k,d-1}\}$ be a complete orthonormal basis of the hyperplane $k^{\bot}=\{k'\in \R^d\,:\, k\cdot k'=0\}$, 
and set $a_{k,\alpha}\stackrel{{\rm def}}{=}a_{-k,\alpha} $ for $k\in \Z^d_-$.
We set 
$$
\sigma_{k,\alpha}(x)=a_{k,\alpha}e^{2\pi \i x\cdot k} \quad x\in \T^d. 
$$ 
Clearly, the vector fields $\sigma_{k,\alpha}$ are smooth and divergence-free for all $k,\alpha$.

\smallskip

Next, we introduce the family of complex Brownian motions $(W^{k,\alpha})_{k,\alpha}$. 
Let $(B^{k,\alpha})_{k,\alpha}$ be a family of independent standard (real) Brownian motions on the above-mentioned filtered probability space $(\O,\A,(\F_t)_{t\geq 0},\P)$. Then, we set
\begin{equation}
\label{eq:complex_BM}
W^{k,\alpha}\stackrel{{\rm def}}{=}
\left\{
\begin{aligned}
&B^{k,\alpha}+ \i \,B^{-k,\alpha},& \qquad& k\in \Z^d_+,\\
&B^{-k,\alpha}- \i \, B^{k,\alpha}, &\qquad& k\in \Z^d_-.
\end{aligned}
\right.
\end{equation}
In particular, $\overline{W^{k,\alpha}}= W^{-k,\alpha}$ for all $k,\alpha$. 
The above conditions can be summarized in the following form
\begin{equation}
\label{eq:decorrelation_complex_BM}
[W^{k,\alpha},W^{h,\beta}]_t= 2t \delta_{k,-h}\delta_{\alpha,\beta} \ \text{ for all }h,k\in \Z^d_0, \ \alpha,\beta \in\{1,\dots,d-1\}.
\end{equation}
For later use, we note that, on a suitable extension of the probability space and filtration, one can construct a family of independent standard Brownian motions $(\wt{B}^{k,\alpha})_{k,\alpha}$ on $[-1/4,\infty)$ (i.e. starting from zero at $t=-1/4$), in such a way that their increments after time $t=0$ coincide with those of the original Brownian motions $(B^{k,\alpha})_{k,\alpha}$.
For instance, this can be obtained by taking a product with an additional probability space carrying another sequence of complex-valued Brownian motions. With a slight abuse of notation, we do not distinguish between the original Brownian motions and their extension above.

\subsection{Helmholtz projection and spaces of divergence-free vector fields} 
\label{ss:Helmh_div_free}
In this subsection, we introduce the Helmholtz projection $\p$. For an $\R^d$-valued distribution $F=(F^j)_{j=1}^{d}\in \D'(\T^d;\R^d)$ on $\T^d$, let $\widehat{F^j}(k) =\langle F^j, e_k \rangle $ be the $k$-th Fourier coefficient, where $j\in \{1,\dots,d\}$, $k=(k_j)_{j=1}^d\in \Z^d$ and $e_k(x)=e^{2\pi \i k\cdot x}$. 

For $F\in \D'(\T^d;\R^d)$, we let $\qq F$ be given by 
\begin{equation*}
(\widehat{ \qq F})(k)\stackrel{{\rm def}}{=}  -\frac{\i}{2\pi} \sum_{1\leq j\leq d} \frac{ k_j }{|k|^2} \widehat{F^j}(k), \qquad 
(\widehat{ \qq F}) (0)\stackrel{{\rm def}}{=}0,
\end{equation*}   
or in other words, $\qq F =\Delta^{-1}(\nabla \cdot F)$. 

The Helmholtz projection $\p$ and its complementary projection $\q$ are given by 
\begin{equation}
\label{eq:def_helmholtz_projection_revised}
\q F \stackrel{{\rm def}}{=} \nabla \qq F \qquad \text{ and }\qquad 
\p F\stackrel{{\rm def}}{=}F- \q F,
\end{equation}
respectively. From standard Fourier analysis, it follows that $\q$ and $\p$ restrict to bounded linear operators on $H^{s,q}(\T^d;\R^d)$ and $B^s_{q,p}(\T^d;\R^d)$ for $s\in \R$ and $q,p\in (1,\infty)$.

To connect the above construction with the corresponding one on domains, note that $\psi_F = \qq F$ solves the following elliptic problem:
$$
\Delta \psi_F = \nabla \cdot F \qquad \text{and} \qquad \langle \psi_F, 1 \rangle = 0.
$$
Let $B = B_\varrho(x_0) \subset \R^d$ be a ball of radius $\varrho \leq 1/2$ centered at $x_0 \in \R^d$. It is often convenient to identify the torus $\T^d$ with the periodic domain $\T^d_B = x_0 + [-1/2, 1/2)^d$, which naturally contains $B$. Under this identification, we denote the corresponding Helmholtz projection by $\p_{\T^d_B}$, and similarly for $\q_{\T^d_B}$ and $\qq_{\T^d_B}$.

\smallskip

Finally, for $s\in \R$ and $q\in (1,\infty)$, we can introduce function spaces of divergence-free vector fields with Bessel-potential regularity: 
\begin{align*}
\Hs^{s,q}(\T^d)\stackrel{{\rm def}}{=}\big\{ F\in H^{s,q}(\T^d;\R^d) \,:\, \nabla \cdot F=0\text{ in }\D'(\T^d)\big\}
\end{align*}
endowed with the induced norms. Similar definitions hold for Lebesgue $\Ls^q(\T^d)$ and Besov $\Bs^{s}_{q,p}(\T^d)$  spaces of divergence-free vector fields, where $p\in (1,\infty)$. 
Often we write $\Hs^{s,q}$ instead of $\Hs^{s,q}(\T^d)$ if no confusion seems likely.

\subsection{It\^o-reformulation}
\label{ss:ito_reformulation}
In this subsection, we formally recast the Stratonovich SPDE \eqref{eq:navier_stokes_intro} into an It\^o stochastic evolution equation. This formulation provides the setting for defining solutions to the stochastic 3D NSEs.

Assuming $\nabla \cdot u_0=0$, and applying the Helmholtz projection $\p$ to the first equation in \eqref{eq:navier_stokes_intro}, we formally obtain 
\begin{equation}
\label{eq:navier_stokes_helmholtz}
\left\{
\begin{aligned}
\partial_t u +\p [\nabla \cdot (u\otimes u)]&=\Delta u +\sqrt{\frac{3\mu}{2}}\sum_{k,\alpha}\theta_k\p\big[ (\sigma_{k,\alpha}\cdot\nabla) u\big]\circ \dot{W}^{k,\alpha}_t,\\
 u(0,\cdot)&=u_0.
 \end{aligned}
\right.
\end{equation} 
Here, we express the nonlinearity in its standard conservative form, which more naturally accommodates the weak setting.
Consequently, the divergence-free condition of $u_0$ is formally preserved along the flow of \eqref{eq:navier_stokes_helmholtz}.
The deterministic and stochastic pressures can be recovered from the velocity field via the identities
\begin{align}
\label{eq:definition_pressures_PwtP_local}
P =- \qq [\nabla \cdot (u\otimes u)] \qquad \text{ and }\qquad
\wt{P}_{k,\alpha} =\theta_k\qq [ (\sigma_{k,\alpha}\cdot\nabla) u].
\end{align}
As in \cite[Section 2.3]{FL19}, due to \eqref{eq:decorrelation_complex_BM} and the divergence-free property of $\sigma_{k,\alpha}$, at least formally, for all normalized radially symmetric $\theta\in \ell^2$, one has
\begin{align*}
\sqrt{\frac{3\mu}{2}}
\sum_{k,\alpha}\theta_k\, \p [ (\sigma_{k,\alpha}\cdot \nabla) u]\circ \dot{W}_t^{k,\alpha}
&=
\LL u+\sqrt{\frac{3\mu}{2}}
\sum_{k,\alpha}\theta_k \,\p [ (\sigma_{k,\alpha}\cdot \nabla) u]\, \dot{W}_t^{k,\alpha},
\end{align*}
where 
\begin{align*}
\LL u
&=
\frac{3\mu}{2}
\sum_{k,\alpha}\theta_k^2 \,\p \big[ \nabla \cdot ( \p[(\sigma_{k,\alpha}\cdot \nabla)u]\otimes \sigma_{-k,\alpha}\big)\big]=\mu \Delta u -\mathcal{P} u
\end{align*}
and
\begin{equation}
\label{eq:Ito_stratonovich_correction_pressure}
\mathcal{P} u \stackrel{{\rm def}}{=} 
\frac{3\mu}{2}\sum_{k,\alpha}\theta_k \,\p \big[ \nabla \cdot ( \nabla \wt{P}_{k,\alpha}\otimes \sigma_{-k,\alpha}\big)\big] 
\end{equation}
where $\wt{P}_{k,\alpha}$ is as in \eqref{eq:definition_pressures_PwtP_local}. In the above, $(\nabla \cdot A)_i= \sum_{1\leq j\leq 3}\partial_j A_{i,j}$ for a matrix valued map $A=(A_{i,j})_{i,j=1}^3$ and we used the well-known fact that for all radially symmetric coefficients $\theta$, it holds that (see e.g., \cite[Section 2.3]{G20} or \cite[Section 2]{FL19})
\begin{equation}
\label{eq:symmetric_equal_diagonal_matrix}
\sum_{k,\alpha} \theta_k^2 \sigma_{k,\alpha}\otimes \sigma_{-k,\alpha}
=
\sum_{k,\alpha} \theta_k^2 a_{k,\alpha}\otimes a_{-k,\alpha}= \frac{2}{3}\, \mathrm{Id}.
\end{equation}
An integration by parts shows that $\LL $ is a negative operator for all $\theta\in \ell^2$, i.e., $\langle \LL v,v \rangle_{H^{-1},H^1}\leq 0$ for all $v\in \Hs^1(\T^3)$. However, in accordance with the pathwise energy inequality \eqref{eq:pathwise_preservation}, the operator $\LL$ does not introduce any additional dissipation; it only balances the energy creation of the transport noise in It\^o-form. 

\smallskip

In light of the previous observations, \eqref{eq:navier_stokes_helmholtz} is formally equivalent to 
\begin{equation}
\label{eq:navier_stokes_helmholtz_Ito}
\left\{
\begin{aligned}
\partial_t u +\p [\nabla \cdot (u\otimes u)]
&=(1+\mu)\Delta u -
\mathcal{P}u \\
&+\sqrt{\frac{3\mu}{2}}\sum_{k,\alpha}\theta_k\p\big[ (\sigma_{k,\alpha}\cdot\nabla) u\big]\, \dot{ W}^{k,\alpha}_t,\\
 u(0,\cdot)&=u_0,
 \end{aligned}
\right.
\end{equation} 
It is now possible to give a precise meaning to solutions of \eqref{eq:navier_stokes_helmholtz_Ito} via vector-valued It\^o-calculus. 
Before doing so, let us emphasize that the above reformulation and the definition in the following section also extend to all dimensions $d\geq 2$ with
$$
\sqrt{\frac{3\mu}{2}}\ \text{ replaced by } \ \sqrt{c_d \mu} \ \text{ where }\  c_d\stackrel{{\rm def}}{=}\frac{d}{d-1}.
$$
This will be used throughout the manuscript without further mention.

\subsection{Solutions to stochastic 3D NSEs}
\label{ss:pq_solution}
Here, we define solutions to the stochastic 3D NSEs \eqref{eq:navier_stokes_intro}. We begin by recalling that the sequence of complex Brownian motions $(W^{k,\alpha})_{k,\alpha}$ satisfying \eqref{eq:decorrelation_complex_BM} induces an $\ell^2$-cylindrical (conjugate-symmetric complex) Brownian motion $\mathcal{W}_{\ell^2}$ via the formula
\begin{equation}
\label{eq:def_cylindrical_noise}
\mathcal{W}_{\ell^2}(f)=\sum_{k,\alpha}\int_{\R_+} f_{k,\alpha}(t)\,\dd W^{k,\alpha}_t 
\ \ \  \text{ for } \ \ \
f=(f_{k,\alpha})_{k,\alpha}\in L^2(\R_+;\ell^2(\R)).
\end{equation}
Note that $\mathcal{W}_{\ell^2}(f)$ is real-valued if $\overline{f_{k,\alpha}}=f_{-k,\alpha}$ as $\overline{W^{k,\alpha}}= W^{-k,\alpha}$ for all $k,\alpha$. Using this and the symmetry of $\sigma_{k,\alpha}$ under the reflection $k\mapsto -k$, one can check that the stochastic perturbation in \eqref{eq:navier_stokes_intro} can be rewritten only using real-valued coefficients and real-valued Brownian motions $(B^{k,\alpha})_{k,\alpha}$ in \eqref{eq:complex_BM}. Indeed,
\begin{align}
\label{eq:equivalence_complex_real_noise}
\sum_{k,\alpha} \theta_k\p [(\sigma_{k,\alpha}\cdot\nabla) u]\, \dot{W}^{k,\alpha}
&= \sum_{k\in \Z^3_+,\alpha\in \{1,2\}} 2\,\theta_k \p[(\Re \sigma_{k,\alpha}\cdot\nabla) u]\, \dot{B}^{k,\alpha}\\
\nonumber
&+ \sum_{k\in \Z^3_-,\alpha\in \{1,2\}} 2\,\theta_k\p[(\Im \sigma_{k,\alpha}\cdot\nabla) u] \, \dot{B}^{k,\alpha}.
\end{align}
Hence, solutions to the stochastic 3D NSEs \eqref{eq:navier_stokes_intro} are naturally real-valued. 
However, the complex formulation of the noise is more convenient for computations and has been widely employed in related works (e.g., \cite{A24_global_small,FGL21,FL19}).

\begin{definition}[Local, unique and maximal $(p,q)$-solutions to stochastic 3D NSEs]
\label{def:p_solution} 
Fix $p,q\in (2,\infty)$, and assume that $u_0\in \Bs^{1-2/p}_{q,p}(\T^3)$ and that $\theta=(\theta_k)_{k\in \Z^3_0}\in \ell^2$ is  normalized and radially symmetric. Let $\tau:\O\to [0,\infty]$ and $u:[0,\tau)\times \O\to \Hs^{1,q}(\T^3) $ 
be a stopping time and a progressively measurable process, respectively.

\begin{itemize}
\item We say that $(u,\tau)$ is a \emph{local $(p,q)$-solution} to \eqref{eq:navier_stokes_intro} if there exists a sequence of stopping times $(\tau_n)_{n\geq 1}$ such that $ \tau_n\uparrow \tau$ a.s., for which the following are satisfied for every $n\geq 1$:
\begin{enumerate}[{\rm(a)}]
\vspace{0.1cm}
\item\label{it:integrability_1} $u\in L^p(0,\tau_n;\Hs^{1,q}(\T^3))$ a.s.;
\vspace{0.1cm}
\item\label{it:integrability_2} $u\otimes u\in L^p(0,\tau_n;L^{q}(\T^3;\R^{3\times 3}))$ a.s.;
\vspace{0.1cm}
\item\label{it:integrability_3} a.s.\ for all $t\in [0,\tau_n]$ it holds that 
\begin{align*}
u(t)-u_0
= &\int_0^t \Big((1+\mu)\Delta u(s) -\mathcal{P}u(s) -\p\big[\nabla \cdot (u(s)\otimes u(s))\big]\Big) \,\dd s\\
+ &\, \sqrt{\frac{3\mu}{2}}
\int_0^t \one_{[0,\tau_n]}\Big(\p[(\theta_k\,\sigma_{k,\alpha}\cdot\nabla)u(s)\big]\Big)_ {k,\alpha}\,\dd 
\mathcal{W}_{\ell^2}.
\end{align*}
\end{enumerate}
\item A local $(p,q)$-solution $(u,\tau)$ to \eqref{eq:navier_stokes_intro} is said to be a \emph{unique (pathwise) local $(p,q)$-solution} to \eqref{eq:navier_stokes_intro} if for any local solution $(u',\tau')$ we have $u'=u$ a.e.\ on $[0,\tau'\wedge \tau)\times \O$. 
\item A unique local $(p,q)$-solution $(u,\tau)$ to \eqref{eq:navier_stokes_intro} is said to be a \emph{unique maximal $(p,q)$-solution} to \eqref{eq:navier_stokes_intro} if for any local solution $(u',\tau')$ we have $\tau'\leq \tau$ a.s.\ and $u'=u$ a.e.\ on $[0,\tau')\times \O$. 
\end{itemize}
\end{definition}

Due to \eqref{it:integrability_1}-\eqref{it:integrability_2} and the fact that $\theta\in \ell^2$, the deterministic and stochastic integrals in \eqref{it:integrability_3} are well-defined as an $\Hs^{-1,q}(\T^3)$-valued Bochner integral and an $\Ls^q(\T^3)$-valued It\^o integral, respectively (see e.g., \cite[Chapter 1]{Analysis1} and \cite[Theorem 5.5]{NVW13}). Thus, the equality in \eqref{it:integrability_3} is understood as an identity in $\Hs^{-1,q}(\T^3)$.
Existence and instantaneous regularization results for $(p,q)$-solutions under the subcriticality assumption $\frac{2}{p}+\frac{3}{q} < 2$ follow from the framework developed in \cite[Section 2]{AV21_NS}.

\section{Statement of the main results}
\label{s:statements}
The main result of this manuscript reads as follows (see Theorem \ref{t:intro} for an informal version). Recall that $(p,q)$-solutions to the stochastic 3D NSEs \eqref{eq:navier_stokes_intro} are defined in Definition \ref{def:p_solution}.

\begin{theorem}[Global smooth solutions to the 3D NSEs by transport noise]
\label{t:global_NSE}
Suppose that $q,p\in (2,\infty)$ satisfy
\begin{equation}
\label{eq:condition_pq_subcritical_scaling_statement}
\frac{2}{p}+\frac{3}{q}<2, \quad {{\normalfont\text{(subcriticality)}}}
\end{equation}
and fix 
\begin{equation*}
M \geq 1  \qquad \text{ and }\qquad \varepsilon\in (0,1).
\end{equation*}
Then there exist $\mu>0$ and $\theta\in \ell^2$ satisfying $\#\{k\,:\, \theta_k\neq 0\}<\infty$ for which the following assertion holds. For any
\begin{equation}
\label{eq:global_NSE_initial_data}
u_0\in \Bs^{1-2/p}_{q,p}(\T^3)\qquad \text{ such that }\qquad \|u_0\|_{B^{1-2/p}_{q,p}(\T^3;\R^3)}\leq M,
\end{equation}
the unique maximal $(p,q)$-solution $(u,\tau)$ to the stochastic {\normalfont{3D NSEs}} \eqref{eq:navier_stokes_intro} is global in time with high probability:
\begin{equation}
\label{eq:global_NSE}
\P(\tau=\infty)>1-\varepsilon.
\end{equation}
Moreover, $(u,\tau)$ has the following regularity properties:
\begin{align}
\label{eq:regularity_NSEs_2}
u&\in L^p_{\loc}([0,\tau);H^{1,q}(\T^3;\R^3))\cap C([0,\tau);B^{1-2/p}_{q,p}(\T^3;\R^3)) \ \text{ a.s.\ }\\
\label{eq:regularity_NSEs_3}
u&\in C^{\vartheta_0,\vartheta_1}_{\loc}((0,\tau)\times \T^3;\R^3)\ \text{ a.s.\ for all }\vartheta_0<1/2,  \ \vartheta_1<\infty.
\end{align}
\end{theorem}

The key point in the above result is the assertion \eqref{eq:global_NSE}, which ensures that strong solutions to the stochastic 3D NSEs \eqref{eq:navier_stokes_intro} are global-in-time with high probability. Existence and regularity up to a positive lifetime $\tau$ are established in \cite[Section 2]{AV21_NS}. 
An outline of the proof strategy of Theorem \ref{t:global_NSE} can be found in Section \ref{s:proof_strategy} (see also Figure \ref{fig:scheme}). The complete proof is given in Subsection \ref{ss:global_proof}.

In the deterministic setting ($\mu=0$), the above result remains unknown and constitutes one of the Millennium Prize Problems \cite{F00_NSproblem}. 

\smallskip

Theorem \ref{t:global_NSE} is complemented by a suitable bound on the solution $(u,\tau)$. More precisely, if $p,q$ additionally satisfy $\frac{2}{p}+\frac{3}{q}>1$ and also $p<\frac{q}{3-q}$ in the case $q<3$, then its proof establishes the existence of $p_0>2p$, $q_0>2q$, and $R>0$---depending only on $M$ and the choice of $p,q$---such that
\begin{equation}
\label{eq:estimate_subcritical_norms_3DNSEs}
\Big\|u-\int_{\T^3} u_0\Big\|_{L^{p_0}(0,\tau_0;L^{q_0}(\T^3))}\leq R,
\end{equation}
where $\tau_0\in (0,\tau]$ is a stopping time satisfying 
$$
\P(\tau_0=\infty)>1-\varepsilon.
$$
Explicit values for $(p_0,q_0)$ can be extracted from the proof of Lemma \ref{l:uniform_N_estimate_LpLq}.
Note that the estimate \eqref{eq:estimate_subcritical_norms_3DNSEs} holds in a space with space-time Sobolev index given by
$$
-\frac{2}{p_0}-\frac{3}{q_0}>
-1,
$$
and this is central to establishing \eqref{eq:global_NSE}. We collect some further observations below.

\begin{remark}\
\begin{itemize}
\item {\rm (Subcriticality of initial data)} The class of initial data considered in Theorem \ref{t:global_NSE} is subcritical as the Sobolev index of $B^{1-2/p}_{q,p}(\T^3)$ is given by
$
1-2/p-3/q>-1.
$
The reader is referred to Subsection \ref{ss:scaling_intro} or \cite[Subsection 1.1]{AV21_NS} for details on scaling of \eqref{eq:navier_stokes_intro}. 
Due to the elementary embedding $B^{\varepsilon}_{3,\infty}(\T^3)\embed B^{1-2/p}_{3,p}(\T^3)$ for $\varepsilon\in (0,1)$ and $p\in (2,2/(1-\varepsilon))$, Theorem \ref{t:global_NSE} applies to initial data in the subcritical space 
$
\Bs^{\varepsilon}_{3,\infty}(\T^3)$ for any $\varepsilon\in (0,1).
$

\smallskip

\item {\rm (Independence of $\mu$ from $\varepsilon$)} The proofs of Theorem \ref{t:global_NSE} and Corollary \ref{cor:enhanced_dissipation} given in Section \ref{s:proof_main_results} show that in the latter results and for the bound \eqref{eq:estimate_subcritical_norms_3DNSEs}, the only object depending on $\varepsilon$ is the sequence of noise coefficients $\theta=(\theta_k)_{k\in \Z^3_0}$.

\smallskip
\item {\rm (Flexibility in the choice of the parameters)}
The choice of $\mu$ and $\theta$ in Theorem \ref{t:global_NSE} is not unique. Indeed, letting $\theta^N$ be as in \eqref{eq:choice_thetaN} for some fixed $a>0$, it follows from its proof that there exist $\mu_0$ and $N_0$ depending only on $p,q,\varepsilon,M$ and $a$ such that the $(p,q)$-solution $(u,\tau)$ to the stochastic 3D NSEs \eqref{eq:navier_stokes_intro} with $\theta=\theta^{N}$ satisfies \eqref{eq:global_NSE}-\eqref{eq:regularity_NSEs_3} provided $\mu\geq \mu_0$ and $N\geq N_0$. As in the above item, $\mu_0$ can be chosen independently of $\varepsilon$.
\end{itemize}
\end{remark}

As discussed below Theorem \ref{t:intro}, the key to understanding this regularizing phenomenon lies in the enhanced dissipation effect of transport noise on the Navier-Stokes dynamics. Crucially, solutions to the stochastic 3D NSEs \eqref{eq:navier_stokes_intro} satisfy the same pathwise energy balance \eqref{eq:pathwise_preservation} as their deterministic counterparts, regardless of the noise parameters $\mu$ and $\theta$. This preservation of energy extends to the spatial deviation from the average:
\begin{equation}
\label{eq:energy_difference_average}
\frac{1}{2}\int_{\T^3 }\Big|u(t,\cdot)-\int_{\T^3}u_0 \Big|^2 +\int_0^t\int_{\T^3}| \nabla u(s,x)|^2 \,\dd x \,\dd s =
\frac{1}{2}\int_{\T^3 }\Big|u_0-\int_{\T^3}u_0\Big|^2 ,
\end{equation}
a.s.\ for all $t<\tau$. 
However, dynamically, transport noise generates large gradients. These gradients force the kinetic energy of the velocity deviation (i.e., $\int_{\T^3}\big|u(t,\cdot)-\int_{\T^3}u_0\big|^2$) to decay much faster than standard bounds derived from the energy equality and Poincaré's inequality would suggest. This effect is rigorously captured in the following result, whose proof is postponed to Subsection \ref{ss:enhanced_proof}.

\begin{corollary}[Enhanced dissipation for 3D NSEs by transport noise]
\label{cor:enhanced_dissipation}
Suppose that $q,p\in (2,\infty)$ satisfy $1<\frac{2}{p}+\frac{3}{q}<2$. Assume further $p<\frac{q}{3-q}$ in case $q<3$. Fix 
\begin{equation*}
M \geq 1 ,  \qquad  \varepsilon\in (0,1), \qquad \lambda\in (0,\infty) \ \quad \text{ and }\ \quad b\in (1,\infty). 
\end{equation*}
Then there exist $\mu>0$ and $\theta\in \ell^2$ such that $\#\{k\,:\, \theta_k\neq 0\}<\infty$ for which the following assertion holds. For all initial data
\begin{equation*}
u_0\in \Bs^{1-2/p}_{q,p}(\T^3)\qquad \text{ such that }\qquad \|u_0\|_{B^{1-2/p}_{q,p}(\T^3;\R^3)}\leq M,
\end{equation*}
let $(u,\tau)$ be the unique maximal $(p,q)$-solution to the stochastic {\normalfont{3D NSEs}} \eqref{eq:navier_stokes_intro}. Then there exists a stopping time $\tau_0\in (0,\tau]$ and a random constant $K\in [1,\infty)$ such that 
\begin{equation}
\label{eq:enhanced_dissipation1}
\P(\tau_0=\infty)>1-\varepsilon 
\qquad \text{ and }\qquad
\E \const^b\lesssim_{p,q,M,\varepsilon,\lambda,b}1,
\end{equation}
for which we have the following enhanced dissipation estimate a.s.\
\begin{align}
\label{eq:enhanced_dissipation2}
\sup_{t<\tau_0}\Big(e^{\lambda t}\Big\|u(t,\cdot)-\int_{\T^3} u_0 \Big\|_{L^2(\T^3)}\Big)
&\leq \const \Big\|u_0-\int_{\T^3}u_0\Big\|_{L^2(\T^3)},\\
\label{eq:enhanced_dissipation3}
\Big\|t\mapsto e^{\lambda t}\Big( u(t,\cdot)-\int_{\T^3} u_0\Big)\Big\|_{L^{2p}(0,\tau_0;L^{2q}(\T^3))}
&\leq \const.
\end{align}
\end{corollary}

The estimate \eqref{eq:enhanced_dissipation2} shows that transport noise induces an exponential decay of the kinetic energy of the velocity fluctuation, with an \emph{a priori} prescribable rate $\lambda$.
Furthermore, assertion \eqref{eq:enhanced_dissipation3} guarantees that this decay persists globally in certain subcritical norms for the 3D NSEs, noting that $-2/(2p)-3/(2q) > -1$. As with the bound \eqref{eq:estimate_subcritical_norms_3DNSEs}, the integrability space $L^{2p}(0,\tau_0;L^{2q})$ can be replaced by $L^{p_0}(0,\tau_0;L^{q_0})$ for some $p_0>2p$ and $q_0>2q$ without affecting the result.

\smallskip

We conclude this section with some additional remarks.

\begin{remark}\
\label{r:high_dimensions}
\begin{itemize}
\item {\rm (Non-trivial deterministic forcing)}
The above results naturally extend to the following forced stochastic 3D NSEs:
\begin{align*}
\partial_t u +(u\cdot \nabla)u&=-\nabla P+ \Delta u +f+\sqrt{\frac{3\mu}{2}}\sum_{k,\alpha}\, [-\nabla \wt{P}_{k,\alpha} +\theta_k (\sigma_{k,\alpha}\cdot\nabla) u]\circ \dot{W}^{k,\alpha}_t,
\end{align*} 
which is complemented by the incompressibility constraint $\nabla\cdot u=0$ and the initial condition $u(0,\cdot)=u_0$.
Fix $\alpha_0>0$. Under the assumptions of Theorem \ref{t:global_NSE} and the additional requirements $\frac{2}{p}+\frac{3}{q}>1$ and $[p<\frac{q}{3-q} \text{ if }q<3]$, the conclusions \eqref{eq:global_NSE} and \eqref{eq:regularity_NSEs_2} hold for the unique maximal $(p,q)$-solution to the above forced stochastic 3D NSEs with noise coefficients $(\mu,\theta)$ depending on $f$ only through $M$ and $\alpha_0$ provided the deterministic forcing $f = f_0 + \nabla \cdot F$ satisfies $\|f_0\|_{ L^1(\R_+;\R^3)}\leq M$,
$$
\| F\|_{ L^p(\R_+;L^{q}(\T^3;\R^{3\times 3}))}\leq M \  \text{ and } \  \| t \mapsto e^{\alpha_0 t}\|F(t,\cdot)\|_{L^2(\T^3;\R^{3\times 3})}\|_{ L^2(\R_+)}\leq M.
$$
Clearly, \eqref{eq:regularity_NSEs_3} is also satisfied provided $f$ is smooth, with spatial derivatives bounded uniformly in time (more flexible results can be obtained via \cite[Theorems 2.4 and 2.7]{AV21_NS}).
To prove the above extension, note that the above conditions on $f$ imply the exponential convergence to the mean spatial velocity used in Theorem \ref{t:scaling_limit_cutoff}; see \eqref{eq:decay_energy_vcut}. The rest of the proof is unchanged. 
A version of Corollary \ref{cor:enhanced_dissipation} also holds for the above forced system. We leave the details to the interested reader.

\smallskip

\item {\rm (Anomalous dissipation by transport noise)}
Combining the techniques used to prove Theorem \ref{t:global_NSE} with those in \cite{A24_anomalous,HPZZ23}, one can also obtain anomalous dissipation of energy induced by transport noise.

\smallskip

\item {\rm (Higher dimensions)}
The results in Theorem \ref{t:global_NSE} and Corollary \ref{cor:enhanced_dissipation} also hold in dimensions $d>3$, under the following additional assumptions on $p,q\in (2,\infty)$:
$$
\frac{2}{p}+\frac{d(d-4)}{2q(d-2)}<1\quad \text{ and }\quad \Big[  p<\frac{q(d-2)}{d-q} \text{ in case }q<d\Big].
$$ 
Note that the above holds if $q$ is large and $p\approx 2$. 
These conditions arise by modifying the proof of Lemma \ref{l:uniform_N_estimate_LpLq} to obtain an oscillation-independent estimate for a suitable cutoff version of the stochastic NSEs \eqref{eq:navier_stokes_intro} in dimension $d$. We refer to Subsection \ref{ss:scaling_intro} for the underlying motivations, and to Theorem \ref{t:scaling_limit_cutoff} for the details.
\end{itemize}

\end{remark}

\section{Stochastic Caccioppoli inequalities via localized energy}
\label{s:caccioppoli}
Let $Q=I\times B$ be a parabolic cylinder, i.e., $Q=Q_\varrho(t_0,x_0)$ for some $\varrho>0$. 
In this section, we prove stochastic Caccioppoli inequalities for local solutions to the stochastic Stokes system on $Q$ (see Subsection \ref{ss:caccioppoli_intro}), 
\begin{equation}
\label{eq:turbulent_stokes_caccioppoli}
\left\{
\begin{aligned}
\partial_t v &=-\nabla \pi+\nu \Delta v+ \nabla \cdot F-\frac{1}{2} \sum_{n} \nabla \cdot (\nabla \wt{\pi}_{n}\otimes \zeta_n) \\ 
&\qquad\qquad\qquad  +\sum_{n\geq 1}\, [-\nabla \wt{\pi}_{n} +(\zeta_{n}\cdot\nabla) v+G_n]\,\dot{W}^{n}_t,\\
 \nabla \cdot v&=0,\\
\Delta \pi
&= \nabla^2 :F - \frac{1}{2} \nabla \cdot \sum_{n} (\zeta_{n}\cdot\nabla)[\nabla \wt{\pi}_{n}] ,\\
\Delta \wt{\pi}_n
&= \nabla \cdot [(\zeta_{n}\cdot \nabla)v]+\nabla \cdot G_n,
 \end{aligned}
\right.
\end{equation} 
where $(W^n)_n$ is a sequence of standard independent Brownian motions starting at $t_0-\varrho^2$ on a filtered probability space $(\O,\F,(\F_t)_t,\P)$, and $(\zeta_n)_n$ is a sequence of $\Progress\otimes \Borel(Q)$-measurable vector fields $\zeta_n:\R_+\times \O\times \T^d\to \R^d$ such that there exist constants $\nu_0<\nu$ and $M>0$ for which the following conditions hold a.e.\ on $I\times \O$,
\begin{align}
\label{eq:ass_caccioppoli1}
\|(\zeta_n)_n\|_{L^\infty(Q;\ell^2)}&\leq M, & &\text{(Boundedness)}\\
\label{eq:ass_caccioppoli2}
\frac{1}{2}\sum_{n} |\zeta_n \cdot \xi|^2 &\leq \nu_0 |\xi|^2\text{ for all } \xi\in \R^d, & &\text{(Parabolicity)}\\
\label{eq:ass_caccioppoli3}
\nabla\cdot \zeta_n &=0 \text{ in }\D'(B) \ \text{ for all }n.   & &\text{(Noise incompressibility)}
\end{align}
Finally, we assume that $F:I\times \O\times B\to \R^{d\times d}$ and $(G_n)_n:I \times \O\times B\to \ell^2(\R^d)$ are progressively measurable processes such that 
\begin{equation}
\label{eq:ass_caccioppoli4}
F, \|(G_n)_n\|_{\ell^2}\in L^2(Q) \text{ a.s. }
\end{equation}

Under the above assumptions, we can define local solutions to \eqref{eq:turbulent_stokes_caccioppoli}.

\begin{definition}[Local weak solutions]
\label{def:local_solutions_NSE}
Suppose that \eqref{eq:ass_caccioppoli1} and \eqref{eq:ass_caccioppoli4} are satisfied.
Let $Q=I\times B$ be a parabolic cylinder with center $(t_0,x_0)$ and radius $\varrho>0$, i.e., 
$$
Q_{\varrho}(t_0,x_0)=(t_0-\varrho^2,t_0)\times B_{\varrho}(x_0).
$$ 
We say that $(v,\pi,(\wt{\pi}_n)_n)$ is a local solution to \eqref{eq:turbulent_stokes_caccioppoli} if the following are satisfied.
\begin{enumerate}[{\rm(1)}]
\item\label{it:local_solutions_NSE1} The processes 
$$
v:I \times \O\to H^1(B;\R^d), \quad \pi:I \times \O\to L^2(B) \ \ \text{ and }\ \  (\wt{\pi}_n)_n:I \times \O\to H^1(B;\ell^2), 
$$
are progressively measurable, and a.s.,
\begin{align*}
v&\in C(\overline{I};L^2(B;\R^d))\cap L^2(I;H^1(B;\R^d)) \\
\pi&\in L^2(Q),\quad \text{ and }\quad  (\wt{\pi}_n)_n \in L^2(I;H^1(B;\ell^2)).
\end{align*}
\item The following conditions hold.
\begin{itemize}
\item a.s.\ for all $t\in I$, and all $\phi\in C^\infty_{{\rm c}}(B;\R^d)$:
\begin{align*}
\int_{B} v(t)\cdot \phi
&=
\int_{B} v(t_0-\varrho^2)\cdot \phi\\
&+\int_{t_0-\varrho^2}^t  \int_{B} \Big[ 
\pi\,\nabla \cdot \phi+ \Big(-F -\nu \nabla v+\frac{1}{2} \sum_{n} (\nabla \wt{\pi}_{n}\otimes \zeta_n)\Big):\nabla \phi\Big]\\
&+\int_{t_0-\varrho^2}^t \Big(\int_{B} \big[(-\nabla \wt{\pi}_n + (\zeta_n\cdot \nabla) v) +G_n\big]\cdot \phi \Big)_{n} \, \dd \mathcal{W}_{\ell^2},
\end{align*}
\item a.e.\ on $I\times \O$, 
in $\D'(B)$,
\begin{align*}
\nabla \cdot v&=0,\\
\Delta \pi
&= \nabla^2 :F - \frac{1}{2} \nabla \cdot \sum_{n}\nabla \cdot(\nabla \wt{\pi}_{n}\otimes \zeta_n )  ,\\
\Delta \wt{\pi}_n
&= \nabla \cdot [(\zeta_{n}\cdot \nabla)v]+\nabla \cdot G_n.
\end{align*}
\end{itemize}
\end{enumerate} 
\end{definition}

For later use, let us note that the evolution equation for the velocity field $v$ can be equivalently formulated as an identity on $H^{-1}(B)$. More precisely, if $(v,\pi,(\wt{\pi}_n)_n)$ is a local solution to \eqref{eq:turbulent_stokes_caccioppoli}, then a.s.\ for all $t\in I$,
\begin{align}
\label{eq:integral_identity_v_Caccioppoli}
v(t)
&= v(t_0-\varrho^2)+ \int_{t_0-\varrho^2}^t \Big(- \nabla \pi + \nu \Delta v+ \nabla \cdot F-\frac{1}{2} \sum_{n} \nabla \cdot ( \nabla \wt{\pi}_{n}\otimes \zeta_n)\Big)\,\dd s\\
\nonumber
&+ \int_{t_0-\varrho^2}^t \big( -\nabla\wt{\pi}_n + (\zeta_n \cdot \nabla) v+G_n \big)_n\,\dd \mathcal{W}_{\ell^2}\quad \text{ in }H^{-1}(B),
\end{align}
where $\mathcal{W}_{\ell^2}$ is the $\ell^2$-cylindrical Brownian motion associated with the sequence of standard independent Brownian motions $(W^n)_n$, cf.\ \eqref{eq:def_cylindrical_noise}. As explained below Definition \ref{def:p_solution}, due to the assumptions in Definition \ref{def:local_solutions_NSE}, the deterministic and stochastic integrals are well-defined in the Bochner and It\^o sense, taking values in $H^{-1}(B)$ and $L^2(B)$, respectively.

\smallskip

Under the above assumptions, we can prove the following local smoothing effect for local solutions to the stochastic Stokes system \eqref{eq:turbulent_stokes_caccioppoli}.

\begin{theorem}[Stochastic Caccioppoli inequality -- Stochastic Stokes system]
\label{t:caccioppoli}
Let $Q=I \times B$ be a parabolic cylinder with side length $\varrho\in (0,1/2]$. Fix $r\in (1,\infty)$ and $\sm\in (0,1)$.
Suppose that \eqref{eq:ass_caccioppoli1}-\eqref{eq:ass_caccioppoli4} hold. Then there is a constant $C>0$ depending only on $d,M,\varrho,\nu,\nu_0,r$ and $\sm$ for which the following estimates hold for all local solutions $(v,\pi,(\wt{\pi}_n)_n)$ to the stochastic Stokes system \eqref{eq:turbulent_stokes_caccioppoli} (see Definition \ref{def:local_solutions_NSE}):
\begin{align*}
\Big(\E\sup_{\sm I}\Big\|v- \fint_{\sm B}v\Big\|_{L^2(\sm B)}^r\Big)^{1/r}+
\big(\E\|\nabla v\|_{L^2(\sm Q)}^r\big)^{1/r}&\leq C\Eno_{  Q}^{\#}, \\
\big(\E\| \pi\|_{L^{2}(\sm Q)}^r\big)^{1/r}&\leq C\Eno_{ Q}^{\#},\\
\big(\E\|( \nabla\wt{\pi}_n)_n\|_{L^{2}( \sm Q;\ell^2)}^r\big)^{1/r}&\leq C\Eno_{ Q}^{\#},
\end{align*} 
where $\Eno^{\#}_{ Q}$ is the sharp energy of the local solution $(v,\pi,(\wt{\pi}_n)_n)$ in $ Q$, i.e., 
\begin{align}
\nonumber
\Eno_{ Q}^{\#}
=\Big( \max\Big\{
\E \Big\|v-\fint_B v \Big\|_{L^2( Q)}^r
,\,
\label{eq:def_En_Q_caccioppoli}
\E \|\pi\|_{L^2(I;H^{-1}( B))}^r,\, 
\E \|( \wt{\pi}_{n} )_{n}\|_{L^{2}( Q;\ell^2)}^r,&\\
\E \|F \|_{L^2( Q)}^r,\,
\E \|(G_n)_n \|_{L^2( Q;\ell^2)}^r & \Big\}\Big)^{1/r}.
\end{align}
\end{theorem}

As discussed in Subsection \ref{ss:caccioppoli_intro}, stochastic Caccioppoli estimates encode a local smoothing effect due to parabolicity and can be thought of as ``reverse Poincar\'e'' estimates. The choice of the spaces for the pressure is consistent with the scaling of the stochastic Stokes system \eqref{eq:turbulent_stokes_caccioppoli}, see Subsection \ref{ss:role_pressure_intro}. 

Clearly, by splitting the complex noise into its real and imaginary parts as in \eqref{eq:equivalence_complex_real_noise}, the real-valued abstract system \eqref{eq:turbulent_stokes_caccioppoli} exactly recovers the oscillating Stokes systems discussed in Section \ref{s:proof_strategy}. Thus, the above stochastic Caccioppoli inequality applies with constants uniform in the parameters determining the oscillations for the stochastic Stokes systems in \eqref{eq:turbulent_Stokes_microscopic_intro} and \eqref{eq:turbulent_Stokes_microscopic_rescaled_intro}.

\smallskip

The estimates in Theorem \ref{t:caccioppoli} hold true with $(v,\pi,(\wt{\pi}_n)_n)$ replaced by 
\begin{equation}
\label{eq:replacying_average_Caccioppoli}
\Big(v,\, \pi-\fint_{B} \pi, \, \Big(\wt{\pi}_n-\fint_{B} \wt{\pi}_n \Big)_n\Big).
\end{equation}
The above follows from the invariance of the local solutions under the addition of time-dependent constants to the pressures. Of course, this is not the case for the velocity field, and some care is needed. Indeed, as shown in Lemma \ref{l:localized_energy_inequality} below, we do not apply It\^o's formula to the kinetic energy $\int_{B}|v|^2$ itself, but rather to 
$$
v\mapsto \int_{B} \phi^2 \Big|v-\int_B \phi^2 v\Big|^2
$$
where $\phi\in C^{\infty}_{{\rm c}}(B)$ is a fixed cutoff function such that $\int_B \phi^2=1$ and $\phi|_{\sm B}=1$. In particular, $\int_B \phi^2 v$ mirrors the average $\fint_B v$ while keeping the functional evaluation $v\mapsto \int_B \phi^2 v$ well-defined for weak solutions.

\smallskip

This section is organized as follows. In Subsection \ref{ss:local_energy_inequality}, we prove local energy inequalities for the velocity in the spirit of Scheffer \cite{Scheffer_partial_regularity} and Caffarelli, Kohn, and Nirenberg \cite{CKN82}. Then, in Subsection \ref{ss:caccioppoli_proof}, we prove Theorem \ref{t:caccioppoli} by combining such local energy inequalities for the velocity with corresponding ones for the pressures. 

As will be clear from the proofs in Subsection \ref{ss:local_energy_inequality}, the stochastic Caccioppoli inequality of Theorem \ref{t:caccioppoli} crucially uses the presence of the It\^o-Stratonovich correction $-\frac{1}{2} \sum_{n} \nabla \cdot ( \nabla \wt{\pi}_{n}\otimes \zeta_n)$ in \eqref{eq:turbulent_stokes_caccioppoli} as it allows for a fundamental cancellation, see Remark \ref{r:necessity_strotonovich}. 
Finally, the argument in this section also offers a framework to investigate \emph{suitable} weak solutions (see e.g., \cite[Definition 6.9]{LePi} for the deterministic case) to stochastic 3D NSEs \eqref{eq:navier_stokes_intro}.

\subsection{Stochastic local energy inequality}
\label{ss:local_energy_inequality}
The aim of this section is to prove the following result, where for $\phi\in C^\infty_{{\rm c}}(B)$ with $\int_B\phi^2=1$, we set
$$
\wt{v}_\phi = v-\overline{v}_\phi \qquad \text{ and }\qquad 
\overline{v}_\phi 
\stackrel{{\rm def}}{=}\int_B \phi^2 v.
$$

\begin{proposition}[Local energy inequality -- Stochastic Stokes system]
\label{prop:local_energy}
Let $Q=I \times B$ be a parabolic cylinder with side length $\varrho\in (0,1/2]$, and fix $r\in (1,\infty)$.
Suppose that \eqref{eq:ass_caccioppoli1}-\eqref{eq:ass_caccioppoli4} holds. Then there is a constant $C>0$ depending only on $d,M,\varrho,\nu,\nu_0$ and $r$ such that, for all $\phi\in C^\infty_{{\rm c}}(B)$ such that $\int_B \phi^2=1$, all $\psi\in C^\infty_{{\rm c}}((t_0-\varrho^2,t_0])$, and all local solutions $(v,\pi,\wt{\pi}_n)$ to the stochastic Stokes system \eqref{eq:turbulent_stokes_caccioppoli}, the following localized energy inequality holds a.s.\
\begin{align}
\nonumber
&\sup_{I}\int_{B} \psi^2\phi^2|\wt{v}_\phi|^2 
+\int_{Q}\psi^2\phi^2  \, |\nabla v|^2 
\leq 
C \int_Q \psi^2 |\nabla \phi|^2 \|(\wt{\pi}_n)_n\|_{\ell^2}^2
\\
\nonumber
&\ \   
+  C \int_{Q} |\wt{v}_\phi|^2 \big( \phi^2 |\partial_t\psi^2 | +\psi^2 |  \nabla  \phi|^2 \big) + C \int_I \psi^2  \|(\nabla \phi) \pi\|_{H^{-1}(B)}^2\\
\nonumber
&  \ \  
+ C \int_{Q}\psi^2  \big( |\nabla \phi|^2 |F|^2+ |\phi|^2 \|(G_n)_n\|_{\ell^2}^2\big) 
\\
&\ \   
+C\sup_I \Big|\sum_{n} \int_{t_0-\varrho^2}^\cdot \psi^2 \Big[\int_B \big( \wt{\pi}_n \wt{v}_\phi-\zeta_n \frac{|\wt{v}_\phi|^2}{2}\big) \cdot \nabla (\phi^2) +(G_n\cdot \wt{v}_\phi) \phi^2\Big]\,\dd W^{n}\Big|.
\label{eq:local_energy_claim}
\end{align}
\end{proposition}

The proof of Proposition \ref{prop:local_energy} will be given at the end of this subsection. As a first preparatory step, we derive a first local balance.

\begin{lemma}
\label{l:localized_energy_inequality}
In the setting of Proposition \ref{prop:local_energy}, for all $\phi\in C^{\infty}_{{\rm c}}(B)$ with $\int_B \phi^2=1$ and $\psi\in C^{\infty}_{{\rm c}}((t_0-\varrho^2,t_0])$, it holds that
\begin{align}
\nonumber
\psi^2(t)\int_{B}  \phi^2 |\wt{v}_\phi(t)|^2 
&\leq
  \int_{t_0-\varrho^2}^t\partial_t (\psi^2)\int_{B} \phi^2 |\wt{v}_\phi|^2 \\
  \nonumber
&+ 2\int_{t_0-\varrho^2}^t \psi^2\int_B \pi \nabla\cdot (\phi^2 \wt{v}_\phi)-\nu \nabla \wt{v}_\phi : \nabla (\phi^2  \wt{v}_\phi) - F: \nabla (\phi^2 \wt{v}_\phi)\\
\nonumber
& + \sum_{n}\int_{t_0-\varrho^2}^t\psi^2\int_B\nabla \wt{\pi}_n \cdot (\zeta_n\cdot \nabla) (\phi^2 \wt{v}_\phi)\\ 
\nonumber
&+\sum_{n}\int_{t_0-\varrho^2}^t\psi^2 \int_B\phi^2 |-\nabla \wt{\pi}_n + (\zeta_n\cdot \nabla) \wt{v}_\phi + G_n|^2\\
&+2 \sum_n \int_{t_0-\varrho^2}^t\psi^2 \Big(\int_B [-\nabla\wt{\pi}_n + (\zeta_n \cdot \nabla) \wt{v}_\phi+G_n ]\cdot( \phi^2 \wt{v}_\phi)\Big)\,\dd W^n,
\label{eq:localized_energy_psiphi}
\end{align} 
 a.s.\ for all $t\in I=(t_0-\varrho^2,t_0)$.
\end{lemma}

\begin{proof}
To prove the claim, we apply the It\^o formula to 
$$
(t,v)\mapsto \psi(t)^2 \mathcal{E}_\phi(v)\quad \text{ where }\quad \mathcal{E}_\phi(v) \stackrel{{\rm def}}{=}\int_B  \phi^2 | \wt{v}_\phi|^2. 
$$
Since
\begin{equation}
\label{eq:Leibnitz_energy_phi}
\dd\big( \psi(t)^2 \mathcal{E}_\phi(v(t))\big)=\mathcal{E}_\phi(v(t)) \partial_t (\psi^2(t))\,\dd t+ \psi^2(t)\, \dd \mathcal{E}_\phi(v(t)) ,
\end{equation}
it suffices to compute $\dd \mathcal{E}_\phi(v(t))$. To this end, note that, as $\int_B\phi^2=1$,  
\begin{equation}
\label{eq:identity_with_weighted_average}
\mathcal{E}_\phi(v)
= \int_{B} |\phi v|^2 - \Big(\int_{B} \phi^2 v\Big)^2 .
\end{equation}
From the smoothness of $\phi$, it follows that 
\begin{align}
\label{eq:integral_identity_w_Caccioppoli}
\dd (\phi v)(t)
&= \phi\Big(- \nabla \pi + \nu \Delta v+ \nabla \cdot F-\frac{1}{2} \sum_{n} \nabla \cdot ( \nabla \wt{\pi}_{n}\otimes \zeta_n)\Big) \,\dd t \\
\nonumber
&+\sum_n \big( \phi\,[-\nabla\wt{\pi}_n + (\zeta_n \cdot \nabla) v+G_n ]\big)\,\dd W^n.
\end{align}
As in Definition \ref{def:local_solutions_NSE}, the above identity in $H^{-1}(B)$ is understood in its natural sense, a.s.\ for all $t\in I$. Now, as $(v,\pi,(\wt{\pi}_n)_n)$ is a local weak solution to \eqref{eq:turbulent_stokes_caccioppoli}, we have 
$$
 \phi v\in L^2(I;H^{1}_0(B;\R^d))\cap C(\overline{I};L^2(B;\R^d)) \ \text{ a.s.\ }
$$
Thus, we can apply the It\^o formula \cite[Theorem 4.2.5]{LR15} to the standard Gelfand triple
$$
H^{1}_0(B;\R^d)\embed L^2(B;\R^d)\embed H^{-1}(B;\R^d),
$$
and thereby obtain
\begin{align}
\label{eq:Ito_formula_L2_localized}
\dd \int_{B} |\phi v(t)|^2  
&= \Big( 2\int_{B}\big[ \pi\, \nabla (\phi^2)\cdot v-\nu \nabla v : \nabla (\phi^2  v) - F: \nabla (\phi^2 v)\big]\\
\nonumber
& + \sum_{n} \int_B \nabla \wt{\pi}_n \cdot (\zeta_n\cdot \nabla) (\phi^2 v)\Big)\,\dd t \\ 
\nonumber
&+\sum_{n} \int_B\phi^2 |-\nabla \wt{\pi}_n + (\zeta_n\cdot \nabla) v + G_n|^2\,\dd t \\
\nonumber
&+2 \sum_n \Big(\int_B [-\nabla\wt{\pi}_n + (\zeta_n \cdot \nabla) v+G_n ]\cdot (\phi^2 v)\Big)\,\dd W^n,
\end{align} 
a.s.\ for all $t\in I$. In the above, we used that $\nabla \cdot v=0$ implies
$$
-\int_{B}\nabla \pi \cdot (\phi^2 v)=
\int_{B} \pi\, \nabla \cdot(\phi^2 v)=
\int_{B} \pi \,\nabla(\phi^2 ) \cdot v.
$$
As for the differential of the second term in \eqref{eq:identity_with_weighted_average}, note that, for $i\in \{1,\dots,d\}$,
\begin{align*}
\dd \int_{B} \phi^2 v_i(t)
&= \Big( \int_B  \pi\partial_i (\phi^2) - \nu \nabla v_i\cdot \nabla (\phi^2) -  F_i \cdot \nabla (\phi^2) 
+\frac{1}{2} \sum_{n}  \partial_i \wt{\pi}_{n} (\zeta_n \cdot \nabla) (\phi^2)\Big)\,\dd t \\
\nonumber
&+\sum_n  \int_B \phi^2\big[ -\partial_i\wt{\pi}_n  + (\zeta_n \cdot \nabla) v_i+G_{n,i}\big] \,\dd W^n,
\end{align*}
where $F_i =(F_{i,j})_{j=1}^d$ and $G_n=(G_{n,i})_{i=1}^d$. 
By It\^o's formula, using the nonnegativity of It\^o's correction and that $\overline{v}_\phi$ is spatially constant,
\begin{align}
\label{eq:Ito_formula_L2_localized1}
\dd \Big|\int_{B} \phi^2 v(t)\Big|^2
&\geq \Big(2\int_B \big[ \pi\,\nabla ( \phi^2)\cdot \overline{v}_\phi - \nu \nabla v : \nabla (\phi^2 \overline{v}_\phi) -  F : \nabla (\phi^2 \overline{v}_\phi) \big] \\
\nonumber
&+\sum_{n}  \int_B \nabla \wt{\pi}_{n}\cdot (\zeta_n \cdot \nabla) (\phi^2\overline{v}_\phi)\Big)\,\dd t  \\
\nonumber
&+2\sum_n  \int_B   \big[-\nabla \wt{\pi}_n + (\zeta_n \cdot \nabla) v+G_n\big]\cdot \phi^2\overline{v}_\phi\,\dd W^n.
\end{align}
Combining \eqref{eq:Ito_formula_L2_localized} with \eqref{eq:Ito_formula_L2_localized1}, and using $\nabla \wt{v}_\phi=\nabla v$, 
\begin{align}
\label{eq:Ito_formula_L2_localized2}
\dd \mathcal{E}_\phi(v(t))  
&\leq \Big( 2\int_{B} \big[\pi \nabla (\phi^2)\cdot \wt{v}_\phi-\nu \nabla \wt{v}_\phi : \nabla (\phi^2  \wt{v}_\phi) - F: \nabla (\phi^2 \wt{v}_\phi)\big]\\
\nonumber
& + \sum_{n} \int_B \nabla \wt{\pi}_n \cdot (\zeta_n\cdot \nabla) (\phi^2 \wt{v}_\phi )\Big)\,\dd t \\ 
\nonumber
&+\sum_{n} \int_B\phi^2 |-\nabla \wt{\pi}_n + (\zeta_n\cdot \nabla) \wt{v}_\phi  + G_n|^2\,\dd t \\
\nonumber
&+2 \sum_n \Big(\int_B [-\nabla\wt{\pi}_n + (\zeta_n \cdot \nabla) \wt{v}_\phi +G_n ]\cdot (\phi^2 \wt{v}_\phi) \Big)\,\dd W^n.
\end{align} 
The estimate \eqref{eq:localized_energy_psiphi} follows by \eqref{eq:Leibnitz_energy_phi}, \eqref{eq:Ito_formula_L2_localized2} and the fact that 
\begin{align*}
\int_B [-\nabla\wt{\pi}_n + (\zeta_n \cdot \nabla)\wt{v}_\phi+G_n ]\cdot (\phi^2 \wt{v}_\phi)
  = 
\int_B \Big( \wt{\pi}_n \wt{v}_\phi-\zeta_n \frac{|\wt{v}_\phi|^2}{2}\Big) \cdot \nabla (\phi^2) +G_n\cdot \wt{v}_\phi \phi^2,
\end{align*}
due to the divergence-free property of $\zeta_n$ and $\wt{v}_\phi$.
\end{proof}

The following is the key auxiliary result to prove Proposition \ref{prop:local_energy}.

\begin{lemma}
\label{l:local_energy}
Under the assumptions of Proposition \ref{prop:local_energy}, for each $\varepsilon\in (0,1)$, there exists a constant $C$ depending only on $d,M,\varrho,\nu,\nu_0,r$ and $\varepsilon$ such that, for all $\phi\in C^{\infty}_{{\rm c}}(B)$ with $\int_B \phi^2=1$ and $\psi\in C^{\infty}_{{\rm c}}((t_0-\varrho^2,t_0])$, the following local energy inequality holds a.s.\
\begin{align}
\nonumber
&\sup_{I}\big( \psi^2 \| \phi \, \wt{v}_\phi\|_{L^2(B)}^2\big)
+\int_{Q}\psi^2\phi^2  \, |\nabla v|^2 
\leq \varepsilon \int_{Q} \psi^2\phi^2 \|(\nabla \wt{\pi}_n)_n\|_{\ell^2}^2\\
\nonumber
&\ \  +C \int_Q \psi^2 |\nabla \phi|^2 \|(\wt{\pi}_n)_n\|_{\ell^2}^2+  C \int_{Q} |\wt{v}_\phi|^2 \big(\phi^2 |\partial_t\psi^2| +\psi^2|  \nabla  \phi|^2 \big)\\
\nonumber
& \  \  + C \int_I \psi^2\|(\nabla \phi) \pi\|_{H^{-1}(B)}^2
+ C \int_{Q} \psi^2\big( |\nabla \phi|^2 |F|^2+ (|\phi|^2+|\nabla \phi|^2) \|(G_n)_n\|_{\ell^2}^2\big) 
\\
& \ \ 
+C\sup_I\Big|\sum_{n} \int_{t_0-\varrho^2}^\cdot \psi^2\Big[\int_B \big( \wt{\pi}_n \wt{v}_\phi-\zeta_n \frac{|\wt{v}_\phi |^2}{2}\big) \cdot \nabla (\phi^2) +G_n\cdot \wt{v}_\phi \phi^2\Big]\,\dd W^{n}\Big|.
\label{eq:local_energy_inequality_intermediate}
\end{align}
\end{lemma}

\begin{proof}
As $\nabla v =\nabla \wt{v}_\phi$, it suffices to prove that \eqref{eq:localized_energy_psiphi} implies \eqref{eq:local_energy_inequality_intermediate} with $v$ replaced by $\wt{v}_\phi$. 
From \eqref{eq:localized_energy_psiphi} in Lemma \ref{l:localized_energy_inequality}, it remains to bound the deterministic integral on the right-hand side of \eqref{eq:localized_energy_psiphi}. 
Let $\delta\in (0,1)$ be fixed later.
First, 
\begin{align*}
\Big|\int_{B} \pi\, \nabla (\phi^2)\cdot \wt{v}_\phi\Big|
&\leq \|(\nabla \phi)\pi \|_{H^{-1}(B)}\|\phi \wt{v}_\phi\|_{H^1_0(B)}\\
&\leq C  \|(\nabla \phi)\pi \|_{H^{-1}(B)} (\|\phi \nabla \wt{v}_\phi\|_{L^2(B)}+ \| (\nabla \phi) \wt{v}_\phi\|_{L^2(B)})\\
&\leq \delta \int_B | \phi|^2 |\nabla \wt{v}_\phi|^2+  C_\delta  \|(\nabla \phi)\pi \|_{H^{-1}(B)}^2 +C_\delta \int_B |\nabla \phi|^2  |\wt{v}_\phi|^2, 
\end{align*}
where we used that $\phi|_{\partial B}=0$ and the Poincar\'e inequality. Second, note that 
\begin{align*}
&-\int_B\big( \nu \nabla \wt{v}_\phi : \nabla (\phi^2  \wt{v}_\phi) + F: \nabla (\phi^2 \wt{v}_\phi)\big)\\
&\leq - \Big( \nu-\frac{\delta}{2}\Big) \int_B |\phi|^2 |\nabla \wt{v}_\phi|^2 + C_\delta  \int_B |\nabla \phi|^2 (|\wt{v}_\phi|^2+|F|^2).
\end{align*}
Third, we rewrite It\^o's correction on the right-hand side of \eqref{eq:Ito_formula_L2_localized}. To begin, let us note that
\begin{align*}
&\int_{B}\phi^2 |-\nabla \wt{\pi}_n + (\zeta_n\cdot \nabla) \wt{v}_\phi + G_n|^2\\
&= \int_{B}\phi^2 |-\nabla \wt{\pi}_n + (\zeta_n\cdot \nabla) \wt{v}_\phi|^2 +2\int_B\phi^2  [-\nabla \wt{\pi}_n + (\zeta_n\cdot \nabla) \wt{v}_\phi]\cdot  G_n + \int_B\phi^2 |G_n|^2,
\end{align*}
and the first term can be rewritten as 
\begin{align}
\nonumber
&\int_{B}\phi^2 |-\nabla \wt{\pi}_n + (\zeta_n\cdot \nabla) \wt{v}_\phi|^2 \\
&\qquad \qquad \qquad 
\label{eq:expansion_pressure_product_rule}
= \int_{B}\phi^2 |\nabla \wt{\pi}_n|^2 - 2\int_B\phi^2 
 \nabla \wt{\pi}_n\cdot (\zeta_n\cdot \nabla) \wt{v}_\phi +\int_B \phi^2 |(\zeta_n\cdot \nabla) \wt{v}_\phi|^2 
\end{align}
Now, we rewrite the term $\int_B\phi^2 |\nabla \wt{\pi}_n|^2$. From the regularity of $v$ and $(\wt{\pi}_n)_n$ in Definition \ref{def:local_solutions_NSE}, it holds that 
\begin{equation}
\label{eq:weak_formulation_pressure}
\int_{B} \nabla \wt{\pi}_n \cdot \nabla \psi= \int_{B} (\zeta_n \cdot \nabla )\wt{v}_\phi \cdot \nabla \psi  +\int_{B} G_n \cdot \nabla \psi
\end{equation}
a.e.\ on $I\times \O$, and for all $\psi\in H^{1}_0(B)$. Using \eqref{eq:weak_formulation_pressure} with $\psi= \phi^2 \wt{\pi}_n $, it holds that 
\begin{align}
\nonumber
\int_{B} \phi^2 |\nabla \wt{\pi}_n|^2 
&=
\int_{B} \nabla \wt{\pi}_n \cdot \nabla (\phi^2 \wt{\pi}_n)
- 
\int_{B} \wt{\pi}_n \nabla \wt{\pi}_n\cdot \nabla (\phi^2)\\
\nonumber
&=
\int_{B}  (\zeta_n \cdot \nabla) \wt{v}_\phi \cdot \nabla (\phi^2 \wt{\pi}_n)+\int_{B} G_n\cdot \nabla (\phi^2\wt{\pi}_n)
- 
\int_{B} \wt{\pi}_n \nabla \wt{\pi}_n\cdot \nabla (\phi^2)\\
\nonumber
&=
\int_{B} \phi^2 (\zeta_n \cdot \nabla) \wt{v}_\phi \cdot \nabla \wt{\pi}_n+\int_{B}  \wt{\pi}_n (\zeta_n \cdot \nabla) \wt{v}_\phi \cdot \nabla (\phi^2 )\\
\label{eq:localized_energy_turbulent_pressure}
&+\int_{B} G_n\cdot \nabla (\phi^2\wt{\pi}_n)
- \int_{B} \wt{\pi}_n \nabla \wt{\pi}_n\cdot \nabla (\phi^2)
\end{align}
Finally, we combine this identity with the It\^o-Stratonovich corrector contribution in \eqref{eq:localized_energy_psiphi}, namely 
$$\int_B\nabla \wt{\pi}_n \cdot (\zeta_n\cdot \nabla) (\phi^2 \wt{v}_\phi)=\int_B (\nabla\wt{\pi}_n \cdot \wt{v}_\phi )(\zeta_n\cdot \nabla) (\phi^2 )+\int_B\phi^2 \, \nabla\wt{\pi}_n \cdot (\zeta_n\cdot \nabla)  \wt{v}_\phi,
$$ 
and we obtain
\begin{align}
\label{eq:Ito_stratonovich_correction_compensation}
&\sum_{n}\int_B \nabla \wt{\pi}_n \cdot (\zeta_n\cdot \nabla) (\phi^2 \wt{v}_\phi)
+\sum_{n}\int_{B}\phi^2 |-\nabla \wt{\pi}_n + (\zeta_n\cdot \nabla) \wt{v}_\phi+ G_n|^2\\
\nonumber
&= \sum_{n}\Big(\int_B \wt{v}_\phi\cdot \nabla \wt{\pi}_n  (\zeta_n\cdot \nabla) (\phi^2) 
+2\int_{B}\phi^2  [-\nabla \wt{\pi}_n + (\zeta_n\cdot \nabla) \wt{v}_\phi]\cdot  G_n \\
&
\nonumber
+ \int_B\phi^2 |G_n|^2+\int_{B} G_n\cdot \nabla (\phi^2\wt{\pi}_n) + \int_B \phi^2 |(\zeta_n\cdot \nabla) \wt{v}_\phi|^2 \\
&
\nonumber
+\int_{B}  \wt{\pi}_n (\zeta_n \cdot \nabla) \wt{v}_\phi \cdot \nabla (\phi^2 )
- \int_{B} \wt{\pi}_n \nabla \wt{\pi}_n\cdot \nabla (\phi^2)\Big).
\end{align}
Remarkably, the higher-order term $\int_{B}\phi^2 
 \nabla \wt{\pi}_n\cdot (\zeta_n\cdot \nabla) v$ containing both the $\nabla \wt{\pi}_n$ and $\nabla v$ is balanced by the It\^o-Stratonovich corrector in \eqref{eq:Ito_formula_L2_localized} and the terms in the identities  \eqref{eq:expansion_pressure_product_rule} and \eqref{eq:localized_energy_turbulent_pressure}.
Finally, from \eqref{eq:ass_caccioppoli1}-\eqref{eq:ass_caccioppoli2} and the Cauchy-Schwarz inequality, it follows that the right-hand side of the previous equality can be estimated by
\begin{align*}
\varepsilon \int_{Q} \phi^2 \|(\nabla \wt{\pi}_n)_n\|_{\ell^2}^2
&+ (2\nu_0+\delta) \int_{Q} \phi^2 |\nabla \wt{v}_\phi|^2\\
&+C_{\delta,\varepsilon} \int_{Q}\big( |\nabla \phi|^2 ( | \wt{v}_\phi|^2+ \|(\wt{\pi}_n)_n\|_{\ell^2}^2) +( |\phi|^2+|\nabla \phi|^2) \|(G_n)_n\|_{\ell^2}^2\big).
\end{align*}
The conclusion follows by collecting the previous estimates and choosing $\delta$ such that $(2\nu_0+\delta)+ \delta <2\nu-\delta$.  
\end{proof}

\begin{remark}[On the It\^o-formulation of transport noise]
\label{r:necessity_strotonovich}
The cancellation in \eqref{eq:Ito_stratonovich_correction_compensation} is solely due to the presence of the Itô-Stratonovich corrector. This corrector yields the integral $\int_B \nabla \wt{\pi}_n \cdot (\zeta_n\cdot \nabla) (\phi^2 \wt{v}_\phi)$ in the energy balance, which perfectly cancels the cross-term $\int_{B}\phi^2 \nabla \wt{\pi}_n\cdot (\zeta_n\cdot \nabla) v$ that has no clear sign and is of high order in both pressure and velocity. 
\end{remark}

We are now ready to prove Proposition \ref{prop:local_energy}.

\begin{proof}[Proof of Proposition \ref{prop:local_energy}]
It remains to bound the stochastic pressure. By \eqref{eq:ass_caccioppoli1}, \eqref{eq:localized_energy_turbulent_pressure} and $\nabla v =\nabla \wt{v}_\phi$, one can readily check that
\begin{align}
\label{eq:local_smoothing_turbulent_pressure}
\int_B|\phi|^2 \|(\nabla \wt{\pi}_n)_n\|_{\ell^2}^2
&\leq C_0\int_{B } \Big[ |\phi|^2 |\nabla \wt{v}_\phi|^2 + (|\nabla \phi|^2+|\phi|^2)\|(G_n)_n\|_{\ell^2}^2\Big]\\
&\nonumber
+C_0\int_{B } |\nabla \phi|^2 \|(\wt{\pi}_n)_n\|_{\ell^2}^2+
\frac{1}{2}
\int_B|\phi|^2 \|(\nabla \wt{\pi}_n)_n\|_{\ell^2}^2 ,
\end{align}
where $C_0$ depends only on $M$ in \eqref{eq:ass_caccioppoli1}. Absorbing the last term on the right-hand side of \eqref{eq:local_smoothing_turbulent_pressure}, the conclusion of Proposition \ref{prop:local_energy} follows from Lemma \ref{l:local_energy} by using the above estimate with the choice $\varepsilon=1/(4C_0)$.
\end{proof}

\subsection{Stochastic Caccioppoli inequality -- Proof of Theorem \ref{t:caccioppoli}}
\label{ss:caccioppoli_proof}
We begin by proving an intermediate estimate involving only the velocity field and the stochastic pressure.

\begin{lemma}[Intermediate estimate on the velocity and stochastic pressure]
\label{l:intermediate_estimate_pressure_deterministic0}
Under the assumptions of Theorem \ref{t:caccioppoli}, for all $\alpha\in (0,1)$,  it holds that  
\begin{align*}
\Big(\E\sup_{\alpha I}\Big\|v-\fint_{\alpha B} v\Big\|_{L^2(\alpha B)}^r\Big)^{1/r}
&+
\big(\E\|\nabla v\|_{L^2(\alpha Q)}^r\big)^{1/r}\\
&+
\big(\E\|(\nabla \wt{\pi}_n)_n\|_{L^{2}(\alpha Q;\ell^2)}^r\big)^{1/r}
\leq C_0 \Eno_{Q}^{\#} ,
\end{align*}
where $C_0$ depends only on $d,M,\varrho,\nu,\nu_0,r$ and $\alpha$ and $\Eno_Q^{\#}$ is as in \eqref{eq:def_En_Q_caccioppoli}.
\end{lemma}

\begin{proof}
Let $\psi\in C^\infty_{{\rm c}}(I)$ and $\phi\in C^\infty_{{\rm c}}(B)$ be such that $\psi|_{\alpha I}=1$ and $\phi|_{\alpha B}=1$. Without loss of generality, we may assume $\| \psi\|_{W^{1,\infty}(I)}, \| \phi\|_{W^{1,\infty}(B)}\lesssim_{\alpha,\varrho} 1$. 
We begin by estimating the stochastic integral term on the right-hand side of \eqref{eq:local_energy_claim} with the previous choice of $(\phi,\psi)$. 
From the Burkholder-Davis-Gundy inequality, we have, for all $r\in (1,\infty)$,
\begin{align*}
&\E\Big[\sup_{I}\Big|\sum_{n} \int_{t_0-\varrho^2}^\cdot \psi^2\int_B \nabla (\phi^2)\cdot \big(\wt{\pi}_n \wt{v}_\phi 
- \zeta_n\frac{|\wt{v}_\phi|^2}{2}\big)\,\dd W^{n}\Big|^{r/2}\Big]\\
&
\leq C_r \E\Big[\int_{I}\psi^4\Big\| \Big(\int_B\phi \nabla \phi\cdot \big(\wt{\pi}_n \wt{v}_\phi - \zeta_n\frac{|\wt{v}_\phi|^2}{2}\big)\Big)_n\Big\|_{\ell^2}^2\Big]^{r/4}\\
&\leq  C_r\E\Big[\int_{I} \psi^4\Big(\int_B|\phi| |\nabla \phi| 
\Big\|\Big( \big(\wt{\pi}_n \wt{v}_\phi -\zeta_n\frac{|\wt{v}_\phi|^2}{2}\big)\Big)_n\Big\|_{\ell^2}\Big)^2\Big]^{r/4}\\
&\leq C_r\E\Big[\int_{I}  \psi^4\Big(\int_B|\phi| |\nabla \phi| 
 |\wt{v}_\phi|\big(\|(\wt{\pi}_n)_n\|_{\ell^2}  + |\wt{v}_\phi|\big)\Big)^2\Big]^{r/4}\\
&\leq C_r\E\Big[\int_{I}\psi^4 \|\phi \wt{v}_\phi\|_{L^2(B)}^2 \big(\|((\nabla \phi) \wt{\pi}_n)_n\|_{L^2(B;\ell^2)}^2 +\|(\nabla \phi)\wt{v}_\phi\|_{L^2(B)}^2\big)\Big]^{r/4}\\
&\leq C_r\E\Big[\big[\sup_{I} \big(\psi^2\|\phi \wt{v}_\phi\|_{L^2(B)}^2\big) \big] 
\Big(\int_{I} \psi^2\big(\|((\nabla \phi)\wt{\pi}_n)_n\|_{L^2(B;\ell^2)}^2 +\|(\nabla \phi) \wt{v}_\phi\|_{L^2(B)}^2\big)\Big)\Big]^{r/4}\\
&\leq \frac{1}{4}\, \E\Big[ \sup_{I}\big( \psi\|\phi \wt{v}_\phi\|_{L^2(B)} \big)^{r}\Big]\\
&+ C_{r} \E \Big[ \Big(\int_{I}\psi^2 \big(\|((\nabla \phi)\wt{\pi}_n)_n\|_{L^2(B;\ell^2)}^2+\|(\nabla \phi)\wt{v}_\phi\|_{L^2(B)}^2\big)\Big)^{r/2}\Big],
\end{align*}
as well as 
\begin{align*}
\E\Big[\sup_I \Big|\sum_{n} \int_{t_0-\varrho^2}^\cdot \psi^2 \int_B (G_n\cdot \wt{v}_\phi) \phi^2\,\dd W^{n}\Big|^{r/2}\Big]
&\leq
\frac{1}{4}\, \E\Big[ \sup_{I}\big( \psi\|\phi \wt{v}_\phi\|_{L^2(B)} \big)^{r}\Big] \\
& +C_r \E\Big[\Big(\int_{Q} \psi^2 \phi^2 \|(G_n)_n\|^2_{\ell^2}\Big)^{r/2}\Big].
\end{align*}
Next, taking the $(r/2)$-th moment in \eqref{eq:local_energy_claim}, using the above inequality, and absorbing on the left-hand side the term $\frac{1}{2} \E \sup_{I}\|\phi \wt{v}_\phi\|_{L^2(B)}^{r}$, we obtain
\begin{align}
\label{eq:caccioppoli_1_intermediate}
\Big(\E\Big[ \sup_{I}\big( \psi\|\phi \wt{v}_\phi\|_{L^2(B)} \big)^{r}\Big]\Big)^{1/r}+
\big(\E\|\psi \phi (\nabla v)\|^{r}_{L^2( Q)}\big)^{1/r}
\leq C_0 \Eno_{ Q}^{\#} ,
\end{align}
where we used the choice of $(\phi,\psi)$ and 
\begin{align*}
\Big\|f-\Big(\int_{ B} \phi^2 f\Big) \Big\|_{L^2( B)}
&\leq \Big\|f-\fint_{  B} f \Big\|_{L^2( B)}
+C\Big| \int_{ B} \phi^2\Big(f-\fint_{ B}f\Big)\Big|\\
&\leq C_\alpha\Big\|f-\fint_{B}f  \Big\|_{L^2(B)},
\end{align*}
for all $f\in L^2(B)$ as $\int_B\phi^2=1$.
Next, we estimate the stochastic pressure. From \eqref{eq:local_smoothing_turbulent_pressure} and the above, it holds that 
\begin{equation}
\label{eq:caccioppoli_2_intermediate}
\Big(\E\Big[\int_Q|\psi|^2 |\phi|^2 \|(\nabla \wt{\pi}_n)_n\|_{\ell^2}^2\Big]^{r/2}\Big)^{1/r}\leq C_0 \Eno_{ Q}^{\#} .
\end{equation}
To conclude, it remains to use the support properties of $\psi$ and $\phi$ as well as 
$$
\Big\|f-\fint_{\alpha B} f \Big\|_{L^2(\alpha B)}\leq \Big\|f-\Big(\int_{B} \phi^2 f\Big) \Big\|_{L^2(\alpha B)}
\leq\Big\|\phi^2 \Big[f-\Big(\int_{B} \phi^2 f\Big)\Big] \Big\|_{L^2( B)}
$$
for all $f\in L^2(B)$.
\end{proof}

As a final ingredient in the proof of Theorem \ref{t:caccioppoli}, we prove an intermediate estimate for the deterministic pressure.

\begin{lemma}[Intermediate estimate on the deterministic pressure]
\label{l:intermediate_estimate_pressure_deterministic}
Under the assumptions of Theorem \ref{t:caccioppoli}, for all $\alpha\in (0,1)$ and $\beta\in (\alpha,1]$, it holds that, a.e.\ on $\beta I\times \O$, 
\begin{align}
\label{eq:caccioppoli_pressure1}
\| \pi\|_{L^{2}(\alpha B)}
\leq C_0 \|F\|_{L^2(\beta  B)}
+ C_0\|(\nabla \wt{\pi}_n)_n\|_{L^2(\beta  B;\ell^2)} +C_0 \|\pi\|_{H^{-1}(\beta B)},
\end{align} 
where $C_0$ depends only on $d,M,\varrho,\nu,\nu_0,r,\alpha$ and $\beta$.
\end{lemma}

\begin{proof}
The proof here follows again by a localization argument. 
Fix $\phi\in C^\infty_{{\rm c}}(\beta B)$ such that $\phi|_{\alpha B}=1$ and $\max_{0\leq k\leq 2}\|\nabla^k \phi\|_{L^\infty(\R^d)}\leq C_{\alpha,\beta,\varrho}$. 
Letting $\Pi= \phi \pi$, it follows from Definition \ref{def:local_solutions_NSE} that 
\begin{align*}
\Delta \Pi 
&= \phi \Delta \pi + 2\nabla \phi \cdot \nabla \pi + \pi \Delta \phi\\
&= \phi \nabla^2 : F -\phi \frac{1}{2} \nabla \cdot \sum_{n}\nabla \cdot(\nabla \wt{\pi}_{n}\otimes \zeta_n ) + 2\nabla \phi \cdot \nabla \pi + \pi \Delta \phi.
\end{align*}
Now, we can clearly extend by periodicity $\Pi$ outside $\T^d_{\beta} = x_0 + (-\beta,\beta]^d$. 
Let $\wt{F}$ be the right-hand side of the above identity. The above identity yields $\Pi =(\mathrm{Id}-\Delta_{\T^d_\beta})^{-1}(\Pi-\wt{F})$, and therefore
\begin{align*}
\|\pi\|_{L^2(\alpha B)}&\stackrel{(i)}{\leq}
\|\Pi \|_{L^2(\T^d_\beta)}\lesssim_\beta \|\wt{F}\|_{H^{-2}(\T^d_\beta)}+ \|\Pi\|_{H^{-2}(\T^d_\beta)}\\
&\stackrel{(ii)}{\lesssim}_{\alpha,\beta,\varrho} 
 \|F\|_{L^2( \beta B)}+ \Big\| \sum_{n}\nabla \wt{\pi}_{n}\otimes \zeta_n \Big\|_{L^2(\beta B)}+ \|\pi\|_{H^{-1}(\beta B)}\\
&\stackrel{(iii)}{\lesssim}_{\alpha,\beta,\varrho,M}
 \|F\|_{L^2(\beta B)}+ \| ( \nabla \wt{\pi}_{n})_n  \|_{L^2(\beta B;\ell^2)}+ \|\pi\|_{H^{-1}( \beta B)}
\end{align*}
where in $(i)$ we used $\phi|_{\alpha B}=1$, in $(ii)$ the bounds assumed on $\phi$, and in $(iii)$ the boundedness assumption on the noise coefficients \eqref{eq:ass_caccioppoli1}. 
\end{proof}

\begin{proof}[Proof of Theorem \ref{t:caccioppoli}]
From Lemma \ref{l:intermediate_estimate_pressure_deterministic0} with $\alpha=\sm$, it remains to prove that 
$$
\big(\E\| \pi\|_{L^{2}(\sm Q)}^r\big)^{1/r}\leq C\Eno_{ Q}^{\#}.
$$
The above follows immediately from chaining Lemma \ref{l:intermediate_estimate_pressure_deterministic} with $\alpha=\sm$ and $\beta=(1+\sm)/2$ with Lemma \ref{l:intermediate_estimate_pressure_deterministic0} applied with $\alpha=(1+\sm)/2$.
\end{proof}

We conclude this section with a possible improvement of Theorem \ref{t:caccioppoli}.

\begin{remark}[Improved stochastic Caccioppoli estimates for the stochastic pressure]
\label{r:improved_Hminus1_turbulent_pressure}
The estimate in Theorem \ref{t:caccioppoli} can be further improved in terms of the local smoothing of the stochastic pressure by using the argument in Lemma \ref{l:intermediate_estimate_pressure_deterministic}. Indeed, by using the divergence-free property of $\zeta_n$ and Definition \ref{def:local_solutions_NSE}, it holds that 
$$
\Delta \wt{\pi}_{n}=\nabla \cdot \big[(\zeta_n\cdot \nabla) \wt{v}+  G_n\big]= \nabla \cdot \big[\nabla \cdot (\wt{v}\otimes \zeta_n)+  G_n\big] \ \text{ in } \ \D'(B)
$$
a.e.\ on $I\times \O$, where $\wt{v}=v-\fint_{\beta B} v$ and $\nabla \wt{v} =\nabla v$.
More precisely, by following the arguments in Lemma \ref{l:intermediate_estimate_pressure_deterministic}, one can prove that, for all $0<\alpha<\beta\leq 1$,
$$
\|(\wt{\pi}_n)_n \|_{L^2(\alpha B;\ell^2)}
\leq C \|(\wt{\pi}_n)_n \|_{\underline{H}^{-1}(\beta B;\ell^2)}
+C \Big\|v-\fint_{\beta B} v \Big\|_{L^2(\beta  B)}+C \|(G_n)_n\|_{L^2(\beta B;\ell^2)}
$$
where $C$ depends only on $M,\alpha$ and $\beta$, see \eqref{eq:ass_caccioppoli1}. 
The above can be combined with the current statement of Theorem \ref{t:caccioppoli} to obtain further local smoothing of the stochastic pressure. However, this will not be used here. 
\end{remark}

\section{Localized stochastic maximal $L^p$-regularity}
\label{s:localized_smr}
Here, we consider the stochastic Stokes system on a parabolic cylinder $Q$:
\begin{equation}
\label{eq:turbulent_Stokes_xi_linear}
\left\{
\begin{aligned}
\partial_t v &=-\nabla \pi+ \nu\Delta v+ \nabla \cdot F-\frac{1}{2} \sum_{n\geq 1} \nabla \cdot (\nabla \wt{\pi}_{n}\otimes \zeta_{n})\\ 
&\qquad\qquad\qquad  + \sum_{n\geq 1} [-\nabla \wt{\pi}_n +( \zeta_n\cdot\nabla) v+G_n]\,\dot{W}^n_t,\\
\Delta \pi
&= \nabla^2 :F -\frac{1}{2}  \nabla \cdot\Big( \sum_{n\geq 1} \nabla \cdot \big(\nabla \wt{\pi}_{n} \otimes \zeta_{n}\big) \Big),\\
\Delta \wt{\pi}_{n}
&= \nabla \cdot [(\zeta_n\cdot \nabla)v]+\nabla \cdot G_n,
\end{aligned}
\right.
\end{equation}
and $(W^n)_n$ is a family of standard independent Brownian motions starting at $t_0-1$ on a filtered probability space $(\O,\F,(\F_t)_t,\P)$. As in Section \ref{s:caccioppoli}, in the above, $v:I\times \O\times B\to \R^d$ and $\pi,\wt{\pi}_n:I\times \O\times B \to \R$ denote the unknown velocity field and pressures, respectively. 

\smallskip

The main result of this section reads as follows, where, for notational convenience, we write $\ell^2$ instead of $\ell^2(\N_{\geq 1};\R^d)$ if no confusion seems likely.

\begin{theorem}[Localized stochastic maximal $L^p$-regularity]
\label{t:localized_SMR}
Fix $\nu>0$ and let $Q=I\times B$ be a parabolic cylinder of side length $\varrho\in (0,1/2]$. Suppose that $(\zeta_n)_n : I\times \O\times B \to \ell^2$ is progressively measurable, and that there exist $M>0$ and $\nu_0\in (0,\nu)$ such that 
\begin{align}
\label{eq:parabolicity_localized_SMR}
&\frac{1}{2}\sum_{n\geq 1} |\zeta_n\cdot \xi |^2\leq \nu_0 |\xi|^2 \ \text{ for all }\xi\in\R^d,  & &({{\normalfont{\text{Parabolicity}}}})\\
\label{eq:smoothness_coefficients_localized_SMR}
&\sup_I \|(\zeta_n)_n\|_{W^{1,\infty}(B;\ell^2)}\leq M. & &({{\normalfont{\text{Regularity}}}})\\
\label{eq:incompressibility_coefficients_localized_SMR}
&\nabla\cdot \zeta_n =0 \text{ in }\D'(B) \ \text{ for all }n.   & &({{\normalfont\text{Noise incompressibility}}})
\end{align}
Fix $r\in (1,\infty)$, $p,q\in (2,\infty)$ and $\sm\in (0,1)$. Then there exists a constant $C_0>0$ depending only on $p,q,r,M,\varrho,d,\nu,\nu_0$ and $\sm$ for which the following estimate holds  for any local solution $(v,\pi,(\wt{\pi}_n)_n)$ to the stochastic Stokes system \eqref{eq:turbulent_Stokes_xi_linear} on $Q$ (see Definition \ref{def:local_solutions_NSE}):
\begin{align}
\nonumber
\Big(\E\sup_{\sm I}\|v\|_{B^{1-2/p}_{q,p}(\sm  B)}^r\Big)^{1/r}
&+ \big(\E\|\nabla v\|_{L^p( \sm I;L^q(\sm B))}^r\big)^{1/r}
\leq C_0
\big( \E\| v\|_{L^2(Q)}^r \big)^{1/r}\\
\nonumber
&+C_0 \big(\E\|\pi\|_{L^2(I;H^{-1}(B))}^r\big)^{1/r}
+C_0\big(\E\|(\wt{\pi}_n)_{n}\|_{L^2(Q;\ell^2)}^r\big)^{1/r}
\\ 
& + C_0\big(\E\|F\|_{L^p(I;L^q(B))}^r\big)^{1/r}+ C_0\big(\E\|(G_n)_n\|_{L^p(I;L^q(B;\ell^2))}^r\big)^{1/r},
\label{eq:localized_SMR0}
\end{align}
whenever the right-hand side is finite. 
\end{theorem}

The estimate encodes well-known local smoothing estimates, see, e.g., \cite{GM12_appunti} for the case of elliptic PDEs. 
Let us point out that, unlike the Caccioppoli inequality of Theorem \ref{t:caccioppoli}, the above result relies on the smoothness assumption on the noise coefficients $(\zeta_n)_n$ in \eqref{eq:smoothness_coefficients_localized_SMR}.

Theorem \ref{t:localized_SMR} is proved in Subsection \ref{sss:proof_localized_smr} and is a localized version of \cite[Theorem 3.2]{AV21_NS}. While the first condition \eqref{eq:parabolicity_localized_SMR} is necessary for  \eqref{eq:localized_SMR0} to hold, we expect that \eqref{eq:smoothness_coefficients_localized_SMR} can be weakened to $(\zeta_n)_n\in C^{\delta}(Q;\ell^2)$ for some $\delta>0$ as for the global-in-space stochastic maximal $L^p$-regularity estimates in \cite[Section 3]{AV21_NS}. 
In the proof of Theorem \ref{t:localized_SMR}, we also obtain local regularity estimates for the pressures $\pi$ and $\wt{\pi}_n$. While local smoothing for the pressures is central in obtaining the bound \eqref{eq:localized_SMR0}, these estimates will not be used in the following sections, and therefore we do not state them explicitly. 
The interested reader is referred to Lemmas \ref{l:local_smoothing_iteration_pressure_turbulent} and \ref{l:local_smoothing_iteration_pressure} below.

\smallskip

Before going into the proof of Theorem \ref{t:localized_SMR}, we state a consequence for local solutions $(v^N,\pi^N,(\wt{\pi}_{k,\alpha}^N)_{k,\alpha})$ of the oscillating stochastic Stokes system at the \emph{microscopic scale} $N^{-1}Q$ (see Subsections \ref{ss:almost_self} and \ref{ss:control_energy_intro}):
\begin{equation}
\label{eq:turbulent_Stokes_microscopic}
\left\{
\begin{aligned}
\partial_t v^N &=-\nabla \pi^N+ (1+\mu)\Delta v^N+ \nabla \cdot F - c_d \mu\sum_{k,\alpha} \theta^N_k \nabla \cdot (\nabla \wt{\pi}_{k,\alpha}^N\otimes \sigma_{-k,\alpha})\\ 
&\qquad\ \  +\sqrt{c_d \mu} \sum_{k,\alpha} [-\nabla \wt{\pi}^{N}_{k,\alpha} +\theta^{N}_{k}(\sigma_{k,\alpha}\cdot\nabla) v^N+\theta^{N}_{k}\sigma_{k,\alpha}\cdot G]\,\dot{W}^{k,\alpha}_t,\\
 \nabla \cdot v^N&=0,\\
 \Delta \pi^N
&= \nabla^2 : F - c_d \mu \, \nabla \cdot\Big( \sum_{k,\alpha} \theta^N_k\nabla \cdot (\nabla \wt{\pi}_{k,\alpha}^N\otimes \sigma_{-k,\alpha})\Big),\\
\Delta \wt{\pi}_{k,\alpha}^N
&= \theta_k^N\nabla \cdot \big[(\sigma_{k,\alpha}\cdot \nabla)v^N+\sigma_{k,\alpha}\cdot G\big],
\end{aligned}
\right.
\end{equation}
where $\sigma_{k,\alpha}$ and $\theta^N$ are as in Subsection \ref{ss:probabilistic}.

\begin{corollary}[Localized $L^\infty$-bounds at microscopic scale -- Oscillating stochastic Stokes]
\label{cor:SMR_microscopic_estimate}
Let $Q=Q_{1/2}(t_0,x_0)$ be a parabolic cylinder with center $(t_0,x_0)\in \R_+\times \R^d$.
Fix $\mu>0$, $r\in (1,\infty)$ and $p,q\in (2,\infty)$ such that 
\begin{equation}
\label{eq:parameters_for_pointwise_evaluation_smr_localized_N}
\frac{2}{p}+\frac{d}{q}<1.
\end{equation}
Then there exists a constant $C_0>0$ depending only on $p,q,r$ and $\mu$ such that for any local solution $(v^N,\pi^N,(\wt{\pi}_{k,\alpha}^N)_{k,\alpha})$ to the oscillating stochastic Stokes system \eqref{eq:turbulent_Stokes_microscopic} at the microscopic scale $N^{-1}Q$, it holds that 
\begin{align}
\label{eq:Linfty_bounds_t0x0_smr}
\big(\E|v^N(t_0,x_0)|^r\big)^{1/r}
&\leq C_0
 \big(\E\| v^N\|_{\underline{L}^2(N^{-1}Q)}^r\big)^{1/r}\\
 \nonumber
 &+C_0  \big(\E\|\pi^N\|_{\underline{L}^2(N^{-1}I;\underline{H}^{-1}(N^{-1}B))}^r\big)^{1/r} \\
\nonumber 
 &+C_0\big(\E\|(\wt{\pi}^N_{k,\alpha})_{k,\alpha}\|_{\underline{L}^2(N^{-1}Q;\ell^2)}^r \big)^{1/r}\\
 \nonumber
& +C_0 \big(\E\|F\|_{\underline{L}^p(I;\underline{L}^q(B))}^r\big)^{1/r}+C_0 \big(\E\|G\|^r_{\underline{L}^p(I;\underline{L}^q(B;\ell^2))}\big)^{1/r}.
\end{align}
\end{corollary}

Local solutions to \eqref{eq:turbulent_Stokes_microscopic} can be defined as in Definition \ref{def:local_solutions_NSE} by using the decomposition \eqref{eq:equivalence_complex_real_noise}, or by allowing $\wt{\pi}_{k,\alpha}$ to be complex valued with the additional symmetry $\overline{\wt{\pi}_{k,\alpha}^N}=\wt{\pi}_{-k,\alpha}^N$. This will be used below without further mention.

\smallskip 

The proof of the above result is given in Subsection \ref{ss:proof_corollary_SMR_microscopic_estimate}, and it is based on a blow-up technique and exploits the `rigidity' of the $L^\infty(B)$-norm on continuous functions, which allows for pointwise evaluations (see Subsection \ref{ss:almost_self}). 
For comments on the role of Corollary \ref{cor:SMR_microscopic_estimate} in the proof of Theorem \ref{t:global_NSE}, the reader is referred to Subsection \ref{ss:almost_self}.
As in Theorem \ref{t:caccioppoli}, one can apply the above result with $(v^N,\pi^N,(\wt{\pi}^N_{k,\alpha})_{k,\alpha})$ replaced by 
\begin{equation*}
\Big(v^N,\, \pi^N-\fint_{N^{-1}B} \pi^N, \, \Big(\wt{\pi}^N_{k,\alpha}-\fint_{N^{-1}B} \wt{\pi}^N_{k,\alpha} \Big)_{k,\alpha}\Big), 
\end{equation*}
and a similar comment holds for Theorem \ref{t:localized_SMR}. This observation will be used in the proof of  
Theorems \ref{t:suboptimal_Lq_finite_time} and \ref{t:suboptimal_Lq_Rpositive}, for which we only need to control the oscillations of the pressures and not the full norms, cf.\ Theorem \ref{t:control_mesoscopic_energy} below.

\subsection{Proof of Corollary \ref{cor:SMR_microscopic_estimate}}
\label{ss:proof_corollary_SMR_microscopic_estimate}
The main idea is to derive Corollary \ref{cor:SMR_microscopic_estimate} by combining Theorem \ref{t:localized_SMR} and a blow-up argument, together with the embedding
\begin{equation}
\label{eq:Besov_embedding_Linfty}
B^{1-2/p}_{q,p}( B)\embed L^\infty(  B)
\end{equation}
which holds due to the assumption \eqref{eq:parameters_for_pointwise_evaluation_smr_localized_N} for each ball $B\subseteq \R^d$.

Next, we collect some useful facts. 
Let $(v^N,\pi^N,(\wt{\pi}_{k,\alpha}^N)_{k,\alpha})$ be a local solution to \eqref{eq:turbulent_Stokes_microscopic} on $N^{-1}Q$.
Denote by $(V^N,\Pi^N,(\wt{\Pi}_{k,\alpha}^N)_{k,\alpha})$ its blow-up around $(t_0,x_0)$:
\begin{align*}
V^N(t,x)&= v^N (t_0+ N^{-2}(t-t_0), x_0+ N^{-1}(x-x_0)),\\
\Pi^N(t,x)&=N^{-1}\pi^N (t_0+ N^{-2}(t-t_0), x_0+ N^{-1}(x-x_0)),\\
\wt{\Pi}_{k,\alpha}^N(t,x)&=\wt{\pi}_{k,\alpha}^N(t_0+ N^{-2}(t-t_0), x_0+ N^{-1}(x-x_0)).
\end{align*}
One can readily check that $(V^N,\Pi^N,(\wt{\Pi}^N_{k,\alpha})_{k,\alpha})$ is a local solution to
\begin{equation*}
\left\{
\begin{aligned}
\partial_t V^N &=-\nabla \Pi^N+  (1+\mu)\Delta V^N+ \nabla \cdot F^N - c_d \mu\sum_{k,\alpha} \theta^N_k\nabla \cdot\big( \nabla \wt{\Pi}_{k,\alpha}^N\otimes \sigma_{-k,\alpha}^{N^{-1}}\big)\\ 
&  + \sqrt{c_d\mu}\sum_{k,\alpha} [-\nabla \wt{\Pi}_{k,\alpha}^N +\theta^{N}_{k}(\sigma_{k,\alpha}^{N^{-1}}\cdot\nabla) V^N+\theta^{N}_{k}\sigma_{k,\alpha}^{N^{-1}} \cdot G^N]\,\dot{\Xi}^{k,\alpha}_t,\\
 \nabla \cdot V^N&=0,\\
  \Delta \Pi^N
&= \nabla^2 : F^N - c_d \mu \, \nabla \cdot \sum_{k,\alpha} 
\theta^N_k\nabla \cdot\big( \nabla \wt{\Pi}_{k,\alpha}^N\otimes \sigma_{-k,\alpha}^{N^{-1}}\big),\\
\Delta \wt{\Pi}_{k,\alpha}^N
&= \theta_k^N\nabla \cdot \big[(\sigma_{k,\alpha}^{N^{-1}}\cdot \nabla)V^N+\sigma_{k,\alpha}^{N^{-1}}\cdot G^N\big],
\end{aligned}
\right.
\end{equation*}
on the parabolic cylinder $Q=Q_{1/2}(t_0,x_0)$, where 
\begin{align*}
\Xi^{k,\alpha}_t &=N \big(W^{k,\alpha}_{t_0+ N^{-2}(t-t_0)}-W^{k,\alpha}_{t_0-N^{-2}/4}\big),\\
\sigma_{k,\alpha}^{N^{-1}}(x) &=\sigma_{k,\alpha}(x_0+ N^{-1}(x-x_0)),\\ 
F^N(t,x) &= N^{-1} F(t_0+ N^{-2}(t-t_0), x_0+ N^{-1}(x-x_0)),\\
G^N(t,x) &= N^{-1} G(t_0+ N^{-2}(t-t_0), x_0+ N^{-1}(x-x_0)),
\end{align*}
and filtration $(\F_{t_0+ N^{-2}(t-t_0)})_{t\geq t_0-1/4}$. In particular, the above rescaled problem is of the form \eqref{eq:turbulent_Stokes_xi_linear} by replacing 
$(\zeta_n)_{n}$ with a suitable rearrangement of $(\sigma_{k,\alpha}^{N^{-1}})_{k,\alpha}$, see \eqref{eq:equivalence_complex_real_noise}. 
Next, we show that the rescaled noise coefficients are uniformly smooth on the large scale $Q=I\times B$ (i.e., the original coefficients are uniformly smooth on $N^{-1}Q$).

\begin{lemma}[Uniform smoothness of noise coefficients at microscopic scale]
\label{l:uniform_smoothness_microscopic_scale}
Let $B=B_{1/2}(x_0)$ be a ball.  
For all $\g>0$, there exists $C>0$ independent of $x_0$ and $N$ such that 
$$
\sum_{k,\alpha} (\theta_k^{N})^2 \|\sigma_{k,\alpha}^{N^{-1}}\|_{C^\g(B)}^2\leq C.
$$
In particular, $(\theta_k^{N}\sigma_{k,\alpha}^{N^{-1}})_{k,\alpha}\in C^\g(B;\ell^2)$ uniformly in $N\geq 1$ for all $\g>0$.
\end{lemma}

\begin{proof}
Without loss of generality, we assume $x_0=0$. Since $\supp\theta^N= \{N\leq |k|\leq 2N\}$, it is enough to show that for all integers $m\geq 1$ and $\alpha\in \{1,\dots,d-1\}$, 
$$
\sup_{N}\sum_{N\leq |k|\leq 2N} (\theta_k^{N})^2 \|\sigma_{k,\alpha}(N^{-1}\cdot)\|_{C^m(B_{1})}^2<\infty.
$$
Clearly, for all $x\in B$, $k\in \supp\theta^N$, integers $m\geq 1$, and all multi-index $\beta\in \N^d$ such that $\sum_{1\leq j\leq d}\beta_j\leq m$, 
$$
|\partial_x^{\beta} [\sigma_{k,\alpha}(N^{-1}x)]|
\leq C \big|\textstyle{\prod_{1\leq j\leq d}}(2\pi \i N^{-1} k_j)^{\beta_j} e^{2\pi \i N^{-1}k\cdot x}\big| \leq C (N^{-1} |k|)^{m} \leq C,
$$
where $C$ is independent of $N$. 
Hence, 
$$
\sum_{N\leq |k|\leq 2N} (\theta_k^{N})^2 \|\sigma_{k,\alpha}(N^{-1}\cdot)\|_{C^m(B_{1})}^2\leq C^2
$$
where we used $\|\theta^N\|_{\ell^2}=1$.
\end{proof}

With Lemma \ref{l:uniform_smoothness_microscopic_scale} at our disposal, we can prove Corollary \ref{cor:SMR_microscopic_estimate}.

\begin{proof}[Proof of Corollary \ref{cor:SMR_microscopic_estimate}]
Since $V^N(t_0,x_0)=v^N(t_0,x_0)$, it follows from \eqref{eq:Besov_embedding_Linfty} that
$$
|v^N (t_0,x_0)|^r
\leq
\sup_{(1/2) I}\|V^N\|_{B^{1-2/p}_{q,p}((1/2)  B)}^r.
$$
Moreover, note that 
\begin{align*}
\| V^N\|_{L^2(Q)}
&\eqsim_d 
\| v^N\|_{\underline{L}^2(N^{-1}Q)},\\
\|\Pi^N\|_{L^2(I;H^{-1}(B))}
&\eqsim_d
\|\pi^N\|_{\underline{L}^2(N^{-1}I;\underline{H}^{-1}(N^{-1}B))},\\
\|(\wt{\Pi}^N_{k,\alpha})_{k,\alpha}\|_{L^2(Q;\ell^2)}&\eqsim_d
\|(\wt{\pi}^N_{k,\alpha})_{k,\alpha}\|_{\underline{L}^2(N^{-1}Q;\ell^2)},
\end{align*}
as well as 
\begin{align*}
\|F^N\|_{L^p(I;L^q(B))}
=N^{-1}\|F\|_{\underline{L}^p(N^{-1}I;\underline{L}^q (N^{-1}B))}
&\leq N^{-1+(2/p+d/q)}\|F\|_{\underline{L}^p(I;\underline{L}^{q}(B))}\\
&\leq \|F\|_{\underline{L}^p(I;\underline{L}^{q}(B))},
\end{align*}
where we used $\frac{2}{p}+\frac{d}{q}<1$. A similar bound holds for $G^N$. 
The estimate \eqref{eq:Linfty_bounds_t0x0_smr} follows from Lemma \ref{l:uniform_smoothness_microscopic_scale} and Theorem \ref{t:localized_SMR} applied to $(V^N, \Pi^N,(\wt{\Pi}^N_{k,\alpha})_{k,\alpha})$ and $\sm=1/2$.
\end{proof}

\subsection{Proof of Theorem \ref{t:localized_SMR}}
\label{ss:localized_SMR_proof}
The proof of the localized stochastic maximal $L^p$-estimate \eqref{eq:localized_SMR0} for local solutions of the stochastic Stokes system \eqref{eq:turbulent_Stokes_xi_linear} relies on an iterative argument, in which at each step we gain a small amount of regularity by shrinking the domain. The iterative regularity gain is divided into three steps:
\begin{itemize}
\item Local gain of smoothness for the velocity field -- Lemma \ref{l:local_smoothing_iteration_velocity}.
\item Local gain of smoothness for the stochastic pressure -- Lemma \ref{l:local_smoothing_iteration_pressure_turbulent}.
\item Local gain of smoothness for the deterministic pressure -- Lemma \ref{l:local_smoothing_iteration_pressure}.
\end{itemize}
This division is dictated by the intrinsic structure of the stochastic Stokes system \eqref{eq:turbulent_Stokes_xi_linear}.
We emphasize that this strategy differs from the one used to prove Caccioppoli inequalities in Section \ref{s:caccioppoli}, as it relies crucially on the additional regularity assumption \eqref{eq:smoothness_coefficients_localized_SMR} enforced in this section.

\subsubsection{Local smoothing for the velocity field}
\label{sss:local_smr_velocity}
In this subsection, we discuss local smoothing effects for the velocity field $v$ in \eqref{eq:turbulent_Stokes_xi_linear}.
The following can be seen as a local version of \cite[Theorem 3.2]{AV21_NS}.

\begin{lemma}[Local smoothing for velocity field -- One step improvement]
\label{l:local_smoothing_iteration_velocity}
Let $Q=Q_{\varrho}(t_0,x_0)$ be a parabolic cylinder of side $\varrho\in (0,1]$ and center $(t_0,x_0)\in \R\times \R^d$. 
Suppose that \eqref{eq:parabolicity_localized_SMR} and \eqref{eq:smoothness_coefficients_localized_SMR} hold.
Fix $r\in (1,\infty)$, $p\in (2,\infty)$ and $\sm\in (0,1)$. Then for all $q_0\in [2,\infty)$ and $q_1\in (q_0,\infty)$ such that
\begin{equation}
\label{eq:improvement_q0_q1_p_local_SMR}
q_1-q_0\leq 1/(dp),
\end{equation}
and all local solutions $(v,\pi,(\wt{\pi}_n)_n)$ to the stochastic Stokes system \eqref{eq:turbulent_Stokes_xi_linear} in $Q$ for which 
\begin{align}
\label{eq:local_smoothing_iteration_velocity1}
F,\, \|(G_n)_n\|_{\ell^2} &\in L^r_\Progress(\O;L^{p}(I;L^{q_1}(B))),\\
\label{eq:local_smoothing_iteration_velocity2}
v&\in L^r(\O;L^{2}(I;H^{1,q_0}(B)))\cap L^r(\O;C(\overline{I};L^{q_0}(B))),\\
\label{eq:local_smoothing_iteration_velocity3}
\pi&\in L^r(\O;L^{2}(I;L^{q_0}(B))),\\
\label{eq:local_smoothing_iteration_velocity4}
(\wt{\pi}_n)_n&\in L^r(\O;L^{2}(I;H^{1,q_0}(B;\ell^2))),
\end{align}
one has 
\begin{equation}
\label{eq:local_smoothing_iteration_velocity_conclusion}
v\in L^r(\O;L^{p}(\sm  I;H^{1,q_1}(\sm B)))\cap L^r(\O;C(\sm\overline{I};B^{1-2/p}_{q_1,p}(\sm B))),
\end{equation} 
with a corresponding estimate depending only on the norms of $F$, $\|(G_n)_n\|_{\ell^2}$, $v$, $\pi$ and $(\wt{\pi}_n)_n$ appearing in \eqref{eq:local_smoothing_iteration_velocity1}-\eqref{eq:local_smoothing_iteration_velocity4}. 
\end{lemma}

Before going into the proof of the above result, we emphasize that, compared to the definition of local solutions (Definition \ref{def:local_solutions_NSE}), no additional time regularity is assumed in \eqref{eq:local_smoothing_iteration_velocity2}-\eqref{eq:local_smoothing_iteration_velocity3}. Indeed, as discussed in Subsection \ref{ss:role_pressure_intro}, the pressure satisfies an elliptic PDE and therefore no time regularization is expected (see also Lemmas \ref{l:local_smoothing_iteration_pressure_turbulent} and \ref{l:local_smoothing_iteration_pressure} below).
However, as we show below, $L^2$-time regularity for the pressures suffices to prove $L^p$-time regularity for the velocity field \eqref{eq:local_smoothing_iteration_velocity_conclusion}.

\begin{proof}
The idea is to use a localization argument together with the stochastic maximal $L^p$-regularity in \cite[Theorem 3.2]{AV21_NS}. 
Without loss of generality, we assume $x_0=(1/2,\dots,1/2)$ and hence $B=B_{\varrho}(x_0)\subseteq \T^d$. 
Let $\phi\in C^{\infty} (Q)$ be such that $\phi|_{\partial_{\pb} Q}=0$ and $\phi|_{\sm Q}=1$, e.g., $\phi=\psi \xi$ where $\psi\in C^{\infty}_{{\rm c}}((t_0-\varrho^2,t_0])$ and $\xi\in C^\infty_{{\rm c}}(B)$ with $\psi|_{(t_0-\eta \varrho^2,t_0]}=\xi|_{\eta B}=1$.
First, from It\^o's formula, the process $\phi v$ solves the SPDE:
\begin{align*}
\partial_t (\phi v) 
&=v (\partial_t \phi)-\nabla (\pi\phi)+\pi\nabla \phi+ \nu\Delta (\phi v)
-2 \nu \nabla \phi\cdot \nabla v -\nu v\Delta \phi+\phi \nabla \cdot F\\
&-\frac{1}{2} \sum_{n} \big[ \nabla \cdot (\nabla(\phi \wt{\pi}_{n})\otimes \zeta_{n})- (\nabla \wt{\pi}_n\otimes \zeta_n)\cdot \nabla\phi  - \nabla\cdot (\nabla \phi \otimes \zeta_n \wt{\pi}_n)\big]\\
&  + \sum_{n} [-\nabla (\phi \wt{\pi}_n)+\wt{\pi}_n\nabla \phi +( \zeta_n\cdot\nabla) (\phi v)- v( \zeta_n\cdot\nabla)\phi+\phi G_n]\,\dot{W}^n_t,
\end{align*}
where the above is understood in its natural integral form on $H^{-1}(B)$, cf.\ \eqref{eq:integral_identity_v_Caccioppoli}. 
Now, since $\supp(\phi(\cdot,t))\subseteq B$ for all $t\in (t_0-\varrho^2/2,t_0]$, all processes appearing in the above identity can be understood as processes with values in distributions on $\T^d$. 
With a slight abuse of notation, in the following, we still denote by $(\zeta_n)_n$ the extension of the coefficients on $\T^d$ which still satisfy \eqref{eq:parabolicity_localized_SMR}-\eqref{eq:smoothness_coefficients_localized_SMR} for some $M>0$ and $\nu_0<\nu$, depending only on $\sm$ and $d$. To construct such an extension, one can proceed as follows. First, consider $\om \zeta_n$ where $\om=1$ on $\supp \xi$ and $\xi$ is as at the beginning of the proof. Then, one can trivially extend $\om\zeta_n$ by periodicity outside $B$. Note that this construction does not preserve the incompressibility of the noise coefficients \eqref{eq:incompressibility_coefficients_localized_SMR}. However, this is not needed to apply the results in \cite[Section 3]{AV21_NS}, where only smoothness and parabolicity are required.  For simplicity, we assume $\varrho=1$ below. The general case is analogous.

\smallskip

Set $w=\p [\phi v]$, where $\p=\p_{\T^d}$ is the Helmholtz projection on $\T^d$ and let $\q=\mathrm{Id}-\p$ be its complementary projection. Recall that $\q=\nabla \qq$ where $\qq$ is as in Subsection \ref{ss:Helmh_div_free}. From the above, $w$ solves the SPDE:
\begin{align}
\label{eq:abstract_spde_localization_w}
\partial_t w 
&= \nu\Delta w+ f_0+ f_1-\frac{1}{2} \sum_{n} \p\big[ \nabla \cdot (\nabla \wt{\Pi}_{n}\otimes \zeta_{n})\big]\\
\nonumber
&  + \sum_{n} \big(\p[( \zeta_n\cdot\nabla) w]+g_{0,n}+g_{1,n}\big)\,\dot{W}^n,
\end{align}
together with the initial condition $w(t_0-1)=0$, where
\begin{align*} 
f_0&=   \p \big[v (\partial_t \phi)+\pi\nabla \phi
-2 \nu \nabla \phi\cdot \nabla v -\nu v\Delta \phi\big]\\
&\qquad-\frac{1}{2} \sum_{n} \p\Big[ \nabla \cdot (\nabla[(\phi \wt{\pi}_{n})-\wt{\Pi}_n- \wt{R}_n]\otimes \zeta_{n})\Big]\\
&\qquad\qquad +\frac{1}{2} \sum_{n} \p\big[ (\nabla \wt{\pi}_n\otimes \zeta_n)\cdot \nabla\phi  + \nabla\cdot (\nabla \phi \otimes \zeta_n \wt{\pi}_n)\big],\\
\wt{\Pi}_n&= \qq [(\zeta_n\cdot \nabla) w],\\
f_1&=  \p \big[\phi \nabla \cdot F\big]
-\frac{1}{2} \sum_{n} \p\big[ \nabla \cdot (\nabla  \wt{R}_n\otimes \zeta_{n})\big],\\
g_{0,n}
&= \p\big[\wt{\pi}_n\nabla \phi +( \zeta_n\cdot\nabla) (\q [\phi v])- v( \zeta_n\cdot\nabla)\phi\big], \qquad
g_{1,n} 
= \p\big[\phi G_n\big],
\end{align*}
and $\wt{R}_n$ is the unique solution to the PDE:
$$
\Delta \wt{R}_n =\phi \nabla\cdot G_n-\textstyle{\int_{\T^d} } \phi \nabla\cdot G_n
\qquad \text{ and }\qquad \textstyle{\int_{\T^d} }\wt{R}_n=0.
$$
The idea behind the above decomposition is as follows. In light of \eqref{eq:local_smoothing_iteration_velocity1}, the forcing terms $(f_0,g_0)$ contain all the contributions to the SPDE \eqref{eq:abstract_spde_localization_w} with time integrability lower than $p$, while all the remaining ones are collected in $(f_1,g_1)$. 
However, all the contributions in $(f_0,g_0)$ are spatially smoother than the ones in $(f_1,g_1)$. In Step 1, we will exploit this fact to obtain a suitable `lifting' of $(f_0,g_0)$ via solving a stochastic heat equation on $\T^d$ having $L^p$-regularity in time, and not only $L^2$ as the pressures \eqref{eq:local_smoothing_iteration_velocity3}-\eqref{eq:local_smoothing_iteration_velocity4}. In Step 2, we conclude the proof of Lemma \ref{l:local_smoothing_iteration_velocity} by applying the global stochastic maximal $L^p$-regularity result in \cite[Theorem 3.2]{AV21_NS}.

\smallskip

\emph{Step 1: Let $w_0$ be the solution to the following stochastic heat equation
$$
\partial_t w_0 = \nu \Delta w_0 +f_0+ \sum_n g_{0,n} \,\dot{W}^n , \qquad w_0(t_0-1)=0,
$$
both on $\T^d$. Then, $w_0\in L^r(\O;L^p(I;H^{1,q_1}(\T^d)))\cap L^r(\O;C(\overline{I};B^{1-2/p}_{q_1,p}(\T^d)))$ with a corresponding estimate provided \eqref{eq:improvement_q0_q1_p_local_SMR} holds.}
We claim that 
\begin{align}
\label{eq:integrability_f0}
f_0&\in L^r(\O;L^2(I;L^{q_0}(\T^d))),\\  
\label{eq:integrability_g0}
(g_{0,n})_n&\in  L^r(\O;L^{2}(I;H^{1,q_0}(\T^d))),
\end{align}
together with an estimate depending only on the norm of $(F,\|(G_n)_n\|_{\ell^2},v,\pi,\wt{\pi}_n)$ in \eqref{eq:local_smoothing_iteration_velocity1}-\eqref{eq:local_smoothing_iteration_velocity4}. 
Before proving \eqref{eq:integrability_f0}-\eqref{eq:integrability_g0}, we show that it yields the claim of Step 1.
Clearly, for $t\in I=(t_0-1,t_0)$, we have $w_0(t)=w_{0,f} (t)+ w_{0,g}(t)$ where
\begin{align*}
w_{0,f} (t)
\stackrel{{\rm def}}{=} \int_{t_0-1}^t e^{\nu\Delta (t-s)} f_0(s)\,\dd s  \quad \text{ and }\quad  w_{0,g}(t) \stackrel{{\rm def}}{=} \int_{t_0-1}^t (e^{\nu\Delta (t-s)}g_{0,n}(s))_n \,\dd \mathcal{W}_{\ell^2}.
\end{align*}
To estimate the above convolutions, let us recall that the operator $-\Delta: H^{\sigma+2,t}(\T^d)\subseteq H^{\sigma,t}(\T^d)\to H^{\sigma,t}(\T^d)$ has a bounded $H^\infty$-calculus of angle $0$ for all $\sigma\in \R$ and $t\in (1,\infty)$ as follows from the periodic version of \cite[Theorem 10.2.25]{Analysis2}. For details on the $H^\infty$-calculus and its role in deterministic and stochastic maximal regularity, the reader is referred to \cite[Chapter 17]{Analysis3} and \cite{MaximalLpregularity}, respectively.
From the deterministic maximal $L^2$-regularity of the Laplace operator on $L^{q_0}(\T^d)$ (see either \cite[Corollary 17.3.6]{Analysis3} or \cite[Theorem 4.4.5]{pruss2016moving}),
\begin{align}
\label{eq:regularity_wg_localized_SMR0}
w_{0,f}
 \in L^r(\O;C(\overline{I};B^{1}_{q_0,2}(\T^d))\cap L^{2}(I;H^{2,q_0}(\T^d))).
\end{align}
Next, recall that $(g_{0,n})_n$ is merely $L^2$-integrable in time, and sharp maximal $L^2$-time regularity estimates are not valid when the process takes values in a Lebesgue space, see \cite[Sections 6 and 7]{MaximalLpregularity}. However, for all $\varepsilon\in (0,1)$, we have 
\begin{align}
\label{eq:regularity_wg_localized_SMR1}
w_{0,g}&\in L^r(\O;C(\overline{I};H^{1,q_0}(\T^d))),\\
\label{eq:regularity_wg_localized_SMR2}
w_{0,g}&\in  L^r(\O;L^{2}(I;H^{2-\varepsilon,q_0}(\T^d))).
\end{align}
Here \eqref{eq:regularity_wg_localized_SMR1} follows from \cite[Theorem 1.1]{VW11_notes}, Young's inequality for convolutions, and the above mentioned $H^{\infty}$-calculus of $-\Delta$ on $L^{q_0}(\T^d)$. Meanwhile, \eqref{eq:regularity_wg_localized_SMR2} follows from \cite[Theorem 4.7]{NVW13} and the well-known smoothing property of the heat operator
$
\|e^{t\Delta}\|_{H^{1,q_0}\to H^{2-\varepsilon,q_0}}\lesssim t^{-(1-\varepsilon)/2}
$ 
for all $t\in (0,1]$. 
In both cases, we used that 
\begin{equation}
\label{eq:gamma_fubini_Bessel}
H^{\sigma,t}(\T^d;\ell^2)=\g(\ell^2,H^{\sigma,t}(\T^d)), \ \  \text{ (equivalent norms)}
\end{equation}
where the space on the right-hand side is the space of $\g$-radonifying operators, see \cite[Chapter 9]{Analysis2}. The identification \eqref{eq:gamma_fubini_Bessel} is a consequence of the $\g$-Fubini theorem, see e.g., \cite[Theorem 9.4.8]{Analysis2} or \cite[Proposition A.2]{AgrSau}.
By \eqref{eq:regularity_wg_localized_SMR0} and \eqref{eq:regularity_wg_localized_SMR1}-\eqref{eq:regularity_wg_localized_SMR2} with $\varepsilon=1/4$, interpolation, and Sobolev embeddings, we obtain 
\begin{equation}
\label{eq:regularity_w01}
w_0 \in L^r(\O;L^{p}(I;H^{1+1/p,q_0}(\T^d)))\embed L^r(\O;L^{p}(I;H^{1,q_1}(\T^d)))
\end{equation}
with a corresponding estimate. Note that the application of the Sobolev embedding is valid provided $1/p-d/q_0\geq -d/q_1$, which holds due to \eqref{eq:improvement_q0_q1_p_local_SMR}. 
Moreover, combining the latter fact, $q_0\geq 2$ and Sobolev embeddings for Besov spaces (see e.g., \cite[Theorem 14.4.19]{Analysis3}), we have
\begin{equation}
\label{eq:regularity_w02}
B^{1}_{q_0,2}(\T^d)\embed H^{1,q_0}(\T^d)\embed B^{1-2/p}_{q_1,p}(\T^d).
\end{equation}
Hence, $w_0 \in L^r(\O;C(\overline{I};B^{1-2/p}_{q_1,p}(\T^d)))$ follows from \eqref{eq:regularity_wg_localized_SMR0} and \eqref{eq:regularity_wg_localized_SMR1}.

\smallskip

The rest of this step is devoted to the proof of \eqref{eq:integrability_f0}-\eqref{eq:integrability_g0}. We begin by proving  \eqref{eq:integrability_g0}.
From \eqref{eq:local_smoothing_iteration_velocity2} and \eqref{eq:local_smoothing_iteration_velocity4} as well as the regularity of the noise coefficients \eqref{eq:smoothness_coefficients_localized_SMR}, it suffices to prove that 
\begin{equation}
\label{eq:claim_integrability_g0_proof}
\q [\phi v]\in L^r(\O;L^2(I;H^{2,q_0}(\T^d))),
\end{equation}
together with a corresponding estimate. It is well-known that $v\mapsto \q[\phi v]$ is a smoothing operator in the case where $v$ is divergence-free on the support of the cutoff function $\phi$ (see \cite[Lemma 3.4 and (3.26)]{AV21_NS} for details):
\begin{equation}
\label{eq:smoothing_property_cutoff_div_free}
\| \q[\phi v]\|_{H^{1+\sigma,t}(\T^d)}\lesssim \|v\|_{H^{\sigma,t}(B)},
\end{equation}
for all $\sigma\in \R$ and $t\in (1,\infty)$ provided $\supp\phi\subseteq B$ and  $\nabla \cdot v=0$ in $\D'(B)$. Thus, \eqref{eq:claim_integrability_g0_proof} is a straightforward consequence of \eqref{eq:local_smoothing_iteration_velocity2} and \eqref{eq:smoothing_property_cutoff_div_free}.

Finally, we prove \eqref{eq:integrability_f0}. 
Using \eqref{eq:local_smoothing_iteration_velocity2}-\eqref{eq:local_smoothing_iteration_velocity4} and \eqref{eq:smoothness_coefficients_localized_SMR}, it suffices to check  
\begin{align}
\label{eq:claim_integrability_f0_proof}
(\phi \wt{\pi}_{n})-\wt{\Pi}_n-\wt{R}_n\in L^r(\O;L^2(I;H^{2,q_0}(\T^d;\ell^2))).
\end{align}
To this end, note that  
$
\Delta \wt{\Pi}_n =\nabla \cdot [(\zeta_n\cdot\nabla )w] 
$
in $\D'(\T^d)$, and from the definition of local solutions to \eqref{eq:turbulent_Stokes_xi_linear} in Definition \ref{def:local_solutions_NSE}, one has 
\begin{align}
\label{eq:PDE_for_difference_turbulent_pressure}
\Delta (\phi \wt{\pi}_n -\wt{\Pi}_n-\wt{R}_n)&=2\nabla \phi \cdot \nabla \wt{\pi}_n + \wt{\pi}_n\Delta \phi
+ \nabla \cdot ((\zeta_n \cdot\nabla) \q[\phi v])\\
\nonumber
&+\textstyle{\int_{\T^d} } \phi \nabla\cdot G_n- \nabla \cdot (v (\zeta_n\cdot \nabla)\phi)- (\nabla \phi )\cdot [(\zeta_n \cdot \nabla) v]  ,
\end{align}
in $\D'(\T^d)$. Let $A_n\stackrel{{\rm def}}{=}\phi \wt{\pi}_n -\wt{\Pi}_n-\wt{R}_n$ and $Z_n$ the right-hand side of the above equality. From the above, it follows that 
$$
A_n= (\mathrm{Id}-\Delta_{\T^d})^{-1}\big(  A_n-Z_n \big).
$$
From \eqref{eq:gamma_fubini_Bessel} and the ideal property of $\g$-radonifying operators (see e.g., \cite[Theorem 9.1.10]{Analysis2}), the above identity implies 
\begin{align*}
&\|(A_n)_n\|_{H^{2,q_0}(\T^d;\ell^2)}
\lesssim \|(Z_n)_n\|_{L^{q_0}(\T^d;\ell^2)}
+\|( A_n )_n\|_{L^{q_0}(\T^d;\ell^2)}\\
&\quad \lesssim \|(Z_n)_n\|_{L^{q_0}(\T^d;\ell^2)}+  \|(\wt{\pi}_n)_n\|_{L^{q_0}(B;\ell^2)}+\| v\|_{H^{1,q_0}(B)}+\|(G_n)_n\|_{L^{q_0}(B;\ell^2)},
\end{align*}
where we also used the definitions of $\widetilde{\Pi}_n$ and $\widetilde R_n$, see the formulas below \eqref{eq:abstract_spde_localization_w}. 

As $Z_n$ denotes the right-hand side of \eqref{eq:PDE_for_difference_turbulent_pressure}, the claim \eqref{eq:claim_integrability_f0_proof} now follows from  \eqref{eq:smoothing_property_cutoff_div_free} and the assumed regularity in \eqref{eq:local_smoothing_iteration_velocity2}-\eqref{eq:local_smoothing_iteration_velocity4}.

\smallskip

\emph{Step 2: It holds that 
$$
w\in L^r(\O;L^p(I;H^{1,q_1}(\T^d)))\cap L^r(\O;C(\overline{I};L^{q_1}(\T^d)))
$$
with a corresponding estimate.} 
Note that $w_1=w-w_0$ solves the following SPDE:
\begin{align*}
\partial_t w_1 
&= \nu \Delta w_1 -\frac{1}{2} \sum_{n} \p\big[ \nabla \cdot (\q[(\zeta_n \cdot \nabla) w_1]\otimes \zeta_{n})\big]+f_2\\
  &+ \sum_{n} \big(\p[( \zeta_n\cdot\nabla) w_1]+g_{2,n}\big)\,\dot{W}^n,
\end{align*}
together with the initial condition $w_1(t_0-1)=0$, and where
\begin{align*}
f_2&=f_1-\frac{1}{2} \sum_{n} \p\big[ \nabla \cdot (\q[(\zeta_n \cdot \nabla) w_0]\otimes \zeta_{n})\big],\\
g_{2,n}&=g_{1,n}+\p[( \zeta_n\cdot\nabla)w_0 ].
\end{align*}
Now, from Step 1 and \eqref{eq:local_smoothing_iteration_velocity1}, it readily follows that 
\begin{align*}
f_2\in L^r(\O;L^{p}(I;\Hs^{-1,q_1}(\T^d))), \qquad   
(g_{2,n})_n\in  L^r(\O;L^{p}(I;\Ls^{q_1}(\T^d))).
\end{align*}
Now, the above and \cite[Theorem 3.2]{AV21_NS} imply
$$
w_1\in L^r(\O;L^p(I;H^{1,q_1}(\T^d)))\cap L^r(\O;C(\overline{I};B^{1-2/p}_{q_1,p}(\T^d)))
$$
with a corresponding estimate. The claim of Step 2 now follows by combining the above and Step 1. 

We point out that \cite[Theorem 3.2]{AV21_NS} is formulated only with $p$-th moments. However, its proof readily extends to the general case by modifying the proof of Step 1 in \cite[Lemma 3.7]{AV21_NS} by allowing for the $r$-th moment instead. The latter case follows by combining the arguments in \cite[Subsection 3.5]{VP18} and \cite[Corollary 7.4]{NVW13}.

\smallskip

\emph{Step 3: Conclusion.} Recall that $\phi|_{\sm Q}=1$ and
\begin{equation}
\label{eq:decomposition_phiv_w_proof}
\phi v = w + \q [\phi v].
\end{equation}
From Step 2, it remains to note that \eqref{eq:smoothing_property_cutoff_div_free}, by interpolation (see e.g., \cite[Theorem 6.4.5(4)]{BeLo} or \cite[Theorem 14.4.31]{Analysis3}), also implies 
\begin{equation*}
\| \q[\phi v]\|_{B^{1+\sigma}_{t,s}(\T^d)}\lesssim \|v\|_{B^{\sigma}_{t,s}(B)} \ \text{
for all $\sigma\in \R$ and $t,s\in (1,\infty)$}.
\end{equation*}
Hence, recalling \eqref{eq:regularity_w02} and combining \eqref{eq:local_smoothing_iteration_velocity2} with the above and  \eqref{eq:smoothing_property_cutoff_div_free}, we obtain 
\begin{align*}
\q [\phi v]
&\in L^r(\O;L^{2}(I;H^{2,q_0}(\T^d)))\cap L^r(\O;C(\overline{I};H^{1,q_0}(\T^d))),\\
&\stackrel{(i)}{\embed} L^r(\O;L^{p}(I;H^{1+2/p,q_0}(\T^d)))\cap L^r(\O;C(\overline{I};B^{1-2/p}_{q_1,p}(\T^d)))\\
&\stackrel{(ii)}{\embed} L^r(\O;L^{p}(I;H^{1,q_1}(\T^d)))\cap L^r(\O;C(\overline{I};B^{1-2/p}_{q_1,p}(\T^d))),
\end{align*}
where $(i)$ follows by interpolation, and $(ii)$ by Sobolev embeddings using \eqref{eq:improvement_q0_q1_p_local_SMR}, once more, see \eqref{eq:regularity_w01} and \eqref{eq:regularity_w02} for a similar argument. The conclusion now follows by combining \eqref{eq:decomposition_phiv_w_proof}, the above, and the result of Step 2.
\end{proof}

\subsubsection{Local smoothing for the deterministic and stochastic pressures}
As outlined at the beginning of this section, we first prove local smoothing for the pressures, beginning with the stochastic pressure $(\wt{\pi}_{n})_n$. 

\begin{lemma}[Local smoothing for the stochastic pressures -- One step improvement]
\label{l:local_smoothing_iteration_pressure_turbulent}
Let $Q=Q_{\varrho}(t_0,x_0)$ be a parabolic cylinder of side $\varrho\in (0,1]$ and center $(t_0,x_0)\in \R\times \R^d$. 
Suppose that \eqref{eq:smoothness_coefficients_localized_SMR} holds. Fix $r\in (1,\infty)$ and $\sm\in (0,1)$.
Let $q_0\in [2,\infty)$ and $q_1\in (q_0,\infty)$ be such that $q_1-q_0\leq 1/d$. Then, for all local solutions $(v,\pi,(\wt{\pi}_n)_n)$ to the stochastic Stokes system \eqref{eq:turbulent_Stokes_xi_linear} in $Q$ satisfying 
\begin{align}
\label{eq:local_smoothing_iteration_pressuret1}
\|(G_n)_n\|_{\ell^2} &\in L^r(\O;L^{2}(I;L^{q_1}(B))),\\
\label{eq:local_smoothing_iteration_pressuret2}
v&\in L^r(\O;L^{2}(I;H^{1,q_1}(B))),\\
\label{eq:local_smoothing_iteration_pressuret3}
(\wt{\pi}_n)_n&\in L^r(\O;L^{2}(I;H^{1,q_0}(B;\ell^2))),
\end{align}
it holds that
$$
(\wt{\pi}_n)_n\in L^r(\O;L^{2}(I;H^{1,q_1}(\sm B;\ell^2)))
$$
with a corresponding estimate depending only on the norms of $\|(G_n)_n\|_{\ell^2}$, $v$ and $(\wt{\pi}_n)_n$ appearing in \eqref{eq:local_smoothing_iteration_pressuret1}-\eqref{eq:local_smoothing_iteration_pressuret3}. 
\end{lemma}

\begin{proof}
Recall from Definition \ref{def:local_solutions_NSE} that, a.e.\ on $I\times \O$, 
$$
\Delta \wt{\pi}_n = \nabla \cdot [(\zeta_n \cdot \nabla)v]+ \nabla \cdot G_n \  \ \text{ in }\ \D'(B).
$$
Let $\phi\in C^{\infty}_{{\rm c}}(B)$ be such that $\phi=1$ on $\sm B$ and $\|\phi\|_{W^{2,\infty}(B)}\lesssim_\sm 1$. Then
$$
\Delta(\phi \wt{\pi}_n )=\phi \nabla \cdot [(\zeta_n \cdot \nabla)v] +\phi \nabla \cdot G_n+ 2\nabla \phi \cdot \nabla \wt{\pi}_n + \wt{\pi}_n \Delta \phi.
$$
Arguing as below \eqref{eq:PDE_for_difference_turbulent_pressure}, it holds, a.e.\ on $I\times \O$, that
\begin{align*}
&\|(\phi \wt{\pi}_n)_n\|_{H^{1,q_1}(\T^d;\ell^2)}\\
&\lesssim \|(\phi \nabla \cdot [(\zeta_n \cdot \nabla)v] )_n\|_{H^{-1,q_1}(\T^d;\ell^2)}+\|(\phi \nabla \cdot G_n)_n\|
_{H^{-1,q_1}(\T^d;\ell^2)}+\|(\wt{\pi}_n)_n\|_{L^{q_1}(B;\ell^2)}\\
&\lesssim \|v \|_{H^{1,q_1}(B)}+\|(G_n)_n\|_{L^{q_1}(B;\ell^2)}
+\|(\wt{\pi}_n)_n\|_{H^{1,q_0}(B;\ell^2)},
\end{align*}
where in the last estimate we used \eqref{eq:smoothness_coefficients_localized_SMR} and the Sobolev embedding $L^{q_0}(B;\ell^2)\embed H^{-1,q_1}(B;\ell^2)$ as $q_1-q_0\leq 1/d$. The conclusion follows by taking the $L^r(\O;L^2(\sm I))$-norm in the above and using that $\phi=1$ on $\sm B$.
\end{proof}

The last ingredient in the proof of Theorem \ref{t:localized_SMR} is the following local smoothing result for the deterministic pressure.

\begin{lemma}[Local smoothing for the deterministic pressure --  One step improvement]
\label{l:local_smoothing_iteration_pressure}
Let $Q=Q_{\varrho}(t_0,x_0)$ be a parabolic cylinder of side $\varrho\in (0,1]$ and center $(t_0,x_0)\in \R\times \R^d$. 
Suppose that \eqref{eq:smoothness_coefficients_localized_SMR} holds. Fix $r\in (1,\infty)$ and $\sm\in (0,1)$.
Let $q_0\in [2,\infty)$ and $q_1\in (q_0,\infty)$ be such that $q_1-q_0\leq 1/d$. Then, for all local solutions $(v,\pi,(\wt{\pi}_n)_n)$ to the stochastic Stokes system \eqref{eq:turbulent_Stokes_xi_linear} in $Q$ satisfying 
\begin{align}
\label{eq:local_smoothing_iteration_pressure1}
F &\in L^r(\O;L^{2}(I;L^{q_1}(B))),\\
\label{eq:local_smoothing_iteration_pressure2}
(\wt{\pi}_n)_n&\in L^r(\O;L^{2}(I;H^{1,q_1}(B;\ell^2))),\\
\label{eq:local_smoothing_iteration_pressure3}
\pi &\in L^r(\O;L^{2}(I;L^{q_0}(B))),
\end{align}
it holds that
$$
\pi \in L^r(\O;L^{2}(I;L^{q_1}(\sm B)))
$$
with a corresponding estimate depending only on the norms of $F$, $\pi$ and $(\wt{\pi}_n)_n$ appearing in \eqref{eq:local_smoothing_iteration_pressure1}-\eqref{eq:local_smoothing_iteration_pressure3}. 
\end{lemma}

\begin{proof}
Fix $\phi\in C^{\infty}_{{\rm c}}(B)$ such that $\phi=1$ on $\sm B$. Since (see Definition \ref{def:local_solutions_NSE})  
$$
\Delta \pi =\nabla^2 : F - \frac{1}{2} \nabla\cdot \sum_{n} \nabla \cdot (\nabla \wt{\pi}_n \otimes \zeta_n).
$$
a.e.\ on $I\times \O$ in $\D'(B)$, we have 
\begin{align*}
\Delta(\phi \pi )
=\phi\nabla^2 : F - \frac{1}{2} \phi\nabla \cdot \sum_{n} \nabla \cdot (\nabla \wt{\pi}_n \otimes \zeta_n)
+
 2\nabla \phi \cdot \nabla \pi + \pi \Delta \phi.
\end{align*}
As in Lemma \ref{l:intermediate_estimate_pressure_deterministic} and in \eqref{eq:PDE_for_difference_turbulent_pressure}, by the maximal $L^{q_1}$-regularity of the Laplacian on the torus, it holds that 
\begin{align*}
\| \pi\|_{L^{q_1}(\sm B)}
\leq 
\|\phi \pi \|_{L^{q_1}(\T^d)}
&\lesssim \|F\|_{L^{q_1}(B)}+ \|\pi\|_{H^{-1,q_1}(B)}+ \Big\|\sum_{n\geq 1} \nabla \wt{\pi}_n \otimes \zeta_n\Big\|_{L^{q_1}(B)},
\end{align*}
a.e.\ on $I\times \O$.
Now, as in the proof of Lemma \ref{l:local_smoothing_iteration_pressure_turbulent}, the embedding $L^{q_0}(B)\embed H^{-1,q_1}(B)$ and the assumption \eqref{eq:smoothness_coefficients_localized_SMR} imply
\begin{align*}
\| \pi \|_{L^{q_1}(\sm B)}
&\lesssim \|F\|_{L^{q_1}(B)}+ \|\pi\|_{H^{-1,q_1}(B)}+ \|(\nabla \wt{\pi}_n )_n\|_{L^{q_1}(B;\ell^2)}\\
&\lesssim \|F\|_{L^{q_1}(B)}+ \|\pi\|_{L^{q_0}(B)}+ \|( \wt{\pi}_n)_n \|_{H^{1,q_1}(B;\ell^2)}
\end{align*}
a.e.\ on $I\times \O$. 
The claim follows by taking $L^r(\O;L^2(I))$-norms in the above.
\end{proof}

\subsubsection{Proof of Theorem \ref{t:localized_SMR}}
\label{sss:proof_localized_smr}
Given Lemmas \ref{l:local_smoothing_iteration_velocity}--\ref{l:local_smoothing_iteration_pressure}, Theorem \ref{t:localized_SMR} follows by a standard iteration argument, exploiting the local smoothing effects previously established on shrinking parabolic cylinders.

\begin{proof}[Proof of Theorem \ref{t:localized_SMR}]
We begin by recalling that, by assumption
\begin{equation}
\label{eq:assumption_iteration_proof}
F,\, \|(G_n)_n\|_{\ell^2} \in L^r_{\Progress}(\O;L^{p}(I;L^{q}(B))).
\end{equation}
Moreover, letting $\sm_0\stackrel{{\rm def}}{=}\frac{1}{2}(1+\sm)>\sm$, Theorem \ref{t:caccioppoli} ensures 
\begin{align}
\label{eq:caccioppoli_sm1}
\left\{
\begin{aligned}
v&\in L^r(\O;L^{2}(\sm_0 I;H^{1}(\sm_0 B)))\cap L^r(\O;C(\sm_0 \overline{I};L^2(\sm_0  B))),\\
\pi&\in L^r(\O;L^{2}(\sm_0 I;L^{2}(\sm_0 B))),\\
(\wt{\pi}_n)_n&\in L^r(\O;L^{2}(\sm_0 I;H^{1}(\sm_0 B;\ell^2))),
\end{aligned}
\right.
\end{align}
with corresponding bounds in terms of the norms
\begin{align}
\label{eq:caccioppoli_sm10}
\max \Big\{ \big(\E \|v\|_{L^2( Q)}^r\big)^{1/r}, & \ 
\big(\E \|\pi\|_{L^2(I;H^{-1}( B))}^r\big)^{1/r},\ 
\big(\E \|( \wt{\pi}_{n} )_{n}\|_{L^{2}( Q;\ell^2)}^r \big)^{1/r}, \ \\
\nonumber
&\ \big(\E \|F \|_{L^p( I;L^q(B))}^r\big)^{1/r},\
\big(\E \|(G_n)_n \|_{L^p(I;L^q( B;\ell^2))}^r \big)^{1/r} \Big\},
\end{align}
as $p,q\in (2,\infty)$.
We now bootstrap further regularity by means of Lemmas \ref{l:local_smoothing_iteration_velocity}-\ref{l:local_smoothing_iteration_pressure} on shrinking parabolic cylinders. To this end, let $K\geq 1$ be an integer such that $2+ \frac{K}{dp}>q$, and fix a collection $(\sm_k)_{k=1}^K$ of numbers satisfying 
$$
\sm=\sm_K<\sm_{K-1}<\dots <\sm_0=\tfrac{1}{2}(1+\sm),
$$ 
e.g., $\sm_k =\sm+ \frac{K-k}{2K}(1-\sm)$. 
Moreover, let $q_0=2$ and $q_k = (q_{k-1}+ \frac{1}{dp})\wedge q$ for $k\in \{1,\dots,K\}$. 
Note that the assumption \eqref{eq:assumption_iteration_proof} and Lemmas \ref{l:local_smoothing_iteration_velocity}-\ref{l:local_smoothing_iteration_pressure} show that 
\begin{align}
\label{eq:assumption_iteration_sm1}
\left\{
\begin{aligned}
v&\in L^r(\O;L^{2}(\sm_{k-1}I;H^{1,q_{k-1}}(\sm_{k-1} B)))\\
&\qquad \qquad \qquad\qquad\cap L^r(\O;C(\sm_{k-1}\overline{I};L^{q_{k-1}}(\sm_{k-1} B))),\\
\pi&\in L^r(\O;L^{2}(\sm_{k-1}I;L^{q_{k-1}}(\sm_{k-1} B))),\\
(\wt{\pi}_n)_n&\in L^r(\O;L^{2}(\sm_{k-1}I;H^{1,q_{k-1}}(\sm_{k-1} B;\ell^2))),
\end{aligned}
\right.
\end{align}
implies 
\begin{align}
\label{eq:assumption_iteration_sm11}
\left\{
\begin{aligned}
v&\in L^r(\O;L^{p}(\sm_{k}I;H^{1,q_{k}}(\sm_k B)))\cap L^r(\O;C(\sm_k\overline{I};B^{1-2/p}_{q_{k},p}(\sm_k B))),\\
\pi&\in L^r(\O;L^{2}(\sm_{k}I;L^{q_{k}}(\sm_k B))),\\
(\wt{\pi}_n)_n&\in L^r(\O;L^{2}(\sm_{k}I;H^{1,q_{k}}(\sm_k B;\ell^2))).
\end{aligned}
\right.
\end{align}
Clearly, \eqref{eq:caccioppoli_sm1} implies \eqref{eq:assumption_iteration_sm1} for $k=1$. 
Now, from the elementary embedding 
$$
B^{1-2/\wh{p}}_{\wh{q},\wh{p}}(B)\embed L^{\wh{q}}(B),
$$ 
which holds for all $\wh{q}\in [2,\infty)$, $\wh{p}\in (2,\infty)$ and balls $B\subseteq \R^d$, the conclusion
\eqref{eq:assumption_iteration_sm11} at step $k$ provides the assumptions in \eqref{eq:assumption_iteration_sm1} at step $k+1$. Hence, by induction, \eqref{eq:assumption_iteration_sm11} holds for every $k\in\{1,\dots,K\}$. Note that the case $k=K$ of \eqref{eq:assumption_iteration_sm11} provides the desired result in Theorem \ref{t:localized_SMR}. Keeping track of the constants in the above iteration and the fact that the initial bound depends only on $(v,\pi, (\wt{\pi}_n)_n,F,G)$ via the norms in \eqref{eq:caccioppoli_sm10}, this completes the proof of \eqref{eq:localized_SMR0}.
\end{proof}

\section{Quantitative localized scaling limits with rescaled coefficients}
\label{s:quantitative_loc_scaling}
In this section, we establish the convergence of local solutions to the stochastic Stokes system with oscillating rescaled coefficients on a parabolic cylinder $2Q$,
\begin{equation}
\label{eq:turbulent_Stokes_scaling_quantitative}
\left\{
\begin{aligned}
\partial_t v^{N,\delta} &=-\nabla \pi^{N,\delta}+ (1+\mu) \Delta v^{N,\delta}+ \nabla \cdot F- c_d \mu \sum_{k,\alpha} \theta_k^N\nabla\cdot  (\nabla \wt{\pi}^{N,\delta}_{k,\alpha}\otimes \sigma^\delta_{-k,\alpha})\\ 
&\qquad\ \  + \sqrt{c_d \mu}\sum_{k,\alpha} [-\nabla \wt{\pi}_{k,\alpha}^{N,\delta} +\theta^{N}_{k}(\sigma^\delta_{k,\alpha}\cdot\nabla) v^{N,\delta}+\theta^{N}_{k}\sigma^\delta_{k,\alpha}\cdot G]\,\dot{W}^{k,\alpha}_t,\\
 \nabla \cdot v^{N,\delta}&=0,\\
\Delta \pi^{N,\delta}
&= \nabla^2 : F -  c_d \mu \nabla \cdot \Big[\sum_{k,\alpha}
\theta_k^N \nabla \cdot (\nabla \wt{\pi}_{k,\alpha}^{N,\delta}\otimes \sigma^\delta_{-k,\alpha})\Big] ,\\
\Delta \wt{\pi}_{k,\alpha}^{N,\delta}
&=\theta^{N}_{k} \nabla \cdot \big[(\sigma^\delta_{k,\alpha}\cdot \nabla)v^{N,\delta}+\sigma^\delta_{k,\alpha}\cdot G\big].
 \end{aligned}
\right.
\end{equation}
where $\theta^N=(\theta_k^N)_{k}$ is as in \eqref{eq:choice_thetaN}, and $\sigma^{\delta}_{k,\alpha}$ denotes the rescaled coefficients around the spatial center of $Q$, i.e.,
$$
\sigma^\delta_{k,\alpha}(x)= \sigma_{k,\alpha}(x_0+ \delta (x-x_0)) \quad \text{ for } \delta\in (0,1] \ \text{ and }\ x\in 2B;
$$ 
towards its \emph{homogenized counterpart} $(\uh^{N,\delta},\ph^{N,\delta})$ on the smaller cylinder $Q=I\times B$,
\begin{equation}
\label{eq:homogenized_PDE_linear}
\left\{
\begin{aligned}
\partial_t \uh^{N,\delta}&=-\nabla \ph^{N,\delta}+ (1+D_d \mu) \Delta \uh^{N,\delta}\\
&\qquad\qquad\qquad+  \nabla \cdot (F-(1-D_d) \mu\, G^\top)  & \quad \text{ on }&Q,\\
 \nabla \cdot \uh^{N,\delta}&=0 & \text{ on }&Q, \\
\uh^{N,\delta}&= v^{N,\delta} -\fint_{B} v^{N,\delta}& \text{ on }&\partial_{\pb} Q, \\
\int_{B} \ph^{N,\delta} \varphi
&=0& \text{ on }&I.
\end{aligned}
\right.
\end{equation}
In the above, $\partial_{\pb}$ denotes the parabolic boundary of $Q$, $\varphi\in C_{{\rm c}}^\infty(B)$ is a given function such that $\int_{B} \varphi=1$, and 
$$
D_d \stackrel{{\rm def}}{=}\frac{d^2-3}{(d-1)(d+2)}\in (0,1).
$$
The condition $\int_{B}\ph^{N,\delta}\varphi=0$ on $I$ will be understood in a distributional sense, and is a standard normalization condition to ensure uniqueness for the pressure.
Note that $D_3=\frac{3}{5}$ and $D_2 =\frac{1}{4}$ correspond to the values already obtained in \cite{FL19,L23_enhanced}.

\smallskip

Before stating the main result of this section, we first discuss the well-posedness of \eqref{eq:homogenized_PDE_linear}. 
Since $(v^{N,\delta},\pi^{N,\delta},(\wt{\pi}^{N,\delta}_{k,\alpha})_{k,\alpha})$ is a local solution to the oscillating stochastic Stokes system \eqref{eq:turbulent_Stokes_scaling_quantitative}, it has limited time regularity. Thus, the problem cannot be recast in the usual variational setting (see e.g., \cite[Theorem 10.9]{B11}). However, recalling  
\begin{equation}
\label{eq:interpolation_and_boundary_uNdelta}
v^{N,\delta} \in  L^2(I;H^1(B)),
\end{equation}
and the boundedness of the trace operator $H^{1}(B) \ni f \mapsto f|_{\partial B}\in L^2(\partial B)$ (see e.g., \cite[Proposition 4.5, Chapter 4]{TayPDE1}), solutions to \eqref{eq:homogenized_PDE_linear} can be defined via the transposition method of Lions-Magenes \cite{LM1_book,LM2_book}, see Definition \ref{def:solution_homogenized} and Remark \ref{r:definition_homogenized}.
From the latter notion, we obtain the following well-posedness result for \eqref{eq:homogenized_PDE_linear}.

\begin{proposition}[Well-posedness of the homogenized Stokes system]
\label{prop:local_well_posedness_homogenized_SPDE}
Let $\mu>0$ and $r\in (1,\infty)$. Let $Q$ be a parabolic cylinder with side length $\varrho\in (0,1/4]$ and center $(t_0,x_0)\in \R_+\times \R^d$.
Assume that 
\begin{equation}
\label{eq:assumption_basic_wellposedness_homogenized}
F,G \in L^r_{\Progress}(\O;L^2(2Q;\R^{d\times d})).
\end{equation}
Let $(v^{N,\delta},\pi^{N,\delta},(\wt{\pi}_{k,\alpha}^{N,\delta})_{k,\alpha})$ be a local solution to the oscillating stochastic Stokes system \eqref{eq:turbulent_Stokes_scaling_quantitative} on $2Q$ for some $N\geq 1$ and $\delta\in (0,1]$ (see Definition \ref{def:local_solutions_NSE}). 
Then there exists a unique solution $(\uh^{N,\delta},\ph^{N,\delta})$ to the homogenized Stokes system \eqref{eq:homogenized_PDE_linear} in the transposition sense (see Definition \ref{def:solution_homogenized} below). Moreover 
\begin{align}
\nonumber
&\big(\E\|\uh^{N,\delta}\|_{L^\infty(I;L^2(B))}^r \big)^{1/r}
+
\big(\E\|\nabla \uh^{N,\delta}\|_{L^2(Q)}^r\big)^{1/r}+ \big(\E\|\ph^{N,\delta}\|_{H^{-1}(I;H^{-1}(B))}^r\big)^{1/r}\\
\nonumber
&\qquad\qquad \lesssim
\Big(\E\Big\|v^{N,\delta}-\fint_{B} v^{N,\delta}\Big\|_{L^\infty(I;L^2(B))}^r\Big)^{1/r} +
\big(\E\|\nabla v^{N,\delta}\|_{L^2(Q)}^r\big)^{1/r}\\ 
&\qquad\qquad
+\big(\E \|(\nabla \wt{\pi}_{k,\alpha}^{N,\delta})_{k,\alpha}\|_{L^2(Q;\ell^2)}^r\big)^{1/r}+ \big(\E\|F\|_{L^2(Q)}^r\big)^{1/r}+\big(\E\|G\|_{L^2(Q)}^r\big)^{1/r}, 
\label{eq:local_well_posedness_homogenized_SPDE_estimate}
\end{align}
where the implicit constant depends only on $\mu,r,\varrho$ and $d$. If, in addition to the above assumptions, it holds that $F,G \in L^r(\O;L^p(2I;L^2(2B)))$ for some $p\in (2,\infty)$, then 
\begin{align}
\nonumber
\big(\E\|\ph^{N,\delta}\|_{H^{-1,p}(I;H^{-1}(B))}^r\big)^{1/r}
 \lesssim
\Big(\E\Big\|v^{N,\delta}-\fint_{B} v^{N,\delta}\Big\|_{L^\infty(I;L^2(B))}^r\Big)^{1/r} +
\big(\E\|\nabla v^{N,\delta}\|_{L^2(Q)}^r\big)^{1/r}&\\ 
+\big(\E \|(\nabla \wt{\pi}_{k,\alpha}^{N,\delta})_{k,\alpha}\|_{L^2(Q;\ell^2)}^r\big)^{1/r}+ \big(\E\|F\|_{L^p(I;L^2(B))}^r\big)^{1/r}+\big(\E\|G\|_{L^p(I;L^2(B))}^r\big)^{1/r}&, 
\label{eq:local_well_posedness_homogenized_SPDE_estimate2}
\end{align}
where the implicit constant depends only on $\mu,p,r,\varrho$ and $d$. 
\end{proposition}

From the stochastic Caccioppoli inequality of Theorem \ref{t:caccioppoli}, the right-hand side of \eqref{eq:local_well_posedness_homogenized_SPDE_estimate} can be further estimated by the sharp energy $\Eno^{\#}_{2Q}$ of the local solution $(v^{N,\delta},\pi^{N,\delta},(\wt{\pi}^{N,\delta}_{k,\alpha})_{k,\alpha})$ on $2Q$ defined in \eqref{eq:def_En_Q_caccioppoli}, which contains lower-order norms of this local solution than those appearing in \eqref{eq:local_well_posedness_homogenized_SPDE_estimate}. A similar remark holds for \eqref{eq:local_well_posedness_homogenized_SPDE_estimate2}.

\smallskip

The following is the main result of this section.
Recall that, for a ball $B$ with center $x_0\in \R^d$, we set $\T^d_{B}= x_0+[-1/2,1/2)^d$ with periodic boundary conditions.

\begin{theorem}[Quantitative localized scaling limit]
\label{t:universal_scaling_limit}
Let $Q=I\times B$ be a parabolic cylinder with side length $\varrho\leq 1/4$ and center $(t_0,x_0)\in \R_+\times \R^d$. Suppose that \eqref{eq:assumption_basic_wellposedness_homogenized} holds and that, a.e.\ on $I\times \O$,
\begin{equation}
\label{eq:assumption_G_trace_zero_quantiative}
\Tr(G)=0\text{ on }2B, \quad \text{ and }\quad \nabla \cdot G =0\text{ in }\D'(2B).
\end{equation}
Fix $p_0,p_1,r\in [2,\infty)$ and $ \sigma>0$. Let $p_2\in [2,\infty)$ satisfy either $p_2>p_1$ or $p_2=p_1=2$. 
Let  $\chi\in C^{\infty}_{{\rm c}}(2B)$ be a cutoff function such that $\chi=1$ on $\kappa B$ for some $\kappa\in (1,2)$.
Then there exist constants $C,\g>0$ depending only on $\varrho, p_0,p_1,p_2,r,\sigma,d,\mu$ and $\chi$ such that for any $N\geq 1$ and $\delta\in (0,1]$ satisfying $\delta N\geq 1$, the following estimates hold: 
\begin{align*}
\Big(\E\Big\| \uh^{N,\delta}-\Big(v^{N,\delta}-\fint_B v^{N,\delta}\Big)\Big\|_{L^{p_0}(I;L^2(B))}^r \Big)^{1/r}
&\leq \frac{C}{(\delta N)^{\g}} \wt{\Eno}^{\#}_{Q},\\
\big(\E\|\nabla \Pi^{N,\delta}-\nabla \Pi^0\|_{L^2(I;H^{-1-\sigma}(\T^d_B))}^r\big)^{1/r} 
&\leq \frac{C}{(\delta N)^{\g}} \wt{\Eno}^{\#}_{Q},\\
\big(\E\|(\nabla\wt{\Pi}_{k,\alpha}^{N,\delta})_{k,\alpha}\|_{L^{p_1}(I;H^{-1}(\T^d_B;\ell^2))}^r\big)^{1/r} 
&\leq \frac{C}{(\delta N)^{\g}} \wt{\Eno}^{\#}_{Q},
\end{align*}
where $(v^{N,\delta},\pi^{N,\delta},(\wt{\pi}^{N,\delta}_{k,\alpha})_{k,\alpha})$ and $(\uh^{N,\delta},\ph^{N,\delta})$ are the local solutions to the stochastic and homogenized Stokes systems \eqref{eq:turbulent_Stokes_scaling_quantitative} and \eqref{eq:homogenized_PDE_linear}, respectively, and $\wt{\Eno}^{\#}_{Q}$ is the sharp modified energy of the local solution $(v^{N,\delta},\pi^{N,\delta},(\wt{\pi}^{N,\delta}_{k,\alpha})_{k,\alpha})$ on $2Q$ given by
\begin{align}
\label{eq:modified_sharp_energy_convergence_rate}
\wt{\Eno}^{\#}_{Q}&\stackrel{{\rm def}}{=}
\Big(\E\Big\|v^{N,\delta}-\fint_{2B}v^{N,\delta}\Big\|_{L^2(2Q)}^r \Big)^{1/r}
+ \big(\E \| \pi^{N,\delta}\|_{L^2(2I;H^{-1}(2B))}^r\big)^{1/r}\\
\nonumber
 &+ \big(\E \|(\wt{\pi}_{k,\alpha}^{N,\delta})_{k,\alpha}\|_{L^2(2Q;\ell^2)}^r\big)^{1/r}+\big(\E\|F\|_{L^2(2Q)}^r\big)^{1/r}
+\big(\E\|G\|_{L^{p_2}(2I;L^2(2B))}^r\big)^{1/r},
\end{align}
and $\Pi^{N,\delta}$, $\wt{\Pi}^{N,\delta}_{k,\alpha}$, and $\Pi^0$ denote the local approximation of the deterministic, stochastic local, and homogenized pressures, respectively, i.e.,
\begin{align}
\label{eq:approximated_pressure_1_quantitative_convergence1}
\Pi^{N,\delta}&\stackrel{{\rm def}}{=}\qq_{\T^d_B} \Big[\chi\big[ \nabla \cdot (F+(1-D_d)\mu \nabla v^{N,\delta})\\
&\qquad \qquad \qquad \qquad 
\nonumber
- c_d \mu \sum_{k,\alpha} \theta_k^N\nabla\cdot  (\nabla \wt{\pi}^{N,\delta}_{k,\alpha}\otimes \sigma^\delta_{-k,\alpha} \big)\big]\Big],\\
\label{eq:approximated_pressure_1_quantitative_convergence2}
\wt{\Pi}_{k,\alpha}^{N,\delta}
&\stackrel{{\rm def}}{=}\theta^N_k \qq_{\T^d_B} \Big[\chi\big((\sigma^{\delta}_{k,\alpha}\cdot \nabla)v^{N,\delta} + \sigma^{\delta}_{k,\alpha} \cdot G\big)\Big],\\
\label{eq:approximated_pressure_1_quantitative_convergence3}
\Pi^0&\stackrel{{\rm def}}{=} \qq_{\T^d_B} \Big[\chi\big( \nabla \cdot ( F-(1-D_d)\mu \,G^\top)\big)\Big],
\end{align}
where $\qq_{\T^d_B}$ is as in Subsection \ref{ss:Helmh_div_free}.
\end{theorem} 

The key point in the above is that the rate of convergence depends only on the cumulative frequency of oscillation $L=\delta N$ (see Subsection \ref{ss:control_energy_intro}). 
Recall from the comments at the beginning of Section \ref{s:proof_strategy} and in the last part of Subsection \ref{ss:control_energy_intro} that for our purposes, it is crucial to have an explicit rate of convergence, as the qualitative arguments in \cite{AL87_compactness} do not extend in our setting. 

\smallskip

Before proceeding to the proofs, we observe that since $\chi|_{\kappa B}=1$ and $\nabla \cdot v^{N,\delta}=0$, the local approximations $\Pi^{N,\delta}$, $(\wt{\Pi}^{N,\delta}_{k,\alpha})_{k,\alpha}$, and $\Pi^0$ defined in \eqref{eq:approximated_pressure_1_quantitative_convergence1}-\eqref{eq:approximated_pressure_1_quantitative_convergence3} satisfy the same equations on $\kappa B$ as the true pressures $\pi^{N,\delta}$, $(\wt{\pi}_{k,\alpha}^{N,\delta})_{k,\alpha}$, and $\ph^{N,\delta}$. Namely, in $\D'(\kappa B)$ and a.e.\ on $I\times \O$, it holds that:
\begin{align}
\label{eq:approximated_pressure_1_quantitative_convergence1PDE}
\Delta \Pi^{N,\delta} 
&=\nabla^2: F- c_d \mu \nabla \cdot \Big[\sum_{k,\alpha} \theta_k^N\nabla\cdot  (\nabla \wt{\pi}^{N,\delta}_{k,\alpha}\otimes \sigma^\delta_{-k,\alpha} \big)\Big],\\
\label{eq:approximated_pressure_1_quantitative_convergence2PDE}
\Delta \wt{\Pi}_{k,\alpha}^{N,\delta}
&=\theta^{N}_{k} \nabla \cdot \Big[(\sigma^\delta_{k,\alpha}\cdot \nabla)v^{N,\delta} +\sigma^\delta_{k,\alpha}\cdot G\Big],\\
\label{eq:approximated_pressure_1_quantitative_convergence3PDE}
\Delta \Pi^0
&=\nabla^2:  \big[F-(1-D_d)\mu \,G^\top\big].
\end{align}
As discussed in Subsection \ref{ss:role_pressure_intro}, this fact will be of central use in Section \ref{s:energy_control_microscopic} below; see the proof of the one-step improvement in Lemma \ref{l:one_step_improvement}.

\subsection{The Stokes system with rough boundary data}
\label{ss:linear_setup}
In this subsection, we prove well-posedness for the Stokes system with rough boundary data:
\begin{equation}
\label{eq:homogenized_PDE_linear_model_problem}
\left\{
\begin{aligned}
\partial_t v &=-\nabla \Pi+\kappa \Delta v+  \nabla \cdot R & \text{ on }&Q,\\
 \nabla \cdot v&=0 & \text{ on }&Q, \\
v&= \xi & \text{ on }&\partial_{\pb}Q, \\
\int_{B} \Pi \varphi
&=0& \text{ on }&I.
\end{aligned}
\right.
\end{equation}
The aim of this section is to set the stage for the proof of Proposition \ref{prop:local_well_posedness_homogenized_SPDE}, which is given in Subsection \ref{ss:local_well_posedness_homogenized_SPDE} below. As in \eqref{eq:homogenized_PDE_linear}, $\varphi$ is a given test function, and $\int_{B} \Pi \varphi=0$ on $I$ is a normalization condition.

Throughout this subsection, without loss of generality, we assume $Q=Q_1(1,x_0)$ for some $x_0\in \R^d$. In particular, $I=(0,1)$, $B=B_1(x_0)$ and $\partial_\pb Q= \{0\}\times B \cup I\times \partial B$. On the boundary data $\xi$ and the forcing $R$, we enforce the following 

\begin{assumption}
\label{ass:homogenized_PDE_linear_model_problem}
The function
$\xi(0)=\xi|_{\{0\}\times B}\in L^2(B;\R^d)$ satisfies $\nabla \cdot \xi(0)=0$ in $\D'(B)$. Moreover
$$
R \in L^2(Q)\qquad \text{ and }\qquad \xi|_{I \times \partial B}\in L^2(I\times \partial B).
$$ 
Finally, $\xi$ satisfies the following compatibility condition 
\begin{equation}
\label{eq:condition_div_free_integral}
\int_{\partial B} \xi \cdot n =0 \text{ for a.e. }t\in I,
\end{equation} 
where $n$ is the exterior normal field on $\partial B$.
\end{assumption}

The condition \eqref{eq:condition_div_free_integral} is a natural compatibility condition to allow for the divergence-free nature of $v$. Indeed, if $\wh{\xi}=\wh{v}|_{\partial B}$ for some divergence-free vector field $\wh{v}$ in $H^1(B;\R^d)$, then a formal integration by parts yields $\int_{\partial B} \wh{\xi}\cdot n = \int_{B} \nabla\cdot \wh{v} =0$.
As follows from our analysis, Assumption \ref{ass:homogenized_PDE_linear_model_problem} can be weakened, still keeping the well-posedness result of Proposition \ref{prop:stokes_rough} valid. For simplicity, we do not pursue this here.

\smallskip

To define solutions to \eqref{eq:homogenized_PDE_linear_model_problem} under the Assumption \ref{ass:homogenized_PDE_linear_model_problem}, we follow the transposition method of Lions and Magenes \cite{LM1_book,LM2_book}. To motivate such a notion of solutions, we argue formally. 
First, we introduce the Helmholtz projection $\p_B$ on a domain $B$.
For $F\in L^2(B;\R^d)$, define $\qq_B F=\psi_F\in H^1(B)$, where $\psi_F$ is the unique solution to the elliptic PDE
\begin{equation}
\label{eq:elliptic_PDE_helmholtz}
\Delta \psi_F =\nabla \cdot F \text{ in }B, \quad \text{ and } \quad\partial_n \psi_F=F\cdot n \text{ on }\partial B ,
\end{equation}
together with the normalization condition $
\int_{B} \psi_F=0.
$
In the above, $\partial_n=n\cdot \nabla $ denotes the exterior normal derivative. 
Clearly, the elliptic problem \eqref{eq:elliptic_PDE_helmholtz} is understood in its natural variational formulation: 
$$
\int_{B} \nabla \psi_F\cdot \nabla \varphi = \int_{B} F\cdot \nabla \varphi \quad \text{ for all } \varphi\in H^1(B)\text{ with }\int_{B}\varphi =0.
$$
The Helmholtz projection on $B$ is then given by
\begin{equation}
\label{eq:def_Helmholtz}
\p_B F= F- \q_B F \quad \text{ where }\quad \q_B F=\nabla \qq_B F.
\end{equation}
One can readily check that $\p_B$ is an orthogonal projection on $L^2(B;\R^d)$, and we let 
$$
\Ls^2(B)=\p_B(L^2(B;\R^d)).
$$
It is well-known that $\Ls^2(B)$ coincides with the closure in $L^2(B;\R^d)$ of divergence-free vector fields in $C^{\infty}_{{\rm c}}(B;\R^d)$, see e.g., \cite{Galdi_book}. 
Second, let $A$ be the (minus) Stokes operator on  $\Ls^2(B)$:
\begin{equation}
\begin{aligned}
\label{eq:stokes_operator}
A: \Do(A)&= H^{2}(B;\R^d)\cap H^1_0(B;\R^d)\cap \Ls^2(B)\to \Ls^2(B),\\
 A u&= -\p_B \Delta u.
\end{aligned}
\end{equation}
It is well-known that $A$ is a positive, invertible, and self-adjoint operator on $\Ls^2(B)$. Thus, it follows from \cite[Corollary 17.3.8]{Analysis3} that $A$ has maximal $L^{2}$-regularity. In particular, for any $\Phi\in L^{2}(Q;\R^d)$, there exists a unique 
strong solution    
\begin{equation}
\label{eq:regularity_backward_stokes_problem}
\psi\in L^{2}(I;\Do(A))\cap H^{1}(I;\Ls^2(B)), \qquad \nabla q\in L^{2}(Q)
\end{equation} 
to the backward Stokes problem
\begin{equation}
\label{eq:Phi_psi_problem}
\left\{
\begin{aligned}
-\partial_t \psi -\kappa\Delta \psi &= -\nabla q + \Phi &\text{ on }&Q,\\
\nabla \cdot \psi&=0  &\text{ on }&Q,\\
\psi&=0  &\text{ on }&I\times \partial B,\\
\psi(1)&=0  &\text{ on }&B.
\end{aligned}
\right.
\end{equation}
In the above, we used that $Q=I\times B$ and $I=(0,1)$ by assumption.
Now, multiplying formally \eqref{eq:homogenized_PDE_linear_model_problem} by $\psi$, an integration by parts readily yields
\begin{align*}
\int_{Q} [(\partial_t-\kappa\Delta) v ]\cdot \psi
&= -\int_{B} \psi(0) \cdot \xi(0)+\kappa\int_{I\times \partial B} \xi\cdot  \partial_n \psi -
\int_{Q} v \cdot [(\partial_t+\kappa\Delta) \psi]\\
&= -\int_{B} \psi(0)\cdot \xi(0)+\kappa\int_{I\times \partial B} \xi\cdot \partial_n \psi+
\int_{Q} v \cdot \Phi -\int_{Q} v\cdot  \nabla q \\
&=- \int_{B} \psi(0)\cdot \xi(0)+\kappa\int_{I\times \partial B} \xi\cdot  \partial_n \psi +
\int_{Q} v \cdot \Phi -\int_{I\times \partial B} \xi\cdot n q.
\end{align*}
On the other hand, $[(\partial_t-\kappa\Delta) v ]=-\nabla \Pi + \nabla \cdot R$, and 
\begin{align*}
\int_Q [(\partial_t-\kappa\Delta) v ]\cdot \psi
=\int_Q (-\nabla \Pi + \nabla \cdot R) \cdot \psi
=-\int_Q   R :\nabla \psi.
\end{align*}

Hence, the above argument motivates the following definition.

\begin{definition}[Transposition solutions to the Stokes system with rough boundary data]
\label{def:solution_homogenized}
Let Assumption \ref{ass:homogenized_PDE_linear_model_problem} be satisfied. We say that $v$ is a solution to \eqref{eq:homogenized_PDE_linear_model_problem} (in the transposition sense) if $v\in L^2(Q)$, and 
for all $\Phi\in L^{2}(Q;\R^d)$, 
$$
\int_{Q} v \cdot \Phi 
=\int_{B} \psi(0)\cdot \xi(0)+\int_{I\times \partial B} \xi\cdot \big( n\,q-\kappa \,\partial_n \psi\big) -\int_Q   R :\nabla \psi,
$$ 
where $(\psi,q)$ is a solution to \eqref{eq:Phi_psi_problem} with regularity as in \eqref{eq:regularity_backward_stokes_problem}. 
\end{definition}

Note that the problem \eqref{eq:Phi_psi_problem} is invariant under the transformation $(\psi,q)\mapsto (\psi,q+C)$ for any $C\in \R$. This is compatible with the above definition due to the condition \eqref{eq:condition_div_free_integral}  in Assumption \ref{ass:homogenized_PDE_linear_model_problem}.

\smallskip

The main result of this subsection reads as follows.

\begin{proposition}[Well-posedness of the Stokes system with rough boundary data]
\label{prop:stokes_rough}
Suppose that Assumption \ref{ass:homogenized_PDE_linear_model_problem} holds. Fix $\varphi\in C^{\infty}_{{\rm c}}(B)$ such that $\int_B \varphi =1$. Then there exists a unique solution 
$v$ 
to \eqref{eq:homogenized_PDE_linear_model_problem} in the transposition sense satisfying 
$$
\nabla \cdot v=0 \text{ in }\D'(B) \text{  a.e.\ on } I,
$$ 
and a unique
$\Pi\in H^{-1}(I;H^{-1}(B))$ such that 
\begin{align*}
&\partial_t v =-\nabla \Pi +\kappa \Delta v+\nabla \cdot R \  \text{ in }\ \D'(Q),\\
&\langle \Pi,\eta\otimes \varphi \rangle =0 \ \text{ for all } \ \eta\in C^{\infty}_{{\rm c}}(I).
\end{align*}
Finally, there exists a constant $C>0$ depending only on $d$ and $\kappa$ such that 
\begin{align}
\label{eq:stokes_rough_bound_claim}
\|v\|_{L^{2}(Q)}
+\|\Pi\|_{H^{-1}(I;H^{-1}(B))}
\leq C\big( \|\xi(0)\|_{L^2(B)}+ \|\xi\|_{L^2(I\times \partial B)}+ \|R\|_{L^2(Q)}\big).
\end{align}
\end{proposition}

In the above, we employ the notation $(\eta\otimes \varphi)(t,x)=\eta(t)\varphi(x)$. As one may anticipate, in the proof of Proposition \ref{prop:local_well_posedness_homogenized_SPDE}, we apply the previous pointwise in $\O$. However, the corresponding solution has lower regularity than the one stated in Proposition \ref{prop:local_well_posedness_homogenized_SPDE}. To obtain the stronger regularity, in Subsection \ref{ss:local_well_posedness_homogenized_SPDE}, we further exploit the fact that $(v^{N,\delta},\pi^{N,\delta},(\wt{\pi}^{N,\delta}_{k,\alpha})_{k,\alpha})$ is a local solution to the SPDE system \eqref{eq:turbulent_Stokes_scaling_quantitative}.

\begin{proof}
We divide the proof into three steps.

\smallskip

\emph{Step 1: Existence and uniqueness of a solution $v\in  L^2(Q)$ to \eqref{eq:homogenized_PDE_linear_model_problem} in the sense of Definition \ref{def:solution_homogenized}.} 
Uniqueness is an immediate consequence of the arbitrariness of $\Phi\in L^2(Q)$. 
As for the existence, consider the linear mapping $L:L^{2}(Q;\R^d)\to \R$ given by
$$
L(\Phi)=
\int_{B} \psi(0) \cdot \xi(0)+\int_{I \times \partial B} \xi\cdot (n\, q-\kappa \partial_n \psi)  -\int_Q   R :\nabla \psi,
$$
where $(\psi,q)$ is the unique solution to \eqref{eq:Phi_psi_problem} such that $\int_{B}q(t) =0$ for a.e.\ $t\in I$. 
Now, from the boundedness of the trace operator (see e.g., \cite[Proposition 4.5, Chapter 4]{TayPDE1})
$H^{1}(B) \ni f \mapsto f|_{\partial B}\in L^2(\partial B)$ and $\nabla \psi\in L^2(I;H^1(B;\R^d))$ by \eqref{eq:regularity_backward_stokes_problem}, 
we have 
$$
\nabla \psi|_{\partial B}\in L^2(I\times \partial B). 
$$
Similarly, we can write $\int_{I \times \partial B} \xi\cdot n q=\int_{I \times \partial B} \xi\cdot n \big(q-\fint_{B} q\big)$ due to \eqref{eq:condition_div_free_integral}, and therefore $q-\fint_{B}q \in L^2(I;H^1(B) )$ by \eqref{eq:regularity_backward_stokes_problem} and Poincar\'e inequality. Hence, 
$$
\Big(q-\fint_{B}q \Big)\Big|_{\partial B}\in L^2(I\times \partial B).
$$
Moreover, as
$\psi\in H^{1}(I;\Ls^2(B))\subset C(\overline{I};\Ls^2(B))$,  
$$
[\Phi \mapsto L(\Phi)]\in (L^2(Q;\R^d))^*.
$$ 
Hence, Riesz' representation theorem implies the existence of $v\in L^2(Q;\R^d)$ such that 
$$
L(\Phi)= \int_{Q} v \cdot \Phi,
$$
therefore proving the existence of a solution $v$ of \eqref{eq:homogenized_PDE_linear_model_problem} in the transposition sense.

\smallskip

\emph{Step 2: $\nabla \cdot v=0$ in $\D'(B)$ a.e.\ on $I$.}    
It suffices to prove that for all $\eta\in \D(I)$ and $\phi\in \D(B)$, it holds that 
$$
\int_{Q} \eta [v\cdot \nabla \phi]=0. 
$$
The above clearly follows as $(\psi,q)=(0,\eta \phi)$ is a solution to \eqref{eq:Phi_psi_problem} with $\Phi=\eta\nabla \phi$.

\smallskip

\emph{Step 3: Estimate for the pressure $\Pi$ and proof of $\partial_t v =-\nabla \Pi +\kappa \Delta v+\nabla \cdot R \text{ in }\D'(Q)$.}
We begin by letting 
\begin{equation}
\label{eq:definition_ansatz_pressure_stokes_homogenized}
P=\kappa\Delta v+\nabla \cdot R-\partial_t v.
\end{equation}
From Step 1 and Assumption \ref{ass:homogenized_PDE_linear_model_problem}, it readily follows that $P\in H^{-1}(I;H^{-2}(B))$. It remains to show the existence of $\Pi\in H^{-1}(I;H^{-1}(B))$ such that 
\begin{equation}
\label{eq:claim_Step_3_proof_pressure_existence}
\nabla \Pi= P \text{ in }\D'(Q),
\end{equation} 
and additionally satisfies the bound in \eqref{eq:stokes_rough_bound_claim}.
The uniqueness of $\Pi$ such that $\langle \Pi,\eta\otimes \varphi\rangle=0$ for all $\eta\in C^{\infty}_{{\rm c}}(I)$ is standard. As for the existence, we claim that 
\begin{equation}
\label{eq:claim_bogosvkii_pressure_construction}
\langle P, w\rangle =0 \text{ for all }w\in C^{\infty}_{{\rm c}}(Q;\R^d)\  \text{ satisfying }\ \nabla \cdot w=0.
\end{equation}
We first derive the existence of $\Pi$ satisfying \eqref{eq:claim_Step_3_proof_pressure_existence} with the required regularity, and afterwards we prove the claim \eqref{eq:claim_bogosvkii_pressure_construction}. 
Let $\mathcal{B}_B$ be the Bogovskii operator on the ball $B$ (see e.g., \cite[eq.\ (III.3.9)]{Galdi_book}), and set
\begin{equation}
\label{eq:formula_bogovskii_pressure_how_to_find_it}
\langle \Pi, \wt{\phi}\rangle = -
\Big\langle P, \mathcal{B}_B\Big(\wt{\phi}- \varphi \int_{B}\wt{\phi}\Big)\Big\rangle \  \text{ for all }\ \wt{\phi}\in \D(Q).
\end{equation}
Note that to define $\Pi$ we used $\mathcal{B}_B:C^{\infty}_{{\rm c}}(B)\to C_{{\rm c}}^\infty(B;\R^d)$ continuously (see \cite[Lemma III.3.1]{Galdi_book}). Moreover, since $\int_B \varphi=1$ by assumption, for all $\eta\in C^{\infty}_{{\rm c}}(I)$, it holds that $
\langle \Pi,\eta\otimes \varphi\rangle=0.
$ 
From \eqref{eq:claim_bogosvkii_pressure_construction}, for all $\phi\in C^{\infty}_{{\rm c}}(Q;\R^d)$  it follows that
\begin{align*}
\langle P-\nabla \Pi, \phi\rangle
&=\langle P, \phi\rangle
+\langle \Pi, \nabla \cdot \phi\rangle\\
&\stackrel{(i)}{=}\langle P, \phi\rangle
-\Big\langle P, \mathcal{B}_B\Big(\nabla \cdot \phi- \varphi \int_{B}\nabla \cdot \phi\Big)\Big\rangle\\
&\stackrel{(ii)}{=}\langle P, \phi-\mathcal{B}_B(\nabla \cdot \phi)\rangle\stackrel{(iii)}{=}0
\end{align*}
where in $(i)$ we used \eqref{eq:formula_bogovskii_pressure_how_to_find_it}, in $(ii)$ that $\int_{B} \nabla \cdot \phi=0$, and in $(iii)$ that $\nabla \cdot [\phi-\mathcal{B}_B(\nabla \cdot \phi)]=0$ as $\mathcal{B}_B$ is a right inverse of the divergence operator (see \cite[Subsection III.3]{Galdi_book}). 
Finally, the estimate for $\Pi$ in \eqref{eq:stokes_rough_bound_claim} follows from \eqref{eq:definition_ansatz_pressure_stokes_homogenized}, \eqref{eq:formula_bogovskii_pressure_how_to_find_it}, the well-known duality $H^{-1}(I;H^{-1}(B))=(H^1_0(I;H_0^1(B)))^*$ and the mapping property of the Bogovskii operator $\mathcal{B}_B: H^{1}_0(B)\to H^2_0(B)$ (see e.g., \cite[Remark III.3.2]{Galdi_book}).

To conclude, it remains to prove \eqref{eq:claim_bogosvkii_pressure_construction}. Note that, by the definition of $P$ in \eqref{eq:definition_ansatz_pressure_stokes_homogenized}, for all divergence-free $w\in C^\infty_{{\rm c}}(Q;\R^d)$, it holds that
\begin{align*}
\langle P, w\rangle= \int_{Q} v\cdot (\partial_t w + \kappa\Delta w)- R: \nabla w=0,
\end{align*}
where in the last step we used that $(\psi,q)=(w,0)$ is a solution to \eqref{eq:Phi_psi_problem} with $\Phi=-(\partial_t w +\kappa \Delta w)$. Thus Definition \ref{def:solution_homogenized} yields the desired equality. 
\end{proof}

\subsection{Extrapolated spaces and Helmholtz projection on domains}
\label{ss:Helmholtz_domains}
In this subsection, we provide a brief introduction to extrapolation spaces for sectorial operators. A complete treatment can be found in either \cite[Chapter V.1]{Am} or \cite[Subsection 6.3]{Haase:1}. In a nutshell, extrapolation spaces for an operator $A$ (e.g., the Laplacian with Dirichlet boundary conditions on $L^2(B)$) can be thought of as the analogue of negative Sobolev spaces on domains, with the additional requirement that $A$ still preserves the spectral properties on such extrapolated spaces (e.g., still generates an analytic semigroup, possesses maximal $L^p$-regularity, etc.). It is well-known that, in the case where $A$ is the Dirichlet Laplacian on $H^{-2}(B)$, $A$ does \emph{not} generate a semigroup and, in a very weak sense, boundary conditions have to be taken into account. Moreover, as commented in Subsection \ref{ss:caccioppoli_intro}, they provide the right framework to show \emph{local mixing estimates} for the noise in \eqref{eq:turbulent_Stokes_scaling_quantitative}.

\smallskip

We divide this subsection into two parts. In the first, we provide a brief introduction to the extrapolated scale in the context of self-adjoint operators on Hilbert spaces (for this case, see also \cite[Chapter 5, Appendix A]{TayPDE1}), together with some basic properties. In the second, we discuss the important relation between the extrapolated spaces associated with the Dirichlet Laplacian and the Stokes operator, as well as consequences for the Helmholtz projection.

\subsubsection{Extrapolated spaces}
\label{sss:extrapolated_spaces}
Let $A: \Do(A)\subseteq X\to X$ be a positive, invertible, and self-adjoint operator on a Hilbert space $X$ with domain $\Do(A)$. 
For any $\sigma<0$, we define the extrapolated scale $\X^{A}_{\sigma}$ as 
\begin{equation}
\label{eq:extrapolated_space1}
(\X^{A}_{\sigma},\|\cdot\|_{\X^{A}_{\sigma}})\stackrel{{\rm def}}{=}
(X, \|A^{\sigma} \cdot\|_{X})^\sim
\end{equation}
where $\sim$ denotes the completion of the normed space and $A^{\sigma}$ denotes the fractional power of the operator $A$, see e.g., \cite[Subsection 15.2]{Analysis3}. For the case of self-adjoint operators, to define fractional powers, one can equivalently use the spectral theorem \cite[Chapter 8, Section 2]{TayPDE2}. 
From the invertibility assumption of $A$, it follows that $A^{\sigma}$ is injective, and therefore $\X^A_\sigma$ is well-defined. 

One can also complete the scale for non-negative smoothness by letting 
\begin{equation}
\label{eq:extrapolated_space2}
\X^A_0\stackrel{{\rm def}}{=}X \quad \text{ and }\quad \X^A_{\sigma}\stackrel{{\rm def}}{=}\Do(A^\sigma) \ \text{ for } \ \sigma>0,
\end{equation}
endowed with the natural norms. Clearly, for all $\sigma\in \R$, the space $\X^A_\sigma$ is Hilbert.

From the invertibility of $A$ and the fact that $\Do(A)\stackrel{{\rm d}}{\embed} X$, the above definitions readily imply 
\begin{equation}
\label{eq:embedding_extrapolated}
\X^A_{\sigma_0}\stackrel{{\rm d}}{\embed} \X^A_{\sigma_1} \ \text{ for all }\  \sigma_0>\sigma_1.
\end{equation}

In the following, we gather some useful properties. 

\begin{lemma}
\label{l:properties_extrapolated_spaces}
Let $A$ be a positive, invertible, and self-adjoint operator on a Hilbert space $X$. Let $(\X^A_\sigma)_{\sigma\in\R}$ be the scale of Hilbert spaces defined in \eqref{eq:extrapolated_space1}-\eqref{eq:extrapolated_space2}. 
Then the following assertions hold:
\begin{itemize}
\item {\rm (Extrapolated operator)}
For each $\sigma <0$ (resp.\ $\sigma>0$), the operator $A$ uniquely extends (resp.\ restricts) to a positive, invertible, and self-adjoint operator $A_\sigma$ on $\X^A_{\sigma}$ with domain $\Do(A_\sigma)=\X_{1+\sigma}^A$. 
\item {\rm (Interpolation)}
For $\sigma_0,\sigma_1\in \R$ and $\theta\in (0,1)$, 
$$
\X_{\sigma}^A = [\X_{\sigma_0}^A ,\X_{\sigma_1}^A]_{\theta} \quad \text{ where }\quad \sigma=(1-\theta)\sigma_0+\theta\sigma_1.
$$
\item {\rm (Duality)} For each $\sigma \geq 0$, it holds that $
 \X_{-\sigma}^{A}
 =(\X_{\sigma}^{A})^*
$ and
$$
\langle x^*,x\rangle_{-\sigma,\sigma}^{A}
= ((A^{-\sigma})_{-\sigma} x^*,A^{\sigma}x)_{X} \ \text{ for all }\ x^*\in \X^A_{-\sigma}, \ x\in \X^A_{\sigma}.
$$
\end{itemize}
\end{lemma}

The above result follows from the definition of the extrapolated spaces and the fact that self-adjoint operators have a bounded $H^\infty$-calculus, see e.g., \cite[Proposition 10.2.23]{Analysis2} and \cite[Theorem 15.3.9]{Analysis3}. The duality part can be found in \cite[Chapter V, Corollary 1.5.13]{Am}.

\subsubsection{Mapping properties of the Helmholtz projection on extrapolated spaces}
In this subsection, we mainly consider balls $B$ in $\R^d$, although the results below also hold for more general bounded smooth domains. The construction in Subsection \ref{sss:extrapolated_spaces} can be applied to either (minus) the Stokes operator $A$ as defined in \eqref{eq:stokes_operator} on $B$, or to the (minus) Dirichlet Laplacian 
\begin{equation}
\begin{aligned}
\label{eq:dirichlet_oper}
A_{\Dir}: \Do(A_{\Dir})&= H^{2}(B)\cap H^1_0(B)\to L^2(B),\\
 A_{\Dir} u&= - \Delta u.
\end{aligned}
\end{equation}
Correspondingly, we can define the two families of extrapolated spaces
\begin{equation}
\label{eq:extrapolated_spaces_sto_dir}
(\X^{\Sto}_\sigma(B))_{\sigma\in \R}\quad \text{ and }\quad 
(\X^{\Dir}_\sigma(B))_{\sigma\in \R},
\end{equation}
respectively.
To connect the above with the notion of weak solutions in the PDE sense, it is worth mentioning that, in the case $\sigma=-1/2$, the extrapolated Stokes operator $A_{-1/2}$, provided by Lemma \ref{l:properties_extrapolated_spaces}, satisfies
\begin{equation}
\label{eq:weak_formulation_stokes}
\langle A_{-1/2} v,\varphi\rangle = \int_{B} \nabla v : \nabla \varphi  \ \  \text{ for all } \ \varphi\in H^{1}_0(B;\R^d)\cap \Ls^2(B).
\end{equation}
This is a standard consequence of \eqref{eq:embedding_extrapolated}, an integration by parts, and the well-known identification $\Do(A^{1/2})=H_0^{1}(B;\R^d)\cap \Ls^2(B)$. A similar identity holds for the extrapolated Dirichlet Laplacian $(A_{\Dir})_{-1/2}$.

\smallskip

We conclude this subsection with the following important result on the extension of the Helmholtz projection $\p_B$ defined in \eqref{eq:def_Helmholtz} on the above-defined extrapolated spaces.
For notational convenience, we also set 
\begin{equation}
\label{eq:extrapolated_spaces_sto_dirRd}
\X^{\Dir}_{\sigma}(B;\R^d)\stackrel{{\rm def}}{=}(\X_{\sigma}^{\Dir}(B))^d \text{ for }\sigma\in\R.
\end{equation} 

\begin{lemma}[Extrapolated Helmholtz projection]
\label{l:boundedness_p_extrapolated_spaces}
Let $B$ be a ball in $\R^d$. The Helmholtz projection $\p_{B}$ can be uniquely extended as a bounded linear operator $\p_{-1,B}$ from $\X_{-1}^{\Dir}(B;\R^d)$ into $\X^{\Sto}_{-1}(B)$. In particular,
$$
\p_{-\sigma,B}\stackrel{{\rm def}}{=}
\p_{-1,B}|_{\X_{-\sigma}^{\Dir}(B;\R^d)} : \X_{-\sigma}^{\Dir}(B;\R^d)\to \X_{-\sigma}^{\Sto}(B) 
$$
boundedly for all $ \sigma \in [0,1]$.
\end{lemma}

The above is proven in \cite[Proposition 9.14]{KKW}, where it is also shown that $\p_{-1,B}$ is surjective. For the reader's convenience, we include a short proof of Lemma \ref{l:boundedness_p_extrapolated_spaces}. 

\begin{proof}
Recall that $A$ is (minus) the Stokes operator as defined in \eqref{eq:stokes_operator} and $\X_{1}^\Sto(B)=\Do(A)$.
By complex interpolation and Lemma \ref{l:properties_extrapolated_spaces}, it suffices to consider $\sigma=1$.  
In the latter case, for any $f\in L^2(B;\R^d)$, the duality in Lemma \ref{l:properties_extrapolated_spaces} yields
\begin{align*}
\|\p_{-1,B} f\|_{\X_{-1}^{\Sto}(B)}
&=\sup_{\|g\|_{\Do(A)}\leq 1} |\langle \p f, g\rangle_{-1,1}^{\Sto}| \\
&\stackrel{(i)}{=}\sup_{\|g\|_{\Do(A)}\leq 1} \Big|\int_{B}\p f\cdot g\Big| \stackrel{(ii)}{=} \sup_{\|g\|_{\Do(A)}\leq 1}\Big|\int_{B} f\cdot g\Big|\\
&\stackrel{(iii)}{\leq }\|f\|_{\X_{-1}^{\Dir}(B;\R^d)}\Big( \sup_{\|g\|_{\Do(A)}\leq 1}\|g\|_{H^{2}(B;\R^d)}\Big) 
\stackrel{(iv)}{\leq} C \|f\|_{\X_{-1}^{\Dir}(B;\R^d)},
\end{align*}
where in $(i)$ we used the compatibility of the duality pairing (see the last item in Lemma \ref{l:properties_extrapolated_spaces}), in $(ii)$ the self-adjointness of $\p$ and $\p g =g$ as $g\in \Do(A)$, in $(iii)$ that $g|_{\partial B}=0$ if $g\in \Do(A)$ and in $(iv)$ the estimate $\|g\|_{H^{2}(B;\R^d)}\leq C \|A g\|_{L^2(B;\R^d)}$ for $g\in \Do(A)$ (i.e., elliptic regularity for the Stokes operator, or the invertibility of $A$). 

The conclusion follows from $\X_{0}^{\Dir}(B;\R^d)=L^2(B;\R^d)$ and density, see \eqref{eq:embedding_extrapolated}.
\end{proof}

To conclude, it is worth noticing that the target space of $\p_{-1,B}$ cannot be replaced by $\X_{-1}^{\Dir}(B;\R^d)$. Indeed, by duality and Lemma \ref{l:properties_extrapolated_spaces}, the boundedness on $\X_{-1}^{\Dir}(B;\R^d)$ implies that the usual Helmholtz operator restricts as a bounded linear operator on $\X_{1}^{\Dir}(B)=H^{2}(B;\R^d)\cap H^1_0(B;\R^d)$, which is notoriously false due to the Dirichlet boundary conditions.

\subsection{Well-posedness of the homogenized SPDE -- Proof of Proposition 
\ref{prop:local_well_posedness_homogenized_SPDE}}
\label{ss:local_well_posedness_homogenized_SPDE}
We begin by collecting some observations. As usual, for simplicity, we assume $\varrho=1$.
As noticed below \eqref{eq:interpolation_and_boundary_uNdelta}, from the regularity of local solutions $(v^{N,\delta},\pi^{N,\delta}, (\wt{\pi}_{k,\alpha}^{N,\delta})_{k,\alpha})$ to \eqref{eq:turbulent_Stokes_scaling_quantitative} (see Definition \ref{def:local_solutions_NSE}), it follows that 
$$
v^{N,\delta}|_{I\times \partial B}\in L^2(I\times \partial B).
$$
Moreover, as $v^{N,\delta}\in L^2(I;H^1(B;\R^d))$ a.s.\ and $\nabla \cdot v^{N,\delta}=0$ in $\D'(Q)$ a.e.\ on $I\times \O$, we have (cf.\ the comments below Assumption \ref{ass:homogenized_PDE_linear_model_problem})
$$
\int_{\partial B} \Big(v^{N,\delta} - \fint_{B} v^{N,\delta}\Big)\cdot n = 0\ \  \text{ a.e.\ on }I\times \O.
$$
Hence, it follows from Proposition \ref{prop:stokes_rough} that there exists a unique solution $(\uh^{N,\delta},\ph^{N,\delta})$ to \eqref{eq:homogenized_PDE_linear} in the transposition sense (see Definition \ref{def:solution_homogenized}) such that   
\begin{equation}
\label{eq:solution_homogenized_with_low_integrability}
\uh^{N,\delta}\in L^{2}(Q)\quad \text{ and }\quad \ph^{N,\delta} \in H^{-1}(I;H^{-1}(B)), 
\end{equation} 
with corresponding bounds. 
In particular, $\uh^{N,\delta}$ satisfies, a.s.,
\begin{align}
\nonumber
\int_{Q} \uh^{N,\delta} \cdot \Phi 
&=\int_{B} \psi(t_1)\cdot v^{N,\delta} (t_1)+\int_{I\times \partial B} \Big(v^{N,\delta}-\fint_B v^{N,\delta}\Big)\cdot \big(n q-  (1+D_d \mu)\partial_n \psi \big)\\
\label{eq:homogenized_PDE_weak_formulation}
& -\int_Q   (F-(1-D_d) \mu\, G^\top) :\nabla \psi,
\end{align}
for all $\Phi\in L^{2}(Q)$, where $\psi$ is the solution to the Stokes backward problem \eqref{eq:Phi_psi_problem} with $\kappa=1+D_d \mu$, and $t_1=t_0-1$. In the above, we used that $\int_{B} \psi(t_1) \cdot \fint_B v^{N,\delta}(t_1)=0$ which follows from the fact that 
\begin{equation}
\label{eq:phi_Ls2_null_average}
\int_{B} \phi=0 \   \text{ for all } \ \phi\in \Ls^2(B).
\end{equation}
The latter is a consequence of a standard density argument, and that $\int_B \phi_i =\int_B  \phi \cdot \nabla x_i=0$ for all divergence-free vector fields $\phi\in C^{\infty}_{{\rm c }}(B;\R^d)$.

\smallskip

To prove Proposition \ref{prop:local_well_posedness_homogenized_SPDE} and Theorem \ref{t:universal_scaling_limit}, let us consider the difference process
\begin{equation}
\label{eq:difference_process_vN}
w^{N,\delta}\stackrel{{\rm def}}{=}\uh^{N,\delta}-\Big(v^{N,\delta}- \fint_B v^{N,\delta}\Big).
\end{equation} 
In the following, we show that $w^{N,\delta}$ has more regularity than what can be obtained from the regularity of $v^{N,\delta}$ and $\uh^{N,\delta}$, see Definition \ref{def:local_solutions_NSE} and \eqref{eq:solution_homogenized_with_low_integrability}, respectively. 
To this end, we first formally deduce an SPDE for the process $w^{N,\delta}$. 
First, note that, by formally taking the spatial average in the SPDE \eqref{eq:turbulent_Stokes_scaling_quantitative}, the process $\fint_B v^{N,\delta}$ solves 
\begin{align}
\label{eq:identity_pressure_mean_formal_derivation_equation_w}
\partial_t \fint_B v^{N,\delta} 
&= P^{N,\delta} + \sqrt{c_d \mu}\sum_{k,\alpha} \wt{P}^{N,\delta}_{k,\alpha} \dot{W}^{k,\alpha} \\
\nonumber
&= -\nabla (p^{N,\delta} \cdot x )-
\sqrt{c_d \mu} \sum_{k,\alpha}\nabla (\wt{p}^{N,\delta}_{k,\alpha} \cdot x) \dot{W}^{k,\alpha} ,
\end{align}
for some random constant-in-space vectors $P^{N,\delta}$, $\wt{P}^{N,\delta}_{k,\alpha}$, $p^{N,\delta} $ and $\wt{p}^{N,\delta}_{k,\alpha}$. In particular, at least formally, they only contribute to the dynamics via pressure terms
(as in Section \ref{s:caccioppoli}, the above reasoning can be made rigorous if instead of $\fint_B v^{N,\delta}$ we consider $\int_B \phi^2 v^{N,\delta}$ for some test function $\phi$ with $\int_B \phi^2=1$).

\smallskip

Taking the difference between \eqref{eq:turbulent_Stokes_scaling_quantitative} and \eqref{eq:homogenized_PDE_linear}, and using the identity \eqref{eq:identity_pressure_mean_formal_derivation_equation_w}, the process $w^{N,\delta}$ formally solves the following SPDE: 
\begin{equation}
\label{eq:turbulent_Stokes_scaling_quantitative_vdifference_formal}
\left\{
\begin{aligned}
\partial_t w^{N,\delta} &=-\nabla R^{N,\delta}+ (1+\mu D_d) \Delta w^{N,\delta}+ \Errd\\ 
&\qquad\ \  + \sum_{k,\alpha} [-\nabla \wt{R}_{k,\alpha}^{N,\delta}+\Errs_{k,\alpha}]\,\dot{W}^{k,\alpha}_t&\text{ on }&Q,\\
 \nabla \cdot w^{N,\delta}&=0&\text{ on }&Q,\\
 w^{N,\delta} &=0 &\text{ on }&\partial_{\pb }Q,
 \end{aligned}
\right.
\end{equation}
where 
$$
R^{N,\delta} = \ph^{N,\delta} -\pi^{N,\delta} + p^{N,\delta}\cdot x, \quad \text{ and }\quad 
\wt{R}_{k,\alpha}^{N,\delta} =\sqrt{c_d \mu} \big(-\wt{\pi}_{k,\alpha}^{N,\delta} +\wt{p}_{k,\alpha}^{N,\delta}\cdot x\big),  
$$
while $\Errd$ and $\Errs=(\Errs_{k,\alpha})_{k,\alpha}$ denote the deterministic and stochastic errors between the local solution $(v^{N,\delta},\pi^{N,\delta},(\wt{\pi}^{N,\delta}_{k,\alpha})_{k,\alpha})$ and the homogenized one $(\uh^{N,\delta},\ph^{N,\delta})$: 
\begin{align}
\label{eq:def_error_det}
\Errd&\stackrel{{\rm def}}{=}
-(1-D_d) \mu(\Delta v^{N,\delta}+\nabla \cdot G^\top) + c_d \mu \sum_{k,\alpha} \theta_k^{N}\nabla\cdot  (\nabla \wt{\pi}^{N,\delta}_{k,\alpha}\otimes \sigma^\delta_{-k,\alpha}), \\
\label{eq:def_error_sto}
\Errs_{k,\alpha}
&\stackrel{{\rm def}}{=}
-\sqrt{c_d \mu}\,\theta^{N}_{k}\big[(\sigma^\delta_{k,\alpha}\cdot\nabla) v^{N,\delta}+\sigma^\delta_{k,\alpha}\cdot G\big].
\end{align}
The convergence of the above errors to zero is more subtle and is proven in Subsections \ref{ss:quantitative_scaling_limits} and \ref{ss:preliminary_results_convergence} below. 
The following uniform-in-$(N,\delta)$ estimates suffice to prove Proposition \ref{prop:local_well_posedness_homogenized_SPDE}.

\begin{lemma}[Uniform bound for errors -- Weak setting]
\label{l:error_uniform_bound}
Let $Q$ be a parabolic cylinder with side length $\varrho\in (0,1/2]$ and center $(t_0,x_0)\in \R_+\times \R^d$.
For all $r\in (1,\infty)$, there exists a constant $C>0$ depending only on $\mu,r,\varrho$ and $d$ such that, for all $N\geq 1$, $\delta\in (0,1]$ and any local solution $(v^{N,\delta},\pi^{N,\delta},(\wt{\pi}^{N,\delta}_{k,\alpha})_{k,\alpha})$ to \eqref{eq:turbulent_Stokes_scaling_quantitative}, it holds that 
\begin{align*}
\big(\E \|\Errd\|_{L^2(I;H^{-1}(B))}^r \big)^{1/r}
+ \big(\E\|(\Errs_{k,\alpha})_{k,\alpha}\|_{L^2(Q;\ell^2)}^r\big)^{1/r}
\leq C  
\big(\E\|G\|_{L^2(Q)}^r\big)^{1/r}\qquad &\\+
C\big(\E\|\nabla v^{N,\delta}\|_{L^2(Q)}^r \big)^{1/r}
+C\big(\E \|(\nabla \wt{\pi}_{k,\alpha}^{N,\delta})_{k,\alpha}\|_{L^2(Q;\ell^2)}^r\big)^{1/r}.&
\end{align*}
\end{lemma}

The previous result is an immediate consequence of the definitions \eqref{eq:def_error_det}-\eqref{eq:def_error_sto} and $\|\theta^N\|_{\ell^2}=1$ for all $N\geq 1$.
Due to Lemma \ref{l:error_uniform_bound} and 
$$
\X^{\Dir}_{-1/2}=(H^1_0(B))^*=H^{-1}(B),
$$ 
(see Lemma \ref{l:properties_extrapolated_spaces} for the last equality), we may formally apply the extrapolated Helmholtz projection $\p_{-1/2,B}$ to \eqref{eq:turbulent_Stokes_scaling_quantitative_vdifference_formal} (see Lemma \ref{l:boundedness_p_extrapolated_spaces}), and obtain the following SPDE written in its natural integrated form: 
\begin{align}
\label{eq:turbulent_Stokes_scaling_quantitative_vdifference_true}
z^{N,\delta}(t) 
&= \int_{t_0-1}^t \big(-(1+D_d \mu) A_{-1/2} z^{N,\delta} + \p_{-1/2,B} \Errd \big) \,\dd s\\
\nonumber
&+ \sum_{k,\alpha} \int_{t_0-1}^t \p_{B} \Errs_{k,\alpha} \,\dd W^{k,\alpha}_s
\end{align}
a.s.\ for all $t\in I$,
where we used $\p_{-1/2,B}=\p_B$ on $L^2(B;\R^d)$. As above, for simplicity, we consider only the case $\varrho=1$.
Here, we formally used that $\p_{-1/2, B}$ annihilates gradients, and we denote the solution to the above SPDE by $z^{N,\delta}$ instead of $w^{N,\delta}$ as, at the moment, we only know that $z^{N,\delta}$ is defined as in \eqref{eq:turbulent_Stokes_scaling_quantitative_vdifference_true}.
The proof that $w^{N,\delta}$ and $z^{N,\delta}$ coincide is part of the proof of Proposition \ref{prop:local_well_posedness_homogenized_SPDE} below.

By the semigroup approach to \cite[Chapter 5]{DPZ}, the stochastic evolution equation \eqref{eq:turbulent_Stokes_scaling_quantitative_vdifference_true} admits a unique solution given by the stochastic convolution:
\begin{align}
\label{eq:definition_vN_rates}
z^{N,\delta}(t) &\stackrel{{\rm def}}{=} \int_{t_0-1}^t e^{-(1+D_d\mu)A_{-1/2}(t-s)} \p_{-1/2,B} \Errd(s)\,\dd s\\
\nonumber
& + \int_{t_0-1}^t \big(e^{-(1+D_d\mu)A_{-1/2}(t-s)}\p_{B} \Errs_{k,\alpha}(s)\big)_{k,\alpha}\,\dd \mathcal{W}_s.
\end{align}
Indeed, in the above, $\mathcal{W}$ is the cylindrical Brownian motion associated with $(W^{k,\alpha})_{k,\alpha}$ (see \eqref{eq:def_cylindrical_noise}), and we used that $A_{-1/2}$ is a positive, invertible, and self-adjoint operator on $\X^A_{-1/2}(B)$, and therefore generates an analytic semigroup $(e^{-(1+D_d\mu)A_{-1/2}t})_{t\geq 0}$.

From Lemma \ref{l:error_uniform_bound}, it readily follows that the right-hand side of \eqref{eq:definition_vN_rates} is well-defined and that the following result holds:

\begin{lemma}
\label{l:regularity_v_u}
Let $Q$ be a parabolic cylinder with side length $\varrho\in (0,1/2]$ and center $(t_0,x_0)\in \R_+\times \R^d$. Fix $r\in (1,\infty)$, $N\geq 1$ and $\delta\in (0,1]$. 
Let $(v^{N,\delta},\pi^{N,\delta},(\wt{\pi}_{k,\alpha}^{N,\delta})_{k,\alpha})$ be a local solution to \eqref{eq:turbulent_Stokes_scaling_quantitative}. Let $\Errd$ and $\Errs$ be as in \eqref{eq:def_error_det} and \eqref{eq:def_error_sto}, respectively. Then, the process $z^{N,\delta}$ is well-defined with paths
$$
z^{N,\delta} \in L^2(I;H^{1}_0(B;\R^d))\cap L^\infty(I;\Ls^2(B)) \ \text{a.s.,}
$$
and moreover,
\begin{align*}
\big(\E\|z^{N,\delta}\|_{L^\infty(I;L^2(B))}^r \big)^{1/r}
+
\big(\E\|z^{N,\delta}\|_{L^2(I;H^1(B))}^r\big)^{1/r}
\lesssim\big( \E\|G\|_{L^2(Q)}^r\big)^{1/r}\qquad \qquad\qquad&\\
+\big(\E\|\nabla v^{N,\delta}\|_{L^2(Q)}^r\big)^{1/r}
 +\big(\E \|(\nabla \wt{\pi}^{N,\delta}_{k,\alpha})_{k,\alpha}\|_{L^2(Q;\ell^2)}^r\big)^{1/r},&
\end{align*}
where the implicit constant depends only on $\mu,r,\varrho$ and $d$.
\end{lemma}

\begin{proof}
From Lemmas \ref{l:boundedness_p_extrapolated_spaces} and \ref{l:error_uniform_bound}, it follows that 
\begin{align*}
\Errd &\in L^r(\O;L^2(I;\X^{\Dir}_{-1/2}(B;\R^d))),\\
\Errs &\in L^r(\O;L^2(I;\ell^2(L^{2}(B;\R^d)))),
\end{align*}
with corresponding bounds depending only on $\mu,r,\varrho$ and $d$.
Thus, the claim follows from the self-adjointness of $A_{-1/2}$ as well as the well-known identity $\X_{1/2}^{A}(B)=\Hs^1_0(B)$ (see \cite[Corollary 7.4]{NVW13} and \cite{VW11_notes}).
\end{proof}

With the above results at our disposal, we are ready to prove Proposition \ref{prop:local_well_posedness_homogenized_SPDE}.

\begin{proof}[Proof of Proposition \ref{prop:local_well_posedness_homogenized_SPDE}]
We begin by recalling that, as discussed at the beginning of this subsection, the homogenized Stokes problem \eqref{eq:homogenized_PDE_linear} admits a unique solution in the sense of Definition \ref{def:solution_homogenized} and is $L^{2}$-integrable, see \eqref{eq:solution_homogenized_with_low_integrability}. 
In the next two steps, we overcome this apparent limitation.

\smallskip

\emph{Step 1: Bound on $\uh^{N,\delta}$.}
From Lemma \ref{l:regularity_v_u}, it suffices to show  
\begin{equation}
\label{eq:identity_uh_uN_wN}
w^{N,\delta}=z^{N,\delta} \ \text{ a.e.\ on }I\times \O,
\end{equation}
where $w^{N,\delta}$ is the difference process defined in \eqref{eq:difference_process_vN}. 
To prove the previous, from the separability of $L^2(Q)$ and density of smooth functions in the latter, it suffices to prove that, for all $\Phi\in C^\infty_{{\rm c}}(I;L^2(B))$,
\begin{equation}
\label{eq:claim_uniqueness_uuhv}
\int_{Q} \Big(\uh^{N,\delta} - v^{N,\delta} +\fint_B v^{N,\delta}- z^{N,\delta}\Big)\cdot \Phi=0 \text{ a.s. }
\end{equation}
Fix $\Phi\in C^\infty_{{\rm c}}(I;L^2(B))$, and let $(\psi,q)$ be the unique solution to the backwards Stokes system \eqref{eq:Phi_psi_problem} with $\kappa=1+D_d \mu$.
To show \eqref{eq:claim_uniqueness_uuhv}, we begin by proving that, a.s., 
\begin{align}
\label{eq:identity_average_Phi}
\int_Q \Big(\fint_B v^{N,\delta} \Big)\cdot \Phi
=
& \int_{I\times \partial B}  \Big(\fint_B v^{N,\delta} \Big)\cdot \Big( n q - (1+D_d \mu)\partial_n \psi\Big).
\end{align}
Note that, letting $\kappa=1+D_d \mu$,
\begin{align*}
&\int_Q \Big(\fint_B v^{N,\delta} \Big)\cdot \Phi
= \int_Q \Big(\fint_B v^{N,\delta} \Big)\cdot \Big(
-\partial_t \psi -\kappa\Delta \psi +\nabla q\Big)\\
&\quad =-\int_I \Big(\fint_B v^{N,\delta} \Big) \cdot \partial_t \Big(\int_B \psi\Big)
-\kappa \int_{I\times \partial B}  \Big(\fint_B v^{N,\delta} \Big)\cdot \partial_n \psi  +\int_{I\times \partial B} \Big(\fint_B v^{N,\delta} \Big)\cdot n q\\
&\quad=
\int_{I\times \partial B}  \Big(\fint_B v^{N,\delta} \Big)\cdot \big( n q - \kappa\,\partial_n \psi\big),
\end{align*}
where in the last equality we used \eqref{eq:phi_Ls2_null_average}.
Next, we rewrite the term $\int_{Q} z^{N,\delta}\cdot \Phi$ more conveniently. Due to the regularity of $\psi$ in \eqref{eq:regularity_backward_stokes_problem}, we can apply It\^o's formula to the functional 
$$
(t,z^{N,\delta})\mapsto  \langle z^{N,\delta}, \psi(t)\rangle_{-1/2,1/2}^A, 
$$
and using the terminal condition $\psi(t_0)=0$ in \eqref{eq:Phi_psi_problem}, it follows from \eqref{eq:turbulent_Stokes_scaling_quantitative_vdifference_true} that, a.s.,
\begin{align}
\label{eq:vN_ito_formula}
0&=\int_{t_1}^{t_0} \big(\langle z^{N,\delta}(s), \partial_t \psi(s)\rangle_{-1/2,1/2}^A- (1+D_d \mu)\langle A_{-1/2} z^{N,\delta}(s),  \psi(s)\rangle_{-1/2,1/2}^A\big) \,\dd s \\
\nonumber
&+\int_{t_1}^{t_0}\langle \p_{-1/2,B}\Errd(s) , \psi (s)\rangle_{-1/2,1/2}^A\,\dd s \\
\nonumber
&+\sum_{k,\alpha}\int_{t_1}^{t_0} \langle \p_B \Errs_{k,\alpha}(s),\psi(s)\rangle_{-1/2,1/2}^A \,\dd W^{k,\alpha}_s,
\end{align}
where $t_1=t_0-1$.
Let us point out that in the above application of the It\^o formula, we used that $ \psi\in C^{\infty}(\overline{I};\Do(A))$ since $\Phi\in C^{\infty}_{{\rm c}}(I;L^2(B))$ by assumption. 

Next, we rewrite the terms on the right-hand side of \eqref{eq:vN_ito_formula}. From the regularity of $\psi$ and the identity in \eqref{eq:weak_formulation_stokes} for $A_{-1/2}$, a.s.\ on $I$, it holds that  
\begin{align*}
\langle z^{N,\delta}, \partial_t \psi \rangle_{-1/2,1/2}^A
&-(1+D_d \mu)\langle A_{-1/2} z^{N,\delta},  \psi \rangle_{-1/2,1/2}^A\\
&= \int_{B} z^{N,\delta} \cdot \partial_t \psi -(1+D_d \mu)\int_{B} \nabla z^{N,\delta} :\nabla \psi  \\
&\stackrel{(i)}{=} \int_{B} z^{N,\delta} \cdot (\partial_t \psi +(1+D_d \mu)  \Delta \psi)  \\
&\stackrel{(ii)}{=} \int_{B} z^{N,\delta} \cdot (\nabla q - \Phi)\\
&\stackrel{(iii)}{=} -\int_{B} z^{N,\delta} \cdot \Phi,
\end{align*}
where in $(i)$ and $(iii)$ we used $z^{N,\delta}\in H^1_0(B;\R^d)$ and $\nabla \cdot z^{N,\delta}=0$ a.s.\ on $I$, respectively; and in $(ii)$ that $\psi$ solves the backward Stokes system  \eqref{eq:Phi_psi_problem} with $\kappa=1+D_d \mu$. 
To rewrite the remaining terms on the left-hand side of \eqref{eq:vN_ito_formula}, let us note that, for all $f\in \X_{-1/2}^{{\Dir}}$ and $\psi\in \Do(A)$,
\begin{equation}
\label{eq:useful_identity_extrapolated_spaces}
\langle \p_{-1/2,B}f , \psi \rangle_{-1/2,1/2}^A
= \langle f , \psi\rangle_{-1/2,1/2}^{\Dir},
\end{equation}
where $\langle\cdot,\cdot \rangle_{-1/2,1/2}^{\Dir}$ denotes the duality pairing between $\X^{\Dir}_{-1/2}(B;\R^d)$ and 
$\X^{\Dir}_{1/2}(B;\R^d)$, see \eqref{eq:extrapolated_spaces_sto_dirRd} and Lemma \ref{l:properties_extrapolated_spaces}. 
To see \eqref{eq:useful_identity_extrapolated_spaces}, it is enough to use the density of the embeddings \eqref{eq:embedding_extrapolated}, and the self-adjointness of $\p$,
\begin{align*}
\langle \p_{-1/2,B}f , \psi \rangle_{-1/2,1/2}^A
&= \lim_{n\to \infty} \langle \p_{-1/2,B}f_n , \psi \rangle_{-1/2,1/2}^A\\
&= \lim_{n\to \infty} \int_{B}f_n\cdot  \psi = \langle f , \psi\rangle_{-1/2,1/2}^{{\Dir}},
\end{align*}
where $(f_n)_n \subseteq L^2(B;\R^d)$ is a sequence such that $f_n\to f $ in $\X_{-1/2}^{{\Dir}}$.

From \eqref{eq:useful_identity_extrapolated_spaces}, $\X_{1/2}^{\Dir}(B;\R^d)=H^{1}_0(B;\R^d)$ and an integration by parts, we infer
\begin{align*}
\langle \p_{-1/2,B}\Errd , \psi \rangle_{-1/2,1/2}^A
&= \langle \Errd, \psi \rangle_{-1/2,1/2}^{{\Dir}}\\
&=
(1-D_d) \mu \int_B (\nabla v^{N,\delta} +  G^\top): \nabla \psi \\
&- c_d \mu \sum_{k,\alpha} \theta_k^N \int_B [\nabla \wt{\pi}^{N,\delta}_{k,\alpha}\otimes \sigma^\delta_{-k,\alpha}]: \nabla \psi,
\end{align*}
as well as 
\begin{align*}
\langle \p_B \Errs_{k,\alpha}(s),\psi(s)\rangle_{-1/2,1/2}^A
=-\sqrt{c_d\mu}\theta^{N}_{k} \int_B [(\sigma^\delta_{k,\alpha}\cdot\nabla) v^{N,\delta}+\sigma^\delta_{k,\alpha}\cdot G]\cdot \psi .
\end{align*}
Putting together the previous identities and using \eqref{eq:vN_ito_formula}, we have obtained, a.s., 
\begin{align}
\label{eq:vN_identity_equality_proof}
 &\int_{Q} z^{N,\delta} \cdot \Phi
 =- \sqrt{c_d\mu} \sum_{k,\alpha} \theta^{N}_{k} \int_I  \int_B [(\sigma^\delta_{k,\alpha}\cdot\nabla) v^{N,\delta}+\sigma^\delta_{k,\alpha}\cdot G]\cdot \psi\,\dd W^{k,\alpha} \\
\nonumber
 &\ \ \  +(1-D_d) \mu \int_{Q} (\nabla v^{N,\delta} +  G^\top): \nabla \psi - c_d \mu \sum_{k,\alpha} \theta_k^N \int_{Q} [\nabla \wt{\pi}^{N,\delta}_{k,\alpha}\otimes \sigma^\delta_{-k,\alpha}]: \nabla \psi.
\end{align}
Next, consider the local solution $(v^{N,\delta},\pi^{N,\delta},(\wt{\pi}^{N,\delta}_{k,\alpha})_{k,\alpha})$ to \eqref{eq:turbulent_Stokes_scaling_quantitative} (see Definition \ref{def:local_solutions_NSE}). As in the previous argument, by It\^o's formula applied to 
$$
(t,v^{N,\delta})\mapsto \langle v^{N,\delta},\psi(t)\rangle_{-1/2,1/2}^{{\Dir}},
$$
we obtain 
\begin{align}
\nonumber
\int_{Q} v^{N,\delta} \cdot \Phi &= \int_{B} \psi(t_1) \cdot v^{N,\delta}(t_1)+\int_{I \times \partial B}v^{N,\delta}\cdot [ n q-(1+D_d \mu) \partial_n \psi] 
\\
\nonumber
&  -\int_Q   F :\nabla \psi +\sqrt{c_d\mu} \sum_{k,\alpha} \theta^{N}_{k} \int_I  \int_B [(\sigma^\delta_{k,\alpha}\cdot\nabla) v^{N,\delta}+\sigma^\delta_{k,\alpha}\cdot G]\cdot \psi\,\dd W^{k,\alpha} \\
&- (1-D_d) \mu \int_{Q} \nabla v^{N,\delta}: \nabla \psi + c_d \mu \sum_{k,\alpha} \theta_k^N \int_{Q} [\nabla \wt{\pi}^{N,\delta}_{k,\alpha}\otimes \sigma^\delta_{-k,\alpha}]: \nabla \psi.
\label{eq:uN_identity_equality_proof}
\end{align}
From \eqref{eq:homogenized_PDE_weak_formulation} and the identities \eqref{eq:identity_average_Phi}, \eqref{eq:vN_identity_equality_proof} and \eqref{eq:uN_identity_equality_proof}, it readily follows that \eqref{eq:claim_uniqueness_uuhv} holds. This proves $\uh^{N,\delta}= z^{N,\delta}+v^{N,\delta}-\fint_B v^{N,\delta}$, i.e.,  \eqref{eq:identity_uh_uN_wN}
as desired.

\smallskip

\emph{Step 2: Bound on $\ph^{N,\delta}$.} Here, we prove the estimate in \eqref{eq:local_well_posedness_homogenized_SPDE_estimate} for $\ph^{N,\delta}$ and \eqref{eq:local_well_posedness_homogenized_SPDE_estimate2}. 
In particular, we assume that $F,G\in L^r_{\Progress}(\O;L^p(I;L^2(B)))$ for some $p\in [2,\infty)$.
Arguing as in Step 3 of the proof of Proposition \ref{prop:stokes_rough} (see \eqref{eq:claim_Step_3_proof_pressure_existence} and \eqref{eq:formula_bogovskii_pressure_how_to_find_it}), it holds that 
$$
\nabla \ph^{N,\delta}  = -\partial_t \uh^{N,\delta} +(1+D_d \mu)\Delta \uh^{N,\delta} + \nabla \cdot (F - (1-D_d)\mu G^\top)\text{ in }\D'(Q).
$$
Now, from Step 1 of the current proof and Lemma \ref{l:regularity_v_u}, it holds that 
$$
\uh^{N,\delta} = \Big(v^{N,\delta}- \fint_B v^{N,\delta}\Big)+ w^{N,\delta}\in L^\infty(I;L^2(B;\R^d)),
$$ 
together with a corresponding estimate. Thus,  
\begin{align*}
\partial_t \uh^{N,\delta}&\in H^{-1,p_0}(I;L^2(B;\R^d)) \ \ \text{ for all }p_0<\infty,\\
 \Delta \uh^{N,\delta}&\in L^{p_0}(I;H^{-2}(B;\R^d))\  \ \ \ \text{ for all }p_0<\infty,
\end{align*}
and 
\begin{align*}
&\big(\E\|\Delta \uh^{N,\delta}\|_{L^{p_0}(I;H^{-2}(B))}^r\big)^{1/r}+
\big(\E\|\partial_t \uh^{N,\delta}\|_{H^{-1,p_0}(I;L^2(B))}^r\big)^{1/r}\\
&\qquad\qquad  \lesssim \Big(\E\Big\|v^{N,\delta} -\fint_B v^{N,\delta}\Big\|_{L^{p_0}(I;L^2(B))}^r\Big)^{1/r}
+\big(\E\|w^{N,\delta}\|_{L^{p_0}(I;L^2(B))}^r\big)^{1/r}\\
&\qquad \qquad \lesssim \Big(\E\Big\|v^{N,\delta} -\fint_B v^{N,\delta}\Big\|_{L^{\infty}(I;L^2(B))}^r\Big)^{1/r}
+\big( \E\|G\|_{L^2(Q)}^r\big)^{1/r}\\
&\qquad \qquad +\big(\E\|\nabla v^{N,\delta}\|_{L^2(Q)}^r\big)^{1/r}
 +\big(\E \|(\nabla \wt{\pi}^{N,\delta}_{k,\alpha})_{k,\alpha}\|_{L^2(Q;\ell^2)}^r\big)^{1/r}.
\end{align*}
Since $ \nabla \cdot (F - (1-D_d)\mu G^\top)\in L^p(I;H^{-1}(B;\R^d))$ a.s.\ by assumption, it follows that 
$$
\nabla \ph^{N,\delta}\in H^{-1,p}(I;H^{-2}(B;\R^d)) \text{ a.s.\ }
$$
and thus $\ph^{N,\delta}\in H^{-1,p}(I;H^{-1}(B))$ a.s.\ follows from the formula \eqref{eq:formula_bogovskii_pressure_how_to_find_it} and the mapping properties of the Bogovskii operator $\mathcal{B}_B$ on a ball $B$, see \cite[Remark III.3.2]{Galdi_book} and the duality $H^{-1,p}(I;H^{-1}(B))=(H^{1,p'}_0(I;H^1_0(B)))^*$.
Note that the same argument also implies the required estimates for $(\E\|\ph^{N,\delta}\|_{H^{-1,p}(I;H^{-1}(B))}^r)^{1/r}$ in \eqref{eq:local_well_posedness_homogenized_SPDE_estimate} and \eqref{eq:local_well_posedness_homogenized_SPDE_estimate2} depending on the value of $p\in [2,\infty)$. 
This concludes the proof of Proposition \ref{prop:local_well_posedness_homogenized_SPDE}.
\end{proof}

\begin{remark}[On the definition of solutions to the homogenized Stokes system]
\label{r:definition_homogenized}
One is tempted to define a solution to \eqref{eq:homogenized_PDE_linear} as $\uh^{N,\delta} = v^{N,\delta}-\fint_B v^{N,\delta}+ z^{N,\delta}$, where $z^{N,\delta}$ solves \eqref{eq:definition_vN_rates}. However, as the latter SPDE is only defined on the extrapolated space $\X_{-1/2}^{A}(B)$, with this approach, it seems challenging to obtain the bound on the pressure $\ph^{N,\delta}$ as stated in \eqref{eq:local_well_posedness_homogenized_SPDE_estimate}. However, estimates for the homogenized pressure $\ph^{N,\delta}$ will be of central importance to prove \emph{local smoothing} for the homogenized Stokes system, see Lemma \ref{l:local_smoothing_homogenized} and Figure \ref{fig:scheme}.
\end{remark}

\subsection{Spatially localized scaling limits -- Proof of Theorem \ref{t:universal_scaling_limit}}
\label{ss:quantitative_scaling_limits}
In this section, we finally prove Theorem \ref{t:universal_scaling_limit}. Here, to exploit the mixing property of the transport noise in \eqref{eq:turbulent_Stokes_scaling_quantitative}, we prove quantitative error estimates in extrapolated spaces of negative smoothness, see Proposition \ref{prop:error_estimates_operators} below. 

We begin by collecting some useful facts. Recall from \eqref{eq:identity_uh_uN_wN} in Step 1 of Proposition \ref{prop:local_well_posedness_homogenized_SPDE} that 
\begin{equation}
\label{eq:equality_wNuNdifference_recall}
w^{N,\delta}= \uh^{N,\delta}- v^{N,\delta}+\fint_{B} v^{N,\delta} \text{ a.e.\ on }I\times \O,
\end{equation}
where $w^{N,\delta}=z^{N,\delta}$ is defined in \eqref{eq:definition_vN_rates} (see Lemma \ref{l:regularity_v_u}) and solves also \eqref{eq:turbulent_Stokes_scaling_quantitative_vdifference_true}. Moreover, the errors $\Errd$ and $\Errs_{k,\alpha}$ are defined in \eqref{eq:def_error_det} and \eqref{eq:def_error_sto}, respectively. To estimate the deterministic error, it is convenient to decompose it into two parts:
\begin{equation*}
\Errd = \Errdloc+ \Errdglo
\end{equation*}
where $\Errdloc$ and $\Errdglo$ are the \emph{local} and \emph{local to global} errors, respectively:
\begin{align}
\label{eq:errore_locale}
\Errdloc&=
-(1-D_d) \mu (\Delta v^{N,\delta} +\nabla \cdot G^\top)+c_d \mu\sum_{k,\alpha} \theta_k^N\nabla\cdot  (\nabla \wt{\Pi}^{N,\delta}_{k,\alpha}\otimes \sigma^\delta_{-k,\alpha})\\
\label{eq:errore_globale}
\Errdglo &= 
c_d \mu \sum_{k,\alpha} \theta_k^N\nabla\cdot  (\nabla[ \wt{\pi}^{N,\delta}_{k,\alpha}-\wt{\Pi}^{N,\delta}_{k,\alpha}]\otimes \sigma^\delta_{-k,\alpha}),
\end{align}
where $\wt{\Pi}^{N,\delta}_{k,\alpha}$ is as in Theorem \ref{t:universal_scaling_limit}, i.e.,
\begin{equation}
\label{eq:wtPi_def_quantitative_lim}
\wt{\Pi}_{k,\alpha}^{N,\delta} 
=\theta^N_k \qq_{\T^d_B} \Big[\chi\Big((\sigma^{\delta}_{k,\alpha}\cdot \nabla)v^{N,\delta} + \sigma^{\delta}_{k,\alpha} \cdot G\Big)\Big],
\end{equation}
for $\chi\in C^{\infty}_{{\rm c}}(2B)$ a fixed cutoff function satisfying $\chi|_{B}=1$, and $\T^d_B=x_0 + [-1/2,1/2)^d$ with periodic boundary conditions.

\smallskip

In light of the splitting of the deterministic errors  \eqref{eq:errore_locale}-\eqref{eq:errore_globale}, it is natural to consider the \emph{local It\^o-Stratonovich corrector}:
\begin{align}
\label{eq:local_stratonovich_corrector}
\mathfrak{L}^{N,\delta}(H)=
(1-D_d)  \nabla \cdot H^\top -c_d\sum_{k,\alpha} (\theta^N_k)^2\nabla\cdot  (\q_{\T^d_B} \big[\chi(\sigma^{\delta}_{k,\alpha}\cdot H)\big]\otimes \sigma^\delta_{-k,\alpha}),
\end{align}
where $H\in L^2(2B;\R^{d\times d})$, and to define the \emph{local to global It\^o-Stratonovich corrector}:
\begin{equation}
\label{eq:local_to_global_stratonovich_corrector}
 \mathfrak{G}^{N,\delta} (h)
=c_d \sum_{k,\alpha} \theta_k^N\nabla\cdot  (\nabla h_{k,\alpha}\otimes \sigma^\delta_{-k,\alpha}).
\end{equation}
The main ingredient in the proof of Theorem \ref{t:universal_scaling_limit} is given by the following result, which provides a quantitative rate of convergence of the deterministic and stochastic errors in the extrapolated spaces 
$
\X_{-\beta}^{{\Dir}}(B;\R^d)
$
defined in \eqref{eq:extrapolated_spaces_sto_dirRd}.
Recall that such spaces are the natural replacement of Sobolev spaces $H^{-2\sigma}(B)$ on which the Dirichlet Laplacian still generates an analytic semigroup (see Subsection \ref{ss:Helmholtz_domains} for details).

\begin{proposition}[Quantitative error estimates in extrapolated spaces]
\label{prop:error_estimates_operators}
Let $B$ be a ball of radius $\varrho\in (0,1/4]$. 
For all $\sigma_0>0$, $\kappa\in (1,2]$ and a cutoff function $\chi$ such that $\chi|_{B}=1$ and $\supp\chi\subseteq \kappa B$, there exist constants $C_0,\g_0>0$ for which the following holds for all $N\geq 1$ and $\delta\in(0,1]$ such that $\delta N\geq 1$.
\begin{enumerate}[{\rm(1)}]
\item\label{it:error_estimates_operators1} 
For all $H\in L^2(B;\R^{d})$,
$$
\Big(
\sum_{k,\alpha}(\theta^N_k)^2 \|\sigma^{\delta}_{k,\alpha}\cdot H\|_{\X_{-\sigma_0}^{\Dir}(B)}^2\Big)^{1/2}
\leq \frac{C_0}{(\delta N)^{\g_0}} \|H\|_{L^2(B)}.
$$
\item\label{it:error_estimates_operators2} 
For all $H\in L^2(\kappa B;\R^{d\times d})$ such that $\Tr(H)=0$ on $\kappa B$ and $\nabla \cdot H=0$ in $\D'(\kappa   B)$, 
$$
 \|\mathfrak{L}^{N,\delta}(H)\|_{\X_{-1/2-\sigma_0}^{{\Dir}}(B;\R^d)}
 \leq \frac{C_0}{(\delta N)^{\g_0}}\|H\|_{L^2(\kappa B)},
$$
\item\label{it:error_estimates_operators3} 
For all $h=(h_{k,\alpha})_{k,\alpha}\in L^2(\kappa  B;\ell^2)$ such that $\Delta h_{k,\alpha}=0$ in $\D'(\kappa B)$,
$$
\| \mathfrak{G}^{N,\delta} (h)\|_{\X_{-1/2-\sigma_0}^{\Dir}(B;\R^d)}
\leq \frac{C_0}{(\delta N)^{\g_0}}\|(\nabla h_{k,\alpha})_{k,\alpha}\|_{L^2(\kappa  B;\ell^2)}.
$$
\end{enumerate}
\end{proposition}

The proof of the above result is postponed to Subsection \ref{ss:preliminary_results_convergence}. The key is that the rate of convergence only depends on the cumulative frequency $L=\delta N$ (see Subsections \ref{ss:control_energy_intro}, \ref{ss:caccioppoli_intro} and \ref{ss:role_pressure_intro} for additional comments).

\smallskip

Let us anticipate that, although Proposition \ref{prop:error_estimates_operators} is formulated for possibly small $\sigma_0>0$, the proof follows by first obtaining a rate for $\sigma_0\gg 1$ and then recovering the case of possibly small $\sigma_0>0$ by interpolation. For $\sigma_0\gg 1$, the emergence of the cumulative frequency is due to a Parseval-type identity (see Lemma \ref{l:Bessel_property} below), which naturally extends the argument with full periodic boundary conditions, see e.g., \cite[Lemma 6.3]{A22} and \cite[Proposition 3.6, Step 1]{FGL21}. 
Finally, let us point out that the estimate of the local to global  It\^o-Stratonovich corrector $ \mathfrak{G}^{N,\delta}(h)
$ does not depend on the harmonicity of $h_{k,\alpha}$, but only on their local smoothing improvement, see \eqref{eq:Caccioppoli_harmonic_functions} below. Indeed, because of the latter, in the products $\nabla h_{k,\alpha}\otimes \sigma^{\delta}_{-k,\alpha}$, only the second component is highly oscillating as $\delta N\to \infty$, and therefore we obtain a standard averaging effect. The appearance of the gradient on the right-hand side in Proposition \ref{prop:error_estimates_operators}\eqref{it:error_estimates_operators3} is due to a local smoothing estimate. The reader is referred to Subsection \ref{sss:proof_error_estimates_operators3} for the proof.

\smallskip

Before going into the proof of Theorem \ref{t:universal_scaling_limit}, we need a preliminary result. Let $\chi$ be a smooth cutoff function such that $\supp\chi\subseteq B$, where $B$ is a ball with radius $\leq 1$. 
One can check that the multiplication operator by $\chi$ is continuous on $\X_{-\sigma}^{{\Dir}}(B)$ with values in $H^{-2\sigma}(\T^d_B)
$ for all $\sigma>0$:
\begin{equation}
\label{eq:multiplication_operator_bounded_extrapolated}
f\mapsto \chi f, \qquad 
\X_{-\sigma}^{{\Dir}}(B) \to H^{-2\sigma}(\T^d_B).
\end{equation} 
Indeed, to prove the latter, fix an integer $k\geq 2\sigma$. By density \eqref{eq:embedding_extrapolated} and duality (see Lemma \ref{l:properties_extrapolated_spaces}), it suffices to note that, for all $f\in L^2(B)\stackrel{{\rm d}}{\embed}\X_{-k}^{{\Dir}}(B) $ ,
\begin{align*}
\Big|\int_{B} \chi f \varphi\Big|
=|\langle  f, \chi\varphi \rangle_{-k,k}^{{\Dir}}|
\leq \|f\|_{\X^{\Dir}_{-k}(B)} \|\chi\varphi\|_{H^{2k}_0(B)}
\lesssim \|f\|_{\X^{\Dir}_{-k}(B)} \|\varphi\|_{H^{2k}(\T^d_B)},
\end{align*}
and hence \eqref{eq:multiplication_operator_bounded_extrapolated} holds with $\sigma=k$. The general case follows by interpolation with the trivial case $\sigma=0$.

\begin{proof}[Proof of Theorem \ref{t:universal_scaling_limit}]
We divide the proof into several steps. For notational convenience, throughout this proof, we assume that
\begin{equation}
\label{eq:chi_has_small_support}
\chi=1 \text{ on }(5/4)B \qquad \text{ and }\qquad 
\supp\chi \subseteq (3/2)B.
\end{equation}
The general case follows analogously.

\smallskip

\emph{Step 1: Reduction to the case $p_0=p_1=2$ and $\sigma=1$.} 
If Theorem \ref{t:universal_scaling_limit} holds in the case $p_0=p_1=2$ and $\sigma=1$, then by interpolation, it suffices to prove the following uniform bound:
\begin{align}
\label{eq:claim_uniform_bound_convergence_zero1}
\Big(\E\Big\| \uh^{N,\delta}-\Big(v^{N,\delta} -\fint_B v^{N,\delta}\Big)\Big\|_{L^{\infty}(I;L^2(B))}^r \Big)^{1/r}&\leq C_0 \wt{\Eno}_Q^{\#},\\
\label{eq:claim_uniform_bound_convergence_zero2}
\big(\E\|\nabla \Pi^{N,\delta}-\nabla \Pi^0\|_{L^2(I;H^{-1}(\T^d_B))}^r\big)^{1/r}&\leq C_0 \wt{\Eno}_Q^{\#}, \\
\label{eq:claim_uniform_bound_convergence_zero3}
\big(\E\|(\nabla\wt{\Pi}_{k,\alpha}^{N,\delta})_{k,\alpha}\|_{L^{p_2}(I;H^{-1}(\T^d_B;\ell^2))}^r\big)^{1/r} &\leq C_0 \wt{\Eno}_Q^{\#},
\end{align}
where $\wt{\Eno}_{Q}^{\#}$ is the modified sharp energy as in \eqref{eq:modified_sharp_energy_convergence_rate}.
The bound \eqref{eq:claim_uniform_bound_convergence_zero1} follows immediately from the stochastic Caccioppoli inequality of Theorem \ref{t:caccioppoli} and Proposition \ref{prop:local_well_posedness_homogenized_SPDE}. Moreover, from \eqref{eq:approximated_pressure_1_quantitative_convergence1} and \eqref{eq:approximated_pressure_1_quantitative_convergence3}, we have 
\begin{align}
\label{eq:difference_approximation_deterministic_pressure_proof}
&\nabla [\Pi^{N,\delta}- \Pi^0]\\
\nonumber
&=\mu\,\q_{\T^d_{B}} \Big[\chi\Big( (1-D_d)\nabla \cdot (G^\top+\nabla v^{N,\delta})- c_d  \sum_{k,\alpha} \theta_k^N\nabla\cdot  (\nabla \wt{\pi}^{N,\delta}_{k,\alpha}\otimes \sigma^\delta_{-k,\alpha} \big)\Big)\Big].
\end{align} 
From \eqref{eq:chi_has_small_support} and the continuity of $\q_{\T^d_B}: H^{-1}(\T^d_B)\to H^{-1}(\T^d_B)$, it follows that 
$$
\|\nabla [\Pi^{N,\delta}- \Pi^0]\|_{H^{-1}(\T^d_B)}
\lesssim \|G\|_{L^2((3/2)B)}+\|\nabla v^{N,\delta}\|_{L^2((3/2)B)}
+\|(\nabla \wt{\pi}_{k,\alpha})_{k,\alpha}^{N,\delta}\|_{L^2((3/2)B;\ell^2)}.
$$
Hence, the estimate \eqref{eq:claim_uniform_bound_convergence_zero2} follows from the above and Theorem \ref{t:caccioppoli}.
Finally, to prove the uniform bound on the stochastic pressures, note that 
\begin{equation}
\begin{aligned}
\label{eq:trick_adding_average_local_pressure}
&\wt{\Pi}_{k,\alpha}^{N,\delta}
= \theta^N_k \qq_{\T^d_B} \big[\chi\big( \nabla \cdot (\wt{v}^{N,\delta} \otimes \sigma^\delta_{k,\alpha})+ \sigma_{k,\alpha}^{\delta}\cdot G\big)\big]  \\ 
& \text{ where } \ \ \ \wt{v}^{N,\delta}= v^{N,\delta} -\fint_{(3/2)B}v^{N,\delta},
\end{aligned}
\end{equation}
as $\nabla \cdot \sigma^{\delta}_{k,\alpha}=0$.
Therefore, 
\begin{align*}
\|(\nabla\wt{\Pi}^{N,\delta}_{k,\alpha})_{k,\alpha}\|_{L^{p_2}(I;H^{-1}(\T^d_B;\ell^2))}
&\leq C \|\wt{v}^{N,\delta}\|_{L^{\infty}(I;L^2((3/2)B))}+ C\|G\|_{L^{p_2}(I;L^2(2B))}.
\end{align*}
Taking the $L^r(\O)$-norm, the uniform bound in \eqref{eq:claim_uniform_bound_convergence_zero3} follows from Theorem \ref{t:caccioppoli}.

\smallskip

In light of Step 1, without further mention, in the rest of the proof we assume $p_0=p_1=p_2=2$ and $\sigma=1$.

\smallskip

\emph{Step 2: Quantitative convergence rate for the velocity difference $\uh^{N,\delta} -\big(v^{N,\delta} -\fint_B v^{N,\delta}\big)$.}
Recall the identity \eqref{eq:equality_wNuNdifference_recall},
where $w^{N,\delta}$ is given in \eqref{eq:definition_vN_rates}, and $\Errd$ and $\Errs$ are as in \eqref{eq:def_error_det} and \eqref{eq:def_error_sto}, respectively.
The decomposition in \eqref{eq:errore_locale} and the definitions \eqref{eq:local_stratonovich_corrector}-\eqref{eq:local_to_global_stratonovich_corrector} yield
\begin{align}
\label{eq:identification_error_loc}
\Errdloc &= -\mu\,\mathfrak{L}^{N,\delta} \big((\nabla v^{N,\delta})^\top + G\big),\\
\label{eq:identification_error_glo}
\Errdglo &= \mu\,\mathfrak{G}^{N,\delta} (\wt{\pi}_{k,\alpha}^{N,\delta}-\wt{\Pi}_{k,\alpha}^{N,\delta}).
\end{align}
Next, by consistency of the family $(A_{-\sigma})_{\sigma\in\R}$, $w^{N,\delta}=z^{N,\delta}$ and \eqref{eq:definition_vN_rates}, it holds that 
\begin{align*}
w^{N,\delta}(t) 
&= \int_{t_0-1}^t e^{-(1+D_d\mu)A_{-1}(t-s)} \p_{-1,B}\Errd(s)\,\dd s\\
\nonumber
 &+ \sum_{k,\alpha}\int_{t_0-1}^t e^{-(1+D_d\mu)A_{-1}(t-s)}\p_{-1/2,B} \Errs_{k,\alpha}(s)\,\dd W^{k,\alpha}_s.
\end{align*}
Note that by the properties of extrapolated operators (see Lemma \ref{l:properties_extrapolated_spaces}) and the self-adjointness of $A$, it follows that 
$$A_{-1}:\Do(A_{-1})\subseteq \X_{-1}^A (B)\to \X_{-1}^A (B)
$$ 
with domain $\Do(A_{-1})= \Ls^2(B)$, has bounded $H^{\infty}$-calculus of angle $0$ (see \cite[Proposition 10.2.23]{Analysis2}). Thus, by stochastic maximal regularity 
(see e.g., \cite[Corollary 7.4]{NVW13}) and Lemma \ref{l:boundedness_p_extrapolated_spaces}, it holds that 
\begin{align*}
&\big(\E \|w^{N,\delta}\|_{L^2(Q)}^r\big)^{1/r}\\
&\lesssim \big(\E\|\Errd\|_{L^2(I;\X_{-1}^{{\Dir}}(B;\R^d))}^r\big)^{1/r} +\big(\E\|\Errs\|_{L^2(I;\ell^2(\X_{-1/2}^{{\Dir}}(B;\R^d)))}^r\big)^{1/r} \\
&\lesssim \big(\E\|\Errdloc\|_{L^2(I;\X_{-1}^{{\Dir}}(B;\R^d))}^r\big)^{1/r}
+\big(\E\|\Errdglo\|_{L^2(I;\X_{-1}^{{\Dir}}(B;\R^d))}^r\big)^{1/r}\\
&+\big(\E\|\Errs\|_{L^2(I;\ell^2(\X_{-1/2}^{{\Dir}}(B;\R^d)))}^r\big)^{1/r} \\
&\lesssim \frac{1}{(\delta N)^{\g}} \Big[\big(\E \|\nabla v^{N,\delta}\|_{L^2((3/2)Q)}^r\big)^{1/r}+\big(\E \|(\nabla \wt{\pi}^{N,\delta}_{k,\alpha})_{k,\alpha}\|_{L^2((3/2)Q;\ell^2)}^r\big)^{1/r} \\
&+ \big(\E \|(\nabla \wt{\Pi}^{N,\delta}_{k,\alpha})_{k,\alpha}\|_{L^2((3/2)Q;\ell^2)}^r\big)^{1/r}+\big(\E \|G\|_{L^2((3/2)Q)}^r\big)^{1/r}\Big],
\end{align*}
where the last step follows from Proposition \ref{prop:error_estimates_operators} and \eqref{eq:identification_error_loc}-\eqref{eq:identification_error_glo}.
Since
$$
 \|(\nabla \wt{\Pi}^{N,\delta}_{k,\alpha})_{k,\alpha}\|_{L^2((3/2)B;\ell^2)}
\lesssim  \|\nabla v^{N,\delta}\|_{L^2((3/2)B)}+ \|G\|_{L^2((3/2)B)}
$$
a.e.\ on $I\times \O$, the conclusion of Step 2 follows from the previous estimates and the stochastic Caccioppoli inequality.

\smallskip

\emph{Step 3: Convergence rate for the local pressures $\nabla [\Pi^{N,\delta} -\Pi^0]$ and $(\nabla \wt{\Pi}_{k,\alpha}^{N,\delta})_{k,\alpha}$, where $\Pi^{N,\delta}$, $\wt{\Pi}_{k,\alpha}^{N,\delta}$ and $\Pi^0$ are as in Theorem \ref{t:universal_scaling_limit}.}
We begin by discussing the convergence of $(\nabla \wt{\Pi}_{k,\alpha}^{N,\delta})_{k,\alpha}$. 
Note that, by Fubini's theorem,
\begin{align*}
&\big(\E \|(\nabla\wt{\Pi}_{k,\alpha}^{N,\delta})_{k,\alpha}\|_{L^2(I;H^{-1}(\T^d_B;\ell^2))}^r\big)^{1/r}\\
&\qquad \qquad 
\lesssim
\big(\E \|(\nabla\wt{\Pi}_{k,\alpha}^{N,\delta})_{k,\alpha}\|_{L^2(I;\ell^2(H^{-1}(\T^d_B)))}^r\big)^{1/r}\\
&\qquad \qquad 
\stackrel{(i)}{\lesssim}
\big(\E \|\theta^N_k\chi[(\sigma^{\delta}_{k,\alpha}\cdot \nabla)v^{N,\delta} + \sigma^{\delta}_{k,\alpha} \cdot G]\|_{L^2(I;\ell^2(H^{-1}(\T^d_B)))}^r\big)^{1/r}\\
&\qquad \qquad 
\stackrel{(ii)}{\lesssim }
\big(\E \|(\theta^N_k(\sigma^{\delta}_{k,\alpha}\cdot \nabla)v^{N,\delta} )_{k,\alpha}\|_{L^2(I;\ell^2(\X^{\Dir}_{-1/2}((3/2)B)))}^r\big)^{1/r}\\
&\qquad \qquad +\big(\E \|(\theta^N_k(\sigma^{\delta}_{k,\alpha} \cdot G ))_{k,\alpha}\|_{L^2(I;\ell^2(\X^{\Dir}_{-1/2}((3/2)B)))}^r\big)^{1/r}
\end{align*}
where in $(i)$ we used the boundedness of $\q_{\T^d_B}$ and in $(ii)$  \eqref{eq:multiplication_operator_bounded_extrapolated}. The claimed convergence follows from Proposition \ref{prop:error_estimates_operators}\eqref{it:error_estimates_operators1} and the stochastic Caccioppoli inequality.

\smallskip

To prove the claimed convergence for the deterministic local pressures, note that, as in \eqref{eq:identification_error_loc}-\eqref{eq:identification_error_glo}, from \eqref{eq:difference_approximation_deterministic_pressure_proof} and \eqref{eq:local_stratonovich_corrector}-\eqref{eq:local_to_global_stratonovich_corrector}, we have 
\begin{align*}
\nabla [\Pi^{N,\delta}- \Pi^0]
=\mu\q_{\T^d_B} \Big[\chi\Big( \mathfrak{L}^{N,\delta} ((\nabla v^{N,\delta} )^\top+ G)+\mathfrak{G}^{N,\delta} ( \wt{\Pi}_{k,\alpha}^{N,\delta}-\wt{\pi}_{k,\alpha}^{N,\delta})\Big)\Big].
\end{align*}
Thus, the claimed convergence for $\nabla [\Pi^{N,\delta}- \Pi^0]$ readily follows from \eqref{eq:multiplication_operator_bounded_extrapolated} and Proposition \ref{prop:error_estimates_operators}\eqref{it:error_estimates_operators2}-\eqref{it:error_estimates_operators3}, and a.e.\ on $I\times \O$, 
\begin{align*}
\Tr((\nabla v^{N,\delta})^\top+ G)=
\nabla \cdot v^{N,\delta}+ \Tr(G)&=0 \quad \text{ on }\ 2B,\\
\nabla \cdot [(\nabla v^{N,\delta})^\top+ G]= \nabla \cdot G&=0 \quad \text{ in }\ \D'(2B),
\end{align*}
where we used $\nabla \cdot v^{N,\delta}=0$ and the assumption \eqref{eq:assumption_G_trace_zero_quantiative} in Theorem \ref{t:universal_scaling_limit}. 
\end{proof}

\subsection{Error estimates in extrapolated spaces -- Proof of Proposition \ref{prop:error_estimates_operators}}
\label{ss:preliminary_results_convergence}
This subsection is devoted to the proof of Proposition \ref{prop:error_estimates_operators}. The key result is the following, where the cumulative frequency 
$
L=\delta N
$ 
for the oscillating rescaled stochastic Stokes system \eqref{eq:turbulent_Stokes_scaling_quantitative} arises.

\begin{lemma}[Emergence of the cumulative frequency for rescaled coefficients]
\label{l:Bessel_property}
Let $B$ be a ball of radius $\varrho\leq 1/2$. Then, for all $f\in L^2(B;\R^d)$ and $\delta\in (0,1]$, 
\begin{equation}
\label{eq:Bessel_property1}
\sum_{k,\alpha} \Big|\int_{B} 
 \sigma^{\delta}_{k,\alpha} \cdot f   \Big|^2
\leq \delta^{-d} \|f\|_{L^2(B)}^2.
\end{equation}
In particular, there exists a constant $C$ such that, for all $N\geq1$, $\delta\in (0,1]$, balls $B$ of radius at most $1/2$, and $f\in L^2(B;\R^d)$, 
\begin{equation}
\label{eq:Parseval_combined_with_decay}
\sum_{k,\alpha}(\theta_k^N)^2 \Big|\int_{B} 
 \sigma^{\delta}_{k,\alpha} \cdot f  \Big|^2
\leq \frac{C}{(\delta N)^d }  \|f\|_{L^2(B)}^2.
\end{equation}
\end{lemma}

The key point in \eqref{eq:Parseval_combined_with_decay} is that the rate depends exclusively on the cumulative frequency $L=\delta N$.

\begin{proof}
Clearly, \eqref{eq:Parseval_combined_with_decay} follows from \eqref{eq:Bessel_property1} and $\|\theta^N\|_{\ell^\infty}\leq C N^{-d/2}$, where the latter is a consequence of the choice of $\theta^N$, see \eqref{eq:choice_thetaN}.
The proof of \eqref{eq:Bessel_property1} is elementary. For the reader's convenience, we include some details. Without loss of generality, we may assume that the ball $B$ is centered at $0$. For $f\in L^2(B;\R^d)$, let $f_B\stackrel{{\rm def}}{=} \one_B f$. Note that $(\wt{\sigma}^\delta_{k,\alpha})_{k,\alpha}$ where 
$$
\wt{\sigma}^\delta_{k,\alpha}(x)\stackrel{{\rm def}}{=}\delta^{d/2}e^{-2\pi \im \delta x\cdot k} a_{k,\alpha}
$$ 
is an incomplete orthonormal system of $L^2((1/2)[-1/\delta,1/\delta)^d;\C^d)$, and hence
\begin{align*}
\sum_{k,\alpha}\Big|\int_{B} 
 \sigma^{\delta}_{k,\alpha} \cdot f   \Big|^2
&=\sum_{k,\alpha} \Big|\int_{(1/2)[-1/\delta,1/\delta)^d}  \sigma^{\delta}_{k,\alpha}\cdot f_B \Big|^2\\
&=\delta^{-d}\sum_{k,\alpha} \Big|\int_{(1/2)[-1/\delta,1/\delta)^d}   \wt{\sigma}^{\delta}_{k,\alpha}\cdot f_B\Big|^2
\leq \delta^{-d}\|f_B\|_{L^2((1/2)[-1/\delta,1/\delta)^d)}^2,
\end{align*}
where we used $\delta\leq 1$, and hence $B\subseteq (1/2)[-1/\delta,1/\delta)^d$.
Since $\|f_B\|_{L^2((1/2)[-1/\delta,1/\delta)^d)}=\|f\|_{L^2(B)}$, the above gives the desired inequality.
\end{proof}

With the above at our disposal, we can immediately prove Proposition \ref{prop:error_estimates_operators}\eqref{it:error_estimates_operators1}.

\begin{proof}[Proof of Proposition \ref{prop:error_estimates_operators}\eqref{it:error_estimates_operators1}]
Since $\|\theta^N\|_{\ell^2}=1$, it follows that 
$$
\Big(
\sum_{k,\alpha}(\theta^N_k)^2 \|\sigma^{\delta}_{k,\alpha}\cdot H\|_{L^2(B)}^2\Big)^{1/2}
\leq C \|H\|_{L^2(B;\R^d)},
$$
where $C$ depends only on $d$. Since $\X^{\Dir}_0(B)=L^2(B)$ and in light of the interpolation result in Lemma \ref{l:properties_extrapolated_spaces} applied to the extrapolated scale associated with the Dirichlet Laplacian \eqref{eq:extrapolated_spaces_sto_dir}, it is enough to prove the estimate in \eqref{it:error_estimates_operators1} for $\sigma_0>1$ sufficiently large. 

Let $(\lambda_j)_{j\geq 1}$ and $(\varphi_j)_{j\geq 1}$ be the eigenvalues and eigenfunctions of (minus) the Dirichlet Laplacian, see \eqref{eq:dirichlet_oper}, respectively. It follows from the definition of the extrapolated spaces \eqref{eq:extrapolated_space1} that 
\begin{align}
\label{eq:estimate_martingale_part_zero}
\sum_{k,\alpha}(\theta^N_k)^2 \|\sigma^{\delta}_{k,\alpha}\cdot H\|_{\X_{-\sigma_0}^{\Dir}(B)}^2
&=\sum_{k,\alpha}\sum_{j\geq 1}(\theta^N_k)^2 \lambda_j^{-2\sigma_0}|\langle \sigma^\delta_{k,\alpha} \cdot H, \varphi_j \rangle^{\Dir}_{-\sigma_0,\sigma_0} |^2\\
\nonumber
&=\sum_{j\geq 1}\lambda_j^{-2\sigma_0} \sum_{k,\alpha}(\theta^N_k)^2 
\Big|\int_{B} \sigma^\delta_{k,\alpha}\cdot (H \varphi_j) \,\dd x \Big|^2\\
\nonumber
&\stackrel{(i)}{\leq} \frac{C}{(\delta N)^d}\sum_{j} \lambda_j^{-2\sigma_0}
\int_B |H \varphi_j|^2\,\dd x \\
\nonumber
&\leq  \frac{C}{(\delta N)^d}
\|H\|_{L^2(B)}^2\Big(\sum_{j} \lambda_j^{-2\sigma_0}
\|  \varphi_j \|_{L^\infty(B)}^2\Big),
\end{align}
where in $(i)$ we applied \eqref{eq:Parseval_combined_with_decay} and the compatibility of the inner products (see Lemma \ref{l:properties_extrapolated_spaces}).
If $\sigma_0\gg 1$ (depending only on $d$), the last sum converges by Weyl's law and the regularity of the eigenfunctions of the Dirichlet Laplacian together with Sobolev embeddings.
\end{proof}

The proofs of Proposition \ref{prop:error_estimates_operators}\eqref{it:error_estimates_operators2} and \eqref{it:error_estimates_operators3} are more elaborate than \eqref{it:error_estimates_operators1}, and will be given in Subsections \ref{sss:proof_error_estimates_operators2} and \ref{sss:proof_error_estimates_operators3} below.

\subsubsection{Local It\^o-Stratonovich corrector error estimate -- Proof of Proposition \ref{prop:error_estimates_operators}\eqref{it:error_estimates_operators2}}
\label{sss:proof_error_estimates_operators2}
Here, we prove quantitative convergence rates for \eqref{eq:local_stratonovich_corrector}, which involves the local It\^o-Stratonovich corrector. As above, the emergence of the cumulative frequency $L=\delta N $ is due to Lemma \ref{l:Bessel_property}.
For notational convenience, we assume $\kappa=2$. The general case is analogous.
To prepare for the proof of Proposition \ref{prop:error_estimates_operators}\eqref{it:error_estimates_operators2}, for $\Phi\in H^1_0(B;\R^d)$ and $H\in L^2(2B;\R^{d\times d})$, we denote by $A^{N,\delta}$ the associated bilinear form (up to a constant):
\begin{align*}
A^{N,\delta}(H,\Phi)
= c_d^{-1}
\langle\mathfrak{L}^{N,\delta}(H),\Phi\rangle_{-1/2,1/2}^{\Dir},
\end{align*}
which is given by
\begin{align}
\label{eq:definition_A_bilinear}
&A^{N,\delta}(H,\Phi) \\
\nonumber
&=\int_{B} \Big( - D_d' H^\top : \nabla \Phi + \sum_{k,\alpha}  (\theta^N_k)^2 \q_{\T^d_B} [\chi(\sigma^{\delta}_{k,\alpha}\cdot H)] \cdot [ (\sigma^{\delta}_{-k,\alpha} \cdot  \nabla) \Phi]\Big)\,\dd x
\end{align} 
where $D_d'=(1-D_d)/c_d=\frac{ (d+1)}{d(d+2)}$ and $\chi\in C^\infty_{{\rm c}}(2B)$ satisfies $\chi|_{B}=1$.

\smallskip

The following is a consequence of Lemma \ref{l:Bessel_property}. For convenience of exposition, we postpone the proof to Appendix \ref{app:cancellation_ito_stratonovich}, see Subsection \ref{ss:local_ito_stratonovich_corrector_proof}.

\begin{lemma}[Bilinear error estimate for local It\^o-Stratonovich corrector]
\label{l:local_ito_stratonovich_corrector}
Let $B$ be a ball of radius $\varrho\in (0,1/4]$. 
Let $\chi$ be a cutoff function such that $\chi|_{B}=1$ and $\supp\chi\subseteq 2 B$. Then there exist constants $C_0,\beta>0$ for which the following holds. For all $N\geq 1$, $\delta \in (0,1]$ such that $\delta N\geq 1$, $\Phi\in  W^{2,\infty}(B;\R^d)\cap H^1_0(B;\R^d)$, and $H\in L^2(2B;\R^{d\times d})$ satisfying 
\begin{equation*}
\Tr (H)=0 \text{ a.e.\ on }2 B, \qquad \text{ and }\qquad 
\nabla \cdot H=0 \ \text{ in }\D'(2 B),
\end{equation*}
it holds that 
\begin{equation}
\label{eq:bilinearity_correction_claim}
|A^{N,\delta}(H,\Phi) |\leq \frac{C_0}{(\delta N)^{\beta}}\|H\|_{L^2(2B)}
\|\Phi\|_{W^{2,\infty}(B)}.
\end{equation}
\end{lemma}

As in Lemma \ref{l:Bessel_property}, the rate depends only on $L=\delta N$. The key point in \eqref{eq:bilinearity_correction_claim} is that no assumption on $H$ is needed besides incompressibility and trace-freeness. This was crucial in the proof of Theorem \ref{t:universal_scaling_limit} in Subsection \ref{ss:quantitative_scaling_limits}.

\begin{proof}[Proof of Proposition \ref{prop:error_estimates_operators}\eqref{it:error_estimates_operators2}]
Clearly, from $\|\theta^N\|_{\ell^2}=1$ and \eqref{eq:local_stratonovich_corrector}, it follows that
$$
 \|\mathfrak{L}^{N,\delta}(H)\|_{H^{-1}(B)}
 \leq C\|H\|_{L^2(\kappa B)},
$$
for some $C>0$ depending on $d$. Thus, as in the proof of Proposition \ref{prop:error_estimates_operators}\eqref{it:error_estimates_operators1}, due to the identity $\X_{-1/2}^{{\Dir}}(B;\R^d)=H^{-1}(B;\R^d)$, it suffices to prove the claim for $\sigma_0\gg 1$ depending only on $d$. 
Let $(e_h)$ be the standard basis in $\R^d$. As in \eqref{eq:estimate_martingale_part_zero}, letting $(\lambda_j)_{j\geq 1}$ and $(\varphi_j)_{j\geq 1}$ be the eigenvalues and eigenfunctions of (minus) the Dirichlet Laplacian, see \eqref{eq:dirichlet_oper}, respectively, we have
\begin{align*}
\|\mathfrak{L}^{N,\delta}(H)\|_{\X_{-\sigma_0}^{{\Dir}}(B;\R^d)}
&\eqsim_d \Big(\sum_{1\leq h\leq d}\sum_{j\geq 1} \lambda^{-2\sigma_0}_j|\langle \mathfrak{L}^{N,\delta}(H),\varphi_j e_h \rangle^{\Dir}_{-\sigma_0,\sigma_0} |^2\Big)^{1/2}\\
&\stackrel{(i)}{\leq}   \frac{C}{(N\delta)^\beta}\|H\|_{L^2(2B)} 
\Big(\sum_{j\geq 1} \lambda^{-2\sigma_0}_j\|\varphi_j\|_{W^{2,\infty}(B)}^2\Big)^{1/2}\\
&\stackrel{(ii)}{\lesssim}   \frac{C}{(N\delta)^\beta}\|H\|_{L^2(2B)}
\end{align*}
where in $(i)$ we used Lemma \ref{l:local_ito_stratonovich_corrector} and 
$$
\langle \mathfrak{L}^{N,\delta}(H),\varphi_j e_h \rangle_{-\sigma_0,\sigma_0}^{\Dir}
=
\langle \mathfrak{L}^{N,\delta}(H),\varphi_j e_h \rangle_{-1/2,1/2}^{\Dir}=A^{N,\delta}(H,\varphi_j e_h)
$$ 
due to the compatibility of the inner products (see Lemma \ref{l:properties_extrapolated_spaces}), and in $(ii)$ the known bounds on the eigenvalues of the Laplacian and $\sigma_0\gg 1$ depending only on $d$.
\end{proof}

\subsubsection{Local to global It\^o-Stratonovich corrector error estimate -- Proof of Proposition \ref{prop:error_estimates_operators}\eqref{it:error_estimates_operators3}}
\label{sss:proof_error_estimates_operators3}
Before going into the proof of Proposition \ref{prop:error_estimates_operators}\eqref{it:error_estimates_operators3}, we begin by collecting some useful facts. 
First, for a function $h\in L^2(\kappa B)$ that is harmonic on $\kappa B$ for some $\kappa>1$, by a Caccioppoli-type argument, one has
\begin{equation}
\label{eq:Caccioppoli_harmonic_functions}
[h]_{{\rm Lip}(B)}\leq C \|h\|_{H^{-1}(\kappa B)}
\end{equation}
with an implicit constant independent of $h$. 
The above can also be proved as a consequence of the mean property of harmonic functions. Indeed, if $h$ is harmonic on $B$, then $h(x)=\int_{\kappa B} \phi(x-y)h(y)\,\dd y$ for all $x\in B$, where $\phi$ is a radial smooth function satisfying $\supp\phi\subseteq (\kappa-1)B$ and $\int_{(\kappa-1)B} \phi=1$. Thus, on $B$,
$$
|\nabla h |\leq\Big|\int_{\kappa B} \nabla \phi(\cdot -y) h(y)\,\dd y \Big|\leq \|\nabla \phi\|_{H^{1}_0(\kappa B)}\|h\|_{H^{-1}(\kappa B)}\ \text{ on }B,
$$
and hence \eqref{eq:Caccioppoli_harmonic_functions} follows. Second, for all $\xi\in \R^d$, it holds that
\begin{equation}
\label{eq:exponential_rate_Hg}
\|x\mapsto e^{i\xi \cdot x}\|_{H^{-\g}(B)}\lesssim |\xi|^{-\g}.
\end{equation}
To see the above, note that, for all $\phi\in C^\infty_{{\rm c}}(\R^d)$ such that $\phi|_{B}=1$,
\begin{align}
\label{eq:exponential_negative_smoothness_scaling}
\|x\mapsto e^{i\xi \cdot x}\|_{H^{-\g}(B)}^2
&\leq C \|x\mapsto  \phi(x)e^{i\xi\cdot x}\|_{H^{-\g}(\R^d)}^2\\
\nonumber
&= \int_{\R^d} (1+|\eta+\xi|^2)^{-\g} |\wh{\phi}(\eta)|^2\,\dd \eta\\
\nonumber
&\leq C_\g (1+|\xi|^2)^{-\g} \int_{\R^d} (1+|\eta|^2)^{\g} |\wh{\phi}(\eta)|^2\,\dd \eta,
\end{align}
where we used Peetre's inequality $(1+|\eta+\xi|^2)^{-\g}\leq C_\g (1+|\xi|^2)^{-\g}(1+|\eta|^2)^{\g} $. 
We are in a position to prove Proposition \ref{prop:error_estimates_operators}\eqref{it:error_estimates_operators3}.

\begin{proof}[Proof of Proposition \ref{prop:error_estimates_operators}\eqref{it:error_estimates_operators3}]
As above, for simplicity, we assume $\kappa=2$. The general case is analogous. 
As in the proofs of Proposition \ref{prop:error_estimates_operators}\eqref{it:error_estimates_operators1}-\eqref{it:error_estimates_operators2}, using the identity $\X_{-1/2}^{{\Dir}}(B)=H^{-1}(B)$ and $\|\theta^N\|_{\ell^2}=1$ for all $N\geq 1$, one has
\begin{equation}
\label{eq:uniform_bound_global_to_local}
\|\mathfrak{G}^{N,\delta} (h)\|_{\X_{-1/2}^{{\Dir}}(B;\R^d)}
\lesssim \|(\nabla h_{k,\alpha})_{k,\alpha}\|_{L^2(B;\ell^2)}.
\end{equation}
Hence, by complex interpolation (see Lemma \ref{l:properties_extrapolated_spaces}), it remains to prove that the estimate in \eqref{it:error_estimates_operators3} holds for $\sigma_0$ large, depending only on $d$.
To this end, integrating by parts, for all $\varphi\in \X^{{\Dir}}_{1/2}(B;\R^d)$,
\begin{align*}
\langle  \mathfrak{G}^{N,\delta} (h), \varphi\rangle_{-1/2,1/2}^{\Dir}
=-c_d  \sum_{k,\alpha}\theta_k^N \int_{B}\nabla h_{k,\alpha}\cdot  \big[ (\sigma^\delta_{-k,\alpha}\cdot \nabla) \varphi\big]=- c_d \sum_{k,\alpha} \theta^N_k \int_{B}  \sigma_{-k,\alpha}^\delta \cdot F_{k,\alpha}
\end{align*}
where 
\begin{equation}
\label{eq:definition_Fkappaalpha}
F_{k,\alpha}=[\nabla \varphi]^{\top}\cdot  \nabla h_{k,\alpha}. 
\end{equation}
Since $H^{-\g}(B)=(H^{\g}(B))^*$ and $H^{\g}(B)=H^{\g}_0(B)$ for $\g<1/2$, it holds that
\begin{align*}
\Big| \sum_{k,\alpha}\theta_k^N  \int_{B} \sigma^\delta_{-k,\alpha}\cdot F_{k,\alpha} \Big|
&=\Big| \sum_{k,\alpha}\theta_k^N  \int_{B} e^{-2\pi \i \delta k\cdot x} a_{k,\alpha}\cdot F_{k,\alpha}(x)\,\dd x \Big|\\
&\leq  \sum_{k,\alpha}\Big(|\theta_k^N| \big\|x\mapsto e^{-2\pi \i \delta k\cdot x}\big\|_{H^{-\g}(B)} \| F_{k,\alpha} \|_{H^{\g}(B;\R^d)}\Big)\\
&\leq \Big(\sum_{k,\alpha} (\theta^N_k)^2 \big\|x\mapsto e^{-2\pi \i \delta k\cdot x}\big\|_{H^{-\g}(B)}^2 \Big)^{1/2}
\Big(\sum_{k,\alpha}  \|F_{k,\alpha}\|_{H^{\g}(B;\R^d)}^2 \Big)^{1/2}\\
&\stackrel{(i)}{\leq} C 
\Big(\sum_{k} (\theta^N_k)^2 (\delta |k|)^{-2\g} \Big)^{1/2}
\Big(\sum_{k,\alpha}  \|F_{k,\alpha}\|_{H^{\g}(B;\R^d)}^2 \Big)^{1/2}\\
&\stackrel{(ii)}{\leq} 
\frac{C}{(\delta N)^{\g}}
\Big(\sum_{k,\alpha}  \|F_{k,\alpha}\|_{H^{\g}(B;\R^d)}^2 \Big)^{1/2},
\end{align*}
where in $(i)$ we used \eqref{eq:exponential_negative_smoothness_scaling}, and in $(ii)$ that $\supp\theta^N =\{k\,:\, N\leq |k|\leq 2N\}$ and $\|\theta^N\|_{\ell^2}=1$ by construction.
To conclude, as $\g<1/2$ and $H^1(B)\embed H^\g(B)$, from \eqref{eq:definition_Fkappaalpha}, it follows that 
\begin{align*}
\Big(\sum_{k,\alpha}  \|F_{k,\alpha}\|_{H^{1}(B;\R^d)}^2\Big)^{1/2}
&\lesssim \|\varphi\|_{C^{2}(\overline{B})}\|(\nabla h_{k,\alpha})_{k,\alpha}\|_{H^1(B;\ell^2)}\\
&\lesssim \|\varphi\|_{C^{2}(\overline{B})}\|(\nabla h_{k,\alpha})_{k,\alpha}\|_{L^2(2B;\ell^2)},
\end{align*}
where the last inequality follows from \eqref{eq:Caccioppoli_harmonic_functions} and Fubini's theorem.

Arguing as in the proof of Proposition \ref{prop:error_estimates_operators}\eqref{it:error_estimates_operators1}, by Weyl's law, the above proves the estimate in Proposition \ref{prop:error_estimates_operators}\eqref{it:error_estimates_operators3} for $\sigma_0$ sufficiently large, depending only on $d$. As mentioned above, the conclusion follows from complex interpolation and the uniform bound \eqref{eq:uniform_bound_global_to_local}.
\end{proof}

\section{Control of energy at microscopic scales}
\label{s:energy_control_microscopic}
In this section, we prove the control of the energy at the microscopic scale as outlined in Subsection \ref{ss:control_energy_intro} by following the argument of Avellaneda and Lin \cite{AL87_compactness}.
In particular, we consider the oscillating stochastic Stokes system:
\begin{equation}
\label{eq:turbulent_Stokes_scaling_original_iteration}
\left\{
\begin{aligned}
\partial_t v^N &=-\nabla \pi^N+ (1+\mu) \Delta v^N+ \nabla \cdot F- c_d \mu \sum_{k,\alpha} \theta_k^N\nabla\cdot  (\nabla \wt{\pi}^N_{k,\alpha}\otimes \sigma_{-k,\alpha})\\ 
&\qquad\ \  + \sqrt{c_d \mu}\sum_{k,\alpha} [-\nabla \wt{\pi}_{k,\alpha}^N +\theta^{N}_{k}(\sigma_{k,\alpha}\cdot\nabla) v^N+\theta^{N}_{k}\sigma_{k,\alpha}\cdot G]\,\dot{W}^{k,\alpha}_t,\\
 \nabla \cdot v^N&=0,\\
\Delta \pi^N
&= \nabla^2 : F -  c_d \mu \nabla \cdot \Big[\sum_{k,\alpha}
\theta_k^N \nabla \cdot (\nabla \wt{\pi}_{k,\alpha}^N\otimes \sigma_{-k,\alpha})\Big] ,\\
\Delta \wt{\pi}_{k,\alpha}^N
&=\theta^{N}_{k} \nabla \cdot\big[(\sigma_{k,\alpha}\cdot \nabla)v^N+\sigma_{k,\alpha}\cdot G\big].
 \end{aligned}
\right.
\end{equation}
where $\theta^N=(\theta_k^N)_{k}$ is as in \eqref{eq:choice_thetaN}.
Below, we employ the following 

\begin{assumption}[Forcing regularity]
\label{ass:forcing_regularity_iteration}
Let $Q$ be a parabolic cylinder with side length $1/2$ and center $(t_0,x_0)\in \R_+\times \R^d$, and $r\in (1,\infty)$.
Let $p,q\in (2,\infty)$ be such that
$$
\frac{2}{p}+\frac{d}{q}<1.
$$ 
Suppose that $F$ and $G$ are progressively measurable processes such that 
$$
F,G \in L^r(\O;L^p(I;L^q(B;\R^{d\times d}))),
$$
and a.e.\ on $I\times \O$, $\Tr(G)=0$ on $B$ as well as $\nabla \cdot G=0$ in $\D'(B)$.
\end{assumption}

The main result of this section reads as follows.

\begin{theorem}[Control of the energy at microscopic scale]
\label{t:control_mesoscopic_energy}
Suppose that Assumption \ref{ass:forcing_regularity_iteration} holds. Let $\mu>0$, $r\in (1,\infty)$ and $p_0\in (2,p)$.
Then there exists a constant $R_0>0$ depending only on $\mu,p,q,r,p_0$ and $d$ such that the following assertion holds. For all local solutions $(v^N,\pi^N,(\wt{\pi}^N_{k,\alpha})_{k,\alpha})$ to the stochastic Stokes system \eqref{eq:turbulent_Stokes_scaling_original_iteration} on $Q$, the following uniform control of the energy at the microscopic scales holds
\begin{align*}
\big(\E \|v^N \|_{\underline{L}^2(N^{-1}Q)}^r\big)^{1/r}
&\leq  R_0\Big(\E\sup_{I} \|v^{N}\|_{L^2(B)}^{r}\Big)^{1/r}+R_0 \En_{Q},
\end{align*}
as well as a uniform control of the oscillation of the deterministic and stochastic pressures at the microscopic scales
\begin{align*}
\Big(\E \Big\|\pi^N-\fint_{N^{-1}B}\pi^N \Big\|_{\underline{L}^2(N^{-1}I;\underline{H}^{-1}(N^{-1} B))}^r\Big)^{1/r}
&\leq R_0 \En_{Q},\\
\Big(\E \Big\|\Big( \wt{\pi}_{k,\alpha}^N-\fint_{N^{-1}B}\wt{\pi}_{k,\alpha}^N\Big)_{k,\alpha}\Big\|_{\underline{L}^{p_0}(N^{-1}I;\underline{L}^2(N^{-1}B;\ell^2))}^r\Big)^{1/r}
&\leq R_0 \En_{Q},
\end{align*}
where $\En_{Q}$ is the $L^p(L^q)$-content of $(v^N,\pi^N,(\wt{\pi}^N_{k,\alpha})_{k,\alpha})$ on the (large) scale $Q$, i.e.,
\begin{align}
\nonumber
\En_{Q}
=\big( \max\big\{\E \|F \|_{L^p(I;L^q(B))}^r,\,
\E \|G \|_{L^p(I;L^q(B))}^r , \,
\E \|v^N \|_{L^2(Q)}^r
,&\\
\label{eq:def_En_Q_Lebesgue}
\E \|\nabla \pi^N\|_{L^2(I;H^{-1}(B))}^r,\, 
\E \|( \nabla \wt{\pi}_{k,\alpha}^N )_{k,\alpha}\|_{L^{p_0}(I;H^{-1}(B;\ell^2))}^r &\,\big\}\big)^{1/r}.
\end{align}
\end{theorem}

The key point in the above estimate is the uniformity of the constant $R_0$ with respect to the oscillation parameter $N$. Note that the condition $2/p+d/q<1$ is optimal for the above result to hold; see Subsection \ref{ss:almost_self}.
Due to the lack of time regularity of the pressures, we can only control the pressure oscillations (cf.\ Corollary \ref{cor:caccioppoli_combined_decay}). 
However, this is sufficient for our purposes; see Theorem \ref{t:almost_Linfty} below. 
Finally, as commented in Subsection \ref{ss:role_pressure_intro}, the condition $p_0>2$ is needed to compensate for the lack of time regularity of the stochastic pressure, while for the deterministic one, we exploit the scaling properties of the $\underline{H}^{-1}(\sm B)$-norm as $\sm \to 0$, see \eqref{eq:scaling_negative_sob_space_intro}.

\smallskip

The proof of Theorem \ref{t:control_mesoscopic_energy} is deferred to the end of this section (see Subsection \ref{ss:control_mesoscopic_energy_proof}). As outlined in Subsection \ref{ss:control_energy_intro}, following \cite{AL87_compactness}, the argument is divided into two parts: a one-step improvement for a rescaled system and a corresponding iteration for the original system \eqref{eq:turbulent_Stokes_scaling_original_iteration}; see Subsections \ref{ss:one_step_improvement} and \ref{ss:iteration}, respectively.

\subsection{One step improvement -- Rescaled oscillating stochastic Stokes system}
\label{ss:one_step_improvement}
Consider the oscillating stochastic Stokes system with rescaled coefficients:
\begin{equation}
\label{eq:turbulent_Stokes_scaling_quantitative_iteration}
\left\{
\begin{aligned}
\partial_t v^{N,\delta} &=-\nabla \pi^{N,\delta}+ (1+\mu) \Delta v^{N,\delta}+ \nabla \cdot F- c_d \mu \sum_{k,\alpha} \theta_k^N\nabla\cdot  (\nabla \wt{\pi}^{N,\delta}_{k,\alpha}\otimes \sigma^\delta_{-k,\alpha})\\ 
&\ \  + \sqrt{c_d \mu}\sum_{k,\alpha} [-\nabla \wt{\pi}_{k,\alpha}^{N,\delta} +\theta^{N}_{k}(\sigma^\delta_{k,\alpha}\cdot\nabla) v^{N,\delta}+\theta^{N}_{k}\sigma^\delta_{k,\alpha}\cdot G]\,\dot{W}^{k,\alpha}_t,\\
 \nabla \cdot v^{N,\delta}&=0,\\
\Delta \pi^{N,\delta}
&= \nabla^2 : F -  c_d \mu \nabla \cdot \Big[\sum_{k,\alpha}
\theta_k^N \nabla \cdot (\nabla \wt{\pi}_{k,\alpha}^{N,\delta}\otimes \sigma^\delta_{-k,\alpha})\Big] ,\\
\Delta \wt{\pi}_{k,\alpha}^{N,\delta}
&=\theta^{N}_{k} \nabla \cdot [(\sigma^\delta_{k,\alpha}\cdot \nabla)v^{N,\delta}+\sigma^\delta_{k,\alpha}\cdot G],
 \end{aligned}
\right.
\end{equation}
where $\theta^N=(\theta_k^N)_{k}$ is as in \eqref{eq:choice_thetaN}, and for $\delta\in (0,1]$, $\sigma_{k,\alpha}^\delta$ are the rescaled coefficients around the spatial center of $Q$, i.e., 
$
\sigma^{\delta}_{k,\alpha}=\sigma_{k,\alpha}(x_0 +\delta (\cdot-x_0))$. 

\smallskip

The following is the version of the one-step improvement that is convenient for our purposes and is the key to proving Theorem \ref{t:control_mesoscopic_energy}. 

\begin{lemma}[One step improvement]
\label{l:one_step_improvement}
Suppose that Assumption \ref{ass:forcing_regularity_iteration} holds. Let $\mu>0$, $r\in (1,\infty)$, $p_0\in (2,p)$, and fix $0<\g_0< (1-\frac{2}{p}-\frac{d}{q})\wedge (1-\frac{2}{p_0})$. 
Then there exist constants 
$$
L_0\geq 1\qquad \text{ and } \qquad \sm\in (0,1/4],
$$ 
depending only on $\mu,p,q,r,p_0,d$ and $\g_0$ for which the following assertion holds whenever $\delta\in (0,1]$ and $N\geq 1$ satisfy 
$$
(\delta N)^{-1}\leq L_0^{-1}.
$$
For all local solutions $(v^{N,\delta},\pi^{N,\delta},(\wt{\pi}^{N,\delta}_{k,\alpha})_{k,\alpha})$ to the rescaled stochastic Stokes system \eqref{eq:turbulent_Stokes_scaling_quantitative_iteration} on $Q$ (see Definition \ref{def:local_solutions_NSE}), it holds that 
\begin{align}
\label{eq:one_step_improvement1}
\Big(\E \Big\|v^{N,\delta} -\fint_{\sm B} v^{N,\delta}  \Big\|_{\underline{L}^2(\sm Q)}^{r}\Big)^{1/r}
 \leq \sm^{\g_0}
\En_{Q}^{\#},\\
\label{eq:one_step_improvement2}
\Big(\E \Big\|\pi^{N,\delta} -\fint_{\sm B} \pi^{N,\delta}  \Big\|_{\underline{L}^2(\sm I;\underline{H}^{-1}(\sm B))}^{r}\Big)^{1/r}
 \leq \sm^{\g_0}
\En_{Q}^{\#},\\
\label{eq:one_step_improvement3}
\Big(\E \Big\|\Big( \wt{\pi}_{k,\alpha}^{N,\delta} -\fint_{\sm B} \wt{\pi}_{k,\alpha}^{N,\delta}\Big)_{k,\alpha} \Big\|_{\underline{L}^{p_0}(\sm I;\underline{L}^2(\sm B; \ell^2))}^{r}\Big)^{1/r}
 \leq \sm^{\g_0}
 \En_{Q}^{\#},
\end{align}
where $\En_Q^{\#}$ denotes the sharp Lebesgue-energy of the local solution $(v^{N,\delta},\pi^{N,\delta},(\wt{\pi}^{N,\delta}_{k,\alpha})_{k,\alpha})$ at the (large) scale $Q$, i.e.,
\begin{align}
\nonumber
\En_Q^{\#}
=&\Big( \max\Big\{\E \|F \|_{\underline{L}^p(I;\underline{L}^q(B))}^r,\,
\E \|G \|_{\underline{L}^p(I;\underline{L}^q(B))}^r , \,
\E \Big\|v^{N,\delta} -\fint_{B} v^{N,\delta}\Big\|_{\underline{L}^2(Q)}^r
,\\
\label{eq:def_En_Q}
&\E \Big\|\pi^{N,\delta}-\fint_{B} \pi^{N,\delta}\Big\|_{\underline{L}^2(I;\underline{H}^{-1}(B))}^r,\, 
\E \Big\|\Big( \wt{\pi}_{k,\alpha}^{N,\delta}-\fint_{B} \wt{\pi}^{N,\delta}_{k,\alpha} \Big)_{k,\alpha}\Big\|_{\underline{L}^{p_0}(I; \underline{L}^2(B;\ell^2))}^r \Big\}\Big)^{1/r}.
\end{align}
\end{lemma}

As will become clear from the iteration argument in Subsection \ref{ss:iteration} and the corresponding proof of Theorem \ref{t:control_mesoscopic_energy} in Subsection \ref{ss:control_mesoscopic_energy_proof}, the key point in the above lemma is the joint condition on $N$ and $\delta$ via the cumulative oscillation frequency $L=\delta N$ (see Subsection \ref{ss:control_energy_intro}). The use of averaged Lebesgue spaces in the definition of $\En_Q^{\#}$ in \eqref{eq:def_En_Q} is primarily for convenience, as their scaling properties perfectly accommodate the iteration performed in Subsection \ref{ss:iteration} below.

\smallskip

The proof of Lemma \ref{l:one_step_improvement} is postponed to Subsection \ref{sss:one_step_iteration}. As a fundamental preparatory step, we first establish local smoothness for the homogenized Stokes system \eqref{eq:homogenized_PDE_linear} in Subsection \ref{ss:preparation_one_step}. We emphasize that, in contrast to \cite{AL87_compactness}, our one-step improvement relies crucially on the quantitative error estimates derived in Section \ref{s:quantitative_loc_scaling}.

\subsubsection{Local smoothing for the homogenized Stokes system}
\label{ss:preparation_one_step}
The following result captures the local smoothing of the homogenized Stokes system \eqref{eq:homogenized_PDE_linear}.

\begin{lemma}[Local smoothing for the homogenized Stokes system]
\label{l:local_smoothing_homogenized}
Suppose that Assumption \ref{ass:forcing_regularity_iteration} holds. Fix $\mu>0$, $r\in (1,\infty)$ and $p_0\in (2,p)$. Let $(v^{N,\delta},\pi^{N,\delta},(\wt{\pi}^{N,\delta}_{k,\alpha})_{k,\alpha})$ be a local solution to the stochastic Stokes system \eqref{eq:turbulent_Stokes_scaling_quantitative_iteration} on $Q$ for some $N\geq 1$ and $\delta\in (0,1]$. Let  $(\uh^{N,\delta},\ph^{N,\delta})$ be the solution to the homogenized Stokes system \eqref{eq:homogenized_PDE_linear} on $(1/2)Q$ provided by Proposition \ref{prop:local_well_posedness_homogenized_SPDE}.
Then there exists a constant $C_0>0$ depending only on $\mu,\varrho,p,q,r,p_0$ and $d$ such that
\begin{equation}
\label{eq:local_smoothing_uhN_proof_one_step_improvement}
\big(
\E\|\uh^{N,\delta}\|^r_{L^p((1/4)I;W^{1,q}((1/4)B))}\big)^{1/r}
\leq C_0 \En_Q^{\#},
\end{equation}
where $\En_Q^{\#}$ is as in \eqref{eq:def_En_Q} and $Q=I\times B$.
\end{lemma}

Comparing \eqref{eq:local_smoothing_uhN_proof_one_step_improvement} with the regularity in Proposition \ref{prop:local_well_posedness_homogenized_SPDE}, it follows that the velocity field $\uh^{N,\delta}$ gains regularity far from the boundary, which is the source of roughness. This is consistent with the well-known fact that the regularity of solutions to elliptic and parabolic PDEs depends only locally on the coefficients and data.

However, due to the limited time regularity of $\ph^{N,\delta}$ (see Proposition \ref{prop:local_well_posedness_homogenized_SPDE}), the above local smoothing is nonstandard, and in particular, we do not expect a control $\partial_t \uh^{N,\delta}$ in $L^{p}((1/4)I;H^{-1,q}((1/4)B))$ that is the maximal time regularity counterpart of \eqref{eq:local_smoothing_uhN_proof_one_step_improvement}. 
This lack of temporal regularity in the homogenized velocity explains why, in the one-step improvement of Lemma \ref{l:one_step_improvement}, we study velocity oscillations within spatial balls rather than the time-space cylinders typically used in parabolic PDEs.

\begin{proof}
Recall from Proposition \ref{prop:local_well_posedness_homogenized_SPDE} that
\begin{equation}
\begin{aligned}
\label{eq:PDE_homogenized_local_smoothing}
\partial_t \uh^{N,\delta} &=-\nabla \ph^{N,\delta}+ (1+D_d \mu) \Delta \uh^{N,\delta}+  \nabla \cdot (F-(1-D_d) \mu\, G^\top), \\
 \nabla \cdot \uh^{N,\delta}&=0 ,
\end{aligned}
\end{equation}
both in $\D'((1/2)Q)$ for some $\ph^{N,\delta} \in H^{-1,p}((1/2)I;H^{-1}((1/2)B))$, with corresponding bounds in terms of $\En_{Q}^{\#}$ due to the estimates \eqref{eq:local_well_posedness_homogenized_SPDE_estimate}-\eqref{eq:local_well_posedness_homogenized_SPDE_estimate2} and the stochastic Caccioppoli inequality applied to $(v^{N,\delta},\pi^{N,\delta}-\fint_B \pi^{N,\delta}, (\wt{\pi}^{N,\delta}_{k,\alpha}-\fint_B \wt{\pi}^{N,\delta}_{k,\alpha})_{k,\alpha})$. 

From the previously displayed formula, we obtain
$$
\Delta \ph^{N,\delta} = \nabla^2 : (F-(1-D_d) \mu\, G^\top)\ \  \text{ in }\ \D'((1/2)Q).
$$
Next, fix a cutoff function $\chi\in C^{\infty}_{{\rm c}}((1/2)Q)$ such that 
\begin{equation*}
\chi=1\ \text{ on }(3/8)Q, \qquad  \text{ and }\qquad\|\nabla^{k}\chi\|_{L^\infty(Q)}\leq C_{k}.
\end{equation*}
As in \eqref{eq:approximated_pressure_1_quantitative_convergence3}, we let $\Pi^0$ be the local pressure on $Q$:
\begin{equation}
\label{eq:regularity_local_pressure_smoothing_homogenized}
\Pi^0=\qq_{\T^d_{B}}[\chi \nabla \cdot (F-(1-D_d)\mu \,G^\top)]\in L^r(\O;L^p(I;L^q(B))).
\end{equation}
Moreover, the choice of the cutoff $\chi$ implies that $\ph^{N,\delta}-\Pi^0$ is harmonic:  
\begin{equation}
\label{eq:harmonicity_difference}
\Delta (\ph^{N,\delta}-\Pi^0)=0 \ \  \text{ in }\ \D'((3/8)Q).
\end{equation} 
Now, we claim that there exists $\mathcal{H}\in L^p((1/2)I;H^{-1}((1/2)B))$ such that 
\begin{equation}
\label{eq:harmonic_distribution_construction1}
\|\mathcal{H}\|_{L^p((1/2)I;H^{-1}((1/2)B))}\leq C\|\ph^{N,\delta}-\Pi^0\|_{H^{-1,p}((1/2)I;H^{-1}((1/2)B))}
\end{equation}
with $C>0$ depending only on $\mu,p$ and $d$, and it satisfies
\begin{equation}
\label{eq:harmonic_distribution_construction2}
\partial_t\mathcal{H}= \ph^{N,\delta}-\Pi^0\ \  \text{ in }\ \D'((1/2)Q), \ \ \ \text{ and }\ \ \  \Delta \mathcal{H}=0\ \  \text{ in }\ \D'((3/8)Q).
\end{equation}
From the above and local smoothing for harmonic functions as in \eqref{eq:Caccioppoli_harmonic_functions}, we obtain 
\begin{align}
\label{eq:smoothing_nabla_H_localized_stokes_smoothness}
\|\nabla \mathcal{H}\|_{L^p((1/4)I;{\rm Lip}((1/4)B))}
&\leq C\| \mathcal{H}\|_{L^p((1/4)I;H^{-1}((3/8)B))}\\
\nonumber
&\leq C \|\ph^{N,\delta}- \Pi^0\|_{H^{-1,p}((1/2)I;H^{-1}((1/2)B))},
\end{align}
where the constant $C>0$ depends only on $p$ and $d$.

The construction of $\mathcal{H}$ is standard (see e.g., \cite[Subsection 8.2]{B11}). We include some details for completeness.
Fix $\varphi_0\in \D((1/2)I)$ such that $\int_{(1/2)I} \varphi_0=1$. Set  
\begin{equation}
\label{eq:explicit_version_harmonic}
\langle \mathcal{H}, \varphi\rangle = -\langle \ph^{N,\delta}-\Pi^0, \phi_\varphi\rangle ,
\end{equation}
where 
$$
\phi_\varphi(t,\cdot)= \int_{(1/2)I \cap (-\infty,t)} \Big( \varphi(s,\cdot) - \varphi_0 (s)\int_{(1/2)I} \varphi(r,\cdot)\,\dd r\Big)\,\dd s .
$$
The first identity in \eqref{eq:harmonic_distribution_construction2} is clear, while the second follows from \eqref{eq:harmonicity_difference}. To check \eqref{eq:harmonic_distribution_construction1}, recall that, arguing as in \cite[Proposition 8.14]{B11} (see \cite[Theorem 1.3.10 and 1.3.21]{Analysis1} for duality of Bochner spaces), the space $H^{-1,p}((1/2)I;H^{-1}((1/2)B))$ can be characterized as the set of derivatives of functions in $L^p((1/2)I;H^{-1}((1/2)B))$. In particular, there exists $R\in L^p((1/2)I;H^{-1}((1/2)B))$ which satisfies the bound \eqref{eq:harmonic_distribution_construction1} with $\mathcal{H}$ replaced by $R$ and $\partial_t R = \ph^{N,\delta} - \Pi^0$. From \eqref{eq:explicit_version_harmonic}, it follows that 
\begin{align*}
\langle \mathcal{H}, \varphi\rangle 
&= -\langle \partial_t R, \phi_\varphi\rangle =\langle  R, \partial_t \phi_\varphi\rangle \\
&=\int_{(1/2)I}   \langle R, \varphi(s,\cdot)\rangle - \varphi_0(s) \Big\langle  R(s),  \int_{(1/2)I} \varphi(r,\cdot)\,\dd r\Big\rangle \,\dd s\\
&=\int_{(1/2)I}   \Big\langle R- \int_{(1/2)I} R(r)\varphi_0 (r)\,\dd r, \varphi(s,\cdot)\Big\rangle\,\dd s ,
\end{align*}
i.e., $\mathcal{H}=R- \int_{(1/2)I} R\varphi_0$ and \eqref{eq:harmonic_distribution_construction1} follows from the corresponding bound for $R$.

\smallskip

Letting $\whom^{N,\delta}=\nabla\mathcal{H}+\uh^{N,\delta} $, the first identity in \eqref{eq:PDE_homogenized_local_smoothing} implies, in $\D'((3/8)Q)$,
\begin{align*}
\partial_t \whom^{N,\delta}
&\stackrel{(i)}{=} \nabla \ph^{N,\delta}-\nabla \Pi^0+ \partial_t \uh^{N,\delta}\\
&=   -\nabla \Pi^0 + (1+D_d \mu)\Delta \uh^{N,\delta}+  \nabla \cdot (F-(1-D_d) \mu\, G^\top)\\
&\stackrel{(ii)}{=}   -\nabla \Pi^0 + (1+D_d \mu) \Delta \whom^{N,\delta}+  \nabla \cdot (F-(1-D_d) \mu\, G^\top)
\end{align*}
where in $(i)$ and $(ii)$, we used the first and the second identities in \eqref{eq:harmonic_distribution_construction2}, respectively. 

Finally, it follows from \eqref{eq:regularity_local_pressure_smoothing_homogenized} that 
$$
-\nabla \Pi^0 + \nabla \cdot (F-(1-D_d)\mu \,G^\top) \in L^r(\O;L^p(I;H^{-1,q}(B)))
$$
with a corresponding bound depending only on the $L^r(\O;L^p(I;L^{q}(B)))$-norms of $F$ and $G$.
The claimed estimate \eqref{eq:local_smoothing_uhN_proof_one_step_improvement} follows from \eqref{eq:smoothing_nabla_H_localized_stokes_smoothness}, standard local smoothing for the heat equation on $ (3/8)Q\supset(1/4)Q$,  and, as mentioned at the beginning of the proof, the fact that the right-hand sides of \eqref{eq:local_well_posedness_homogenized_SPDE_estimate}-\eqref{eq:local_well_posedness_homogenized_SPDE_estimate2} are controlled by $\En_{Q}^{\#}$ thanks to Theorem \ref{t:caccioppoli}.
\end{proof}

\subsubsection{One step improvement -- Proof of Lemma \ref{l:one_step_improvement}}
\label{sss:one_step_iteration}
We begin by proving a `sharp' variant of the Poincar\'e inequality. Below, as usual, we let $H^0(B)=L^2(B)$. 

\begin{lemma}[Sharp Poincar\'e inequality]
\label{l:sharp_poincare}
For all $\sigma\in [0,1/2)$ and balls $B\subseteq \R^d$, there exists a constant $C>0$ independent of the center of $B$ such that, for all $f\in L^2(B)$ and $\nabla f \in H^{-1-\sigma}(B)$, 
\begin{align}
\label{eq:poincare_sharp_lemma}
\Big\|f-\fint_{B} f\Big\|_{H^{-\sigma}(B)}
\leq C \|\nabla f\|_{H^{-1-\sigma}(B)}.
\end{align}
In particular, $
\big\|f-\fint_{B} f\big\|_{H^{-1}(B)}
\leq C \|\nabla f\|_{H^{-1}(B)}$.
\end{lemma}

The restriction $\sigma<1/2$ is natural, as $\one_{B}$ is a pointwise multiplier in $H^{\sigma}(B)$ and \eqref{eq:poincare_sharp_lemma} can be extended to any $f\in H^{-\sigma}(B)$ if we interpret $\fint f =|B|^{-1}\langle f,\one_{B}\rangle $. A version of Lemma \ref{l:sharp_poincare} also holds by replacing $\fint_{B}f$ by $\int_B \phi f$ for a fixed $\phi\in C_{{\rm c}}^{\infty}(B)$ such that $\int_B \phi=1$, as used in Section \ref{s:caccioppoli}, but this will not be needed here.
The above might be known to experts. For the reader's convenience, we include a proof.

\begin{proof}
Note that, for all $\varphi\in H^\sigma(B) =H^{\sigma}_0(B)$ (as $\sigma<1/2$),
\begin{align*}
\Big|\int_{B}\Big(f-\fint_{B} f\Big) \varphi \Big|
&=\Big|\int_{B}f \Big(\varphi-\fint_{B} \varphi\Big)  \Big|.
\end{align*}
Now, by using the Bogovskii operator (see e.g., \cite[Subsection III.3]{Galdi_book}), there exists $v\in H^{1+\sigma}_0(B)=H^{1+\sigma}(B)\cap H^1_0(B)$ such that $\|v\|_{H^{1+\sigma}(B)}\leq C \|\varphi\|_{H^{\sigma}(B)}$ and
$$
\left\{
\begin{aligned}
\nabla \cdot v& = \varphi-\fint_{B} \varphi  &\text{ on }& B, \\
v&=0 &\text{ on }&\partial B.
\end{aligned}
\right.
$$
Thus, 
\begin{align*}
\Big|\int_{B}\Big(f-\fint_{B} f\Big) \varphi \Big|
=\Big|\int_{B}f\, \nabla\cdot v\Big|
&=|\langle \nabla f,  v \rangle |\\
&\leq \|\nabla f\|_{H^{-1-\sigma}(B)}\|v\|_{H^{1+\sigma}_0(B)}\\
&\leq C\|\nabla f\|_{H^{-1-\sigma}(B)}\|\varphi\|_{H^{\sigma}_0(B)}.
\end{align*}
The conclusion follows by taking the supremum over $\varphi\in H^\sigma_0(B)$ with unit norm.
\end{proof}

Before going into the proof of Lemma \ref{l:one_step_improvement}, let us recall the elementary inequality: 
\begin{equation}
\label{eq:scaling_negative_LH1}
\|f\|_{\underline{H}^{-1}(\sm B)}\leq C \sm \|f\|_{\underline{L}^{q_0}(\sm B)} \quad \text{ for all }q_0\geq 2.
\end{equation}

\begin{proof}[Proof of Lemma \ref{l:one_step_improvement}]
We begin the proof by recalling that 
\begin{align*}
\Pi^{N,\delta}&=\qq_{\T^d_{B}} \Big[\chi\big[ \nabla \cdot (F+(1-D_d)\mu \nabla v^{N,\delta})- c_d \mu \sum_{k,\alpha} \theta_k^N\nabla\cdot  (\nabla \wt{\pi}^{N,\delta}_{k,\alpha}\otimes \sigma^\delta_{-k,\alpha} \big)\big]\Big],\\
\wt{\Pi}_{k,\alpha}^{N,\delta}
&=\theta^N_k \qq_{\T^d_{B}} \Big[\chi\big((\sigma^{\delta}_{k,\alpha}\cdot \nabla)v^{N,\delta} + \sigma^{\delta}_{k,\alpha} \cdot G\big)\Big],\\
\Pi^0&= \qq_{\T^d_{B}} \Big[\chi\big( \nabla \cdot (F-(1-D_d)\mu\, G^\top)\big)\Big].
\end{align*}
For later use, in the proof below, we use $\chi\in C^{\infty}_{{\rm c}}(B)$ such that 
\begin{equation}
\label{eq:choice_chi_34}
\chi=1 \text{ on }(5/8)B \qquad \text{ and }\qquad 
\supp\chi \subseteq (3/4)B.
\end{equation}
Let $(\uh^{N,\delta},\ph^{N,\delta})$ be the unique solution to the homogenized Stokes system \eqref{eq:homogenized_PDE_linear} provided by Proposition \ref{prop:local_well_posedness_homogenized_SPDE} with $2Q$ replaced by $Q$.
Next, we decompose the local solution $(v^{N,\delta},\pi^{N,\delta},(\wt{\pi}^{N,\delta}_{k,\alpha})_{k,\alpha})$ to the stochastic Stokes system \eqref{eq:turbulent_Stokes_scaling_quantitative} as 
\begin{align}
\label{eq:decomposition_one_step_iteration1}
v^{N,\delta}&= \Big(\uh^{N,\delta}+\fint_{(1/2)B} v^{N,\delta}\Big)+\Big(v^{N,\delta} -\fint_{(1/2)B} v^{N,\delta} -\uh^{N,\delta}\Big) ,\\
\label{eq:decomposition_one_step_iteration2}
\pi^{N,\delta} &=\big[\Pi^0+(\pi^{N,\delta}-\Pi^{N,\delta})\big]+(\Pi^{N,\delta}-\Pi^0 )  ,\\
\label{eq:decomposition_one_step_iteration3}
\wt{\pi}^{N,\delta}_{k,\alpha}&= (\wt{\pi}_{k,\alpha}^{N,\delta} - \wt{\Pi}_{k,\alpha}^{N,\delta})+\wt{\Pi}_{k,\alpha}^{N,\delta}.
\end{align}
The idea behind the above decomposition is that, to obtain the desired contraction at scale $\sm$ in \eqref{eq:one_step_improvement1}-\eqref{eq:one_step_improvement3}, we combine the following two ingredients.
\begin{itemize}
\item  Fix a scale $\sm\in (0,1/4]$ at which such objects are contractive of order $\leq \sm^{\g_0}/2$ by using the smoothness of terms $\uh^{N,\delta}+\fint_{(1/2)B} v^{N,\delta}$, $\Pi^0$, $\pi^{N,\delta}-\Pi^{N,\delta}$ and $\wt{\pi}_{k,\alpha}^{N,\delta} - \wt{\Pi}_{k,\alpha}^{N,\delta}$. 
\item  Fix $L_0\geq 1$ such that the quantities $v^{N,\delta} -\fint_{(1/2)B} v^{N,\delta} -\uh^{N,\delta}$, $\Pi^{N,\delta}-\Pi^0$ and $\wt{\Pi}_{k,\alpha}^{N,\delta}$ are also contractive at scale $\leq \sm^{\g_0}/2$ provided the cumulative frequency $L=\delta N\geq L_0$ is large enough (Theorem \ref{t:universal_scaling_limit}).
\end{itemize}
We now divide the proof into two steps, following the roadmap above.

\smallskip

\emph{Step 1: Choice of $\sm\in (0,1/4]$.} 
We divide this step into four substeps. 

\smallskip

\emph{Substep 1a: Velocity field.}
From Lemma \ref{l:local_smoothing_homogenized}, we know that 
$$
\uh^{N,\delta} \in L^r(\O;L^p((1/4)I;W^{1,q}((1/4)B)))\ \text{ a.s., }
$$ 
with a corresponding estimate on the $r$-th moment depending only on $\En^\#_Q$. 
From $\frac{2}{p}+\frac{d}{q}<1$, it follows that $q>d$, and therefore
$$
W^{1,q}((1/4)B)\embed C^{1-d/q}((1/4)B)
$$
from Sobolev embeddings. In particular, a.s.,
\begin{align*}
\Big\|\uh^{N,\delta}-\fint_{\sm B} \uh^{N,\delta}\Big\|_{\underline{L}^2(\sm Q)}
&\leq \Big\|\uh^{N,\delta}-\fint_{\sm B} \uh^{N,\delta}\Big\|_{\underline{L}^p(\sm I; \underline{L}^2(\sm B))}\\
&\leq C \sm^{1-d/q} 
\|\uh^{N,\delta}\|_{\underline{L}^p(\sm I;C^{1-d/q}( (1/4)B))}\\
&\stackrel{(i)}{\leq} C_0 \sm^{1-d/q-2/p} 
\|\uh^{N,\delta}\|_{L^p( (1/4)I;C^{1-d/q}((1/4) B))}\\
&\leq C \sm^{1-d/q-2/p} 
\|\uh^{N,\delta}\|_{L^p( (1/4)I;W^{1,q}((1/4) B))},
\end{align*}
where in $(i)$ we used $|\sm I|= \sm^2$. Thus, from Lemma \ref{l:local_smoothing_homogenized}, we get
\begin{align}
\nonumber
\Big(\E\Big\|\Big(\uh^{N,\delta}+\fint_{(1/2)B} v^{N,\delta}\Big)-\fint_{\sm B} \Big(\uh^{N,\delta}+\fint_{(1/2)B} v^{N,\delta}\Big)\Big\|_{\underline{L}^2(\sm Q)}^r\Big)^{1/r}  &\\
\label{eq:choice_eta_one_step_1}
 =
\Big(\E\Big\|\uh^{N,\delta}-\fint_{\sm B} \uh^{N,\delta}\Big\|_{\underline{L}^2(\sm Q)}^r\Big)^{1/r}
&\leq C\sm^{1-d/q-2/p} \En_Q^{\#}.
\end{align}

\smallskip

\emph{Substep 1b: Deterministic pressure.}
First, recall that from the Poincar\'e inequality and a scaling argument, it holds that $\|f-\fint_{\lambda B} f\|_{\underline{H}^{-1}(\lambda B)}\lesssim \lambda \|f\|_{\underline{L}^2(\lambda B)}$ for all $\lambda>0$ and $f\in L^2(\lambda B)$. 
Thus,
\begin{align*}
\Big\| \Pi^0- \fint_{\sm B} \Pi^0\Big\|_{\underline{L}^2(\sm I;\underline{H}^{-1}(\sm B))}
&\leq  C\sm \| \Pi^0\|_{\underline{L}^2(\sm Q )}\\
&\leq C\sm \| \Pi^0\|_{\underline{L}^p(\sm I;\underline{L}^{q}(\sm B))}\\
&\leq C\sm^{1-2/p -d/q} \| \Pi^0\|_{\underline{L}^p( I;\underline{L}^{q}( B))}\\
&\leq C  \sm^{1-2/p -d/q} \big(\| F\|_{\underline{L}^p( I;\underline{L}^{q}( B))}+
\| G\|_{\underline{L}^p( I;\underline{L}^{q}( B))}\big),
\end{align*}
where in the last step, we used the choice of the cutoff function $\chi$ and  
\begin{equation}
\label{eq:boundedness_qq_proof_one_step}
\qq_{\T^d_B}: H^{\sigma,q}(\T^d_B)\to H^{1+\sigma,q}(\T^d_B)
\end{equation} 
continuously for all $\sigma \in \R$ and $q\in (1,\infty)$.

Second, we estimate the term 
$$
h^{N,\delta}\stackrel{{\rm def}}{=}\pi^{N,\delta}-\Pi^{N,\delta}. 
$$
From the definition of local solutions in Definition \ref{def:local_solutions_NSE}, 
it follows that $h^{N,\delta}$ is harmonic on $(5/8)B$ by construction as $\nabla \cdot v^{N,\delta}=0$ in $\D'(B)$ a.e.\ on $I\times \O$ and \eqref{eq:choice_chi_34}. Using \eqref{eq:scaling_negative_LH1} and the local smoothing for harmonic functions \eqref{eq:Caccioppoli_harmonic_functions}, for all $\sm\leq 1/2$,
\begin{align*}
\Big\|h^{N,\delta}(t,\cdot)-\fint_{\sm B}h^{N,\delta}\Big\|_{\underline{H}^{-1}(\sm B)}
&\leq \sm \Big\|h^{N,\delta}(t,\cdot)-\fint_{\sm B}
h^{N,\delta}\Big\|_{\underline{L}^2(\sm B)}\\
&\leq C \sm^2 [h^{N,\delta}(t,\cdot)]_{{\rm Lip} ((1/4)B)}\\
&\leq C \sm^2
\Big\|h^{N,\delta}(t,\cdot)-\fint_{B}h^{N,\delta}(t,\cdot)\Big\|_{\underline{H}^{-1}(B)}. 
\end{align*}
Taking the $\underline{L}^2(\sm I)$-norm, we obtain,
\begin{align*}
\Big\|h^{N,\delta}-\fint_{\sm B}h^{N,\delta}\Big\|_{\underline{L}^2(\sm I;\underline{H}^{-1}(\sm B))}
&\leq C \sm^2 
\Big\|h^{N,\delta}-\fint_{ B}h^{N,\delta}\Big\|_{\underline{L}^2(\sm I;\underline{H}^{-1}(B))}\\
&\leq C \sm
\Big\|h^{N,\delta}-\fint_{B}h^{N,\delta}\Big\|_{\underline{L}^2((1/2)I;\underline{H}^{-1}(B))}\\
&\leq C \sm 
\Big\|\pi^{N,\delta}-\fint_{B}\pi^{N,\delta}\Big\|_{\underline{L}^2((1/2)I;\underline{H}^{-1}( B))}\\
&+C \sm \Big\|\Pi^{N,\delta}-\fint_{B}\Pi^{N,\delta}\Big\|_{\underline{L}^{2}((1/2)I;\underline{H}^{-1}( B))}.
\end{align*}
To conclude, it remains to note that, from Lemma \ref{l:sharp_poincare},
\begin{align*}
\nonumber
\Big(\E\Big\|\Pi^{N,\delta}-\fint_{B}\Pi^{N,\delta}\Big\|_{\underline{L}^2((1/2)I;\underline{H}^{-1}( B))}^r \Big)^{1/r}
&\leq 
 C\big(\E\|\nabla \Pi^{N,\delta}\|_{\underline{L}^2((1/2)I;\underline{H}^{-1}(B))}^r \Big)^{1/r}\\
 \nonumber
&\leq C \big(\E\|\nabla \Pi^{N,\delta}\|_{\underline{L}^2((1/2)I;H^{-1}(\T^d_B))}^r \Big)^{1/r}.
\end{align*}
Moreover, from \eqref{eq:boundedness_qq_proof_one_step}, $\|\theta^N\|_{\ell^2}=1$ for all $N\geq1$, the smoothness of $\chi$ and \eqref{eq:choice_chi_34},
\begin{align*}
\big(\E\|\nabla \Pi^{N,\delta}\|_{\underline{L}^2((1/2)I;H^{-1}(\T^d_B))}^r \Big)^{1/r}
&\leq C \big(\E\|F\|_{\underline{L}^2(Q)}^r \big)^{1/r}
+ C \big(\E\|\nabla v^{N,\delta}\|_{\underline{L}^2((3/4)Q)}^r \big)^{1/r}\\
&+ C \big(\E\|(\nabla \wt{\pi}_{k,\alpha}^{N,\delta})_{k,\alpha}\|_{\underline{L}^2((3/4)Q;\ell^2)}^r \big)^{1/r}\\
 &\leq C \En^\#_Q,
\end{align*}
where the last step follows from the stochastic Caccioppoli inequality of Theorem \ref{t:caccioppoli}.

Putting together the above estimates, we obtain that 
\begin{align}
 \label{eq:choice_eta_one_step_2}
\Big(\E\Big\|(\Pi^0+\pi^{N,\delta}-\Pi^{N,\delta})-\fint_{\sm B}(\Pi^0+\pi^{N,\delta}-\Pi^{N,\delta})\Big\|_{\underline{L}^2(\sm I;\underline{H}^{-1}(\sm B))}^r \Big)^{1/r}&\\
\nonumber
\leq C \sm^{1-2/p-d/q} \En^\#_Q&.
\end{align}

\smallskip

\emph{Substep 1c: Stochastic pressures.}
Here, we apply an argument similar to the one used in Substep 1b.
By construction and Definition \ref{def:local_solutions_NSE}, it follows that 
$$
\wt{h}^{N,\delta}_{k,\alpha}\stackrel{{\rm def}}{=}\wt{\pi}^{N,\delta}_{k,\alpha}-\wt{\Pi}^{N,\delta}_{k,\alpha}
$$ is harmonic on $(1/2)B$. Using again the estimate \eqref{eq:Caccioppoli_harmonic_functions} for harmonic functions and the Poincar\'e inequality of Lemma \ref{l:sharp_poincare}, we have, a.s.\ for all $t\in I$,
\begin{align*}
[\wt{h}^{N,\delta}_{k,\alpha}]_{\mathrm{Lip}(B/4)}
&\leq C \Big\|\wt{h}^{N,\delta}_{k,\alpha}-\fint_{B}\wt{h}^{N,\delta}_{k,\alpha}\Big\|_{\underline{H}^{-1}(B)}\\
&\leq C \Big\|\wt{\pi}^{N,\delta}_{k,\alpha}-\fint_{B}\wt{\pi}^{N,\delta}_{k,\alpha}\Big\|_{\underline{H}^{-1}(B)}+ 
C \Big\|\wt{\Pi}^{N,\delta}_{k,\alpha}-\fint_{B}\wt{\Pi}^{N,\delta}_{k,\alpha}\Big\|_{\underline{H}^{-1}(B)}\\
&\leq C \Big\|\wt{\pi}^{N,\delta}_{k,\alpha}-\fint_{B}\wt{\pi}^{N,\delta}_{k,\alpha}\Big\|_{\underline{L}^2(B)}+ 
C \|\nabla \wt{\Pi}^{N,\delta}_{k,\alpha}\|_{H^{-1}(\T^d_B)}.
\end{align*}
Hence, a.s., 
\begin{align*}
\Big\|\Big(\wt{h}^{N,\delta}_{k,\alpha}- \fint_{\sm B}\wt{h}^{N,\delta}_{k,\alpha} \Big)_{k,\alpha}\Big\|_{\underline{L}^2(\sm B;\ell^2)}
\leq C \sm\Big( \sum_{k,\alpha}
[\wt{h}^{N,\delta}_{k,\alpha}]_{\mathrm{Lip}(B/4)}^2\Big)^{1/2}\qquad \qquad &\\
\leq C \sm \Big\|\Big(\wt{\pi}^{N,\delta}_{k,\alpha}-\fint_{B}\wt{\pi}^{N,\delta}_{k,\alpha}\Big)_{k,\alpha}\Big\|_{\underline{L}^2(B;\ell^2)}
+ C\sm   \big\|\big(\nabla \wt{\Pi}^{N,\delta}_{k,\alpha}\big)_{k,\alpha}\big\|_{H^{-1}(\T^d_B;\ell^2)}&.
\end{align*}
Now, by taking the $\underline{L}^{p_0}(\sm I)$-norm, a.s., it holds that    
\begin{align*}
&\Big\|\Big(\wt{h}^{N,\delta}_{k,\alpha}- \fint_{\sm B}\wt{h}^{N,\delta}_{k,\alpha} \Big)_{k,\alpha}\Big\|_{\underline{L}^{p_0}(\sm I;\underline{L}^2(\sm B;\ell^2))}\\
&\leq C\sm^{-2/p_0}
\Big\|\Big(\wt{h}^{N,\delta}_{k,\alpha}- \fint_{\sm B}\wt{h}^{N,\delta}_{k,\alpha} \Big)_{k,\alpha}\Big\|_{\underline{L}^{p_0}((1/2) I;\underline{L}^2(\sm B;\ell^2))}\\
&\leq C\sm^{1-2/p_0}\Big(
\Big\|\Big(\wt{\pi}^{N,\delta}_{k,\alpha}-\fint_{B}\wt{\pi}^{N,\delta}_{k,\alpha}\Big)_{k,\alpha}\Big\|_{\underline{L}^{p_0}( I;\underline{L}^2(B;\ell^2))}
+\big\|\big(\nabla \wt{\Pi}^{N,\delta}_{k,\alpha}\big)_{k,\alpha}\big\|_{\underline{L}^{p_0}((1/2)I;H^{-1}(\T^d_B;\ell^2))}\Big).
\end{align*}
Therefore, it remains to bound the $L^r(\O)$-norm of the last term by $\En_Q^{\#}$.  
To this end, as in \eqref{eq:trick_adding_average_local_pressure}, letting $\wt{v}^{N,\delta}=v^{N,\delta} -\fint_{(3/4)B}v^{N,\delta}$, from $\nabla \cdot \sigma^{\delta}_{k,\alpha}=0$, it follows that 
\begin{align*}
\nabla \wt{\Pi}^{N,\delta}_{k,\alpha}
&= \theta^N_k \q_{\T^d_{B}} \Big[\chi\Big(\nabla \cdot (\wt{v}^{N,\delta}\otimes \sigma^{\delta}_{k,\alpha}) + \sigma^{\delta}_{k,\alpha} \cdot G\Big)\Big].
\end{align*}
Hence, a.s.\ for all $t\in I$, by using $\q_{\T^d_B}=\nabla \qq_{\T^d_B}$ and \eqref{eq:boundedness_qq_proof_one_step} once more,
\begin{align*}
\|\nabla \wt{\Pi}^{N,\delta}_{k,\alpha}\|_{H^{-1}(\T^d_B)}
&\leq C|\theta^N_k|\big[ \|\chi(\nabla \cdot (\wt{v}^{N,\delta}\otimes \sigma^{\delta}_{k,\alpha})) \|_{H^{-1}(\T^d_B)}
+ \|\sigma^{\delta}_{k,\alpha} \cdot G\|_{L^2(B)}\big]\\
&\leq C|\theta_k^N| \big[\|\wt{v}^{N,\delta} \|_{L^{2}((3/4)B)}
+ \|G\|_{L^2(B)}\big].
\end{align*}
Since $\|\theta^N\|_{\ell^2}=1$ for all $N\geq 1$, by Fubini's theorem and the stochastic Caccioppoli inequality in Theorem \ref{t:caccioppoli}, we have  
\begin{align*}
&\big(\E\|(\nabla \wt{\Pi}^{N,\delta}_{k,\alpha})_{k,\alpha}\|_{\underline{L}^{p_0}((1/2)I;H^{-1}(\T^d_B;\ell^2))}^r \big)^{1/r}\\
\nonumber
&\quad 
\leq  C\big(\E\|\wt{v}^{N,\delta} \|_{L^\infty((3/4)I;\underline{L}^{2}((3/4)B))}^r \big)^{1/r}
+C \big(\E\|G\|_{\underline{L}^{p_0}(I;\underline{L}^2(B))}^r \big)^{1/r} \leq  C\En^\#_Q,
\end{align*}
where, in the last inequality, we used $p_0\leq p$.
Therefore, putting together the estimates in Substep 1c, we have
\begin{equation}
\label{eq:choice_eta_one_step_3}
\Big(\E\Big\|\Big((\wt{\pi}^{N,\delta}_{k,\alpha}-\wt{\Pi}^{N,\delta}_{k,\alpha})-\fint_{\sm B}(\wt{\pi}^{N,\delta}_{k,\alpha}-\wt{\Pi}^{N,\delta}_{k,\alpha})\Big)_{k,\alpha}\Big\|_{\underline{L}^{p_0}(\sm I;\underline{L}^{2}(\sm B;\ell^2))}^r \Big)^{1/r}
\leq C \sm^{1-2/p_0} \En^\#_Q.
\end{equation}

\smallskip

\emph{Substep 1d: Choice of $\sm\in (0,1/4]$.} 
Let $\g_0\in (0,(1-\frac{2}{p}-\frac{d}{q})\wedge (1-\frac{2}{p_0}))$ be as fixed in the statement of Lemma \ref{l:one_step_improvement}.
By the estimates \eqref{eq:choice_eta_one_step_1}, \eqref{eq:choice_eta_one_step_2} and \eqref{eq:choice_eta_one_step_3} we can choose $\sm\in (0,1/4]$ 
depending only on $\mu,p,q,r,p_0,d$ and $\g_0$ so small that the oscillation of the processes appearing as the first terms on the right-hand side of \eqref{eq:decomposition_one_step_iteration1}-\eqref{eq:decomposition_one_step_iteration3} is less than 
$$
(\sm^{\g_0}\En_{Q}^{\#})/2.
$$ 

\smallskip
 
\emph{Step 2: Choice of lower bound $L_0$ for the cumulative frequency $\delta N$.}
Let $\sm$ be as in Step 1.
We begin by analyzing the difference in the velocity fields 
$$
w^{N,\delta}=\uh^{N,\delta} - \Big(v^{N,\delta}-\fint_{(1/2)B} v^{N,\delta}\Big). 
$$
Note that, by Theorem \ref{t:universal_scaling_limit} and $\sm\leq 1/2$, 
\begin{align}
\nonumber
\Big(\E \Big\|w^{N,\delta} -\fint_{\sm B} w^{N,\delta}  \Big\|_{\underline{L}^2(\sm Q)}^{r}\Big)^{1/r}
&\leq C \big(\E \|w^{N,\delta} \|_{\underline{L}^2(\sm Q)}^{r}\big)^{1/r}\\
\nonumber
&\leq C_\sm \big(\E \|w^{N,\delta} \|_{\underline{L}^2((1/2)Q)}^{r}\big)^{1/r}\\
\label{eq:eta_choice_convergence_v}
&\leq \frac{C_\sm C}{(\delta N)^\g}\En_Q^{\#},
\end{align}
where $C_\sm$ is a constant depending only on $d$ and $\sm$. Recall that the latter depends only on $\mu,p,q,r,p_0,d$ and $\g_0$ due to Step 1.
 
We apply a similar argument to bound the remaining terms in the deterministic and stochastic pressures. 
Let us first consider $\Pi^{N,\delta}-\Pi^0$. Fix $\sigma\in (0,1/2)$. By the Poincar\'e inequality of Lemma \ref{l:sharp_poincare},
\begin{align*}
\Big\| (\Pi^{N,\delta}-\Pi^0) -\fint_{\sm B} (\Pi^{N,\delta}-\Pi^0)\Big\|_{\underline{H}^{-1}(\sm B)}
&\leq C_\sm \Big\| (\Pi^{N,\delta}-\Pi^0) -\fint_{\sm B} (\Pi^{N,\delta}-\Pi^0)\Big\|_{\underline{H}^{-\sigma}(\sm B)}\\
&\leq C_\sm\| \nabla \Pi^{N,\delta}-\nabla \Pi^0 \|_{\underline{H}^{-1-\sigma}(\sm B)}\\
&\leq C_\sm \|  \nabla\Pi^{N,\delta}-\nabla\Pi^0\|_{\underline{H}^{-1-\sigma}(\T^d_B)},
\end{align*}
where again $C_\sm$ depends only on $\sm,d$ and the choice of $\sigma$. 
Thus, Theorem \ref{t:universal_scaling_limit} ensures that 
\begin{align}
\label{eq:eta_choice_convergence_Pi}
\Big(\E\Big\| (\Pi^{N,\delta}-\Pi^0) -\fint_{\sm B} (\Pi^{N,\delta}-\Pi^0)\Big\|_{\underline{L}^2(\sm I;\underline{H}^{-1}(\sm B))}^r\Big)^{1/r}
\leq \frac{ C C_\sm }{(\delta N)^\g} \En_Q^{\#}.
\end{align}
Similarly, by Lemma \ref{l:sharp_poincare} with $\sigma=0$, for the stochastic pressure, Theorem \ref{t:universal_scaling_limit} applied with $p_1=p_0$ and $p_2=p$ ensures
\begin{align}
\nonumber
\Big(\E\Big\|\Big( \wt{\Pi}_{k,\alpha}^{N,\delta} -\fint_{\sm B} \wt{\Pi}_{k,\alpha}^{N,\delta}\Big)_{k,\alpha}\Big\|_{\underline{L}^{p_0}(\sm I; \underline{L}^2(\sm B;\ell^2))}^r\Big)^{1/r}
&\leq C_\sm \big(\E \|(\nabla \wt{\Pi}_{k,\alpha}^{N,\delta})_{k,\alpha}\|_{\underline{L}^{p_0}(I;\underline{H}^{-1}(\T^d_B;\ell^2))}^r\big)^{1/r}\\
\label{eq:eta_choice_convergence_wtPi}
&\leq \frac{C C_\sm }{(\delta N)^\g}  \En_Q^{\#},
\end{align}
where, as above, $C_\sm$ depends only on $\sm,d$ and the choice of $p_0$.

Now, from \eqref{eq:eta_choice_convergence_v}, \eqref{eq:eta_choice_convergence_Pi} and \eqref{eq:eta_choice_convergence_wtPi}, it follows that we can choose $L_0\geq 1$ 
depending only on $\mu,p,q,r,p_0,d$ and $\g_0$ such that the oscillations in the second terms on the right-hand side of \eqref{eq:decomposition_one_step_iteration1}-\eqref{eq:decomposition_one_step_iteration3} are less than 
$$
(\sm^{\g_0}\En_{Q}^{\#})/2.
$$ 
The above, together with the result of Substep 1d, yields the claim of Lemma \ref{l:one_step_improvement}.
\end{proof}

\subsection{Iteration via blow-up -- Oscillating stochastic Stokes system}
\label{ss:iteration}
As illustrated in Subsection \ref{ss:control_energy_intro}, the one-step improvement of Lemma \ref{l:one_step_improvement} allows us to prove the following result on the original Stokes system \eqref{eq:turbulent_Stokes_scaling_original_iteration}.

\begin{lemma}[Iteration via blow-up]
\label{l:control_mesoscopic}
Suppose that Assumption \ref{ass:forcing_regularity_iteration} holds. Let $\mu>0$, $r\in (1,\infty)$ and $p_0\in (2,p)$.
Fix $0<\g_0< (1-\frac{2}{p}-\frac{d}{q})\wedge (1-\frac{2}{p_0})$.  
Then there exist constants 
$$
L_0\geq 1\qquad \text{ and } \qquad \sm\in (0,1/4],
$$ 
depending only on $\mu,p,q,r,p_0,d$ and $\g_0$, for which the following assertion holds whenever $N\geq 1$ and $m\geq 1$ satisfy 
$$
N^{-1}\leq \sm^{m-1} L_0^{-1}.
$$
For all local solutions $(v^N,\pi^N,(\wt{\pi}^N_{k,\alpha})_{k,\alpha})$ to the stochastic Stokes system \eqref{eq:turbulent_Stokes_scaling_original_iteration} on $Q$, it holds that 
\begin{align*}
\Big(\E \Big\|v^N -\fint_{\sm^\ell B} v^N  \Big\|_{\underline{L}^2(\sm^\ell Q)}^{r}\Big)^{1/r}
 \leq \sm^{\ell \g_0}
\En_{Q}^\#,\\
\Big(\E \Big\|\pi^N -\fint_{\sm^\ell B} \pi^N  \Big\|_{\underline{L}^2(\sm^\ell I;\underline{H}^{-1}(\sm^\ell B))}^{r}\Big)^{1/r}
 \leq \sm^{\ell \g_0}
\En_{Q}^\#,\\
\Big(\E \Big\|\Big( \wt{\pi}_{k,\alpha}^N -\fint_{\sm^\ell B} \wt{\pi}_{k,\alpha}^N\Big)_{k,\alpha} \Big\|_{\underline{L}^{p_0}(\sm^{\ell} I;\underline{L}^2(\sm^\ell B; \ell^2))}^{r}\Big)^{1/r}
 \leq  \sm^{\ell \g_0}
 \En_{Q}^\#,
\end{align*}
for all $\ell\in \{1,\dots,m\}$,
where $\En_{Q}^\#$ is the sharp Lebesgue-energy of $(v^N,\pi^N,(\wt{\pi}^N_{k,\alpha})_{k,\alpha})$ at (large) scale $Q$, i.e.,\begin{align}
\nonumber
\En_Q^{\#}
=&\Big( \max\Big\{\E \|F \|_{\underline{L}^p(I;\underline{L}^q(B))}^r,\,
\E \|G \|_{\underline{L}^p(I;\underline{L}^q(B))}^r , \,
\E \Big\|v^{N} -\fint_{B} v^{N}\Big\|_{\underline{L}^2(Q)}^r
,\\
\label{eq:def_En_Q2}
&\E \Big\|\pi^{N}-\fint_{B} \pi^{N}\Big\|_{\underline{L}^2(I;\underline{H}^{-1}(B))}^r,\, 
\E \Big\|\Big( \wt{\pi}_{k,\alpha}^{N}-\fint_{B} \wt{\pi}^{N}_{k,\alpha} \Big)_{k,\alpha}\Big\|_{\underline{L}^{p_0}(I; \underline{L}^2(B;\ell^2))}^r \Big\}\Big)^{1/r}.
\end{align}
\end{lemma}

\begin{proof}
The claimed estimates of Lemma \ref{l:control_mesoscopic} follow by a blow-up argument and the one-step improvement. Below, the central importance of the joint condition on the parameters $\delta$ and $N$ in Lemma \ref{l:one_step_improvement} will become transparent.

\smallskip
 
Fix $\g_0\in (0,(1-\frac{2}{p}-\frac{d}{q})\wedge (1-\frac{2}{p_0}))$. Let $\sm$ and $L_0$ be the corresponding values provided by Lemma \ref{l:one_step_improvement}. 
By induction, assume that, for some $1\leq \ell\leq m-1$ with 
\begin{equation}
\label{eq:interpolation_resulting_estimates0}
N^{-1}\leq\sm^{\ell-1}L_0^{-1}, 
\end{equation}
the following estimates hold
\begin{align}
\label{eq:interpolation_resulting_estimates1}
\Big(\E \Big\|v^N -\fint_{\sm^{\ell} B} v^N  \Big\|_{\underline{L}^2(\sm^{\ell} Q)}^{r}\Big)^{1/r}
 \leq  \sm^{\ell \g_0}
\En_{Q}^\#,\\
\label{eq:interpolation_resulting_estimates2}
\Big(\E \Big\|\pi^N -\fint_{\sm^{\ell} B} \pi^N  \Big\|_{\underline{L}^2(\sm^{\ell} I;\underline{H}^{-1}(\sm^{\ell} B))}^{r}\Big)^{1/r}
 \leq \sm^{\ell \g_0}
\En_{Q}^\#,\\
\label{eq:interpolation_resulting_estimates3}
\Big(\E \Big\|\Big( \wt{\pi}_{k,\alpha}^N -\fint_{\sm^{\ell} B} \wt{\pi}_{k,\alpha}^N \Big)_{k,\alpha} \Big\|_{\underline{L}^{p_0}(\sm^\ell I; \underline{L}^2(\sm^{\ell} B; \ell^2))}^{r}\Big)^{1/r}
 \leq  \sm^{\ell \g_0}
 \En_{Q}^\#.
\end{align}
Next, we prove that \eqref{eq:interpolation_resulting_estimates1}-\eqref{eq:interpolation_resulting_estimates3} also hold with $\ell$ replaced by $\ell+1$ (the case $\ell=1$ follows from Lemma \ref{l:one_step_improvement}).

To this end, we define the following process on $Q_0\stackrel{{\rm def}}{=} Q_{1/2}(t_0,x_0)=(0,1/4]\times B_{1/2}(x_0)$, where $t_0=1/4$,
\begin{align}
\label{eq:rescaling1_iteration}
v^{N}_{\ell}(t,x)
&=v^N( t_0+\sm^{2\ell} (t-t_0),x_0+ \sm^\ell (x-x_0)),\\
\label{eq:rescaling2_iteration}
\pi^{N}_\ell(t,x)
&=\sm^\ell \pi^N(t_0+\sm^{2\ell} (t-t_0),x_0+ \sm^\ell (x-x_0)),\\
\label{eq:rescaling3_iteration}
\wt{\pi}^{N}_{k,\alpha,\ell}(t,x)
&=\wt{\pi}^N_{k,\alpha}( t_0+\sm^{2\ell} (t-t_0), x_0+ \sm^\ell (x-x_0)).
\end{align}
Arguing as in the proof of Corollary \ref{cor:SMR_microscopic_estimate} given in Subsection \ref{ss:proof_corollary_SMR_microscopic_estimate}, one can check that 
$$
(v^{N}_\ell,\pi^{N}_\ell,(\wt{\pi}^{N}_{k,\alpha,\ell})_{k,\alpha})
$$ 
is a local solution to the rescaled stochastic Stokes system \eqref{eq:turbulent_Stokes_scaling_quantitative_iteration} on $I_0=(0,1/4)$ (see Definition \ref{def:local_solutions_NSE}) with 
$$
\delta= \sm^\ell
$$
and $(F,G,(W^{k,\alpha})_{k,\alpha})$ replaced by 
\begin{align}
\label{eq:rescaling4_iteration}
F_{\sm,\ell}(t,x) &= \sm^{\ell} F(t_0+\sm^{2\ell} (t-t_0),x_0+ \sm^{\ell} (x-x_0)),\\
\label{eq:rescaling5_iteration}
G_{\sm,\ell}(t,x) &= \sm^{\ell} G(t_0+\sm^{2\ell} (t-t_0),x_0+ \sm^{\ell} (x-x_0)),\\
\label{eq:rescaling6_iteration}
\wt{W}^{k,\alpha}_{t} &=\sm^{-\ell}\big( W^{k,\alpha}_{t_0+\sm^{2\ell}(t-t_0)}-W^{k,\alpha}_{t_0(1-\sm^{2\ell})}).
\end{align}
From \eqref{eq:interpolation_resulting_estimates0} and $\ell\leq m-1$, we have 
$
N^{-1}\leq \sm^{m-1}L_0^{-1}\leq  \sm^{\ell} L_0^{-1} ,
$
and therefore
$$
(\sm^\ell N)^{-1}\leq L_0^{-1}.
$$ 
Hence, we can apply the one-step improvement in Lemma \ref{l:one_step_improvement} with $\delta=\sm^\ell$ and with
$$
(v^{N,\delta},\pi^{N,\delta},(\wt{\pi}_{k,\alpha}^{N,\delta})_{k,\alpha})\ \ \text{ replaced by } \ \ 
\big(v^{N}_\ell, \pi^{N}_\ell,(\wt{\pi}^{N}_{k,\alpha,\ell})_{k,\alpha} \big),
$$
and the above changes of forcings and noise \eqref{eq:rescaling4_iteration}-\eqref{eq:rescaling6_iteration} as well as domain $Q_0=I_0\times B_0$ where $I_0=(0,1/4]$ and $B_0=B_{1/2}(x_0)$.
Applying Lemma \ref{l:one_step_improvement}, we obtain 
\begin{equation}
\label{eq:velocity_rescaling_quantities_proof_iteration}
\Big(\E \Big\|v^{N} -\fint_{\sm^{1+\ell} B} v^N  \Big\|_{\underline{L}^2(\sm^{1+\ell} Q)}^{r}\Big)^{1/r}
=
\Big(\E \Big\|v^{N}_\ell -\fint_{\sm B_0} v^{N}_\ell  \Big\|_{\underline{L}^2(\sm Q_0)}^{r}\Big)^{1/r}
\leq \sm^{\g_0} \wt{\En}_Q^{\#}
\end{equation}
where 
\begin{align}
\nonumber
\wt{\En}_Q^{\#}
&=\Big( \max\Big\{\E \|F_{\sm,\ell} \|_{\underline{L}^p(I_0;\underline{L}^q(B_0))}^r,\,
\E \|G_{\sm,\ell} \|_{\underline{L}^p(I_0;\underline{L}^q(B_0))}^r , \\
\nonumber
&\quad\quad\quad  \E \Big\|v^{N}_\ell  -\fint_{B_0} v^{N}_\ell \Big\|_{\underline{L}^2(Q_0)}^r
,\, \E \Big\|\pi^{N}_\ell-\fint_{B_0} \pi^{N}_\ell\Big\|_{\underline{L}^2(I_0;\underline{H}^{-1}(B_0))}^r,\\
\label{eq:velocity_rescaling_quantities_proof_iteration1}
&\quad \quad\quad \quad \quad \E \Big\|\Big( \wt{\pi}_{k,\alpha,\ell}^{N}-\fint_{B_0} \wt{\pi}^{N}_{k,\alpha,\ell} \Big)_{k,\alpha}\Big\|_{\underline{L}^{p_0}(I_0; \underline{L}^2(B_0;\ell^2))}^r \Big\}\Big)^{1/r}
\leq \sm^{\ell \g_0}\En_Q^{\#},
\end{align}
where the last inequality is a consequence of the inductive assumption  \eqref{eq:interpolation_resulting_estimates1}-\eqref{eq:interpolation_resulting_estimates3} and the scaling properties of the averaged spaces, 
\begin{align}
\label{eq:F_is_decreasing_eta0}
\|F_{\sm,\ell} \|_{\underline{L}^p(I_0;\underline{L}^q(B_0))}
&=\sm^\ell \|F\|_{\underline{L}^p(\sm^\ell I;\underline{L}^q(\sm^\ell B))}\\
\nonumber
&\leq \sm^{\ell(1-2/p-d/q)}  \|F\|_{\underline{L}^p( I;\underline{L}^q( B))}\leq \sm^{\ell\g_0}  \|F\|_{\underline{L}^p( I;\underline{L}^q( B))},
\end{align}
as $\g_0<(1-2/p-d/q)$ and $\sm\leq 1$, and
\begin{align*}
\Big(\E \Big\|\pi^{N} -\fint_{\sm^{1+\ell} B} \pi^N  \Big\|_{\underline{L}^2(\sm^{1+\ell} I;\underline{H}^{-1}(\sm^{1+\ell} B))}^{r}\Big)^{1/r}\qquad\qquad\qquad\qquad &\\
=
\Big(\E \Big\|\pi^{N}_\ell -\fint_{\sm B_0} \pi^{N}_\ell \Big\|_{\underline{L}^2(\sm I_0;\underline{H}^{-1}(\sm B_0))}^{r}\Big)^{1/r},&
\end{align*}
due to \eqref{eq:scaling_negative_sob_space_intro}.
A similar consideration holds for forcing $G_{\sm,\ell}$ and stochastic pressure $(\wt{\pi}^{N}_{k,\alpha,\ell})_{k,\alpha}$.

\smallskip

From \eqref{eq:velocity_rescaling_quantities_proof_iteration} and \eqref{eq:velocity_rescaling_quantities_proof_iteration1}, we infer
$$
\Big(\E \Big\|v^N -\fint_{\sm^{1+\ell} B} v^N  \Big\|_{\underline{L}^2(\sm^{1+\ell} Q)}^{r}\Big)^{1/r}
\leq \sm^{\g_0(1+\ell)}\En_Q^{\#}.
$$
The proof that \eqref{eq:interpolation_resulting_estimates2}-\eqref{eq:interpolation_resulting_estimates3} holds with $\ell$ replaced by $1+\ell$ is analogous.
\end{proof}

\subsection{Proof of Theorem \ref{t:control_mesoscopic_energy}}
\label{ss:control_mesoscopic_energy_proof}
We begin by stating a useful consequence of Lemma \ref{l:control_mesoscopic} and of the stochastic Caccioppoli inequality in Theorem \ref{t:caccioppoli}. 

\begin{corollary}
\label{cor:caccioppoli_combined_decay}
In the setting of Lemma \ref{l:control_mesoscopic}, there exists a constant $C$ 
depending only on $\mu,p,q,r,p_0,d$ and $\g_0$ such that, for any local solution $(v^N,\pi^N,(\wt{\pi}^N_{k,\alpha})_{k,\alpha})$ to \eqref{eq:turbulent_Stokes_scaling_original_iteration}, it holds that
\begin{align}
\label{eq:caccioppoli_combined_decay}
\Big(\E\sup_{ \sm^{\ell+1} I} \Big\|v^N-\fint_{\sm^{\ell+1} B} v^N\Big\|_{\underline{L}^2(\sm^{\ell+1} B)}^r\Big)^{1/r}
&\leq 
C\sm^{\ell \g_0} \En_Q^{\#},\\
\label{eq:caccioppoli_combined_decay_p}
\sm^{\ell+1} \Big( \E\Big\|\pi^N-\fint_{\sm^{\ell+1} B} \pi^N\Big\|_{\underline{L}^2(\sm^{\ell+1} Q)}^r\Big)^{1/r}
&\leq 
C \En_Q^\#,
\end{align}
where $0\leq \ell\leq m$ and $m\geq 1$ is such that $N^{-1}\leq \sm^{m-1} L_0^{-1}$.
\end{corollary}

The key point in \eqref{eq:caccioppoli_combined_decay} is the independence of the constant $C$ in \eqref{eq:caccioppoli_combined_decay} and \eqref{eq:caccioppoli_combined_decay_p} from $\ell\in \{0,\dots,m\}$. 
Before going into the proof of the above result, we mention that Theorem \ref{t:caccioppoli} also provides additional information on $\nabla v^N$ and $(\nabla \wt{\pi}_{k,\alpha}^N)_{k,\alpha}$ at the scale $\sm^{\ell+1}Q$. However, this will not be used here. 
Moreover, the different scaling in \eqref{eq:caccioppoli_combined_decay}-\eqref{eq:caccioppoli_combined_decay_p} reflects the invariance under the dilation in \eqref{eq:rescaling1_iteration}-\eqref{eq:rescaling3_iteration}.

\begin{proof}
We again employ a blow-up argument as in the proof of Lemma \ref{l:control_mesoscopic}. 
Let $(v^{N}_\ell,\pi^{N}_\ell, (\wt{\pi}^{N}_{k,\alpha,\ell})_{k,\alpha})$ be as in \eqref{eq:rescaling1_iteration}-\eqref{eq:rescaling6_iteration}. As shown in the proof of Lemma \ref{l:control_mesoscopic}, the latter is a local solution to \eqref{eq:turbulent_Stokes_scaling_quantitative_iteration} on $Q_0=Q_{1/2}(t_0,x_0)=I_0\times B_0$ where $t_0=1/4$ and with data and noise as in \eqref{eq:rescaling4_iteration}.

From the Caccioppoli inequality of Theorem \ref{t:caccioppoli} applied to the local solution (see also \eqref{eq:replacying_average_Caccioppoli} for the pressures)  
$$
\Big(v^{N}_\ell,\pi^{N}_\ell-\fint_{B_0}\pi^{N}_\ell,\Big(\wt{\pi}^{N}_{k,\alpha,\ell}-\fint_{ B_0}\wt{\pi}^{N}_{k,\alpha,\ell}\Big)_{k,\alpha}\Big)
$$ 
and the invariance property of the averaged Lebesgue spaces, it follows that:
\begin{align*}
\Big(\E\sup_{ \sm^{\ell+1}I} \Big\|v^N-\fint_{\sm^{\ell+1} B} v^N\Big\|_{\underline{L}^2(\sm^{\ell+1} B)}^r\Big)^{1/r}
&=
\Big(\E\sup_{ \sm I_0} \Big\|v_\ell^N-\fint_{\sm B_0} v_\ell^N\Big\|_{\underline{L}^2(\sm B_0)}^r\Big)^{1/r}\\
&\leq C_\sm \Eno_{\sm Q,0},
\end{align*}
where 
\begin{align*}
\Eno_{\sm Q,0}
&= \Big( \max\Big\{
\E \Big\|v^{N}_\ell -\fint_{ B_0} v^{N}_\ell  \Big\|_{L^2( Q_0)}^r
,\,
\E \Big\|\pi^{N}_\ell-\fint_{B_0}\pi^{N}_\ell\Big\|_{L^2(I_0;H^{-1}( B_0))}^r,\, \\
&\quad \E \Big\|\Big( \wt{\pi}^{N}_{k,\alpha,\ell} -\fint_{B_0}\wt{\pi}_{k,\alpha,\ell}^{N}\Big)_{k,\alpha}\Big\|_{L^{2}( Q_0;\ell^2)}^r,\,
\E \|F_{\sm,\ell} \|_{L^2( Q_0)}^r,\,
\E \|G_{\sm,\ell} \|_{L^2( Q_0)}^r  \Big\}\Big)^{1/r}\\
&\leq C \sm^{\ell \g_0} \En_Q^{\#},
\end{align*}
and the last inequality follows from Lemma \ref{l:control_mesoscopic} and \eqref{eq:F_is_decreasing_eta0} (again, the same estimate applies also to $G_{\sm,\ell}$).
As above, in the application of the stochastic Caccioppoli inequality, we used the fact that the rescaled noise coefficients $\sqrt{c_d\mu}(\theta_k^N\sigma_{k,\alpha}^\delta)_{k,\alpha}$ satisfy the conditions \eqref{eq:ass_caccioppoli1}-\eqref{eq:ass_caccioppoli3} uniformly in $N$ and $\delta$.

The proof of \eqref{eq:caccioppoli_combined_decay_p} is analogous, with the additional observations that 
$$
\Big\|\pi^N-\fint_{\sm^{\ell+1} B} \pi^N\Big\|_{\underline{L}^2(\sm^{\ell+1} B)}
\leq \Big\|\pi^N-\fint_{\sm^{\ell} B} \pi^N\Big\|_{\underline{L}^2(\sm^{\ell+1}  B)},
$$
and the factor $ \sm^{\ell+1}$ on the left-hand side of \eqref{eq:caccioppoli_combined_decay_p} comes from the rescaling \eqref{eq:rescaling2_iteration}.
\end{proof}

With the above preparation, we are ready to prove Theorem \ref{t:control_mesoscopic_energy}.

\begin{proof}[Proof of Theorem \ref{t:control_mesoscopic_energy}]
We begin by recalling some notation. Recall that $p_0,p,q\in (2,\infty)$ and $r$ are fixed in Theorem \ref{t:control_mesoscopic_energy}. Fix $0<\g_0<(1-\frac{2}{p}-\frac{d}{q})\wedge(1-\frac{2}{p_0})$. Let $L_0>0$ and $\sm\in (0,1)$ be the parameters provided by Lemma \ref{l:control_mesoscopic} with such a choice of $(p_0,p,q,r)$. For $N\geq L_0$, let $m\geq 1$ be such that 
\begin{equation}
\label{eq:choice_N_Lsm}
\sm^{m}L_0^{-1}< N^{-1}\leq \sm^{m-1} L_0^{-1}.
\end{equation}		 
From Poincar\'e's inequality (see Lemma \ref{l:sharp_poincare}) and Fubini's theorem, it holds that 
\begin{align}
\label{eq:Poincare_easy1_proof}
\Big\| \pi^N-\fint_{B} \pi^N\Big\|_{H^{-1}(B)}& \leq
C
\|\nabla \pi^N\|_{H^{-1}(B)}\\
\label{eq:Poincare_easy2_proof}
\Big\| \Big(\wt{\pi}_{k,\alpha}^N-\fint_{B} \wt{\pi}_{k,\alpha}^N\Big)_{k,\alpha}\Big\|_{L^2(B;\ell^2)}
&\leq C
 \|( \nabla \wt{\pi}_{k,\alpha}^N )_{k,\alpha}\|_{H^{-1}(B;\ell^2)}.
\end{align}
Thus, it suffices to prove the estimates in Theorem \ref{t:control_mesoscopic_energy} with $\En_Q$ replaced by $\En_{Q}^\#$ as defined in \eqref{eq:def_En_Q2}.
For convenience, we divide the proof into four steps. 

\smallskip

\emph{Step 1: (Uniform control of the averages along scales). There exists a constant $C_0$ depending only on $\mu,p,q,r,p_0,d$ and $\g_0$ such that, for all local solutions $(v^N,\pi^N,(\wt{\pi}^N_{k,\alpha})_{k,\alpha})$ to \eqref{eq:turbulent_Stokes_scaling_original_iteration} and $\ell\in \{1,\dots, m-1\}$, it holds that}
$$
\Big(\E \Big\|
\fint_{\sm^\ell B} v^{N}\Big\|_{\underline{L}^2(\sm^\ell I)}^{r}\Big)^{1/r}
\leq C_0 \Big(\E\sup_{I} \|v^{N}\|_{\underline{L}^2(B)}^{r}\Big)^{1/r}+ C_0 \En_Q^{\#}.
$$
To prove the claim of this step, we employ a standard telescoping argument. Indeed, we write 
\begin{align*}
\fint_{\sm^\ell B} v^{N}
&=\fint_{\sm B} v^{N}
+ \sum_{j=1}^{\ell -1} \fint_{\sm^{j+1}B}\Big( v^{N}
-\fint_{\sm^jB} v^{N}\Big).
\end{align*}
Then, as $\sm^\ell I \subseteq \sm^{j} I$ for all $j\leq \ell-1$ as $\sm<1$, we obtain
\begin{align*}
&\Big(\E \Big\|
\fint_{\sm^\ell B} v^{N}\Big\|_{\underline{L}^2(\sm^\ell I)}^{r}\Big)^{1/r}\\
&\leq \Big(\E \Big\|\fint_{\sm B } v^{N}\Big\|_{\underline{L}^2(\sm^\ell I)}^{r}\Big)^{1/r}
+ C_\sm \sum_{j=1}^{\ell -1} \Big(\E \Big\| v^{N}-\fint_{\sm^jB} v^{N}\Big\|_{\underline{L}^2(\sm^\ell I;\underline{L}^2(\sm^{j} B))}^r\Big)^{1/r}\\
&\leq \Big(\E\sup_{I} \Big|\fint_{\sm B } v^{N}\Big|^{r}\Big)^{1/r}
+ C_\sm \sum_{k=0}^{\ell -2} \Big(\E \sup_{\sm^{k+1}I}\Big\| v^{N}
-\fint_{\sm^{k+1} B} v^{N}\Big\|_{\underline{L}^2(\sm^{k+1} B)}^r\Big)^{1/r}\\
&\leq C_\sm \Big(\E\sup_{I} \| v^{N}\|_{\underline{L}^2(B)}^{r}\Big)^{1/r}+ C_\sm \En_{Q}^{\#} \sum_{k=0}^\infty \sm^{\g_0 k}
\end{align*}
where the last inequality follows from \eqref{eq:caccioppoli_combined_decay} in Corollary \ref{cor:caccioppoli_combined_decay}.
Since $\sm<1$, the series converges, completing the proof of Step 1.

\smallskip

\emph{Step 2: Uniform control of the energy at the microscopic scale.} 
We begin by noticing that, if $N\leq L_0$, then the uniform bound of Theorem \ref{t:control_mesoscopic_energy} is trivial as 
$$
\big(\E \|v^N \|_{\underline{L}^2(N^{-1}Q)}^r\big)^{1/r}
\leq L_0^{1+d/2} \big(\E \|v^N \|_{\underline{L}^2(Q)}^r\big)^{1/r}.
$$ 
Next, we assume $N\geq L_0$. In particular, there exists $m\geq 1$ for which \eqref{eq:choice_N_Lsm} holds.
It follows from Lemma \ref{l:control_mesoscopic} and Step 1 that there exists a constant $C>0$ depending only on $\mu,p,q,r,p_0,d$ and $\g_0$ such that
\begin{align}
\label{eq:uniform_norm_energy_conclusion_theorem_energy_microscopic}
&\big(\E \|v^N \|_{\underline{L}^2(\sm^{m-1}Q)}^r\big)^{1/r}\\
\nonumber
&\leq \Big(\E \Big\|\fint_{\sm^{m-1}B} v^{N}\Big\|_{\underline{L}^2(\sm^{m-1} I)}^{r}\Big)^{1/r}
+ \Big(\E \Big\|v^N -\fint_{\sm^{m-1} B} v^N\Big\|_{\underline{L}^2(\sm^{m-1} Q)}^r\Big)^{1/r}\\
\nonumber
&\leq  C\Big(\E\sup_{I} \|v^{N}\|_{\underline{L}^2(B)}^{r}\Big)^{1/r}+ C \En_Q^{\#}.
\end{align}
Hence, from \eqref{eq:choice_N_Lsm}, it follows that 
\begin{align*}
\big(\E \|v^N \|_{\underline{L}^2(N^{-1}Q)}^r\big)^{1/r}
\lesssim_{L_0,\sm} \big(\E \|v^N \|_{\underline{L}^2(L_0^{-1}\sm^{m-1}Q)}^r\big)^{1/r}
\lesssim_{L_0} 
\big(\E \|v^N \|_{\underline{L}^2(\sm^{m-1}Q)}^r\big)^{1/r},
\end{align*}
where we used $L_0\geq 1$. The conclusion follows from \eqref{eq:uniform_norm_energy_conclusion_theorem_energy_microscopic}.

\smallskip

\emph{Step 3: Uniform estimate for the oscillation of the stochastic pressure.}
The proof is similar to Step 2. First, in case $N\leq L_0$, one has 
\begin{align*}
&\Big\|\Big(\wt{\pi}_{k,\alpha}^N- \fint_{N^{-1}B} \wt{\pi}_{k,\alpha}^N\Big)_{k,\alpha}\Big\|_{\underline{L}^{2}(N^{-1}B;\ell^2)}\\
&\leq 
\Big\|\Big(\wt{\pi}_{k,\alpha}^N- \fint_{B} \wt{\pi}_{k,\alpha}^N\Big)_{k,\alpha}\Big\|_{\underline{L}^{2}(N^{-1}B;\ell^2)}
+\Big\|\Big(\fint_{N^{-1}B} \Big(\wt{\pi}_{k,\alpha}^N- \fint_{B} \wt{\pi}_{k,\alpha}^N\Big) \Big)_{k,\alpha}\Big\|_{\ell^2}\\
&\leq C_{L_0}
\Big\|\Big(\wt{\pi}_{k,\alpha}^N- \fint_{B} \wt{\pi}_{k,\alpha}^N\Big)_{k,\alpha}\Big\|_{\underline{L}^{2}(B;\ell^2)},
\end{align*}
and thus
$$
\Big\|\Big(\wt{\pi}_{k,\alpha}^N- \fint_{N^{-1}B} \wt{\pi}^N_{k,\alpha}\Big)_{k,\alpha}\Big\|_{\underline{L}^{p_0}(N^{-1}I;\underline{L}^{2}(N^{-1}B;\ell^2))}
\leq C_{L_0}\Big\|\Big(\wt{\pi}_{k,\alpha}^N- \fint_{B} \wt{\pi}^N_{k,\alpha}\Big)_{k,\alpha}\Big\|_{\underline{L}^{p_0}(I;\underline{L}^{2}(B;\ell^2))}.
$$
Similarly, if \eqref{eq:choice_N_Lsm} holds, then one has 
\begin{align*}
&\Big\|\Big(\wt{\pi}_{k,\alpha}^N- \fint_{N^{-1}B} \wt{\pi}_{k,\alpha}^N\Big)_{k,\alpha}\Big\|_{\underline{L}^{2}(N^{-1}B;\ell^2)}\\&\leq 
\Big\|\Big(\wt{\pi}_{k,\alpha}^N- \fint_{\sm^{m-1}B} \wt{\pi}_{k,\alpha}^N\Big)_{k,\alpha}\Big\|_{\underline{L}^{2}(N^{-1}B;\ell^2)}
+\Big\|\Big(\fint_{ N^{-1}B} \Big(\wt{\pi}_{k,\alpha}^N- \fint_{\sm^{m-1}  B} \wt{\pi}_{k,\alpha}^N\Big) \Big)_{k,\alpha}\Big\|_{\ell^2}\\
&\lesssim_{L_0}
\Big\|\Big(\wt{\pi}_{k,\alpha}^N- \fint_{\sm^{m-1} B} \wt{\pi}_{k,\alpha}^N\Big)_{k,\alpha}\Big\|_{\underline{L}^{2}(\sm^{m-1} B;\ell^2)},
\end{align*}
since $L_0\geq 1$, and thus
\begin{align*}
&\Big\|\Big(\wt{\pi}_{k,\alpha}^N- \fint_{N^{-1}B} \wt{\pi}^N_{k,\alpha}\Big)_{k,\alpha}\Big\|_{\underline{L}^{p_0}(N^{-1}I;\underline{L}^{2}(N^{-1}B;\ell^2))}\\
&\leq C_{L_0,\sm,d}
\Big\|\Big(\wt{\pi}_{k,\alpha}^N- \fint_{\sm^{m-1} B} \wt{\pi}^N_{k,\alpha}\Big)_{k,\alpha}\Big\|_{\underline{L}^{p_0}( \sm^{m-1} I;\underline{L}^{2}(\sm^{m-1}B;\ell^2))}.
\end{align*}
The claim of Step 3 now follows from Lemma \ref{l:control_mesoscopic}.

\smallskip

\emph{Step 4: Uniform estimate for the oscillation of the deterministic pressure.} 
Using the argument in Step 3 and \eqref{eq:choice_N_Lsm}, one can readily check that 
\begin{align*}
N^{-1} \Big\|\pi^N- \fint_{N^{-1}B} \pi^N \Big\|_{\underline{L}^2(N^{-1}Q)}
\lesssim_{L_0,\sm,d}
\sm^{m-1} \Big\|\pi^N- \fint_{\sm^{m-1} B} \pi^N \Big\|_{\underline{L}^2( \sm^{m-1}Q)}.
\end{align*}
We emphasize that the implicit constant in the above estimate depends only on $\mu,p,q,r,p_0,d$ and $\g_0$.
Hence, from \eqref{eq:caccioppoli_combined_decay_p} in Corollary \ref{cor:caccioppoli_combined_decay} and the scaling of $\underline{H}^{-1}$ in \eqref{eq:scaling_negative_sob_space_intro}, it follows that 
\begin{align*}
\Big(\E \Big\|\pi^N-\fint_{N^{-1}B}\pi^N \Big\|_{\underline{L}^2(N^{-1}I;\underline{H}^{-1}(N^{-1} B))}^r\Big)^{1/r}  
&\leq N^{-1} \Big( \E \Big\|\pi^N- \fint_{N^{-1}B} \pi^N \Big\|_{\underline{L}^2(N^{-1}Q)}^r \Big)^{1/r} \\
&\leq C_{L_0,\sm,d} C \En_Q^{\#}.
\end{align*}
As commented below \eqref{eq:Poincare_easy1_proof}-\eqref{eq:Poincare_easy2_proof}, the above completes the proof of Theorem \ref{t:control_mesoscopic_energy}.
\end{proof}

\section{$L^p(L^q)$-estimates for the oscillating stochastic Stokes system}
\label{s:uniform_Lp_estimates}
In this section, we prove $L^p(L^q)$-bounds for the oscillating stochastic Stokes system that are uniform in the oscillation parameter $N$ and globally in the spatial domain. More precisely, we consider the following stochastic Stokes system on the $d$-dimensional torus $\T^d$:
\begin{equation}
\label{eq:turbulent_Stokes_scaling_uniform_estimates}
\left\{
\begin{aligned}
\partial_t w^N &=(1+\mu) \Delta w^N+\p[\nabla \cdot F]- c_d \mu\,\p \sum_{k,\alpha} \theta_k^N\nabla\cdot  (\nabla \wt{P}^N_{k,\alpha}\otimes \sigma_{-k,\alpha})\\ 
&\qquad\ \  + \sqrt{c_d \mu}\sum_{k,\alpha} \theta^{N}_{k}\p[(\sigma_{k,\alpha}\cdot\nabla) w^N]\,\dot{W}^{k,\alpha}_t,\\
\nabla \wt{P}^N_{k,\alpha}
&= \theta^{N}_{k}\q[(\sigma_{k,\alpha}\cdot\nabla) w^N],\\
w^N(0)&=w_{0},
 \end{aligned}
\right.
\end{equation}
$\p$ denotes the Helmholtz projection on $\T^d$ and $\theta^N=(\theta_k^N)_{k}$ is as in \eqref{eq:choice_thetaN}. 

Similarly to Subsection \ref{ss:ito_reformulation}, the system \eqref{eq:turbulent_Stokes_scaling_uniform_estimates} formally corresponds to the following stochastic Stokes system with \emph{Stratonovich} noise:
$$
\partial_t w^N = \Delta w^N+\p[\nabla \cdot F]
+ \sqrt{c_d \mu}\sum_{k,\alpha} \theta^{N}_{k}\p[(\sigma_{k,\alpha}\cdot\nabla) w^N]\circ \dot{W}^{k,\alpha}_t.
$$ 
The following result is well-known and provides the basis for our investigation.

\begin{lemma}[Energy solutions to the stochastic Stokes system]
\label{l:global_energy_estimate}
There exists a constant $C>0$ such that, for any $N\geq 1$, any interval $I=(0,T)\subseteq \R_+$ with $T<\infty$,
$$
w_0\in L^0_{\F_{0}}(\O;\Ls^2(\T^d))\qquad \text{ and }\qquad 
F\in L^0_{\Progress}(\O;L^2(I\times \T^d;\R^{d\times d})),
$$
there exists a unique progressively measurable process 
$$
w^N\in C(\overline{I};\Ls^2(\T^d))\cap L^2(I;\Hs^{1}(\T^d))\ \text{ a.s. }
$$ 
which solves the oscillating stochastic Stokes system \eqref{eq:turbulent_Stokes_scaling_uniform_estimates} in its natural integrated form (cf.\ Definition \ref{def:p_solution}), and satisfies the following \emph{pathwise} energy inequality: 
\begin{align}
\label{eq:energy_estimate_pathwise_u0}
\sup_I\|w^N\|_{L^{2} (\T^d)}
+\|\nabla w^N\|_{L^{2}(I\times \T^d)}
\leq C \|w_0\|_{L^{2}(\T^d)} + C  \|F\|_{L^2(I\times \T^d)}\ \text{ a.s. }
\end{align}
\end{lemma}

In the following, solutions provided by Lemma \ref{l:global_energy_estimate} are referred to as \emph{energy solutions} to \eqref{eq:turbulent_Stokes_scaling_uniform_estimates}.
The proof of Lemma \ref{l:global_energy_estimate} is standard and follows directly from the variational approach to SPDEs, see e.g., \cite[Chapter 4]{LR15} and \cite[Appendix A]{AV21_NS}. 

\smallskip

As explained in Subsection \ref{ss:scaling_intro}, the norms in \eqref{eq:energy_estimate_pathwise_u0} are too weak to handle the Navier-Stokes nonlinearity. 
Below, we establish suboptimal $L^p(L^q)$-estimates that are uniform in $N$ and allow for a subcritical Sobolev index $-2/p-d/q>-1$. 

\begin{theorem}[$L^p(L^q)$-estimates uniform in the oscillation -- Finite time horizon]
\label{t:suboptimal_Lq_finite_time}
Fix $r\in (1,\infty)$ and an interval $I=(0,T)$ for some $T<\infty$. Let $p,q\in (2,\infty)$ be such that the Sobolev index $\Sob\stackrel{{\rm def}}{=}1-\frac{2}{p}-\frac{d}{q}$ satisfies
\begin{equation}
\label{eq:assumption_suboptimal_Lq_finite_time1}
-\frac{d}{2}\leq \Sob <0, \qquad \text{ and }\qquad p<\frac{d}{-\Sob},
\end{equation}
Moreover, set $\qmax\stackrel{{\rm def}}{=}-\frac{d^2}{d-2}\frac{1}{\Sob}$, 
with the convention $\qmax=\infty$ if $d=2$.
Then for all $p_0\in (2,\infty)$ and $q_0\in (2,\qmax)$ satisfying 
\begin{equation}
\label{eq:suboptimal_Lq_finite_time_suboptimality_condition}
-\frac{2}{p_0}-\frac{d}{q_0}<\Sob,
\end{equation}
there exists a constant $C_0$ such that, for all $N\geq 1$, 
\begin{equation}
\label{eq:suboptimal_Lq_finite_time_data}
w_0\in L^r_{\F_0}(\O;\Bs^{1-2/p}_{q,p}(\T^d))\qquad \text{ and }\qquad F\in L^r_{\Progress}(\O;L^p(I ;L^q(\T^d;\R^{d\times d}))),
\end{equation}
the unique energy solution $w^N$ to the oscillating stochastic Stokes system \eqref{eq:turbulent_Stokes_scaling_uniform_estimates} provided by Lemma \ref{l:global_energy_estimate} satisfies
\begin{align}
\label{eq:suboptimal_Lq_finite_time1}
\big(\E\|w^N\|_{L^{p_0}(I;L^{q_0} (\T^d))}^{r} \big)^{1/r}
&\leq C_0 \big(\E\|w_0\|_{B^{1-2/p}_{q,p}(\T^d)}^r\big)^{1/r}\\
\nonumber
& + C_0  \big(\E\|F\|_{L^p(I;L^q(\T^d))}^r\big)^{1/r}.
\end{align}
\end{theorem}

The fundamental feature of Theorem \ref{t:suboptimal_Lq_finite_time} is the independence of the constant $C_0$ from the oscillation parameter $N$ as well as the data $(w_0,F)$.

Before going further, let us discuss the sharpness of the estimate \eqref{eq:suboptimal_Lq_finite_time1} from a scaling point of view. As the notation suggests, $\Sob$ is the Sobolev index or regularity content of the maximal regularity space for the problem \eqref{eq:turbulent_Stokes_scaling_uniform_estimates} with the data (see e.g., \cite[Section 3]{AV25_survey} or \cite[Theorem 3.2]{AV21_NS}), i.e., 
\begin{equation}
\label{eq:maximal_regularity_space_associated_LpLq}
C(\overline{I};\Bs_{q,p}^{1-2/p}(\T^d))\cap L^p(I;\Hs^{1,q}(\T^d)).
\end{equation}
Note that the above space has space-time Sobolev index\footnote{Recall that the (parabolic) space-time Sobolev index of  $L^{p}(I;H^{s,q}(\T^d))$ is $-\frac{2}{p}+s-\frac{d}{q}$.} given by $\Sob$. 
However, the scaling of the norm $L^{p_0}_t(L^{q_0}_x)$ appearing on the left-hand side of \eqref{eq:suboptimal_Lq_finite_time1} has Sobolev index given by 
$$
\Sob_0 \stackrel{{\rm def}}{=}-\frac{2}{p_0}-\frac{d}{q_0},
$$
and the condition \eqref{eq:suboptimal_Lq_finite_time_suboptimality_condition} implies that the regularity content of the space for $w^N$ in \eqref{eq:suboptimal_Lq_finite_time1} is strictly lower, but arbitrarily close to the optimal $\Sob$.
This small loss of sharpness in the above result is a consequence of the loss of scaling in the $L^\infty$-estimate of Theorem \ref{t:almost_Linfty}. The sharp case of Theorem \ref{t:suboptimal_Lq_finite_time} goes beyond the scope of this manuscript. It is worth noticing that, for the sake of the results stated in Section \ref{s:statements}, this sharpness loss is harmless and perfectly matches the loss of regularity in the compactness argument behind the scaling limit described in Subsection \ref{ss:scaling_intro} (see also Section \ref{s:scaling_limit} below).  
 
\smallskip

The limitation on the integrability $\qmax$ is less transparent and follows from our proof strategy. The proof of Theorem \ref{t:suboptimal_Lq_finite_time} is based on an interpolation argument between the two endpoints (cf.\ the condition \eqref{eq:assumption_suboptimal_Lq_finite_time1}):
\begin{itemize}
\item Energy inequality corresponding to the Sobolev index $-\frac{d}{2}$, see Lemma \ref{l:global_energy_estimate}.
\item Almost $L^\infty$-bounds corresponding  to the Sobolev index $1-\frac{2}{p}-\frac{d}{q}\approx 0$, see Theorem \ref{t:almost_Linfty} below.
\end{itemize}
As depicted in Figure \ref{fig:scheme}, the results of Sections \ref{s:localized_smr} and \ref{s:energy_control_microscopic} are central to proving Theorem \ref{t:almost_Linfty} and hence Theorem \ref{t:suboptimal_Lq_finite_time} as well as the infinite-time version below. 

Before going further, let us discuss how the value $\qmax$ appears from the above-mentioned interpolation. We restrict to the case $d\geq 3$, as $d=2$ is similar, see \eqref{eq:endpoint_2_UN_interpolation} below. 
Interpolating the endpoints $L^2(I;H^1(\T^d))\embed L^2(I;L^{2d/(d-2)}(\T^d))$ and $L^\infty(I\times \T^d)$ (corresponding to Lemma \ref{l:global_energy_estimate} and Theorem \ref{t:almost_Linfty}, respectively), and with interpolation parameter $\vartheta=1+\frac{2}{d}\,\Sob $ (see \eqref{eq:vartheta_choice_interpolation} for details), one sees that the resulting spatial integrability is 
$$
\frac{(d-2)(1-\vartheta)}{2d}=\frac{1}{\qmax} \qquad \Longrightarrow \qquad \qmax=-\frac{d^2}{(d-2)}\frac{1}{\Sob}.
$$

Finally, we formulate a version of Theorem \ref{t:suboptimal_Lq_finite_time} on the half-line, which we employ to prove Theorem \ref{t:global_NSE}, and follows from the latter result and a standard localization argument. The reader is referred to Subsection \ref{ss:suboptimal_Lq_Rpositive} for details.

\begin{theorem}[$L^p(L^q)$-estimates uniform in the oscillation -- Infinite time horizon]
\label{t:suboptimal_Lq_Rpositive}
Let $p,q\in (2,\infty)$ be such that $\Sob=1-\frac{2}{p}-\frac{d}{q}$ satisfies  \eqref{eq:assumption_suboptimal_Lq_finite_time1}, and set $\qmax\stackrel{{\rm def}}{=}-\frac{d^2}{d-2}\frac{1}{\Sob}$, 
with the convention $\qmax=\infty$ if $d=2$. Then, for all $p_0\in (p,\infty)$, $r\in (1,p_0]$, and $q_0\in (2,\qmax)$ satisfying \eqref{eq:suboptimal_Lq_finite_time_suboptimality_condition}, there exists a constant $C_0>0$ such that, for all $N\geq 1$, any interval $I=(0,T)\subseteq \R_+$ with $T\in (0,\infty]$, and 
\begin{equation}
\label{eq:suboptimal_Lq_Rpositive_data}
w_0\in L^r_{\F_0}(\O;\Bs^{1-2/p}_{q,p}(\T^d))\qquad \text{ and }\qquad F\in L^r_{\Progress}(\O;L^p(I ;L^q(\T^d;\R^{d\times d}))),
\end{equation}
such that $\int_{\T^d} w_0 =0$ a.s., the unique energy solution $w^N$ to the oscillating stochastic Stokes system \eqref{eq:turbulent_Stokes_scaling_uniform_estimates} provided by Lemma \ref{l:global_energy_estimate} satisfies
\begin{align}
\nonumber
\big(\E\|w^N\|_{L^{p_0}(I;L^{q_0} (\T^d))}^{r} \big)^{1/r}
&
\leq C_0 \big(\E\|w_0\|_{B^{1-2/p}_{q,p}(\T^d)}^r\big)^{1/r} + C_0  \big(\E\|F\|_{L^p(I;L^q(\T^d))}^r\big)^{1/r}\\
\label{eq:suboptimal_Lq_Rpositive_estimate}
 &+ C_0\big(\E\|w^N\|_{L^{p}(I;H^{-1,q} (\T^d))}^r \big)^{1/r}. 
\end{align}
\end{theorem}

Comparing the above with Theorem \ref{t:suboptimal_Lq_finite_time}, we are able to prove a similar estimate with a constant that is independent of $N$, $w_0$ and $F$ and the length of the time interval $T$, at the expense of a lower order term on the right-hand side of \eqref{eq:suboptimal_Lq_Rpositive_estimate} and with an additional restriction on the integrability in the probability space $r$. We expect that the latter is not essential for the estimate \eqref{eq:suboptimal_Lq_Rpositive_estimate}. However, \eqref{eq:suboptimal_Lq_Rpositive_estimate} is sufficient for our purposes.
To conclude, let us mention that for certain choices of $(p,q)$ the lower-order term on the right-hand side of \eqref{eq:suboptimal_Lq_Rpositive_estimate} can be estimated via Lemma \ref{l:global_energy_estimate}.
However, we preferred the formulation \eqref{eq:suboptimal_Lq_Rpositive_estimate} as it is easier to apply to prove our main results.

\smallskip

The proofs of Theorems \ref{t:suboptimal_Lq_finite_time} and \ref{t:suboptimal_Lq_Rpositive} are given in Subsections \ref{ss:suboptimal_Lq_finite_time_proof} and \ref{ss:suboptimal_Lq_Rpositive} below.

\subsection{Uniform $L^p(L^q)$-estimates on bounded intervals -- Proof of Theorem \ref{t:suboptimal_Lq_finite_time}}
\label{ss:suboptimal_Lq_finite_time_proof}
As is common in the study of maximal $L^p$-regularity estimates for evolution equations (see e.g., \cite[Subsection 3.3]{AV25_survey}), it is convenient to consider the following variant of the oscillating Stokes system \eqref{eq:turbulent_Stokes_scaling_uniform_estimates} with additive noise on $I=(t_0-T,t_0)$ for some $T<\infty$ such that $t_0\geq T-1/4$:
\begin{equation}
\label{eq:turbulent_Stokes_scaling_uniform_estimates_modified}
\left\{
\begin{aligned}
\partial_t v^N &=(1+\mu) \Delta v^N+\p[\nabla \cdot F]- c_d \mu\,\p \sum_{k,\alpha} \theta_k^N\nabla\cdot  (\nabla \wt{\pi}^N_{k,\alpha}\otimes \sigma_{-k,\alpha})\\ 
&\qquad\ \  + \sqrt{c_d \mu}\sum_{k,\alpha} \theta^{N}_{k}\p[(\sigma_{k,\alpha}\cdot\nabla) v^N+\sigma_{k,\alpha}\cdot G]\,\dot{W}^{k,\alpha}_t,\\
\nabla \wt{\pi}^N_{k,\alpha}
&= \theta^{N}_{k}\q[(\sigma_{k,\alpha}\cdot\nabla) v^N+ \sigma_{k,\alpha}\cdot G],\\
 v^N(t_0-T)&=0.
 \end{aligned}
\right.
\end{equation}
Energy solutions to the above can be introduced similarly to those for the original system  \eqref{eq:turbulent_Stokes_scaling_uniform_estimates}. Moreover, similarly to Lemma \ref{l:global_energy_estimate} (see e.g., \cite[Appendix A]{AV21_NS}), one can check that, for any $r\in (1,\infty)$, there exists a constant $C>0$ independent of $N$, $t_0$, $F$ and $G$ such that  
\begin{align}
\label{eq:energy_estimate_vN}
\big(\E\sup_I\|v^N\|_{L^{2} (\T^d)}^r\big)^{1/r}
&+\big(\E\|\nabla v^N\|^r_{L^{2}(I\times \T^d)}\big)^{1/r}\\
\nonumber
&\leq C  \big(\E\|F\|_{L^2(I\times \T^d)}^r\big)^{1/r}+C  \big(\E\|G\|_{L^2(I\times \T^d)}^r\big)^{1/r} ,
\end{align}
where $v^N$ is the energy solution to \eqref{eq:turbulent_Stokes_scaling_uniform_estimates_modified}.

The reasoning behind the system \eqref{eq:turbulent_Stokes_scaling_uniform_estimates_modified} is that, if $w^N$ is an energy solution to \eqref{eq:turbulent_Stokes_scaling_uniform_estimates}, then $v^N(t)= w^N (t)-z(t)$ where $z(t)=e^{(1+\mu)\Delta t}w_0$ solves \eqref{eq:turbulent_Stokes_scaling_uniform_estimates_modified} with $G=[\nabla z]^\top$ and $t_0=T$. Thus, estimates for solutions to \eqref{eq:turbulent_Stokes_scaling_uniform_estimates_modified} immediately yield corresponding bounds for \eqref{eq:turbulent_Stokes_scaling_uniform_estimates}.

\smallskip

In this subsection, we let $t_0-T\in (-1/4,\infty)$ and therefore the initial time in \eqref{eq:turbulent_Stokes_scaling_uniform_estimates_modified} is possibly negative. Recall that, as mentioned in Subsection \ref{ss:probabilistic}, here we implicitly assume that the complex-valued Brownian motions $(W^{k,\alpha})_{k,\alpha}$ have been extended to times in $(-1/4,\infty)$. This extension will be exploited in the covering argument used in the proof of the following result (see also Figure \ref{fig:cubes_nested}).

\smallskip

The fundamental result for proving Theorems \ref{t:suboptimal_Lq_finite_time} and \ref{t:suboptimal_Lq_Rpositive} is given by the following result, which yields supremum moment estimates for the energy solution to \eqref{eq:turbulent_Stokes_scaling_uniform_estimates_modified} and is a consequence of the localized stochastic maximal $L^p$-regularity in Corollary \ref{cor:SMR_microscopic_estimate}, and the control of the energy at the microscopic scale in Theorem \ref{t:control_mesoscopic_energy}.

\begin{theorem}[Supremum moment bounds on a finite interval]
\label{t:almost_Linfty}
Fix $0<T<\infty$. Let $p,q\in (2,\infty)$ be such that 
\begin{equation}
\label{eq:almost_Linfty0}
\frac{2}{p}+\frac{d}{q}<1.
\end{equation}
Then for all $r\in (1,\infty)$ there exists $C_0>0$ such that for any $N$, any interval $I=(t_0-T,t_0)\subset (-1/4,\infty)$ with length $T<\infty$, and
\begin{equation}
\label{eq:almost_Linfty01}
F,G\in L^r_{\Progress}(\O; L^p(I;L^q(\T^d;\R^{d\times d}))),
\end{equation}
satisfying
\begin{equation}
\label{eq:almost_Linfty02}
\Tr (G)=0 \text{ a.e.\ on }\T^d \qquad \text{ and }\qquad \nabla \cdot G =0\text{ in }\D'(\T^d),
\end{equation} 
a.e.\ on $I\times \O$, it holds that 
\begin{align}
\label{eq:almost_Linfty1}
\sup_{ I\times \T^d}\big( \E|v^N|^r \big)^{1/r}
&\leq C_0  \big(\E\|F\|_{L^p( I;L^q(\T^d))}^r\big)^{1/r}+ C_0 \big(\E\|G\|_{L^p( I;L^q(\T^d))}^r\big)^{1/r},
\end{align}
where $v^N$ is the unique energy solution to \eqref{eq:turbulent_Stokes_scaling_uniform_estimates_modified}.
In particular, for all $r_0\in (1,\infty)$, there exists $C_1>0$ independent of $N,t_0,F$ and $G$ for which  
\begin{align}
\label{eq:almost_Linfty2}
\big( \E\|v^N\|_{L^{r_0}( I\times \T^d)}^{r} \big)^{1/{r}}
\leq C_1  \big(\E\|F\|_{L^p( I;L^q(\T^d))}^{r}\big)^{1/r}+ C_1 \big(\E\|G\|_{L^p( I;L^q(\T^d))}^{r}\big)^{1/{r}}.
\end{align}
\end{theorem}

Recall from the discussion below Theorem \ref{t:suboptimal_Lq_finite_time} that the estimate \eqref{eq:almost_Linfty2} is suboptimal from a scaling point of view compared to the regularity of the space \eqref{eq:maximal_regularity_space_associated_LpLq}.
Indeed, the norm of the left-hand side of \eqref{eq:almost_Linfty2} has Sobolev index $\Sob_0 =-\frac{2+d}{r_0}\approx 0$, while the one of the maximal regularity space \eqref{eq:maximal_regularity_space_associated_LpLq} is $\Sob=1-\frac{2}{p}-\frac{d}{q}>0$ by \eqref{eq:almost_Linfty0}. Thus, Theorem \ref{t:almost_Linfty} is almost optimal. 

\smallskip

We now proceed with the proof. One can easily see that Corollary \ref{cor:SMR_microscopic_estimate} and Theorem \ref{t:control_mesoscopic_energy} imply the pointwise bound 
$$
\sup_{ \T^d}(\E|v^N(t_0,\cdot)|^r)^{1/r}\leq C_0  \big(\E\|F\|_{L^p( I;L^q(\T^d))}^r\big)^{1/r}+ C_0 \big(\E\|G\|_{L^p( I;L^q(\T^d))}^r\big)^{1/r}.
$$
An additional argument is needed to allow the supremum over $I$.

\begin{proof}[Proof of Theorem \ref{t:almost_Linfty}]
For simplicity, we consider the case $I=(t_0-1/4,t_0)$ for some $t_0\geq 0$. The general case is analogous. 
We begin by proving \eqref{eq:almost_Linfty1}. To this end, let 
\begin{equation}
\label{eq:definition_Fhat_Ghat}
\wh{F}\stackrel{{\rm def}}{=}\one_{I} F \qquad \text{ and }\qquad 
\wh{G}\stackrel{{\rm def}}{=}\one_{I} G,
\end{equation}
where $\one_{I}$ denotes the zero extension outside $I$.
Clearly, 
$$
\wh{v}^N\stackrel{{\rm def}}{=}\one_{I}v^N
$$ 
is the unique energy solution to \eqref{eq:turbulent_Stokes_scaling_uniform_estimates_modified} on $ J\times \T^d $ where $J=(-1/4,t_0)$. 
Note that, as $\wh{v}^N$ is a solution to \eqref{eq:turbulent_Stokes_scaling_uniform_estimates_modified}, letting 
\begin{align}
\label{eq:definition_p_almost_Linfty_proof}
\wh{\pi}^N
&\stackrel{{\rm def}}{=}\qq \Big[ \nabla \cdot \wh{F}- c_d \mu\, \sum_{k,\alpha} \theta_k^N\nabla\cdot  (\nabla \wt{\Pi}^N_{k,\alpha}\otimes \sigma_{-k,\alpha})\Big],\\
\label{eq:definition_wtp_almost_Linfty_proof}
 \wt{\Pi}^N_{k,\alpha}
&\stackrel{{\rm def}}{=} \theta^{N}_{k}\qq\big[(\sigma_{k,\alpha}\cdot\nabla) \wh{v}^N+ \sigma_{k,\alpha}\cdot \wh{G}\,\big],
\end{align}
it follows from the definition of $\qq$ that (see Subsection \ref{ss:Helmh_div_free}), a.e.\ on $J\times \O$, 
\begin{align*}
\Delta \wh{\pi}^N
&=\nabla^2 : \wh{F}- c_d \mu\, \nabla \cdot \Big[\sum_{k,\alpha} \theta_k^N\nabla\cdot  (\nabla \wt{\Pi}^N_{k,\alpha}\otimes \sigma_{-k,\alpha})\Big],\\
\Delta \wt{\Pi}^N_{k,\alpha}
&=\theta^{N}_{k}\nabla \cdot \big[ (\sigma_{k,\alpha}\cdot\nabla) \wh{v}^N+ \sigma_{k,\alpha}\cdot \wh{G}\,\big],
\end{align*}
in $\D'(\T^d)$.
In particular, for all parabolic cylinders (see Figure \ref{fig:cubes_nested})
$$
Q_*=I_*\times B_*\stackrel{{\rm def}}{=}Q_{1/2}(t_*,x_*)
$$ 
with $(t_*,x_*)\in I \times \T^d$, the process $(\wh{v}^N,\wh{\pi}^N,(\wt{\Pi}^N_{k,\alpha})_{k,\alpha})$ is a local solution to the oscillating stochastic Stokes system \eqref{eq:turbulent_Stokes_scaling_original_iteration} in the sense of Definition \ref{def:local_solutions_NSE}.

\begin{figure}[htbp]
    \centering
    \begin{tikzpicture}[scale=1.2, >=stealth]

        \def\W{4.5}          
        \def\H{4.0}          
        \def\yline{2.0}      
        
        \def\ytoprect{2.7}   
        \def\ybotrect{0.7}   
        
        \def\xleftrect{1.3}  
        \def\xrightrect{3.1} 

        \draw[thick] (0,0) node[left] {\small $-1$} -- (\W,0);
        \draw[thick] (0,\H) -- (\W,\H);
        
        \draw[dashed] (0,0) -- (0,\H);
        \draw[dashed] (\W,0) -- (\W,\H);
        
        \node[left] at (-0.2, \H) {\small $t_0$};
        
        \draw[thick] (0,\yline) node[left] {\small $t_0-1/4$} -- (\W,\yline);
        
        \draw[thick, pattern=north east lines] (\xleftrect,\ybotrect) rectangle (\xrightrect,\ytoprect);
        
        \node at (\xleftrect - 0.3, 1.2) {\small $Q_*$};
        
        \fill[blue] (2.2, \ytoprect) circle (2pt) node[above=2pt, black] {\small $(t_*,x_*)$};
        
        \draw[<->, thick] (\xrightrect + 0.3, \ybotrect) -- (\xrightrect + 0.3, \ytoprect) node[pos=0.25, right] {\small $I_*$};

        \draw[<->, thick] (\W - 0.6, \yline) -- (\W - 0.6, \H) node[midway, left] {\small $I$};
        
        \draw[<->, thick] (\W + 0.5, 0) -- (\W + 0.5, \H) node[midway, right] {\small $J=(-1/4,t_0)$};
        
        \draw[<->, thick] (\xleftrect, 0.5) -- (\xrightrect, 0.5) node[midway, below] {\small $\mathbb{T}^d$};

    \end{tikzpicture}
    \caption{Construction in the proof of the estimate \eqref{eq:almost_Linfty1} where $I=(t_0-1/4,t_0)$.}
    \label{fig:cubes_nested}
\end{figure}
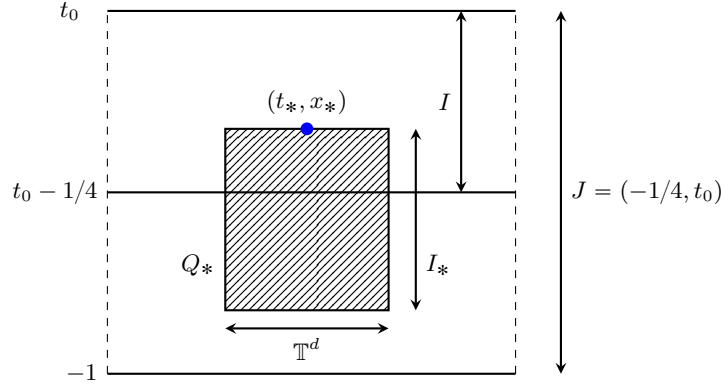

Applying Corollary \ref{cor:SMR_microscopic_estimate} to the local solution 
$$
\Big( \wh{v}^N ,\, \wh{\pi}^N -\fint_{N^{-1} B_*} \wh{\pi}^N , \,\Big(\wt{\Pi}_{k,\alpha}^N -\fint_{N^{-1} B_*} \wt{\Pi}_{k,\alpha}^N\Big)_{k,\alpha} \Big)
$$
on $N^{-1}Q_*\subseteq Q_*$ yields that 
\begin{align*}
\big(\E|\wh{v}^N(t_*,x_*)|^r\big)^{1/r}
&\leq C_0
 \big(\E\| \wh{v}^N\|_{\underline{L}^2(N^{-1}Q_*)}^r\big)^{1/r}\\
 \nonumber
 &+C_0  \Big(\E\Big\|\wh{\pi}^N -\fint_{N^{-1} B_* } \wh{\pi}^N \Big\|_{\underline{L}^2(N^{-1}I_*;\underline{H}^{-1}(N^{-1}B_* ))}^r\Big)^{1/r} \\
\nonumber 
 &+C_0\Big(\E\Big\|\Big(\wt{\Pi}^N_{k,\alpha}-\fint_{N^{-1} B_* } \wt{\Pi}_{k,\alpha}^N\Big)_{k,\alpha}\Big\|_{\underline{L}^2(N^{-1}Q_* ;\ell^2)}^r \Big)^{1/r}\\
 \nonumber
& +C_0 \big(\E\|F\|_{L^p(I;L^q(\T^d))}^r\big)^{1/r}+C_0 \big(\E\|G\|^r_{L^p(I;L^q(\T^d))}\big)^{1/r},
\end{align*}
where $C_0$ is independent of $N, t_*,x_*,F$ and $G$. 
Combining the above with the control of the energy at the microscopic scale in Theorem \ref{t:control_mesoscopic_energy}, we have  
\begin{align}
\nonumber
\big(\E|\wh{v}^N(t_*,x_*)|^r\big)^{1/r}
&\leq C_1
\Big(\E\sup_{I_* } \|\wh{v}^{N}\|_{L^2(B_*)}^{r}\Big)^{1/r}\\
\nonumber
&+ C_1\big( \E \|\nabla \wh{\pi}^N\|_{L^2(I_*;H^{-1}(B_* ))}^r\big)^{1/r}\\
\nonumber
&+C_1\big(\E \|( \nabla \wt{\Pi}_{k,\alpha}^N )_{k,\alpha}\|_{L^{p}(I_* ; H^{-1}(B_* ;\ell^2))}^r \big)^{1/r}\\
\label{eq:proof_almost_Linfty_bound}
& +C_1 \big(\E\|F\|_{L^p( I;L^q(\T^d))}^r\big)^{1/r}+C_0 \big(\E\|G\|^r_{L^p( I;L^q(\T^d))}\big)^{1/r},
\end{align}
where $C_1$ is independent of $N, t_*,x_*,F$ and $G$. 

The key point in the estimate \eqref{eq:proof_almost_Linfty_bound} is that all the terms on the right-hand side of \eqref{eq:proof_almost_Linfty_bound} are evaluated on the large scale $Q_*$ (of order $1$) and have the regularity content of an energy, see \eqref{eq:energy_estimate_vN}. Thus, at this point, these terms can be estimated directly by using \eqref{eq:energy_estimate_vN}.  Indeed, from the latter, there exists $C_2>0$ independent of $N$, $F$ and $G$ for which  
\begin{align}
\label{eq:energy_estimate_almost_Linfty_proof}
\Big(\E\sup_{J} \|\wh{v}^{N}\|_{L^2(\T^d)}^{r}\Big)^{1/r}
&+
\big(\E \|\nabla \wh{v}^{N}\|_{L^2(J \times \T^d)}^{r}\big)^{1/r}\\
\nonumber
&\leq 
 C_2 \big(\E \|F\|_{L^2(I\times \T^d)}^r \big)^{1/r}
+C_2 \big(\E \|G\|_{L^2(I\times \T^d)}^r \big)^{1/r},
\end{align}
due to \eqref{eq:definition_Fhat_Ghat}.
From \eqref{eq:definition_wtp_almost_Linfty_proof}, the fact that $\nabla \qq:H^{-1}(\T^d;\R^d)\to H^{-1}(\T^d;\R^d)$ is bounded, and the identity $(\sigma_{k,\alpha}\cdot \nabla) \wh{v}^N = \nabla\cdot ( \wh{v}^N\otimes \sigma_{k,\alpha})$, it follows that 
\begin{align}
\nonumber
\big(\E \|( \nabla \wt{\Pi}_{k,\alpha}^N )_{k,\alpha}\|_{L^{p}(I_*; H^{-1}(B_*;\ell^2))}^r \big)^{1/r}
&\lesssim \big(\E \|(\theta^N_k [ \wh{v}^N\otimes \sigma_{k,\alpha}] )_{k,\alpha}\|_{L^{\infty}(I_*;L^2(\T^d;\ell^2))}^r \big)^{1/r}\\
\nonumber
&+\big(\E \|(\theta^N_k \sigma_{k,\alpha} \cdot \wh{G})_{k,\alpha}\|_{L^{p}(I_*;L^2( \T^d;\ell^2))}^r \big)^{1/r}\\
\nonumber
&  
\stackrel{(i)}{\lesssim} \Big(\E\sup_{I_*} \| \wh{v}^N \|_{L^2(\T^d)}^r \Big)^{1/r}
+\big(\E\|G\|^r_{L^p(I;L^q(\T^d))}\big)^{1/r}\\
&
\stackrel{(ii)}{\lesssim} \big(\E\|F\|_{L^p(I;L^q(\T^d))}^r\big)^{1/r}+ \big(\E\|G\|^r_{L^p(I;L^q(\T^d))}\big)^{1/r}
\label{eq:estimate_pressure_almost_Linfty_proof1}
\end{align}
where in $(i)$ we used $\|\theta^N\|_{\ell^2}=1$, and in $(ii)$ the energy estimate \eqref{eq:energy_estimate_almost_Linfty_proof}. Similarly, from the boundedness of $\nabla \qq$ on $L^2(\T^d;\R^d)$ and \eqref{eq:energy_estimate_almost_Linfty_proof}, 
\begin{equation*}
\big(\E \|( \nabla \wt{\Pi}_{k,\alpha}^N )_{k,\alpha}\|_{L^{2}(I_*\times \T^d;\ell^2)}^r \big)^{1/r}
\lesssim \big(\E\|F\|_{L^p(I;L^q(\T^d))}^r\big)^{1/r}+ \big(\E\|G\|^r_{L^p(I;L^q(\T^d))}\big)^{1/r}.
\end{equation*}
Hence, from \eqref{eq:definition_p_almost_Linfty_proof} and the above, we infer that
\begin{align}
\label{eq:estimate_pressure_almost_Linfty_proof3}
&\big( \E \|\nabla \wh{\pi}^N\|_{L^2(I_*;H^{-1}(\T^d))}^r\big)^{1/r}\\
\nonumber
&\lesssim  \big(\E\|F\|_{L^p(I;L^q(\T^d))}^r\big)^{1/r}
+ \Big(\E\Big\|\sum_{k,\alpha}\theta^N_k \, \nabla \wt{\Pi}^N_{k,\alpha}\otimes \sigma_{-k,\alpha} \Big\|_{L^2(I_*\times \T^d)}^r\Big)^{1/r}\\
\nonumber
&\lesssim  \big(\E\|F\|_{L^p(I;L^q(\T^d))}^r\big)^{1/r}+
\big(\E \|( \nabla \wt{\Pi}_{k,\alpha}^N )_{k,\alpha}\|_{L^{2}(I_* \times \T^d;\ell^2)}^r \big)^{1/r}\\
\nonumber
&\lesssim \big(\E\|F\|_{L^p(I;L^q(\T^d))}^r\big)^{1/r}+ \big(\E\|G\|^r_{L^p(I;L^q(\T^d))}\big)^{1/r}.
\end{align} 
Since $B_*\subseteq \T^d$, the estimate \eqref{eq:almost_Linfty1} follows by using in \eqref{eq:proof_almost_Linfty_bound} the bounds \eqref{eq:energy_estimate_almost_Linfty_proof}, \eqref{eq:estimate_pressure_almost_Linfty_proof1}, \eqref{eq:estimate_pressure_almost_Linfty_proof3}, and the uniformity of the constants with respect to $(t_*,x_*)\in I\times \T^d$.

\smallskip

Finally, we prove \eqref{eq:almost_Linfty2}. First, by H\"older's inequality, it suffices to show the claim for $r_0\geq r$. In the latter situation, we employ Lenglart's domination, see e.g., \cite{SharpLeng,Lenglart}. In particular, let $\tau:\O\to [t_0-1/4,t_0]$ be a stopping time, and let $v^N$ be the unique energy solution to \eqref{eq:turbulent_Stokes_scaling_uniform_estimates_modified}. Consider $V^N$, the unique energy solution to \eqref{eq:turbulent_Stokes_scaling_uniform_estimates_modified} on $I\times \T^d$ with $(F,G)$ replaced by $(\one_{[t_0-1/4,\tau]}F,\one_{[t_0-1/4,\tau]}G)$. From the uniqueness of energy solutions, it follows that $v^N=V^N$ on $(t_0-1/4,\tau)\times \O$ and therefore, 
\begin{align}
\label{eq:Lengart_domination_application}
&\E\|v^N\|_{L^{r_0}((t_0-1/4,\tau)\times \T^d)}^{r_0}
\leq \E\|V^N\|_{L^{r_0}(I \times \T^d)}^{r_0}\\
&\nonumber
 \lesssim_T \sup_{I \times \T^d} \E |V^N|^{r_0}
\lesssim \E\big(\|F\|_{L^p( (t_0-1/4,\tau);L^q(\T^d))}^{r_0}+\|G\|_{L^p((t_0-1/4,\tau);L^q(\T^d))}^{r_0}\big),
\end{align}
where we used Fubini's theorem and, in the last step, the just-proved estimate \eqref{eq:almost_Linfty1}.

As the constant in \eqref{eq:Lengart_domination_application} is independent of the stopping time $\tau$, a standard application of Lenglart's domination \cite[Theorem 1.1]{SharpLeng} yields the claim for all $ r\in (1,r_0]$.
\end{proof}

Interpolating Theorem \ref{t:almost_Linfty} with the energy inequality \eqref{eq:energy_estimate_vN}, we obtain the following version of Theorem \ref{t:suboptimal_Lq_Rpositive} for energy solutions to \eqref{eq:turbulent_Stokes_scaling_uniform_estimates_modified}.

\begin{proposition}[$L^p(L^q)$-estimates uniform in the oscillation -- Modified Stokes system]
\label{prop:almost_Linfty}
Let $I\subset (-1/4,\infty)$ be an interval of length $T<\infty$.
Let $p,q\in (2,\infty)$ be such that the Sobolev index $\Sob=1-\frac{2}{p}-\frac{d}{q}$ satisfies  \eqref{eq:assumption_suboptimal_Lq_finite_time1}, and set $\qmax\stackrel{{\rm def}}{=}-\frac{d^2}{d-2}\frac{1}{\Sob}$, 
with the convention $\qmax=\infty$ if $d=2$.
Then, for all $p_0\in (2,\infty)$ and $q_0\in (2,\qmax)$ for which \eqref{eq:suboptimal_Lq_finite_time_suboptimality_condition} holds, and $r\in (1,\infty)$, there exists a constant $C_0>0$ such that, for all $N$, and
progressively measurable processes $F,G$
satisfying \eqref{eq:almost_Linfty01}-\eqref{eq:almost_Linfty02}, the unique energy solution $v^N$ to \eqref{eq:turbulent_Stokes_scaling_uniform_estimates_modified} satisfies
\begin{align}
\label{eq:almost_Linfty_prop1}
\big(\E\|v^N\|_{L^{p_0}(I;L^{q_0} (\T^d))}^{r} \big)^{1/r}
&\leq 
 C_0  \big(\E\|F\|_{L^p(I;L^q(\T^d))}^r\big)^{1/r}\\
 \nonumber
& +C_0 \big(\E\|G\|_{L^p(I;L^q(\T^d))}^r\big)^{1/r}.
\end{align}
\end{proposition}

Note that $C_0$ in \eqref{eq:almost_Linfty_prop1} depends on the interval $I$ only through its length $T$.

\begin{proof}
We begin by noticing that the case $\Sob=-\frac{d}{2}$ follows trivially from \eqref{eq:energy_estimate_vN}. Hence, below we assume $\Sob>-\frac{d}{2}$. 
Let $[\cdot,\cdot]_\vartheta$ be the complex interpolation functor, see e.g., \cite{BeLo,InterpolationLunardi} and \cite[Appendix C]{Analysis1}. We recall that, for any couple of compatible Banach spaces $(X_0,X_1)$ (i.e., there exists a Hausdorff topological vector space $V$ such that $X_0\embed V$ and $X_1\embed V$), $r,q_0,q_1\in (1,\infty)$ and $\vartheta\in (0,1)$, it holds that  
\begin{equation}
\label{eq:interpolation_Banach_spaces_progress}
[L^r_{\Progress}(\O;L^{q_0}(I;X_0)),L^r_{\Progress}(\O;L^{q_1}(I;X_1))]_{\vartheta}=
L^r_{\Progress}(\O;L^{q}(I;[X_0,X_1]_\vartheta)),
\end{equation}
where $\frac{1}{q}=\frac{1-\vartheta}{q_0}+\frac{\vartheta}{q_1}$.
It is worth noticing that the identity \eqref{eq:interpolation_Banach_spaces_progress} does not follow from standard interpolation of Bochner spaces (see e.g., \cite[Theorem 2.2.6]{Analysis1}) due to the progressive measurability constraint which mixes the outer and the inner measure spaces in $L^r(\O;L^{q_i}(I;X_i))$.
To show \eqref{eq:interpolation_Banach_spaces_progress}, we employ the retraction-coretraction method, see e.g., either \cite[Subsection I.2.3]{Am} or \cite[Subsection 1.2(b)]{DK13_mixed_order}. To this end, let $\E_\Progress \stackrel{{\rm def}}{=}\E[\cdot| \Progress]$ be the conditional expectation on the progressive $\sigma$-algebra on $I$. 
By \cite[Corollary 2.9]{LvN19}, it follows that 
$$
\E_{\Progress}: L^{r}(\O;L^{q_i}(I;X_i))\to L^{r}_{\Progress}(\O;L^{q_i}(I;X_i))\  \text{ boundedly}, 
$$
and 
$$
\E_{\Progress}(L^{r}(\O;L^{q_i}(I;X_i)))=L^{r}_{\Progress}(\O;L^{q_i}(I;X_i)) \ \ \text{ for all } \ i\in \{0,1\},
$$
cf.\ \cite[Definition 1.48]{DK13_mixed_order}. 
Moreover, for all $i\in \{0,1\}$, $\E_\Progress$ is a retraction (see \cite[Definition 1.47]{DK13_mixed_order}) with coretraction the identity map $e(f)=f$. 
Hence, \eqref{eq:interpolation_Banach_spaces_progress} is a consequence of the above observations, \cite[Lemma 1.51]{DK13_mixed_order} and \cite[Theorem 2.2.6]{Analysis1}.

The proof of Proposition \ref{prop:almost_Linfty} now follows by interpolating the solution operator 
\begin{equation*}
(F,G)\mapsto \LinS^N(F, G)\stackrel{{\rm def}}{=} v^N
\end{equation*} 
associated to \eqref{eq:turbulent_Stokes_scaling_uniform_estimates_modified} where the endpoints are given by the bounds in \eqref{eq:energy_estimate_vN} and the result of Theorem \ref{t:almost_Linfty}. 
However, to ensure the restrictions \eqref{eq:almost_Linfty02} on $G$, we compose the solution operator with the following matrix-valued extension of the Helmholtz projection (see Subsection \ref{ss:Helmh_div_free})
\begin{align*}
\wh{\mathbb{G}_{\T^d}(G)}(k)\stackrel{{\rm def}}{=} \wh{G}(k)P_k-\frac{\Tr(\wh{G}(k)P_k)}{d-1}P_k \ \ \text{ and }
\ \ \wh{\mathbb{G}_{\T^d}(G)}(0)\stackrel{{\rm def}}{=}\wh{G}(0)-\frac{\Tr(\wh{G}(0))}{d} \mathrm{Id},
\end{align*}
where $k\in \Z^d_0$, $P_k=\mathrm{Id}-(k\otimes k)/|k|^2$, $G\in \D'(\T^d;\R^{d\times d})$  and $\widehat{G}(k) =(\langle e_k ,G_{i,j}\rangle)_{i,j=1}^d $. 
From standard Fourier multiplier techniques, $\mathbb G_{\T^d}$ is bounded on $L^{\wh{q}}(\T^d;\R^{d\times d})$ for every $\wh{q}\in(1,\infty)$. Moreover, one can check that $\mathbb G_{\T^d}$ is a projection onto the space of
trace-free and divergence-free matrix-valued fields. In particular,
\begin{equation}
\label{eq:projection_G_equality_proof_interpolation}
\mathbb{G}_{\T^d}(G)=G  \text{ whenever $G$ satisfies \eqref{eq:almost_Linfty02}}.
\end{equation}
Therefore, in the following, we interpolate the following operator
\begin{equation}
\label{eq:solution_operator_UN}
(F,G)\mapsto \Lin^N(F,G) \stackrel{{\rm def}}{=}\LinS^N(F,\mathbb{G}_{\T^d}(G)).
\end{equation}
To this end, recall that $\Sob=1-\frac{2}{p}-\frac{d}{q}\in (-\frac{d}{2},0)$. Let 
$
\Sob_1=-\frac{d}{2}
$
and fix 
$
\Sob_2>0
$.
Clearly, 
\begin{equation}
\label{eq:vartheta_choice_interpolation}
\Sob=(1-\vartheta)\Sob_1 + \vartheta \Sob_2\ \  \text{ where }\ \  
\vartheta=\frac{\Sob-\Sob_1}{\Sob_2-\Sob_1}<1+\frac{2}{d}\,\Sob\in (0,1).
\end{equation}
Moreover, $\vartheta\to 1+\frac{2}{d}\,\Sob$ as $\Sob_2\to 0$.  
Let $p_2,q_2 \in (2,\infty)$ be defined as 
\begin{equation}
\label{eq:choice_p1q1_interpolation_finite_interval}
\frac{1-\vartheta}{2}+\frac{\vartheta}{p_2}=\frac{1}{p} \qquad \text{ and }\qquad 
\frac{1-\vartheta}{2}+\frac{\vartheta}{q_2}=\frac{1}{q}.
\end{equation}
Note that the existence of $p_2,q_2\in (2,\infty)$ for which the above conditions hold is equivalent to $p,q<\frac{d}{-\Sob}$ provided $\Sob_2$ is sufficiently small depending only $p,q$ and $d$. Note that $q<\frac{d}{-\Sob}$ follows automatically from $p>2$.
From the choice of the parameters, it follows that
$$
\frac{2}{p_2}+\frac{d}{q_2}=1-\Sob_2<1.
$$
In particular, we can apply Theorem \ref{t:almost_Linfty} with $(p,q)$ replaced by $(p_2,q_2)$. In particular, the solution operator $\Lin^N$ defined in \eqref{eq:solution_operator_UN} defines a bounded linear operator:
\begin{equation}
\label{eq:interpolation_mapping_almost_Linfty}
[L^{r}_{\Progress}(\O;L^{p_2}(I;L^{q_2}(\T^d;\R^{d\times d})))]^2\to L^{r}_{\Progress}(\O;L^{r_0}(I\times \T^d;\R^d)) \text{ for all }r_0>r.
\end{equation}
Moreover, \eqref{eq:energy_estimate_vN} implies that $\Lin^N$ has the following mapping property: 
\begin{align}
\label{eq:interpolation_mapping_energy}
[L^{r}_{\Progress}(\O;L^{2}(I \times \T^d;\R^{d\times d}))]^2
&\to L^{r} (\O;L^\infty(I;L^{2}(\T^d;\R^d)))\\
\nonumber
&\qquad \cap  L^{r} (\O;L^2(I;H^{1}(\T^d;\R^d))).
\end{align}
Now, interpolating \eqref{eq:interpolation_mapping_almost_Linfty} with the first output space in \eqref{eq:interpolation_mapping_energy}, and using \eqref{eq:interpolation_Banach_spaces_progress} and \eqref{eq:choice_p1q1_interpolation_finite_interval}, we obtain 
\begin{equation}
\label{eq:endpoint_1_UN_interpolation}
\Lin^N :
[L^{r}_{\Progress}(\O;L^{p}(I;L^{q}(\T^d;\R^{d\times d})))]^2
\to 
L^r(\O;L^{\overline{p}_{1}}(I;L^{\overline{q}_{1}}(\T^d;\R^d))),
\end{equation}
with
$$
\frac{\vartheta}{r_0}=\frac{1}{\overline{p}_{1}}\qquad \text{ and }\qquad 
\frac{1-\vartheta}{2}+\frac{\vartheta}{r_0}=\frac{1}{\overline{q}_{1}},
$$
and where $r_0\in (r,\infty)$ and $\vartheta<1+ \frac{2}{d}\,\Sob$ are arbitrary, because $\Sob_2>0$ in the present case. Thus, 
in the range space of $\Lin^N$ in \eqref{eq:endpoint_1_UN_interpolation}, we can choose any 
\begin{equation*}
\overline{p}_{1}<\pmax\stackrel{{\rm def}}{=}\infty
\qquad \text{ and }\qquad \overline{q}_{1}<\qmin\stackrel{{\rm def}}{=}-\frac{d}{\Sob}.
\end{equation*}
Note that $-2/\pmax-d/\qmin=\Sob$, and thus the result  \eqref{eq:endpoint_1_UN_interpolation} is almost optimal.

\smallskip

Next, interpolating \eqref{eq:interpolation_mapping_almost_Linfty} with the second output space in \eqref{eq:interpolation_mapping_energy} and Sobolev embeddings (recall $d\geq 2$), we have 
\begin{align}
\label{eq:endpoint_2_UN_interpolation}
\Lin^N :
[L^{r}_{\Progress}(\O;L^{p}(I;L^{q}(\T^d;\R^{d\times d})))]^2
&\to 
L^r(\O;L^{\overline{p}_{2}}(I;H^{1-\vartheta,\overline{p}_{2}}(\T^d)))\\
\nonumber
&\embed 
L^r(\O;L^{\overline{p}_{2}}(I;L^{\overline{q}_2}(\T^d;\R^d)))
\end{align}
where 
$$
\frac{1-\vartheta}{2}+\frac{\vartheta}{r_0}=\frac{1}{\overline{p}_{2}}\qquad \text{ and }\qquad 
1-\vartheta-\frac{d}{\overline{p}_{2}}=-\frac{d}{\overline{q}_{2}}.
$$
As above, from the arbitrariness of $r_0\in (r,\infty)$ and $\Sob_2>0$, it follows in \eqref{eq:endpoint_2_UN_interpolation} that in the range space of $\Lin^N$, we can choose any 
\begin{equation*}
\overline{p}_{2}<\pmin\stackrel{{\rm def}}{=}-\frac{d}{\Sob}\qquad \text{ and }\qquad 
\overline{q}_{2}<\qmax=-\frac{d^2}{(d-2)\Sob}.
\end{equation*}
Note that $-2/\pmin-d/\qmax=\Sob$, and as for \eqref{eq:endpoint_1_UN_interpolation}, the statement \eqref{eq:endpoint_2_UN_interpolation} with the above choice of the parameters is almost optimal.

\smallskip

Now, one can check that if 
\begin{equation}
\label{eq:p0q0choice_paramter_proof_interpolation_last_step}
p_0\in (\pmin,\pmax) \qquad \text{ and }\qquad q_0\in (\qmin,\qmax)
\end{equation}
satisfy the condition $-\frac{2}{p_0}-\frac{d}{q_0}<\Sob$ as assumed in \eqref{eq:suboptimal_Lq_finite_time_suboptimality_condition}, the claimed estimate in Proposition \ref{prop:almost_Linfty} follows by interpolating once more the (almost) endpoints \eqref{eq:endpoint_1_UN_interpolation} and \eqref{eq:endpoint_2_UN_interpolation}. Moreover, in the cases where $p_0$ and $q_0$ satisfy  $-\frac{2}{p_0}-\frac{d}{q_0}<\Sob$ and either 
$$
\big[ p_0\in (2,\pmin] \ \  \text{ and }\ \  q_0\in (\qmin,\qmax)\big]
\ \ \  \text{ or }\ \  \
\big[ p_0\in (\pmin,\pmax) \ \  \text{ and }\ \  q_0\in (2,\qmin]\big],
$$
then it follows from the case \eqref{eq:p0q0choice_paramter_proof_interpolation_last_step} and H\"older's inequality.
Finally, the case $p_0\leq \pmin$ and $q_0\leq \qmin$ follows trivially from the above cases.

The estimate \eqref{eq:almost_Linfty_prop1} follows by the above interpolation argument and \eqref{eq:projection_G_equality_proof_interpolation}.
\end{proof}

We finally prove Theorem \ref{t:suboptimal_Lq_finite_time}.

\begin{proof}[Proof of Theorem \ref{t:suboptimal_Lq_finite_time}]
Let $w^N$ be the unique energy solution to \eqref{eq:turbulent_Stokes_scaling_uniform_estimates}. 
Let 
$$
z(t,x)\stackrel{{\rm def}}{=}[e^{(1+\mu)\Delta t} w_0](x) \ \ \text{ for }t\in I.
$$
Note that
\begin{equation}
\label{eq:G_smoothness1}
\|z\|_{L^\infty(I;B^{1-2/p}_{q,p}(\T^d))\cap L^p(I;H^{1,q}(\T^d))}
\lesssim_{p,q} \|w_0\|_{B^{1-2/p}_{q,p}(\T^d)}.
\end{equation}
Moreover, it is easy to see that $v^N = w^N- z$ solves \eqref{eq:turbulent_Stokes_scaling_uniform_estimates_modified} with 
\begin{equation}
\label{eq:reduction_to_initial_data_zero}
G(t,x)= [\nabla z(t,x)]^\top \ \text{ for }(t,x)\in I\times \T^d.
\end{equation}
Moreover, as $\nabla \cdot w_0 =0$ in $\D'(\T^d)$ by assumption, it holds that 
$
\nabla \cdot z=0$ on $ I\times\T^d$. From the latter, it readily follows that 
$$
\Tr(G)=0 \text{ a.e.\ on }\T^d, \quad \text{ and }\quad \nabla \cdot G =0 \text{ in }\D'(\T^d).
$$
Hence, the assumptions of Proposition \ref{prop:almost_Linfty} are satisfied, and we have the bound \eqref{eq:almost_Linfty_prop1} for $v^N$. It remains to show that \eqref{eq:G_smoothness1} implies an analogous bound for $z$. To see this, note that, by interpolation,
\begin{align}
\label{eq:G_smoothness2}
L^\infty(I;B^{1-2/p}_{q,p}(\T^d))\cap L^p(I;H^{1,q}(\T^d)) 
\embed 
L^{p_1}(I;H^{s_\vartheta,q}(\T^d)),
\end{align}
where $\vartheta\in (0,1)$, $s_\vartheta=(1-\vartheta)(1-\frac{2}{p})+\vartheta$ and $
\frac{1}{p_1}=\frac{\vartheta}{p}$.
Clearly, the Sobolev index of $L^{p_1}(I;H^{s_\vartheta,q}(\T^d))$ is given by  $-\frac{2}{p_1}+s_\vartheta-\frac{d}{q}=1-\frac{2}{p}-\frac{d}{q}=\Sob$. Thus, it is routine to check that, by Sobolev embeddings, with an appropriate choice of $\vartheta\in (0,1)$, the space on the left-hand side of \eqref{eq:G_smoothness2} embeds in $L^{p_0}(I;L^{q_0}(\T^d))$ provided \eqref{eq:suboptimal_Lq_finite_time_suboptimality_condition} holds.
\end{proof}

\subsection{Uniform $L^p(L^q)$-estimates on the half-line -- Proof of Theorem \ref{t:suboptimal_Lq_Rpositive}}
\label{ss:suboptimal_Lq_Rpositive}
The proof of Theorem \ref{t:suboptimal_Lq_Rpositive} follows from its finite-time version in Theorem \ref{t:suboptimal_Lq_finite_time}  and a standard localization strategy; see, for instance, \cite{Kry_book}.

\begin{proof}[Proof of Theorem \ref{t:suboptimal_Lq_Rpositive}]
Arguing as in the proof of Theorem \ref{t:suboptimal_Lq_finite_time}, it suffices to establish an estimate for the energy solution to \eqref{eq:turbulent_Stokes_scaling_uniform_estimates_modified} with $G$ as in \eqref{eq:reduction_to_initial_data_zero}. Indeed, recall that $\int_{\T^d} w_0=0$ and $w_0\in B^{1-2/p}_{q,p}(\T^d)\embed L^q(\T^d)$ a.s.\ by assumption (as $p>2$). 
The invertibility of the Laplacian on mean-zero functions in $L^m(\T^d)$ (where $1 < m < \infty$) ensures that the bounds and embeddings from \eqref{eq:G_smoothness1} and \eqref{eq:G_smoothness2} remain valid even over the extended interval $I=\R_+$.

\smallskip

Let $v^N$ be the unique energy solution to \eqref{eq:turbulent_Stokes_scaling_uniform_estimates_modified} with $G$ as in \eqref{eq:reduction_to_initial_data_zero} and $t_0=0$. Since    
$
\int_{\T^d} \sigma_{k,\alpha}\cdot G=\int_{\T^d} (\sigma_{k,\alpha}\cdot\nabla) z = 0
$
as $\nabla \cdot \sigma_{k,\alpha}=0$,
we have
\begin{equation}
\label{eq:vN_has_zero_mean_proof}
\int_{\T^d}v^N=0\ \text{ a.e.\ on } \ \R_+\times \O.
\end{equation} 
Next, let $p_0,q_0$ be as in Theorem \ref{t:suboptimal_Lq_Rpositive}. From Lenglart's domination, it suffices to prove that there exists a constant $C_0>0$ such that, for all $N\geq 1$, $F$ and $w_0$ as above, and any stopping time $\tau:\O\to [0,\infty]$, it holds that  
\begin{align}
\nonumber
\big(\E\|v^N\|_{L^{p_0}(0,\tau;L^{q_0} (\T^d))}^{p_0} \big)^{1/p_0}
&
\leq C_0  \big(\E\|F\|_{L^p(0,\tau;L^q(\T^d))}^{p_0}\big)^{1/{p_0}}+ C_0  \big(\E\|G\|_{L^p(0,\tau;L^q(\T^d))}^{p_0}\big)^{1/{p_0}}\\
 &+ C_0\big(\E\|v^N\|_{L^{p}(0,\tau;H^{-1,q} (\T^d))}^{p_0} \big)^{1/{p_0}}. 
 \label{eq:estimate_LqLq_claim_halfline}
\end{align}

In the following, with a slight abuse of notation, as in \eqref{eq:definition_Fhat_Ghat}, we extend by zero the processes $v^N$, $F$, and $G$ on $(-1/4,0)$. 
Let 
$
\zeta\in C^{\infty}_{{\rm c}}(\R) 
$
be such that $\zeta\geq 0$, $\supp\zeta \subseteq (-1/4,1/4)$ and $\int_{\R} \zeta^{p_0}=1$. For all $t_0\in \R$, the process
$$
v^N_{t_0}(t)\stackrel{{\rm def}}{=}\zeta(t-t_0)v^N(t),
$$
solves \eqref{eq:turbulent_Stokes_scaling_uniform_estimates_modified} on $(-1/4,\infty)$ with
$$
(F,G)\quad \text{ replaced by }
\quad
(\partial_t \zeta(\cdot-t_0) [\Delta^{-1}\nabla v^N]+ \zeta(\cdot-t_0)F , \zeta(\cdot-t_0)G ).
$$
Note that in the above, we used that $v^N = \nabla \cdot [\Delta^{-1}\nabla v^N]$ as $v^N$ has zero mean, see \eqref{eq:vN_has_zero_mean_proof}. Similarly to \eqref{eq:Lengart_domination_application}, by uniqueness of energy solutions to \eqref{eq:turbulent_Stokes_scaling_uniform_estimates_modified}, it follows from Proposition \ref{prop:almost_Linfty} that there exists a constant $C_1>0$ independent of $t_0,N,F$ and $G$ such that 
\begin{align}
\nonumber
\E\int_{t_0-1/4}^{t_0+1/4} \one_{[0,\tau)} 
\| v^N_{t_0}\|_{L^{q_0}(\T^d)}^{p_0}
\leq C_1  \E\Big[ \Big(\int_{t_0-1/4}^{t_0+1/4}\one_{[0,\tau)}|\partial_t \zeta(\cdot-t_0)|^p \|v^N\|_{H^{-1,q}(\T^d)}^p \Big)^{p_0/p}\Big]&\\
\label{eq:localized_LpLq_estimate_on_t_0}
+ C_1  \E \Big[\Big(\int_{t_0-1/4}^{t_0+1/4}\one_{[0,\tau)} (\zeta(\cdot-t_0))^p\big(\|F\|_{L^q(\T^d)}^p + \|G\|_{L^q(\T^d)}^p\big) \Big)^{p_0/p}\Big]&.
\end{align}
To prove \eqref{eq:estimate_LqLq_claim_halfline}, it remains to take the integral over $t_0\in (-1/4,\infty)$ on both sides of \eqref{eq:localized_LpLq_estimate_on_t_0}. Indeed,  as $\supp\zeta \subset(-1/4,1/4)$ and $\int_{\R} \zeta^{p_0}=1$,
\begin{align*}
&\E\int_{-1/4}^\infty \int_{t_0-1/4}^{t_0+1/4} \one_{[0,\tau)} 
\| v^N_{t_0}\|_{L^{q_0}(\T^d)}^{p_0}\,\dd t \,\dd t_0\\
&=\E\int_{-1/4}^\infty \int_{\R_+} (\zeta(t-t_0))^{p_0} \one_{[0,\tau)} 
\| v^N\|_{L^{q_0}(\T^d)}^{p_0}\,\dd t \,\dd t_0= 
\E \int_{0}^{\tau} 
\| v^N\|_{L^{q_0}(\T^d)}^{p_0}.
\end{align*}
Moreover, by Young's inequality as well as the assumption $p_0\geq p$, it holds that 
\begin{align*}	
&\E\Big[\int_{-1/4}^\infty \Big(\int_{t_0-1/4}^{t_0+1/4}\one_{[0,\tau)}|\partial_t \zeta(\cdot-t_0)|^p \|v^N(t)\|_{H^{-1,q}(\T^d)}^p \,\dd t \Big)^{p_0/p} \,\dd t_0\Big]\\
&\qquad  \leq \big\||\partial_{t}\zeta|^p\big\|_{L^{p_0/p}(\R)}^{p_0/p}
\E\big[ \big\| \one_{[0,\tau)} \|v^N\|_{H^{-1,q}(\T^d)}^p\big\|_{L^1(\R)}^{p_0/p} \big]\\
&\qquad \lesssim
\E\Big[ \Big(  \int_0^\tau  \|v^N\|_{H^{-1,q}(\T^d)}^p\Big)^{p_0/p} \Big],
\end{align*}
and similarly,
\begin{align*}
&\E\int_{-1/4}^\infty \Big[\Big(\int_{t_0-1/4}^{t_0+1/4}\one_{[0,\tau)} (\zeta(t-t_0))^p\big(\|F(t)\|_{L^q(\T^d)}^p + \|G(t)\|_{L^q(\T^d)}^p \big)\,\dd t \Big)^{p_0/p}\,\dd t_0\Big]\\
&\lesssim 
\E\Big[ \Big(  \int_0^\tau  \|F\|_{L^q(\T^d)}^p\Big)^{p_0/p} \Big]
+
\E\Big[ \Big(  \int_0^\tau  \|G\|_{L^q(\T^d)}^p\Big)^{p_0/p} \Big].
\end{align*}
Therefore, \eqref{eq:estimate_LqLq_claim_halfline} follows from \eqref{eq:localized_LpLq_estimate_on_t_0} and the above findings.
\end{proof}

\begin{remark}[Necessity of interpolation with mixed integrability]
As the proof of Theorem \ref{t:suboptimal_Lq_Rpositive} shows, it was essential to allow $r\neq p$ in \eqref{eq:interpolation_Banach_spaces_progress} as the localization argument used to prove \eqref{eq:estimate_LqLq_claim_halfline} (and thus \eqref{eq:suboptimal_Lq_Rpositive_estimate}) requires $r=p_0>p$.
\end{remark}

\section{Scaling limits for the stochastic 3D NSEs}
\label{s:scaling_limit}
Fix $p,q\in (2,\infty)$ and $R>0$.
In this section, we consider the following stochastic 3D NSEs \eqref{eq:navier_stokes_intro} with cutoff on $\T^3$:
\begin{equation}
\label{eq:navier_stokes_cutoff}
\left\{
\begin{aligned}
\partial_t \vcut^N &+ (\overline{u}_0\cdot \nabla) \vcut^N+\phi^{R}_{p,q}(\cdot,\vcut^N)\,\p [\nabla \cdot (\vcut^N\otimes \vcut^N)]\\
&\qquad  
=\Delta \vcut^N +\sqrt{\frac{3\mu}{2}}\sum_{k,\alpha}\theta_k^N\p[ (\sigma_{k,\alpha}\cdot\nabla) \vcut^N]\circ \dot{ W}^{k,\alpha}_t,\\
 \nabla \cdot \vcut^N&=0,\\
\vcut^N(0,\cdot)&=u_0-\overline{u}_0,
 \end{aligned}
\right.
\end{equation} 
where $u_0\in \Bs^{1-2/p}_{q,p}(\T^3)$ is a given initial datum, $\overline{u}_0=\int_{\T^3} u_0$ denotes its average, $\p$ is the Helmholtz projection on $\T^3$,  
\begin{equation}
\label{eq:def_cutoff_NSE}
\phi^R_{p,q}(t,\vcut^N)
= \phi \big(R^{-1}\|\vcut^N\|_{L^{2p}(0,t;L^{2q}(\T^3;\R^3))}\big)
\end{equation}
and $\phi\in C^\infty_{{\rm c}}([0,\infty))$ is such that $0\leq \phi\leq 1$, $\phi|_{(0,1)}=1$, and $\phi|_{(2,\infty)}=0$. 

\smallskip

As shown in Subsection \ref{ss:global_proof} below, the cutoff system \eqref{eq:navier_stokes_cutoff} has the following fundamental property: If $\phi_{p,q}^R(\cdot,\vcut^N)=1$ on a given stochastic interval, then $u =\vcut^N+\overline{u}_0$ is a solution to the original stochastic 3D NSEs \eqref{eq:navier_stokes_intro} on the same interval. 
The key advantage of having initial data with zero mean is that the Navier-Stokes dynamics preserves such a condition, and exponential decay of solutions to \eqref{eq:navier_stokes_cutoff} can be easily obtained via energy methods, see \eqref{eq:decay_energy_vcut} below.
Finally, the norm in the cutoff \eqref{eq:def_cutoff_NSE} is chosen to obtain a suitable bound on the Navier-Stokes nonlinearity in an $L^p(L^q)$-setting. 
Throughout this section, we assume that 
\begin{equation}
\label{eq:condition_pq_subcritical_scaling_cutoff}
1<\frac{2}{p}+\frac{3}{q}<2 \qquad \text{ and }\qquad \Big[ p<\frac{q}{3-q} \text{ if }q<3\Big].
\end{equation}
The upper bound in the first condition in \eqref{eq:condition_pq_subcritical_scaling_cutoff} coincides with the subcriticality condition \eqref{eq:condition_pq_subcritical_scaling_statement}, and
automatically arises when proving uniform-in-$N$ estimates for the cutoff system \eqref{eq:navier_stokes_cutoff}, see Lemma \ref{l:uniform_N_estimate_LpLq} below. The lower bound in the first condition and the second inequality \eqref{eq:condition_pq_subcritical_scaling_cutoff} are a consequence of the application of Theorem \ref{t:suboptimal_Lq_Rpositive}, see \eqref{eq:assumption_suboptimal_Lq_finite_time1}.

\smallskip

We begin by discussing the global well-posedness of \eqref{eq:navier_stokes_cutoff}, where $(p,q)$-solutions are defined as in Definition \ref{def:p_solution} with trivial modifications.
The proof of the following result is postponed to Subsection \ref{ss:global_cutoff}.

\begin{lemma}[Global well-posedness of stochastic 3D NSEs with cutoff]
\label{l:global_cutoff}
Let $\mu>0$. Let $p,q\in (2,\infty)$ be such that \eqref{eq:condition_pq_subcritical_scaling_cutoff} holds. 
Then, for all $N\geq 1$, $R>0$, and $u_0\in \Bs^{1-2/p}_{q,p}(\T^3)$, there exists a unique \emph{global} $(p,q)$-solution  $\vcut^N$ to \eqref{eq:navier_stokes_cutoff}.
\end{lemma}

We are now ready to formulate the main result of this subsection. Below, for notational convenience, for $M\geq 1$, we set
\begin{equation}
\label{eq:ball_initial_data_cutoff} 
\mathcal{B}_{p,q}(M)\stackrel{{\rm def}}{=}
\big\{u_0\in \Bs^{1-2/p}_{q,p}(\T^3)\,:\, \|u_0\|_{B^{1-2/p}_{q,p}(\T^3;\R^3)}\leq M\big\}.
\end{equation}

\begin{theorem}[Scaling limit for 3D NSEs with cutoff on the half-line]
\label{t:scaling_limit_cutoff}
Let $M\geq 1$, $\mu>0$ and $R>0$. Suppose that $p,q\in (2,\infty)$ satisfy  \eqref{eq:condition_pq_subcritical_scaling_cutoff}.
Then, for all $\varepsilon\in (0,1)$, 
\begin{equation}
\label{eq:scaling_limit_cutoff}
\lim_{N \to \infty}\sup_{u_0\in \mathcal{B}_{p,q}(M)}\P\big(\|\vcut^N - \vcutd\|_{L^{2p}(\R_+;L^{2q}(\T^3;\R^3))}\geq \varepsilon\big)=0,
\end{equation}
where $\vcut^N$ is the global $(p,q)$-solution to \eqref{eq:navier_stokes_cutoff} with initial data $u_0-\overline{u}_0$, and $\vcutd$ is the unique global $(p,q)$-solution to the following deterministic {{\normalfont{3D NSEs}}} with cutoff and \emph{increased viscosity} on $\T^3$:
\begin{equation}
\label{eq:navier_stokes_cutoff_det}
\left\{
\begin{aligned}
\partial_t \vcutd 
&+(\overline{u}_0\cdot \nabla) \vcutd \\
&+\phi_{p,q}^R (\cdot,\vcutd) \,\p\big[ \nabla\cdot (\vcutd \otimes \vcutd)\big] 
= \Big(1 + \frac{3\mu}{5}\Big)\Delta \vcutd, \\
\vcutd(0)&=u_0-\overline{u}_0.
\end{aligned}
\right.
\end{equation}
\end{theorem}

As above, $(p,q)$-solutions to \eqref{eq:navier_stokes_cutoff_det} can be defined as in Definition \ref{def:p_solution}, and the global well-posedness of this deterministic system follows as in Lemma \ref{l:global_cutoff}.

In contrast to the existing literature on scaling limits (see e.g., \cite{A24_global_small,FGL21,FL19}), Theorem \ref{t:scaling_limit_cutoff} is formulated on the half-line rather than on a finite time interval.
This allows us to give a streamlined proof of global smoothness of solutions to stochastic 3D NSEs (see Theorem \ref{t:global_NSE}) in Subsection \ref{ss:global_proof}. Finally, the proof of Theorem \ref{t:scaling_limit_cutoff} below shows that our arguments also yield \eqref{eq:scaling_limit_cutoff} with $L^{2p}(\R_+;L^{2q})$ replaced by $L^{p_0}(\R_+;L^{q_0})$ for some $p_0\in (2p,\infty)$ and $q_0\in (2q,\infty)$ depending only on $p,q$, see Lemma \ref{l:compactness} below.

\smallskip

The proof of Theorem \ref{t:scaling_limit_cutoff} is given in Subsection \ref{ss:proof_scaling_limit} below.
As outlined in Subsection \ref{ss:scaling_intro}, the proof of Theorem \ref{t:scaling_limit_cutoff} relies on the following two ingredients: uniform-in-$N$ estimates for solutions to the stochastic 3D NSEs with cutoff, and a compactness result on the half-line, see Subsections \ref{ss:uniform_estimates_N_cutoff} and \ref{ss:compactness}, respectively. In the proof of the uniform-in-$N$ estimate of Lemma \ref{l:uniform_N_estimate_LpLq}, we make use of the estimate proven in Theorem \ref{t:suboptimal_Lq_Rpositive}.
However, we first prove Lemma \ref{l:global_cutoff}. 

\subsection{Global well-posedness with cutoff -- Proof of Lemma \ref{l:global_cutoff}}
\label{ss:global_cutoff}
The proof of Lemma \ref{l:global_cutoff} is standard. For the reader's convenience, we include a sketch of the proof.
The reader is referred to \cite[Theorem 4.1]{A22} for a similar situation.

\begin{proof}[Proof of Lemma \ref{l:global_cutoff} -- Sketch]
As in the proof of \cite[Theorem 4.1]{A22}, from the local well-posedness of the stochastic 3D NSEs \eqref{eq:navier_stokes_intro} under the subcriticality assumption \eqref{eq:condition_pq_subcritical_scaling_cutoff} and the maximal $L^p$-regularity estimates of \cite[Section 3]{AV21_NS}, it follows that there exists a unique maximal $(p,q)$-solution $(\vcut^N,\tcut^N)$ to \eqref{eq:navier_stokes_cutoff} satisfying $\tcut^N>0$ a.s. Moreover, the following blow-up criterion holds: 
\begin{equation}
\label{eq:blow_up_criterion_cutoff}
\P\Big(\tcut^N<\infty,\, \|\phi_{p,q}^R(\cdot,\vcut^N) (\vcut^N\otimes \vcut^N)\|_{L^p(0,\tcut^N;L^{q})}<\infty\Big)=0.
\end{equation}
We check the above blow-up criterion exploiting the presence of the cutoff. Let $\chi^{R}_{p,q}$ be the stopping time given by 
\begin{equation*}
\chi^{R}_{p,q}\stackrel{{\rm def}}{=}\inf\{t\in [0,\tcut^N)\,:\, \|\vcut^N\|_{L^{2p}(0,t;L^{2q})}\geq 2R\} \quad \text{ and }\quad \inf\emptyset\stackrel{{\rm def}}{=}\tcut^N.
\end{equation*}
From the definition of the cutoff \eqref{eq:def_cutoff_NSE}, it follows that 
\begin{align}
\label{eq:nonlinearity_cutoff_bounds1}
&\big\|\phi_{p,q}^R (\cdot,\vcut^N) (\vcut^N \otimes \vcut^N) \big\|_{L^p(0,\tcut^N;L^q)}\\
\nonumber
&\qquad\qquad\qquad 
= \big\|\phi_{p,q}^R (\cdot,\vcut^N) (\vcut^N \otimes \vcut^N) \big\|_{L^p(0,\tcut^N\wedge \chi^R_{p,q};L^q)}\\
\nonumber
& \qquad\qquad \qquad
\leq \|\vcut^N \|_{L^{2p}(0,\tcut^N\wedge \chi_{p,q}^R;L^{2q})}^{2}
\leq 4R^2,
\end{align}
where the last inequality is a consequence of the definition of $\chi^{R}_{p,q}$.
In particular, 
$$
\|\phi_{p,q}^R(\cdot,\vcut^N) (\vcut^N\otimes \vcut^N)\|_{L^p(0,\tcut^N;L^{q})}<\infty\ \text{ a.s., }
$$ 
and \eqref{eq:blow_up_criterion_cutoff} yields $\tcut^N=\infty$ a.s., i.e., $\vcut^N$ is global in time.
\end{proof}

\subsection{Estimates uniform in the oscillations -- 3D NSEs with cutoff}
\label{ss:uniform_estimates_N_cutoff}
The essential result in the proof of Theorem \ref{t:scaling_limit_cutoff} is given by the following uniform-in-$N$ estimates for solutions to the stochastic oscillating 3D NSEs with cutoff \eqref{eq:navier_stokes_cutoff}. Recall that $\mathcal{B}_{p,q}(M)$ is defined in \eqref{eq:ball_initial_data_cutoff}.

\begin{lemma}[Uniform-in-$N$ estimates in $L^p(L^q)$ for 3D NSEs with cutoff]
\label{l:uniform_N_estimate_LpLq}
Let $M\geq 1$, $\mu>0$ and $R>0$. Suppose that $p,q\in (2,\infty)$ satisfy \eqref{eq:condition_pq_subcritical_scaling_cutoff}.
There exists a constant $C\geq 1$ and exponents $p_0\in (2p,\infty)$, $q_0\in (2q,\infty)$ independent of $N\geq 1$ and $u_0\in \mathcal{B}_{p,q}(M)$ for which the unique global $(p,q)$-solution to the stochastic oscillating 3D NSEs \eqref{eq:navier_stokes_cutoff} provided by Lemma \ref{l:global_cutoff} satisfies
$$
\E\int_0^{\infty} \|\vcut^N\|_{L^{q_0}(\T^3;\R^3)}^{p_0} \leq C.
$$
\end{lemma}

Before going into the proof of Lemma \ref{l:uniform_N_estimate_LpLq}, we first collect some observations.
First, as the initial data $w_0\stackrel{{\rm def}}{=}u_0-\overline{u}_0$ in \eqref{eq:navier_stokes_cutoff} has zero mean, we have 
\begin{equation}
\label{eq:vcut_has_mean_zero}
\int_{\T^3} \vcut^N (t)=0\ \text{ a.s.\ for all }t>0.
\end{equation}
Second, due to the well-known cancellation 
$$
\int_{\T^3} \phi_{p,q}^R (\cdot,w)\, \p[\nabla\cdot (w\otimes w)]\cdot w =0 \ \text{ for all }\ w\in \Hs^1(\T^3),
$$
as $\phi_{p,q}^R (\cdot, w)$ is spatially independent, from It\^o's formula (see e.g., \cite[Theorem 4.2.5]{LR15}), it follows that, a.s.\ for all $t>0$, the following energy balance holds
\begin{equation}
\label{eq:energy_balance_vcut}
\frac{1}{2}\|\vcut^N (t)\|_{L^2(\T^3)}^2 +\int_0^t \int_{\T^3} |\nabla \vcut^N|^2= \frac{1}{2}\| w_0\|_{L^2(\T^3)}^2.
\end{equation}
Finally, due to \eqref{eq:vcut_has_mean_zero} and the Poincar\'e inequality, 
\begin{equation}
\label{eq:decay_energy_vcut}
\|\vcut^N(t)\|_{L^2(\T^3)} \leq e^{- C t }\|w_0\|_{L^2(\T^3)} \lesssim_{q,p} e^{- C t} \|w_0\|_{B^{1-2/p}_{q,p}(\T^3)}
\end{equation}
where $C$ is a universal constant, and we used the elementary embedding $B^{1-2/p}_{q,p}\embed L^2$ as $p>2$ and $q\geq 2$.
We are now ready to prove Lemma \ref{l:uniform_N_estimate_LpLq}.

\begin{proof}[Proof of Lemma \ref{l:uniform_N_estimate_LpLq}]
In light of \eqref{eq:vcut_has_mean_zero}, we can apply Theorem \ref{t:suboptimal_Lq_Rpositive} for some
\begin{equation}
\label{eq:choice_p0q0_scaling_limit}
p_0>2p\qquad \text{ and }\qquad q_0>2q,
\end{equation}
independent of $N,T$ and $u_0$. 
Indeed, $2q<\frac{d^2}{d-2}\frac{1}{-\Sob}=\frac{9}{-\Sob}$ is equivalent to $\frac{2}{p}+\frac{3}{q}-1<\frac{9}{2q}$,
which is automatically satisfied as $p,q\in (2,\infty)$ and $d=3$ (see also the first item in Remark \ref{r:high_dimensions}). Moreover, one can check that the choice as in \eqref{eq:choice_p0q0_scaling_limit} is possible as 
$$
-\frac{2}{2p}-\frac{3}{2q}<\Sob =1-\frac{2}{p}-\frac{3}{q}\qquad \Longleftrightarrow\qquad
\frac{2}{p}+\frac{3}{q}<2,
$$
which is satisfied due to \eqref{eq:condition_pq_subcritical_scaling_cutoff}.
Therefore, from Theorem \ref{t:suboptimal_Lq_Rpositive} with $(p_0,q_0)$ satisfying \eqref{eq:choice_p0q0_scaling_limit} and $r=p_0$, there exists a constant $C_0>0$ such that, for all $N\geq 1$, $T<\infty$ and $u_0\in \mathcal{B}_{p,q}(M)$, the unique global $(p,q)$-solution to \eqref{eq:navier_stokes_cutoff} provided by Lemma \ref{l:global_cutoff} satisfies
\begin{align}
\label{eq:vcut_N_uniform_estimate_proof}
\big(\E\|\vcut^N\|_{L^{p_0}(0,T;L^{q_0})}^{p_0}\big)^{1/p_0}
&
\leq C_0 \|u_0\|_{B^{1-2/p}_{q,p}}\\
\nonumber
& + C_0  \big(\E\big\|\phi_{p,q}^R (\cdot,\vcut^N) (\vcut^N \otimes \vcut^N) \big\|_{L^p(0,T;L^q)}^{p_0}\big)^{1/p_0}\\
\nonumber
 &+ C_0\big(\E\|\vcut^N\|_{L^{p}(0,T;L^{q})}^{p_0} \big)^{1/p_0},
\end{align}
where we used that $(\overline{u}_0\cdot \nabla )\vcut^N=\nabla \cdot (\overline{u}_0\otimes\vcut^N)$.
From \eqref{eq:nonlinearity_cutoff_bounds1} with $\tcut^N=\infty$ (the latter was shown in the proof of Lemma \ref{l:global_cutoff}), we have
\begin{align}
\label{eq:nonlinearity_cutoff_bounds}
\big(\E\big\|\phi_{p,q}^R (\cdot,\vcut^N) (\vcut^N \otimes \vcut^N) \big\|_{L^p(0,T;L^q)}^{p_0}\big)^{1/p_0}
 \leq 4R^2.
\end{align}
Hence, it remains to estimate the last term on the right-hand side of \eqref{eq:vcut_N_uniform_estimate_proof}. To handle this term, we exploit the fact that it is of lower-order type compared to the norm appearing on the left-hand side of \eqref{eq:vcut_N_uniform_estimate_proof}, and that we have the uniform decay \eqref{eq:decay_energy_vcut} of the energy (which is even of lower-order compared to $L^p(L^q)$-norm).
By interpolation and $q_0>2q>q$, there exists $\kappa\in (0,1)$ such that $\|f\|_{L^q}\leq \|f\|_{L^2}^{1-\kappa}\|f\|_{L^{q_0}}^{\kappa}$ for all $f\in L^{q_0}$. Letting $p_1\in (1,\infty)$ be defined via the identity $\frac{1}{p_1}+\frac{\kappa}{p_0}=\frac{1}{p}$, H\"older's inequality ensures
\begin{align}
\label{eq:lower_order_term_interpolation_estimate_proof_uniform_bound}
\|\vcut^N\|_{L^{p}(0,T;L^{q} )}
&\leq \Big\|\|\vcut^N\|_{L^2}^{1-\kappa} \|\vcut^N\|_{L^{q_0}}^{\kappa}\Big\|_{L^p(0,T)}\\
\nonumber
&\leq \Big\|\|\vcut^N\|_{L^2}^{1-\kappa}\Big\|_{L^{p_1}(0,T)} \Big\|\|\vcut^N\|_{L^{q_0}}^{\kappa}\Big\|_{L^{p_0/\kappa}(0,T)}\\
\nonumber
&= \|\vcut^N\|_{L^{p_1(1-\kappa)}(0,T;L^2)}^{1-\kappa} 
\|\vcut^N\|_{L^{p_0}(0,T;L^{q_0})}^\kappa\\
\nonumber
&\lesssim_M \|\vcut^N\|_{L^{p_0}(0,T;L^{q_0})}^\kappa
\nonumber
\end{align}
where the last step follows from \eqref{eq:decay_energy_vcut} and $\|u_0\|_{B^{1-2/p}_{q,p}}\leq M$. Let us stress that the implicit constant in the above inequality is independent of $T<\infty$.
In particular, using \eqref{eq:nonlinearity_cutoff_bounds}-\eqref{eq:lower_order_term_interpolation_estimate_proof_uniform_bound} and Young's inequality in the estimate  \eqref{eq:vcut_N_uniform_estimate_proof}, it follows that there exists a constant $C_1$ independent of $T<\infty$, $N\geq 1$ and $u_0\in \mathcal{B}_{p,q}(M)$ such that  
\begin{equation}
\label{eq:lower_order_treatment_proof}
\big(\E\|\vcut^N\|_{L^{p_0}(0,T;L^{q_0})}^{p_0} \big)^{1/p_0}\leq C_1 +\frac{1}{2}
\big(\E\|\vcut^N\|_{L^{p_0}(0,T;L^{q_0})}^{p_0} \big)^{1/p_0}.
\end{equation}
Using the latter and \eqref{eq:nonlinearity_cutoff_bounds} in the estimate \eqref{eq:vcut_N_uniform_estimate_proof}, we obtain the existence of a constant $C_0\geq 1$  independent of $T<\infty$, $N\geq 1$ and $u_0\in \mathcal{B}_{p,q}(M)$ for which
$$
\big(\E\|\vcut^N\|_{L^{p_0}(0,T;L^{q_0})}^{p_0}\big)^{1/p_0}
\leq C_0,
$$
where we absorbed on the left-hand side the last term in \eqref{eq:lower_order_treatment_proof}.
The claimed estimate in Lemma \ref{l:uniform_N_estimate_LpLq} now follows from the above and Fatou's lemma by letting $T\uparrow \infty$.
\end{proof}

We conclude this section with the following standard result, which is a well-known consequence of the global energy balance \eqref{eq:energy_balance_vcut} and the structure of the noise (see Subsection \ref{ss:probabilistic}).

\begin{lemma}[Uniform-in-$N$ time regularity]
\label{l:time_regularity}
Fix $R,M,\mu>0$, and let $p,q\in (2,\infty)$ be such that \eqref{eq:condition_pq_subcritical_scaling_cutoff} holds.
For all $b\in (2,\infty)$ and $T\in (0,\infty)$, let $\vcut^N$ be the unique global $(p,q)$-solution to \eqref{eq:navier_stokes_cutoff} provided by Lemma \ref{l:global_cutoff} for some $u_0\in \mathcal{B}_{p,q}(M)$. 
Set 
\begin{equation}
\label{eq:martingale_vcutN_def}
\mathcal{M}^N(t)\stackrel{{\rm def}}{=}\sum_{k,\alpha} \theta_k^N \int_0^t (\sigma_{k,\alpha}\cdot \nabla) \vcut^N\,\dd W^{k,\alpha}_s.
\end{equation}
Then there exist constants $K_T,\kappa,\g_0,\g_1>0$ independent of $u_0\in \mathcal{B}_{p,q}(M)$ and $N\geq 1$ such that 
\begin{align}
\label{eq:time_regularity1}
\E [\|\mathcal{M}^N \|_{C^{\g_0}(0,T;H^{-\g_1})}^{b}]
&\leq K_T N^{-\kappa b} (1+\|u_0\|_{B^{1-2/p}_{q,p}}^{b}),\\
\label{eq:time_regularity2}
\E [\|\vcut^N \|_{C^{\g_0}(0,T;H^{-\g_1})}^{b}]
&\leq K_T (1+\|u_0\|_{B^{1-2/p}_{q,p}}^{b}),
\end{align}
where $\g_0,\g_1$ are additionally independent of $T>0$.
\end{lemma}

The proof of the above follows verbatim from the arguments in either \cite[p.\ 1779]{FGL21} or \cite[Lemma 6.3]{A22} combined with the energy balance \eqref{eq:energy_balance_vcut}.

\subsection{A compactness result on the half-line}
\label{ss:compactness}
We begin by introducing some notation. Let $X$ be a Banach space. For a sequence $\mathbf{a}=(a_n)_{n}$ of (strictly) positive numbers and $\g>0$, we let $C_{\loc}^{\mathbf{a},\g}([0,\infty);X)$ be the set of all maps $f:[0,\infty)\to X$ such that $f|_{[0,n]}\in C^{\g}(0,n;X)$ for all $n$ and
$$
\|f\|_{C_{\loc}^{\mathbf{a},\g}([0,\infty);X)}
\stackrel{{\rm def}}{=} \sup_{n} \big( a_n \, \|f\|_{C^{\g}(0,n;X)}\big)<\infty.
$$
Clearly, $C_{\loc}^{\mathbf{a},\g}([0,\infty);X)$ is a Banach space.
Moreover, for $\eta>0$, we denote by $C_{\eta}([0,\infty);X)$ the set of all continuous maps $f:\R_+\to X$ such that 
$$
\|f\|_{C_{\eta}([0,\infty);X)}
\stackrel{{\rm def}}{=}
 \sup_{t\in \R_+} (e^{\eta t }\|f(t)\|_{X})<\infty.
$$
Finally, $C_{0}([0,\infty);X)$ denotes the set of bounded and continuous functions vanishing at infinity with the sup-norm.
The main result of this subsection reads as follows.

\begin{lemma}[Compactness on the half-line]
\label{l:compactness}
Fix 
$$
\g,\sigma,\eta,\varepsilon>0, \qquad 1<p_0,q_0<\infty, \qquad p_1\in (2,p_0),\ q_1\in (2,q_0),
$$
and a sequence $\mathbf{a}=(a_n)_{n}$ of strictly positive numbers. Let
\begin{align*}
\Y
&=C_{\loc}^{\mathbf{a},\g}([0,\infty);H^{-\sigma}) \cap C_{\eta}([0,\infty);L^2)\cap  L^2 (\R_+;H^1)\cap L^{p_0}(\R_+;L^{q_0}),\\
\X
&=L^{p_1}(\R_+;L^{q_1})\cap C_{0}([0,\infty);H^{-\varepsilon}).
\end{align*}
Then, the embedding $\Y\embed \X$ is compact. Moreover, for all $K\geq 1$, 
$$
\X_{K}=\big\{f\in \X\,:\, \|f\|_{L^2(\R_+;H^1)}\leq K\big\}
$$
is a closed subset of $\X$.
\end{lemma}

\begin{proof}
Note that $\Y$ is a Banach space. Next, we prove that, for all $M_0\geq 0$,
\begin{equation}
\label{eq:compactness_set_M_YX}
\{v\in \Y\,:\, \|v\|_{\Y}\leq M_0\} \ \text{ is pre-compact in }\ \X.
\end{equation}
To this end, it suffices to prove that for each sequence $(f_n)_{n}\subseteq \Y$ such that $\sup_n\|f_n\|_{\Y}\leq M_0$ there exists a subsequence $(f_{n_j})_j$ that converges in $\X$.  

To begin, as $(f_n)_n\subseteq C_{\loc}^{\mathbf{a},\g}([0,\infty);H^{-\sigma})$, by the Arzel\`a-Ascoli theorem and a diagonal argument, for all $\sigma_0>\sigma$, there exists a (not-relabeled) subsequence of $(f_{n})_n$ such that 
\begin{equation}
\label{eq:locally_convergence_sequence}
\text{ for all }T<\infty, \quad
f_n \to f \ \text{ in } \ C([0,T];H^{-\sigma_0}).
\end{equation}
Next, fix $\sigma_0>\sigma$. We claim that $f_n\to f$ in $C_0([0,\infty);H^{-\sigma_0})$. Indeed, for each $\varepsilon>0$, fix $T_\varepsilon>0$ so that, for all $t\geq T_\varepsilon$ and $n\geq 1$,
$$
\|f_n(t)\|_{H^{-\sigma_0}}\leq C_{\sigma_0,\sigma}M_0 e^{-\eta T_\varepsilon}\leq \varepsilon/3\ .
$$
Due to \eqref{eq:locally_convergence_sequence}, it also holds that $\|f(t)\|_{H^{-\sigma_0}}\leq \varepsilon/3$ for $t\geq T_\varepsilon$.
Moreover, \eqref{eq:locally_convergence_sequence} ensures the existence of $N_\varepsilon$ for which $\|f_n-f\|_{C([0,T_\varepsilon];H^{-\sigma_0})}\leq \varepsilon/3$ for all $n\geq N_\varepsilon$. Therefore, for all $n \geq N_\varepsilon$,
\begin{align*}
\sup_{t\in \R_+}\|f_n(t)-f(t)\|_{H^{-\sigma_0}}
\leq  \sup_{t\leq T_\varepsilon}\|f_n(t)-f(t)\|_{H^{-\sigma_0}}
+ \sup_{t> T_\varepsilon}\|f_n(t)-f(t)\|_{H^{-\sigma_0}}\leq \varepsilon.
\end{align*}
This together with the uniform bound of $(f_n)_n$ in $C_{\eta}([0,\infty);L^2)$ implies $f_n \to f$ in $C_0([0,\infty); H^{-\sigma_0})$. Combining this, standard interpolation inequalities, and the uniform bound in $C_\eta([0,\infty);L^2)$, one has $f_n\to f$ in $C_{\beta}([0,\infty); H^{-\varepsilon})$ for all $\varepsilon\in (0,1)$ and $\beta>0$ small depending on $\varepsilon$. 
By the arbitrariness of $\varepsilon\in (0,1)$ and the uniform bound in $L^2(\R_+;H^1)$, it follows from interpolation that
$$
f_n \to f \text{ in }L^{r}(\R_+;L^2) \ \text{ for all }\ r\in (0,\infty).
$$
Now, one can readily check that, for each $p_1\in (2,p_0)$, $q_1\in (2,q_0)$, there exists $\varphi\in (0,1)$ and $r\in (0,\infty)$ for which
$$
\frac{1}{q_1}=\frac{1-\varphi}{2}+\frac{\varphi}{q_0}\quad \text{ and }\quad 
\frac{1}{p_1}=\frac{1-\varphi}{r}+\frac{\varphi}{p_0},
$$
and thus, as $n\to \infty$,
\begin{align*}
\|f-f_n\|_{L^{p_1}(\R_+;L^{q_1})}
&\lesssim \|f-f_n\|_{L^{r}(\R_+;L^2)}^{1-\varphi}\|f-f_n\|_{L^{p_0}(\R_+;L^{q_0})}^\varphi \\
&\lesssim_M  \|f-f_n\|_{L^{r}(\R_+;L^2)}^{1-\varphi}\to 0.
\end{align*}
This proves \eqref{eq:compactness_set_M_YX} and concludes the first part of the proof. The closedness of $\X_K$ in $\X$ follows from the lower semicontinuity of $L^2$ and $L^2(\R_+;H^1)$ norms.
\end{proof}

\subsection{Proof of Theorem \ref{t:scaling_limit_cutoff}}
\label{ss:proof_scaling_limit}
In this subsection, we finally prove Theorem \ref{t:scaling_limit_cutoff}. As usual in the case of scaling limits (see e.g., \cite[Theorem 6.1]{A22} or \cite[Proposition 3.7]{FGL21}), the underlying compactness argument naturally produces solutions to \eqref{eq:navier_stokes_cutoff_det} that are weaker than $(p,q)$-solutions. We make this precise in the following

\begin{definition}[Weak solutions to the 3D NSEs with cutoff]
\label{def:NSE_cutoff}
Fix $R,\mu,\varepsilon>0$ and $q,p\in (2,\infty)$ satisfying \eqref{eq:condition_pq_subcritical_scaling_cutoff}. Let $u_0\in \Bs^{1-2/p}_{q,p}(\T^3)$ and 
$$
\vcutd \in L^{2p}(\R_+;\Ls^{2q}(\T^3))\cap L^2_{\loc}([0,\infty);\Hs^1(\T^3))\cap C([0,\infty);\Hs^{-\varepsilon}(\T^3)) .
$$
We say that $\vcutd$ is a weak $(p,q)$-solution to the {{\normalfont{3D NSEs}}} with cutoff \eqref{eq:navier_stokes_cutoff_det} if for all divergence-free vector fields $\varphi\in C^{\infty}(\T^3;\R^3)$ and all $t\in \R_+$, it holds that 
\begin{align*}
\langle \vcutd(t),\varphi\rangle
&=\int_{\T^3} \varphi\cdot (u_0-\overline{u}_0)
+\Big(1+\frac{3\mu}{5}\Big)\int_0^t\int_{\T^3}
\, \vcutd\cdot\Delta \varphi \,\dd x \,\dd s\\
&+\int_0^t\int_{\T^3}
\big[ (\vcutd\otimes \overline{u}_0)+ \phi_{p,q}^R(\cdot,\vcutd)\,(\vcutd\otimes \vcutd)\big] : \nabla \varphi\,\dd x \,\dd s.
\end{align*}
\end{definition}

As a rule of thumb, due to the cutoff in a subcritical space, weak $(p,q)$-solutions to \eqref{eq:navier_stokes_cutoff_det} are actually smooth and unique.

\begin{lemma}[Uniqueness of weak solutions]
\label{l:uniqueness_lemma_weak_sol}
Fix $R,\mu>0$ and $p,q\in (2,\infty)$ satisfying \eqref{eq:condition_pq_subcritical_scaling_cutoff}. Then, for all $u_0\in \Bs^{1-2/p}_{q,p}(\T^3)$, weak solutions to \eqref{eq:navier_stokes_cutoff_det} are unique, and coincide with the $(p,q)$-solution provided by Lemma \ref{l:global_cutoff}. Moreover, if $(u_0^N)_N$ is a sequence of initial data weakly converging to $u_0$ in $B^{1-2/p}_{q,p}$, then 
\begin{equation}
\label{eq:weak_continuity_implies_strong_compactness}
\vcutd^N \to \vcutd \text{ in }L^{2p}(\R_+;L^{2q}(\T^3;\R^3))
\end{equation}
where $\vcutd^N$ denotes the solution to \eqref{eq:navier_stokes_cutoff_det} with initial data $u_0^N-\int_{\T^3}u_0^N$.
\end{lemma}

\begin{proof}
We begin by proving the uniqueness. The idea is to prove that any weak solution is actually a $(p,q)$-solution to \eqref{eq:navier_stokes_cutoff_det}, and thus uniqueness follows from Lemma \ref{l:global_cutoff}. Recall that, from the definition of weak solutions to \eqref{eq:navier_stokes_cutoff_det}, it follows that 
$$
\vcutd\in L^{2p}(0,T;L^{2q}(\T^3;\R^3)) \ \ \text{ for all }T<\infty.
$$
Next, fix $T<\infty$. From the above, we have 
$$
\phi_{p,q}^R(\cdot,\vcutd) (\vcutd\otimes \vcutd)\in L^{p}(0,T;L^{q}(\T^3;\R^{3\times 3})).
$$
By the well-known maximal $L^p(L^q)$-regularity of the operator $v\mapsto (1+\frac{3\mu}{5})\Delta v$ on the space $H^{-1,q}(\T^3)$ with domain $H^{1,q}(\T^3)$ (e.g., the periodic version of \cite[Theorem 17.4.1]{Analysis3}) as well as a standard perturbation argument to deal with the lower order operator $v\mapsto -(\overline{u}_0\cdot \nabla ) v$, it follows that 
\begin{equation}
\label{eq:LpH1q_solutions_weak}
\vcutd \in L^{p}(0,T;H^{1,q}(\T^3))\cap W^{1,p}(0,T;H^{-1,q}(\T^3)).
\end{equation}
Therefore, $\vcutd$ is a $(p,q)$-solution to \eqref{eq:navier_stokes_cutoff_det} and uniqueness holds by Lemma \ref{l:global_cutoff} (cf.\ Definition \ref{def:p_solution}).

Next, we prove \eqref{eq:weak_continuity_implies_strong_compactness}. Let $(u_0^N)_N$ be as in the statement of Lemma \ref{l:uniqueness_lemma_weak_sol}, and set $K\stackrel{{\rm def}}{=}\sup_N \|u_0^N\|_{B^{1-2/p}_{q,p}(\T^3)}<\infty$. As in \eqref{eq:decay_energy_vcut}, due to $\int_{\T^3} \vcutd(t)=0$ for all $t>0$ and the energy and Poincar\'e's inequalities, it follows that 
$$
\sup_N\int_0^\infty \|\nabla \vcutd^N\|_{L^2}^2\,\dd s \lesssim K^2 \quad \text{ and }\quad
\sup_N\|\vcutd^N(t)\|_{L^2}^2 \lesssim K^2 e^{-C t},
$$
for some universal constant $C>0$.
Again, from the argument below \eqref{eq:G_smoothness2}, we have  
\begin{equation}
\label{eq:embedding_Lp0Lq0_proof_uniqueness}
L^{p}(\R_+;H^{1,q}(\T^3))\cap W^{1,p}(\R_+;H^{-1,q}(\T^3))
\subseteq L^{p_0}(\R_+;L^{q_0}(\T^3))
\end{equation}
for some $p_0\in (2p,\infty)$, $q_0\in (2q,\infty)$ depending only on $q$ and $p$. Let $H^{\sigma,q}_{{\rm mf}}(\T^3)$ be the subspace of mean-free distributions of $H^{\sigma,q}(\T^3)$ where $\sigma\in \R$.
Now, the invertibility of $-\Delta: H^{1,q}_{{\rm mf}}(\T^3)\subseteq H^{-1,q}_{{\rm mf}}(\T^3)\to H^{-1,q}_{{\rm mf}}(\T^3)$ ensures that $L^p(L^q)$-maximal regularity estimates hold on the half-line on such spaces. This and \eqref{eq:embedding_Lp0Lq0_proof_uniqueness} imply
$$
\max\Big\{\sup_N  \|\vcutd^N \|_{C^{1-1/p}(\R_+;H^{-1,q}(\T^3))}\,,\, 
\sup_{N}\|\vcutd^N\|_{L^{p_0}(\R_+;L^{q_0}(\T^3))}\Big\}<\infty,
$$
where the inhomogeneity $\phi_{p,q}^R(\cdot,\vcutd^N) (\vcutd^N\otimes \vcutd^N)$ and $(\overline{u}_0^N\cdot\nabla) \vcutd^N$ are estimated as in \eqref{eq:nonlinearity_cutoff_bounds} and \eqref{eq:lower_order_term_interpolation_estimate_proof_uniform_bound}, respectively. In the above, we also used the Sobolev embedding $W^{1,p}(\R_+;X)\subseteq C^{1-1/p}(\R_+;X)$, which holds for every Banach space $X$.
Hence, \eqref{eq:weak_continuity_implies_strong_compactness} is a consequence of the above, Lemma \ref{l:compactness}, and the just-proved uniqueness of weak solutions to \eqref{eq:navier_stokes_cutoff_det}. 
\end{proof}

With the above results at our disposal, we can give the proof of Theorem \ref{t:scaling_limit_cutoff}.

\begin{proof}[Proof of Theorem \ref{t:scaling_limit_cutoff}]
As usual in the context of scaling limits for SPDEs (see e.g., \cite{A22,FGL21,FL19}), in light of the results of Subsections \ref{ss:uniform_estimates_N_cutoff} and \ref{ss:compactness}, we prove Theorem \ref{t:scaling_limit_cutoff} by contradiction. If \eqref{eq:scaling_limit_cutoff} does not hold, then there exist $\varepsilon_0\in (0,1)$ and a sequence $(u_0^N)_N \subseteq \mathcal{B}_{p,q}(M)$ such that 
\begin{equation}
\label{eq:contradiction_scaling_limit}
\lim_{N\to \infty} \P(\|\vcut^N - \vcutd^N\|_{L^{2p}(\R_+;L^{2q})}\geq \varepsilon_0)\geq \varepsilon_0,
\end{equation}
where $\vcut^N$ and $\vcutd^N$ denote the solutions to \eqref{eq:navier_stokes_cutoff} and \eqref{eq:navier_stokes_cutoff_det} with initial data $u_0^N-\int_{\T^3}u_0^N$, respectively.
By reflexivity of $B^{1-2/p}_{q,p}$, without loss of generality, we can assume that $u_0^N$ converges weakly in $B^{1-2/p}_{q,p}$ to some $u_0\in\mathcal{B}_{p,q}(M)$. To obtain the desired contradiction, it suffices to show that, up to extracting a (not relabeled) subsequence,
\begin{align}
\label{eq:scaling_limit_cutoff_proof_1}
\vcut^N &\to  \vcutd\,  \text{ in probability in }  L^{2p}(\R_+;L^{2q}),\\
\label{eq:scaling_limit_cutoff_proof_2}
\vcutd^N &\to \vcutd\,  \text{ in }  L^{2p}(\R_+;L^{2q}),
\end{align}
where $\vcutd$ is the unique $(p,q)$-solution to \eqref{eq:navier_stokes_cutoff_det} with initial data $u_0-\overline{u}_0$, and $u_0$ satisfies $u_0^N \to u_0$ weakly in $B^{1-2/p}_{q,p}$.
Note that \eqref{eq:scaling_limit_cutoff_proof_2} follows from \eqref{eq:weak_continuity_implies_strong_compactness} in Lemma \ref{l:uniqueness_lemma_weak_sol}.
Thus, it remains to prove \eqref{eq:scaling_limit_cutoff_proof_1}.

\smallskip

Let $p_0,q_0\in (2,\infty)$ be as in Lemma \ref{l:uniform_N_estimate_LpLq}. From the latter and \eqref{eq:decay_energy_vcut}, there exist constants $C,\eta>0$ independent of $N$ such that  
\begin{equation}
\label{eq:bound_inform_proof_scaling1}
\E\int_0^{\infty} \|\vcut^N\|_{L^{q_0}}^{p_0} \leq C \qquad \text{ and }\qquad \sup_{t>0} \big(e^{\eta t}\|\vcut^N\|_{L^2}\big) \leq C.
\end{equation}
Let $(K_T)_{T>0}$ be as in Lemma \ref{l:time_regularity}, and set $\mathbf{a}=(1/(2^k K_k^{1/p_0}))_{k\geq 1}$. Then, from Lemma \ref{l:time_regularity} and the continuous embedding $\ell^1\embed \ell^\infty$, it follows that 
\begin{equation}
\label{eq:bound_inform_proof_scaling2}
\E\big[ \|\vcut^N\|_{C_{\loc}^{\mathbf{a},\g_0}([0,\infty);H^{-\g_1})}^{p_0}\big]
\leq \sum_{k\geq 1} 2^{-kp_0}K_{k}^{-1} \E\big[\|\vcut^N\|_{C^{\g_0}(0,k;H^{-\g_1})}^{p_0}\big]
\lesssim_M 1.
\end{equation}
In particular, the implicit constant in the above is independent of $N\geq 1$. 

Fix $\varepsilon\in (0,1)$. 
Prokhorov's theorem together with Lemma \ref{l:compactness} with $p_1=2p$ and $q_1=2q$, and the uniform estimates in \eqref{eq:bound_inform_proof_scaling1} and \eqref{eq:bound_inform_proof_scaling2}, ensures that the laws $(\nu^N)_N$ are pre-compact, where 
$$
\nu^N (A)= \P(\vcut^N \in A) \  \text{ for }\  A \in \Borel(L^{2p}(\R_+;\Ls^{2q})\cap C_0([0,\infty);\Hs^{-\varepsilon})),
$$
and $\Borel$ denotes the Borel sets of $L^{2p}(\R_+;\Ls^{2q}(\T^3))\cap C_0([0,\infty);\Hs^{-\varepsilon}(\T^3))$. Moreover, from the energy balance \eqref{eq:energy_balance_vcut} and its consequence \eqref{eq:decay_energy_vcut}, there exists a constant $K>0$ depending only on $p,q$ and $M$ such that 
$$
\nu^N (\mathcal{S}_{K})=1 \text{ for all }N\geq 1,
$$
where
\begin{align*}
\mathcal{S}_{K}\stackrel{{\rm def}}{=} \Big\{w\in L^{2p}(\R_+;\Ls^{2q})\cap C_0([0,\infty);\Hs^{-\varepsilon}) \,&:\,
 \int_0^\infty \| w(t)\|_{H^1}^2\leq K\text{ and }\\
&\,\langle w(t),1 \rangle=0 \text{ for all }t>0\Big\}.
\end{align*}
Thus, up to extracting a subsequence, $\nu^N \to \nu$ weakly converging in $L^{2p}(\R_+;\Ls^{2q})\cap C_0([0,\infty);\Hs^{-\varepsilon})$ for some probability measure $\nu$. Moreover, as $\nu^N(\mathcal{S}_K)=1$ for all $N$, it follows that $\nu(\mathcal{S}_K)=1$.
Next, we claim that 
\begin{equation}
\label{eq:claim_probability_concentrated_on_weak_solutions}
\nu \big(w\in \mathcal{S}_K\text{ is a weak $(p,q)$-solution to }\eqref{eq:navier_stokes_cutoff_det}\big)=1.
\end{equation}
Next, we show that the above implies \eqref{eq:scaling_limit_cutoff_proof_1} and thus \eqref{eq:contradiction_scaling_limit} cannot hold, therefore proving the claim \eqref{eq:scaling_limit_cutoff}.
Indeed, if \eqref{eq:claim_probability_concentrated_on_weak_solutions} is valid, then Lemma \ref{l:uniqueness_lemma_weak_sol} yields 
$$
\{w\in \mathcal{S}_K\text{ is a weak $(p,q)$-solution to }\eqref{eq:navier_stokes_cutoff_det}\}=\{\vcutd\}.
$$ 
Hence, \eqref{eq:claim_probability_concentrated_on_weak_solutions} implies $\nu=\delta_{\vcutd}$. 
Thus, \eqref{eq:scaling_limit_cutoff_proof_1} is a straightforward consequence of the Portmanteau theorem (see e.g., the argument below \cite[eq.\ (6.12)]{A22} for details).

\smallskip

The proof of \eqref{eq:claim_probability_concentrated_on_weak_solutions} is standard in the context of scaling limits (see e.g., \cite[Proposition 4.2]{FL19}, \cite[Step 1, Theorem 6.1]{A22}, or \cite[Proposition 3.7]{FGL21}). The cancellation of the martingale part in the SPDE \eqref{eq:navier_stokes_cutoff} is a consequence of \eqref{eq:time_regularity1} in Lemma \ref{l:time_regularity}. For the reader's convenience, we give some details. 
Fix a smooth divergence-free vector-field $\phi$ on $\T^3$, and consider the mapping $T_\phi $ defined as 
\begin{align*}
T_\phi (w)
&\stackrel{{\rm def}}{=} \langle w(t),\phi\rangle - \int_{\T^3} \big(u_0-\overline{u}_0\big)\cdot \phi  -\int_0^t \int_{\T^3} \Big(1+\frac{3\mu}{5}\Big) w\cdot \Delta \phi\\
&-\int_0^t \int_{\T^3} (w\otimes \overline{u}_0): \nabla \phi -\int_0^t \int_{\T^3} \phi_{p,q}^R (\cdot,w) [w\otimes w] : \nabla \phi
\end{align*}
for $w\in \mathcal{S}_K$. Next we show that $T_\phi:\mathcal{S}_{K} \to C([0,\infty))$ is continuous, where $C([0,\infty))$ is endowed with the Fr\'echet topology induced by the seminorms $w\mapsto\|w\|_{C([0,M])}$ where $M$ is an integer. To see the latter, it suffices to show that, for each $T<\infty$, the assignment 
$$
\mathcal{S}_{K} \ni w \mapsto\Big[ \int_0^{\cdot} \int_{\T^3} \phi_{p,q}^R (\cdot,w) [w\otimes w] : \nabla \phi\Big] \in C([0,T])
$$ 
is continuous. The latter is an immediate consequence of the continuity of the immersion $\mathcal{S}_{K} \subseteq L^{2p}(0,T;L^{2q}(\T^3;\R^3))$.
In particular, $T_\phi^\# \nu^N \to T_\phi^\# \nu$ weakly as probability measures on $C([0,\infty))$, where $T_\phi^{\#}$ denotes the pushforward of a measure.

Fix $\varepsilon_0\in (0,1)$ and $T<\infty$, and set $A_{\varepsilon_0,T} = \{f\in C([0,\infty))\,:\, \|f\|_{C([0,T])} > \varepsilon_0 \}$. Clearly, $A_{\varepsilon_0,T}$ is an open set of $C([0,\infty))$. Note that  
\begin{align*}
T_{\phi}^\# \nu^N (A_{\varepsilon_0,T})
= \nu^N((T_\phi )^{-1}(A_{\varepsilon_0,T}))\leq \frac{1}{\varepsilon_0^2} \,\E\|T_\phi (\vcut^N)\|_{C([0,T])}^2, 
\end{align*}
and 
\begin{align*}
T_\phi \vcut^N 
&= \int_{\T^3} \big(u_0^N-\overline{u}_0^N-u_0+\overline{u}_0 \big)\cdot \phi -\int_0^t \int_{\T^3} [\vcut^N\otimes (\overline{u}_0-\overline{u}_0^N)]: \nabla \phi\\
&- \int_0^\cdot \int_{\T^3} \vcut^N\cdot  \Big(\mathcal{P}^N-\frac{2\mu}{5}\Delta\Big)\phi +\sqrt{\frac{3\mu}{2}} \langle \mathcal{M}^N , \phi\rangle,
\end{align*}
where $\mathcal{P}^N$ is the It\^o-Stratonovich corrector defined in \eqref{eq:Ito_stratonovich_correction_pressure} with $\theta=\theta^N$ and $\mathcal{M}^N$ is as defined in \eqref{eq:martingale_vcutN_def}.
From the above identity, \eqref{eq:time_regularity1} and the energy inequality combined with the It\^o-Stratonovich convergence result of \cite[Theorem 5.1]{FL19} (or the periodic version of Proposition \ref{prop:error_estimates_operators}\eqref{it:error_estimates_operators2}), it follows that
$$ 
\lim_{N \to \infty }T_{\phi}^\# \nu^N (A_{\varepsilon_0,T}) =0.
$$
Thus, as a consequence of the Portmanteau theorem, the arbitrariness of $\varepsilon_0\in (0,1)$ and $T<\infty$, we have
$$
T_{\phi}^\# \nu (f\,:\, \|f\|_{C([0,T])}=0 \text{ for all }T<\infty) =1.
$$ 
By choosing $\phi$ in $\big\{ e^{2\pi \i k\cdot x}a_{k,\alpha}\,:\, k\in \Z^3_0,\ \alpha\in \{1,2\}\big\}$ (where, as in Subsection \ref{ss:probabilistic}, $\{a_{k,1},a_{k,2}\}$ is an orthonormal basis of $k^\perp$), a standard density argument shows that the above implies \eqref{eq:claim_probability_concentrated_on_weak_solutions} (see Definition \ref{def:NSE_cutoff}).
\end{proof}

\section{Proof of the main results}
\label{s:proof_main_results}
In this section, we prove the results stated in Section \ref{s:statements}. More precisely, Subsections \ref{ss:global_proof} and \ref{ss:enhanced_proof} are devoted to the proofs of Theorem \ref{t:global_NSE} and Corollary \ref{cor:enhanced_dissipation}, respectively.

\subsection{Global smooth solutions -- Proof of Theorem \ref{t:global_NSE}}
\label{ss:global_proof}
Due to Theorem \ref{t:scaling_limit_cutoff}, the last ingredient in the proof of Theorem \ref{t:global_NSE} is the following well-known result on the deterministic 3D NSEs with increased viscosity (see Subsection \ref{ss:scaling_intro}):
\begin{equation}
\label{eq:navier_stokes_cutoff_high_viscosity}
\left\{
\begin{aligned}
\partial_t \ud 
&  +\p\big[ \nabla\cdot (\ud \otimes \ud)\big]= \Big(1 + \frac{3}{5}\mu\Big)\Delta \ud, \\
\ud(0)&=u_0.
\end{aligned}
\right.
\end{equation}
Clearly, the definition of $(p,q)$-solutions given in Definition \ref{def:p_solution} trivially extends to the above system. 
As before, $\overline{u}_0=\int_{\T^3} u_0$ denotes the average of the initial data.

\begin{lemma}[Global well-posedness of 3D NSEs with high viscosity]
\label{l:NS_increased_viscosity}
Let $p,q\in (2,\infty)$ be such that 
\begin{equation}
\label{eq:subcriticality_NS_deterministic}
\frac{2}{p}+\frac{3}{q}<2.
\end{equation}
Then, for all $M\geq 1$, there exist constants $R_0,\mu_0>0$ for which the following assertion holds. For all $\mu\geq \mu_0$ and $u_0\in \Bs^{1-2/p}_{q,p}(\T^3)$ such that 
$$
\|u_0\|_{B^{1-2/p}_{q,p}(\T^3;\R^3)}\leq M,
$$
there exists a unique global-in-time $(p,q)$-solution $\ud$ to the deterministic {{\normalfont{3D NSEs}}} with increased viscosity \eqref{eq:navier_stokes_cutoff_high_viscosity} such that 
$$
\ud\in W^{1,p}_{\loc}([0,\infty);\Hs^{-1,q}(\T^3))\cap L^p_{\loc}([0,\infty);\Hs^{1,q}(\T^3)),$$ 
and it satisfies
\begin{align}
\label{eq:R0_uniform_bound_velocity_det}
\|\ud-\overline{u}_0\|_{L^{2p}(\R_+;L^{2q}(\T^3;\R^3))}\leq R_0.
\end{align}
\end{lemma}

The proof of the above is standard; see e.g., \cite[Theorem 5.1]{A24_global_small} for a similar result. 
As commented below Theorem \ref{t:global_NSE}, the condition \eqref{eq:subcriticality_NS_deterministic} ensures the subcriticality of $B^{1-2/p}_{q,p}(\T^3)$ for the 3D NSEs, see Subsection \ref{ss:scaling_intro}.
Clearly, solutions provided by Lemma \ref{l:NS_increased_viscosity} are smooth in space for every positive time, and \eqref{eq:R0_uniform_bound_velocity_det} holds with $L^{2p}(\R_+;L^{2q})$ replaced by $L^{p_0}(\R_+;L^{q_0})$ for some $p_0>2p$ and $q_0>2q$ depending only on $p,q$ (cf.\ the argument below \eqref{eq:G_smoothness2}). 

\begin{proof}[Proof of Theorem \ref{t:global_NSE}]
Below, $q,p\in (2,\infty)$, $M\geq 1$ and $\varepsilon\in (0,1)$ are fixed as in Theorem \ref{t:global_NSE}.
Without loss of generality, by compatibility of $(p,q)$-solutions (see e.g., \cite[Remark 2.5]{AV21_NS} or \cite[Corollary 5.11]{AV25_survey}), we assume $1<\frac{2}{p}+\frac{3}{q}<2$, and additionally $p<\frac{q}{3-q}$ in case $q<3$. For instance, we can replace $(p,q)$ by $q_0=3$ and $p_0\in (2,\infty)$ so that 
$$
B^{1-2/p}_{q,p}(\T^3)\embed B^{1-2/p_0}_{3,p_0}(\T^3).
$$ 
The latter is always possible as $p\in (2,\infty)$ and $\frac{2}{p}+\frac{3}{q}<2$. 

From Lemma \ref{l:NS_increased_viscosity}, there exist $\mu>0$ and $R>0$ depending only on $M,q$ and $p$ such that the 3D NSEs \eqref{eq:navier_stokes_cutoff_high_viscosity} have a unique global $(p,q)$-solution such that \eqref{eq:R0_uniform_bound_velocity_det} holds.
From the latter, it follows that 
$$
\vcutd\stackrel{{\rm def}}=\ud-\overline{u}_0
$$
is the unique global $(p,q)$-solution to \eqref{eq:navier_stokes_cutoff_det} with $R=R_0+1$.

Applying Theorem \ref{t:scaling_limit_cutoff} with $R=R_0+1$ and $\mu$ as above, it follows that there exists $N_0$ depending only on $M,q,p$ and $\varepsilon$ such that, for all $u_0\in \mathcal{B}_{p,q}(M)$ and $N\geq N_0$, it holds that 
\begin{equation}
\label{eq:conclusion_proof_global_main_step}
\P(\|\vcut^N-\vcutd \|_{L^{2p}(\R_+;L^{2q})}\leq 1/2)>1-\varepsilon.
\end{equation}
Fix $N \geq N_0$, and let $\tau_0$ be the stopping time given by 
$$
\tau_0 \stackrel{{\rm def}}{=}\inf\{t\in[0,\infty)\,:\, \|\vcut^N\|_{L^{2p}(0,t;L^{2q})}\geq R\} 
$$
where $\inf\emptyset\stackrel{{\rm def}}{=}\infty$. From \eqref{eq:conclusion_proof_global_main_step}, $R=R_0+1$ and \eqref{eq:R0_uniform_bound_velocity_det}, we infer
\begin{equation}
\label{eq:tau0_lives_long_proof_global}
\P(\tau_0=\infty)>1-\varepsilon,
\end{equation}
and 
$$
\phi^R_{p,q}(t,\vcut^N)=1 \ \ \text{ a.s.\ for all }\ t\in [0,\tau_0).
$$
In particular, $(U,\tau_0)$, where $U=\vcut^N+ \int_{\T^3} u_0$, is a local $(p,q)$-solution to the stochastic 3D NSEs \eqref{eq:navier_stokes_intro}.
From the maximality of $(p,q)$-solutions constructed in \cite[Theorem 2.4]{AV21_NS}, it follows that the $(p,q)$-solution $(u,\tau)$ to \eqref{eq:navier_stokes_intro} satisfies $\tau\geq \tau_0$ a.s. 
Thus, the claim \eqref{eq:global_NSE} of Theorem \ref{t:global_NSE} immediately follows from \eqref{eq:tau0_lives_long_proof_global}.
\end{proof}

\subsection{Enhanced dissipation -- Proof of Corollary \ref{cor:enhanced_dissipation}}
\label{ss:enhanced_proof}
Finally, we prove the assertions in Corollary \ref{cor:enhanced_dissipation}. As noticed below Theorem \ref{t:scaling_limit_cutoff}, the conclusion \eqref{eq:scaling_limit_cutoff} also holds with $(2p,2q)$ replaced by $(p_0,q_0)$ for some $p_0>2p$ and $q_0>2q$, which depend only on $(p,q)$. 
In particular, the proof of Theorem \ref{t:global_NSE} in Subsection \ref{ss:global_proof} yields the estimate \eqref{eq:estimate_subcritical_norms_3DNSEs}.
Therefore, \eqref{eq:enhanced_dissipation3} follows by interpolating the latter with \eqref{eq:enhanced_dissipation2} and using the arbitrariness of $\lambda$. 

It remains to show \eqref{eq:enhanced_dissipation2}. Again, by the proof of Theorem \ref{t:scaling_limit_cutoff}, it suffices to show the corresponding bound on the half-line for the $(p,q)$-solution of the corresponding cutoff problem \eqref{eq:navier_stokes_cutoff}:
\begin{equation}
\label{eq:navier_stokes_cutoff_mixing}
\left\{
\begin{aligned}
\partial_t \vcut^N &+ (\overline{u}_0\cdot \nabla) \vcut^N+\phi^{R}_{p,q}(\cdot,\vcut^N)\,\p [\nabla \cdot (\vcut^N\otimes \vcut^N)]\\
&\qquad  
=\Delta \vcut^N +\sqrt{\frac{3\mu}{2}}\sum_{k,\alpha}\theta_k^N\p[ (\sigma_{k,\alpha}\cdot\nabla) \vcut^N]\circ \dot{W}^{k,\alpha}_t,\\
 \nabla \cdot \vcut^N&=0,\\
\vcut^N(0,\cdot)&=u_0-\overline{u}_0.
 \end{aligned}
\right.
\end{equation} 
The following is the last ingredient for the proof of \eqref{eq:enhanced_dissipation2}.
Below, for convenience, we often write $w_0$ instead of $u_0-\overline{u}_0$.

\begin{lemma}
\label{l:iteration_enhanced}
Fix $R>0$, $M\geq 1$ and $\eta\in (0,1)$. Let $p,q\in (2,\infty)$ be such that \eqref{eq:condition_pq_subcritical_scaling_cutoff} holds.
Then, there exist $\mu_0>0$ and $N_0\geq 1$ such that, for all $\mu\geq \mu_0$, $N\geq N_0$ and $n\geq 1$, the $(p,q)$-solution to \eqref{eq:navier_stokes_cutoff_mixing} provided by Lemma \ref{l:global_cutoff} with data $u_0\in \mathcal{B}_{p,q}(M)$ satisfies
\begin{equation}
\label{eq:contraction_iterative_step}
\E\|\vcut^N(n+1)\|^2_{L^2}\leq \eta \,\E\|\vcut^N (n)\|^2_{L^2}.
\end{equation}
\end{lemma}

\begin{proof}
Here, we argue as in \cite[Section 4]{L23_enhanced}. Without loss of generality, we assume $\mu\geq 1$.
As in \eqref{eq:energy_balance_vcut}, the following pathwise energy estimate holds for the $(p,q)$-solution $\vcut^N$ to \eqref{eq:navier_stokes_cutoff_mixing}:
\begin{equation}
\label{eq:energy_balance_vcut_st}
\frac{1}{2}\|\vcut^N (t)\|_{L^2(\T^3)}^2 +\int_s^t \int_{\T^3} |\nabla \vcut^N|^2= \frac{1}{2}\| \vcut^N(s)\|_{L^2(\T^3)}^2
\end{equation}
a.s.\ 
for all $0\leq s<t<\infty$. In particular, $\int_0^\infty \int_{\T^3} |\nabla \vcut^N|^2\leq M^2$ and
\begin{equation}
\label{eq:L2_norm_iteration_enhanced_dissipation}
\E\|\vcut^N(n+1)\|^2_{L^2}\leq \E \int_n^{n+1} \|\vcut^N(s)\|_{L^2}^2\,\dd s.
\end{equation}
Let $(e^{t\Delta})_{t\geq 0}$ be the heat semigroup.
From \eqref{eq:navier_stokes_cutoff_mixing} (see Definition \ref{def:p_solution}), it follows that 
\begin{align*}
\vcut^N(t)=I_0(t)+ I_1(t)+I_2(t) +I_3(t)
\end{align*}
where 
\begin{align*}
I_0(t)&= e^{(1+\frac{3}{5}\mu)(t-n) \Delta } \vcut^N (n)
-\int_n^t e^{(1+\frac{3}{5}\mu)(t-s) \Delta }(\overline{u}_0\cdot \nabla) \vcut^N\,\dd s,\\
I_1 (t)&= 
\int_n^t e^{(1+\frac{3}{5}\mu)(t-s)\Delta } \Big(\frac{2}{5}\mu\Delta-\mathcal{P}^N\Big)\vcut^N(s)\,\dd s, \\
I_2(t)&=
-\int_n^t e^{(1+\frac{3}{5}\mu)(t-s) \Delta } \phi^R_{p,q}(\vcut^N)\p\big[\nabla \cdot (\vcut^N(s)\otimes \vcut^N(s))\big]\,\dd s, \\
I_3(t)&
=\sqrt{c_d\mu}\sum_{k,\alpha}\theta^N_k\int_n^t e^{(1+\frac{3}{5}\mu)(t-s)\Delta} \p[(\sigma_{k,\alpha}\cdot \nabla) \vcut^N(s) ]\,\dd W^{k,\alpha}_s,
\end{align*}
and $\mathcal{P}^N$ denotes the It\^o-Stratonovich corrector associated with the noise coefficient $\theta^N$, see \eqref{eq:Ito_stratonovich_correction_pressure}. 
Firstly, as in \cite[Section 4]{L23_enhanced} (see also \cite[Lemma 7.8]{AKT26}), using \eqref{eq:energy_balance_vcut_st}, one can check that there exists another universal constant $\g_0>0$ for which 
\begin{align*}
\E\|I_1\|_{L^2(n,n+1;L^2)}^2 +
 \E \|I_3\|_{L^2(n,n+1;L^2)}^2\lesssim \frac{1}{N^{\g_0}}\E\|\vcut^N(n)\|_{L^2}^2,
\end{align*}
with an implicit constant independent of $\mu$. 
Secondly, from \eqref{eq:energy_balance_vcut_st}, $\int_{\T^3} \vcut^N=0$ and standard energy estimates for the heat equation, it follows that (cf.\ \cite[Lemma 7.7]{AKT26})
\begin{align*}
\E\|I_0\|_{L^2(n,n+1;L^2)}^2
&\lesssim \frac{1}{\mu} \Big( \E\|\vcut^N(n)\|_{L^2}^2+ \E \int_n^{n+1} \|(\overline{u}_0\cdot \nabla) \vcut^N\|_{H^{-2}}^2\,\dd s \Big)\\
&\lesssim\frac{1}{\mu}  \E\|\vcut^N(n)\|_{L^2}^2
\end{align*}
as well as
\begin{align*}
\E \|I_2\|_{L^2(n,n+1;L^2)}^2
&
\lesssim \frac{1}{\mu}
\E\Big( \int_n^{n+1} \Big\|\phi_{p,q}^R(\cdot,\vcut^N)\,\p\big[\nabla \cdot (\vcut^N\otimes \vcut^N)\big]\Big\|_{H^{-1}}\Big)^2
\\
&
\lesssim \frac{1}{\mu}
\E\Big( \int_n^{n+1} \|\vcut^N\|_{L^4}^2 \Big)^2\\
&
\stackrel{(i)}{\lesssim} \frac{1}{\mu}
\E\Big( \int_n^{n+1} \|\nabla \vcut^N\|_{L^2}^2 \Big)^2
\stackrel{(ii)}{\lesssim} \frac{M^2}{\mu}\E\|\vcut^N(n)\|_{L^2}^2,
\end{align*}
where all the implicit constants are independent of $N$, in $(i)$ we applied the Sobolev embedding $H^1(\T^3)\embed L^6(\T^3)$ and the fact that $\vcut^N$ has zero mean, and in $(ii)$ \eqref{eq:energy_balance_vcut_st}.

The claim \eqref{eq:contraction_iterative_step} now follows by using the above estimates in \eqref{eq:L2_norm_iteration_enhanced_dissipation}.
\end{proof}

\begin{proof}[Proof of Corollary \ref{cor:enhanced_dissipation}]
As discussed at the beginning of Subsection \ref{ss:enhanced_proof}, it only remains to prove that \eqref{eq:enhanced_dissipation2} holds with $u-\int_{\T^3} u_0$ replaced by $\vcut^N$ for $\tau_0=\infty$ and suitably chosen $N$ and $\mu$. The latter is a standard consequence of the Borel-Cantelli lemma and Lemma \ref{l:iteration_enhanced}, see e.g.,
\cite[Theorem 2.3]{L23_enhanced}
or 
\cite[Subsection 7.1]{AKT26}.
\end{proof}

\appendix

\section{Local It\^o-Stratonovich corrector -- Proof of Lemma \ref{l:local_ito_stratonovich_corrector}}
\label{app:cancellation_ito_stratonovich}
Here, we prove the convergence result of Lemma \ref{l:local_ito_stratonovich_corrector} for the bilinear form $A^{N,\delta}$ associated with the local It\^o-Stratonovich corrector defined in \eqref{eq:local_stratonovich_corrector}, which was the essential tool to obtain the quantitative localized scaling limit result of Theorem \ref{t:universal_scaling_limit}. 
Throughout this subsection, for notational convenience, we employ the Einstein summation convention for repeated indices. 

We begin by recalling the setting:
\begin{align}
\label{eq:definition_A_bilinear_appendix}
&A^{N,\delta}(H,\Phi) \\
\nonumber
&=\int_{B} \Big(- D_d' H^\top : \nabla \Phi + \sum_{k,\alpha}  (\theta^N_k)^2 \q_{\T^d_B} [\chi(\sigma^{\delta}_{k,\alpha}\cdot H)] \cdot [ (\sigma^{\delta}_{-k,\alpha} \cdot  \nabla) \Phi]\Big)\,\dd x
\end{align}
where $D_d'= \frac{(d+1)}{d(d+2)}$, $B\subseteq \R^d$ is a ball with center $x_0\in \R^d$ and radius $\leq 1/4$, and
\begin{itemize}
 \item  $\T^d_B =x_0 + [-1/2,1/2)^d$ with periodic boundary conditions and $\q_{\T^d_B}$ the corresponding complementary Helmholtz projection on $\T^d_B$ (see Subsection \ref{ss:Helmh_div_free});
 \item $\theta^N=(\theta^N_k)_{k}$ denotes the noise coefficients defined in \eqref{eq:choice_thetaN};
 \item $\sigma_{k,\alpha}^{\delta}=\sigma_{k,\alpha}(x_0+\delta(\cdot-x_0))$ denotes the rescaled vector fields around $x_0$;
 \item $\chi\in C^{\infty}_{{\rm c}}(2B)$ is such that $\chi|_{B}=1$;
 \item $\Phi\in W^{2,\infty}(B;\R^d)\cap H^1_0(B;\R^d)$;
 \item $H\in L^2(2B;\R^{d\times d})$ satisfies 
 \begin{equation*}
\Tr (H)=0 \text{ a.e.\ on }2B, \ \ \  \text{ and }\ \  \
\nabla \cdot H=0 \ \text{ in }\D'(2B).
\end{equation*}
\end{itemize}
As claimed in Lemma \ref{l:local_ito_stratonovich_corrector}, here we prove that $A^{N,\delta}(H,\Phi)\to 0$ quantitatively with a rate depending only on the \emph{cumulative frequency} $L=\delta N$. The main obstacle is the lack of smoothness of $H$, which was of key importance in the proof of Theorem \ref{t:universal_scaling_limit}, see Subsection \ref{ss:quantitative_scaling_limits}.

\smallskip

We first outline the main idea behind the proof of Lemma \ref{l:local_ito_stratonovich_corrector}. Our method differs significantly from the Fourier-based techniques of Flandoli and Luo \cite[Appendix A]{FL19} (cf. \cite[Appendix A]{L23_enhanced}), which are limited to periodic boundary conditions. Here, we work on the spatial domain rather than frequency space. 
The primary advantage of our approach is that it makes the transition from the \emph{non-local} It\^o-Stratonovich correctors to their \emph{local} limit completely transparent. Specifically, this locality arises from the behavior of the covariance operator associated with the transport noise for the oscillating rescaled stochastic Stokes system \eqref{eq:turbulent_Stokes_scaling_quantitative}:
\begin{equation}
\label{eq:expression_covariance_operator}
\Cov^{N,\delta}(z)
=\sum_{k}(\theta_k^N)^2  e^{2\pi \i \delta z\cdot k} \Big(\mathrm{Id}-\frac{k\otimes k}{|k|^2}\Big),
\end{equation}
which clearly satisfies
\begin{align}
\label{eq:expression_covariance_operator2}
\Cov^{N,\delta}(x-y)
=\sum_{k,\alpha}(\theta_k^N)^2 \sigma^\delta_{k,\alpha}(x)\otimes \sigma^\delta_{-k,\alpha}(y) \quad \text{ and }\quad \Cov^{N,\delta}(0)=\frac{d-1}{d}\mathrm{Id},
\end{align}
where the second identity in the above is the $d$-dimensional analogue of \eqref{eq:symmetric_equal_diagonal_matrix}. 
In particular, the covariance function $\Cov^{N,\delta}$ is independent of $x_0$. Let us point out that the vanishing in the scaling limit $N\to \infty$ and $\delta=1$ of the covariance operator on $\T^d$ has already been noted and exploited, see e.g., \cite{FGL21,FGL21_eddy,FGL21_quantitative} and the references therein.

To see the appearance of the covariance operator $\Cov^{N,\delta}$ in the limiting behavior of \eqref{eq:definition_A_bilinear_appendix}, let us begin by noticing that, from Subsection \ref{ss:Helmh_div_free} and standard Fourier methods, we have
$$
[\q_{\T^d_B} f]_i(x) =\frac{1}{d}f^i(x)+  {\normalfont{\text{P.V.}}}\int_{\T^d_B} \Gamma_{i,j}(x-y)f_j (y)\,\dd y \ \ \text{ for }i\in \{1,\dots,d\},\ x\in \T^d_B,
$$
where ${\normalfont{\text{P.V.}}}$ stands for principal value, and the (singular) kernel $\Gamma_{i,j}$ defined on $\T^d\setminus\{0\}$ is uniquely identified via the identity:
\begin{equation}
\label{eq:definition_gamma_torus}
 {\normalfont{\text{P.V.}}}(\Gamma_{i,j})=\mathcal{F}_{\T^d}^{-1}(m_{i,j}),\quad \text{ where } \quad m_{i,j}(k)=\frac{k_ik_j}{|k|^2}-\frac{\delta_{i,j}}{d} \ \text{ for }\ k\neq 0,
\end{equation}
and $m_{i,j}(0)=-\delta_{i,j}/d$.
The fact that $\mathcal{F}_{\T^d}^{-1}(m_{i,j})$ is a principal-value distribution follows by periodizing the corresponding Euclidean homogeneous case. Indeed, the map $\xi\mapsto (\xi_i\xi_j)/|\xi|^2-\delta_{i,j}/d$ is smooth and homogeneous of degree zero on $\R^d\setminus\{0\}$, and has zero spherical mean: $\int_{\mathbb{S}^{d-1}}(\sigma_i\sigma_j-\delta_{i,j}/d)\,\dd \mathcal{H}^{d-1}(\sigma)=0$ where $\mathbb{S}^{d-1}\subseteq \R^d$ is the unit sphere. 
Hence, its Euclidean inverse Fourier transform is a homogeneous principal-value distribution of degree $-d$; see \cite[Proposition 2.4.7]{Grafakos1}. Restricting the multiplier to $\Z^d\setminus\{0\}$ and applying the standard periodization argument yields the corresponding periodic principal-value distribution; see, for instance, \cite[Theorems 6.2 and 6.4]{RuzhanskyTurunen}.

Due to $\chi|_{B}=1$ and $\supp\chi\subseteq 2B$, we can express the second term of $A^{N,\delta}$ as
\begin{align}
\label{eq:expression_via_covariance_ito_strat}
&\sum_{k,\alpha}  (\theta^N_k)^2\int_B \q_{\T^d_B} [\chi(\sigma^{\delta}_{k,\alpha}\cdot H)] \cdot [ (\sigma^{\delta}_{-k,\alpha} \cdot  \nabla) \Phi]\,\dd x= E^{N,\delta}(H,\Phi)+F(H,\Phi)  ,
\end{align}
where, employing the Einstein summation convention for repeated indices, 
\begin{align*}
&E^{N,\delta}(H,\Phi) 
\\
&\stackrel{{\rm def}}{=}\sum_{k,\alpha}  (\theta^N_k)^2
{\normalfont{\text{P.V.}}}\,
\int_{ B\times 2B} \big(\Gamma_{i,j}(x-y)\cdot [\sigma_{k,\alpha}^\delta(y) \cdot H (y)]_j\, \chi(y) \big)\cdot [(\sigma_{-k,\alpha}^{\delta}\cdot \nabla )\Phi]_i\,\dd x\,\dd y \\
&=
{\normalfont{\text{P.V.}}}
\int_{B\times 2B} \Gamma_{i,j}(x-y)  \Cov^{N,\delta}_{m,n}(x-y)\chi(y) H_{m,j}(y)\partial_n \Phi_i \,\dd x\,\dd y,
\end{align*}
and 
\begin{align*}
F(H,\Phi)
\stackrel{{\rm def}}{=}\frac{1}{d}\sum_{k,\alpha}  (\theta^N_k)^2\int_B (\sigma^{\delta}_{k,\alpha}\cdot H)\cdot  (\sigma^{\delta}_{-k,\alpha} \cdot  \nabla) \Phi 
=\frac{d-1}{d^2}\int_B  H^\top :  \nabla \Phi.
\end{align*}
Let us point out that in the above formula, we used the second identity in \eqref{eq:expression_covariance_operator2}.

\smallskip

Having the representation formula \eqref{eq:expression_via_covariance_ito_strat} and subsequent identities at our disposal, we can formulate the last observation leading to the proof of Lemma \ref{l:local_ito_stratonovich_corrector}. More precisely, from Lemma \ref{l:Bessel_property} and using that finite sums of the form $\sum_{k} f_k (x) g_k(y)$ are dense in $L^2(B\times 2B)$, one can check that, for all $F\in L^2(B\times 2B;\R^{d\times d})$,  
\begin{equation}
\label{eq:convergence_zero_covariance}
\int_{B\times 2B} \Cov^{N,\delta}(x-y): F(x,y)\,\dd x\,\dd y  \to 0 \quad \text{ as } \quad \delta N\to \infty,
\end{equation}
see Lemma \ref{l:convergence_covariance} for a quantitative version.
Comparing \eqref{eq:expression_via_covariance_ito_strat} and \eqref{eq:convergence_zero_covariance}, the obstacle to the convergence is the singularity of $\Gamma_{i,j}$ on the diagonal $x=y$. More precisely, for all $\varepsilon>0$ and all cutoff functions $\om_\varepsilon$ such that $\supp\om_{\varepsilon}\subseteq B_{2\varepsilon}(0)$,
\begin{align*}
\int_{B\times 2B}(1-\om_\varepsilon(x-y)) \Gamma_{i,j}(x-y) \Cov^{N,\delta}_{m,n}(x-y) \chi(y) H_{m,j}(y)\partial_n \Phi_i \,\dd x\,\dd y\to0
\end{align*}
as $\delta N\to \infty$.
 In particular, any distributional limit point as $\delta N \to \infty$ of the kernel associated with the bilinear form \eqref{eq:expression_via_covariance_ito_strat} by means of the Schwartz kernel theorem (see e.g., \cite[Theorem 6.1]{TayPDE1}) is concentrated on the diagonal of $B\times B$. Hence, the corresponding operator is \emph{local}.

\smallskip 

The proof of Lemma \ref{l:local_ito_stratonovich_corrector} is given in Subsection \ref{ss:local_ito_stratonovich_corrector_proof} and essentially follows the above reasoning. One key ingredient is a quantitative version of the assertion \eqref{eq:convergence_zero_covariance}, and it will be investigated in the following subsection.

\subsection{Decay estimates for the covariance operator}
\label{ss:decay_covariance}
As discussed above \eqref{eq:convergence_zero_covariance}, the latter convergence is a consequence of Lemma \ref{l:Bessel_property} and the $L^2$-density of finite sums of products of functions of one variable. To make a quantitative version of the latter, we need to assume smoothness in one of the variables. This is the content of the following result.
For later use, let us note that  
\begin{equation}
\label{eq:covariance_Linfty_bound}
\|\Cov^{N,\delta}\|_{L^\infty(\R^d)}\lesssim_d 1,
\end{equation}
with an implicit constant independent of $N$ and $\delta$, as $\|\theta^N\|_{\ell^2}=1$ for all $N\geq 1$.

\begin{lemma}[Quantitative convergence of the tested covariance operator]
\label{l:convergence_covariance}
Let $B_1,B_2$ and $B$ be balls in $\R^d$ such that $B_1\cup B_2\subseteq B$, and $B$ has radius $\leq 1/2$. Then there exist constants $C,\g>0$ such that, for all $N\geq 1$, $\delta\in (0,1]$, and $F\in L^2(B_1\times B_2;\R^{d\times d})$ such that $\nabla_x F\in L^2(B_1\times B_2)$, it holds that 
\begin{equation}
\label{eq:covariance_operator_decay}
\Big|\int_{B_1\times B_2} \Cov^{N,\delta}(x-y): F(x,y)\,\dd x \,\dd y\Big| \leq \frac{C}{(\delta N)^{\g}}\big(\|F\|_{L^2(B_1\times B_2)}+\|\nabla_x F\|_{L^2(B_1\times B_2)}\big).
\end{equation}
\end{lemma}

Of course, by symmetry, one can replace $\|\nabla_x F\|_{L^2(B_1\times B_2)}$ by $\|\nabla_y F\|_{L^2(B_1\times B_2)}$ on the right-hand side of \eqref{eq:covariance_operator_decay} together with a corresponding regularity assumption.

\begin{proof}
We divide the proof into two steps. 

\smallskip

\emph{Step 1: There exists a constant $C_0>0$ such that, for all $N\geq 1$, $\delta\in (0,1]$, and $f_i\in L^2(B_i;\R^{d})$ with $i\in \{1,2\}$, it holds that }
\begin{equation}
\label{eq:covariance_operator_decay_diagonal}
\Big|\int_{B_1\times B_2} [\Cov^{N,\delta}(x-y)\cdot f_1(x)]\cdot f_2(y)\,\dd x\,\dd y\Big| \leq \frac{C_0}{(\delta N)^{d}} \|f_1\|_{L^2(B_1)}\|f_2\|_{L^2(B_2)}.
\end{equation}
Replacing $(f_1,f_2)$ by $(\one_{B}f_1,\one_{B}f_2)$, it suffices to prove the claim of Step 1 with $B_1=B_2$. 
Let $\overline{\Cov}^{N,\delta}$ be the bilinear form associated with $
{\Cov}^{N,\delta}$ on $L^2(B;\R^d)$:
\begin{equation}
\label{eq:covariance_bilinear_form}
\overline{\Cov}^{N,\delta}(f_1,f_2)=\int_{B\times B} [\Cov^{N,\delta}(x-y)\cdot f_1(x)]\cdot f_2(y)\,\dd x\,\dd y.
\end{equation}
Clearly, due to \eqref{eq:expression_covariance_operator}, $\overline{\Cov}^{N,\delta}$ is symmetric.
It follows from Lemma \ref{l:Bessel_property} that there exists a constant $C>0$ such that, for all $N\geq 1$, $\delta\in (0,1]$ and $f\in L^2(B;\R^d)$ for which
$$
0\leq \overline{\Cov}^{N,\delta}(f,f)=\sum_{k,\alpha} (\theta^N_k)^2\Big|\int_{B} \sigma^\delta_{k,\alpha}\cdot f  \Big|^2
\leq \frac{C}{(\delta N)^{d}} \|f\|_{L^2(B)}^2.
$$
By the polarization identity, for all $f_1,f_2\in L^2(B;\R^d)$ satisfying $\|f_i\|_{L^2(B;\R^d)}=1$,
\begin{equation*}
|\overline{\Cov}^{N,\delta}(f_1,f_2)|\leq\frac{1}{4}\big(|\overline{\Cov}^{N,\delta}(f_1+f_2,f_1+f_2)|+|\overline{\Cov}^{N,\delta}(f_1-f_2,f_1-f_2)|\big) \leq \frac{2C}{(\delta N)^d} .
\end{equation*}
Thus, \eqref{eq:covariance_operator_decay_diagonal} follows by applying the above estimate with $(f_1,f_2)$ replaced by their $L^2$-normalizations $(f_1/\|f_1\|_{L^2(B;\R^d)},f_2/\|f_2\|_{L^2(B;\R^d)})$.

\smallskip

\emph{Step 2: Proof of \eqref{eq:covariance_operator_decay}.}
Without loss of generality, we may assume $F_{n,m}=f$ for some $f\in L^2(B_2;H^1(B_1))$ and $n,m\in \{1,\dots,d\}$, and $F_{i,j}=0$ if $(i,j)\neq (n,m)$. 
Hence, for a.e.\ $y\in B_2$, it holds that $f(\cdot,y)\in H^1(B_1)$. Thus, the generalized Fourier expansion yields
$$
f(x,y)= \sum_{j\geq 1} e_j (x) \wh{f}_j(y)\quad \text{ and }\quad \wh{f}_j(y)=\int_{B_1} e_j(x') f(x',y)\,\dd x',
$$
where $(\lambda_j)_{j\geq 1}$ and $(e_j)_{j\geq 1}$ are the eigenvalues and eigenfunctions of the Dirichlet Laplacian on $B_1$, respectively. Next, fix $\sigma\in (0,1/2)$. Recall that   
$$
H^1(B_1)\embed H^{\sigma}(B_1) = 
\Do((-\Delta_{B_1}^{\Dir})^{\sigma/2}),
$$
where $\Delta_{B_1}^{\Dir}$ is the Dirichlet Laplacian on $L^2(B_1)$, see Subsection \ref{ss:Helmholtz_domains}.
Therefore,
\begin{align*}
\sum_{j\geq 1}\lambda_j^{\sigma}\int_{B_2}  |\wh{f}_j(y)|^2\,\dd y\lesssim_\sigma
\int_{B_2}\| f(\cdot,y) \|_{H^1(B_1)}^2\,\dd y.
\end{align*}
From Weyl's law $\lambda_j \sim j^{2/d}$ and the $L^2$-orthogonality of the eigenfunctions $(e_j)_{j\geq 1}$, we have, for all integers $R$,
\begin{align}
\label{eq:estimate_high_order_peices_generalized_Fourier}
\int_{B_1\times B_2} \Big|\sum_{j\geq  R} e_j(x) \wh{f}_j(y) \Big|^2 \,\dd x \,\dd y
&=
\sum_{j\geq R}\int_{B_2} | \wh{f}_j(y) |^2\,\dd y\\
\nonumber
&\leq 
R^{-(2\sigma)/d}
\sum_{j\geq R} \lambda^{\sigma}_j\int_{B_2} | \wh{f}_j(y) |^2\,\dd y\\
\nonumber
&\lesssim_\sigma
R^{-(2\sigma)/d}
\int_{B_2}\|f(\cdot,y) \|_{H^1(B_1)}^2\,\dd y.
\end{align}
Therefore, letting $\Cov_{n,m}^{N,\delta}$ denote the components of the covariance function $\Cov^{N,\delta}$,  and using the notation introduced in \eqref{eq:covariance_bilinear_form}, for all integers $R$, it holds that 
\begin{align*}
&\Big|\int_{B_1\times B_2} \Cov^{N,\delta}(x-y): F(x,y)\,\dd x \,\dd y\Big|\\
&\qquad\qquad\leq \Big|\int_{B_1\times B_2} \Cov^{N,\delta}_{n,m}(x-y) \sum_{j\geq R} e_j(x) \wh{f}_j(y)\,\dd x\,\dd y\Big|
+ \sum_{j=1}^R |\overline{\Cov}^{N,\delta}(e_j,\wh{f}_j)|\\
&\qquad\qquad\stackrel{(i)}{\leq} C \int_{B_1\times B_2} \Big| \sum_{j\geq R} e_j(x) \wh{f}_j(y)\Big|\,\dd x\,\dd y
+ \frac{C}{(\delta N)^d}\sum_{j=1}^R \|e_j\|_{L^2(B_1)}\|\wh{f}_j\|_{L^2(B_2)}\\
&\qquad\qquad\stackrel{(ii)}{\leq} C \Big(\frac{1}{R^{\sigma/d}} 
+ \frac{R}{(\delta N)^d}\Big) \Big(
\int_{B_2}\|f(\cdot,y)\|_{H^1(B_1)}^2\,\dd y\Big)^{1/2},
\end{align*}
where in $(i)$ we used \eqref{eq:covariance_Linfty_bound} and Step 1, and in $(ii)$ \eqref{eq:estimate_high_order_peices_generalized_Fourier}.
Letting $R\eqsim (\delta N)^{d^2/(d+\sigma)}$, we have  
$R^{-(\sigma)/d} 
+ R (\delta N)^{-d}\lesssim (\delta N)^{-\gamma}$ where $\gamma>0$ depends only on $d$ and $\sigma$.
\end{proof}

\subsection{Proof of Lemma \ref{l:local_ito_stratonovich_corrector}}
\label{ss:local_ito_stratonovich_corrector_proof}
Recall that the setting of Lemma \ref{l:local_ito_stratonovich_corrector} was already introduced at the beginning of Appendix \ref{app:cancellation_ito_stratonovich}. In particular, $E^{N,\delta}(H,\Phi)$ and $\Gamma_{i,j}$ are as defined below \eqref{eq:expression_via_covariance_ito_strat} and \eqref{eq:definition_gamma_torus}, respectively.
Arguing as in \eqref{eq:expression_via_covariance_ito_strat}, from standard Fourier techniques, for $f=(f_j)_{j=1}^d\in L^2(\T^d_B;\R^d)$, the $i$-th component of $\q_{\T^d_B}f(x)= \frac{1}{d}f^i(x)+{\normalfont{\text{P.V.}}} \int_{\T^d_B} \Gamma_{i,j}(x-y)f_j(y)\,\dd y$ where the latter is given by 
\begin{align}
\label{eq:helmholtz_projection_PV}
{\normalfont{\text{P.V.}}} \int_{\T^d_B} \Gamma_{i,j}(x-y)f_j(y)\,\dd y
= 
\lim_{r\downarrow 0}\int_{\T^d_B} \one_{\{\mathrm{d}_{\T^d_B}(x,y)\geq r\} }\Gamma_{i,j}(x-y)f_j(y)\,\dd y,
\end{align}
where $\mathrm{d}_{\T^d_B}(x,y)$ denotes the distance on $\T^d_B$, and the limit holds, for instance, in the weak $L^2(\T^d_B;\R^d)$-topology (see e.g., \cite[Chapter 4]{Grafakos1} for more refined results). Let $\mathcal{F}^{-1}_{\R^d}$ be the inverse Fourier transform on $\R^d$ and let $\mathcal{G}_{i,j}$ be the kernel defined on $\R^d\setminus \{0\}$ such that
\begin{equation}
\label{eq:kernel_function_onRd}
{\normalfont{\text{P.V.}}} (\mathcal{G}_{i,j})\stackrel{{\rm def}}{=} \mathcal{F}^{-1}_{\R^d}\Big[\xi\mapsto\Big(\frac{\xi_i\xi_j}{|\xi|^2}-\frac{\delta_{i,j}}{d}\Big) \Big],
\end{equation} 
where, as in \eqref{eq:definition_gamma_torus}, the distribution on the right-hand side of the above is of principal value type because the corresponding symbol has mean zero over the unit sphere $\mathbb{S}^{d-1}$.
Moreover, $\mathcal{G}_{i,j}$ is a Calder\'on-Zygmund kernel, and it is given by
\begin{align}
\label{eq:kernel_function_onRd2}
\mathcal{G}_{i,j}(x)=
\frac{1}{|\mathbb{S}^{d-1}|}
\frac{\delta_{ij}|x|^2-dx_ix_j}{|x|^{d+2}}.
\end{align}
By using the Green's kernel of the Laplacians on $\R^d$ and $\T^d$ together with a simple cutoff argument, one can readily check that 
$$
\Gamma_{i,j}(x-y)= \mathcal{S}_{i,j}(x,y)+ 
\mathcal{G}_{i,j} (x-y)\ \text{ for } x\in 2B \text{ and } y\in \supp\chi,
$$
where $\mathcal{S}_{i,j}$ is a smooth function on $\T^d\times \T^d$.
Let $\om\in C^{\infty}_{{\rm c}}(\R^d)$ be such that $\om|_{B_1(0)}=1$ and $\om|_{\R^d\setminus B_2(0)}=0$, and set  $\om_\varepsilon=\om(\cdot/\varepsilon)$. Take $\Psi\in W^{2,\infty}(2B;\R^d)$ such that 
\begin{equation}
\label{eq:choice_Phi_phi}
\Psi|_{B}=\Phi\qquad \text{ and }\qquad\|\Psi\|_{W^{2,\infty}(2B)}\lesssim \|\Phi\|_{W^{2,\infty}(B)}.
\end{equation}
Below $H_{i,m}$ and $\Phi_j$ denote 
the components of the matrix field $H$ and vector field $\Phi$, respectively. The same convention is used for the
remaining objects.
Arguing as in \eqref{eq:expression_via_covariance_ito_strat}, using $\supp\chi\subseteq 2B$ and $\#\{k\,:\, \theta^N_k \neq 0\}<\infty$ for all $N\geq 1$, we have   
\begin{align}
\label{eq:decomposition_Ito_stratonovich_proof_convergence}
&E^{N,\delta}(H,\Phi)\\
\nonumber
&={\normalfont{\text{P.V.}}}\int_{B} \int_{2B} \Gamma_{i,j}(x-y) \Cov^{N,\delta}_{m,n}(x-y) \chi(y)H_{m,j}(y)\partial_n \Phi_i(x)\,\dd y\,\dd x\\
\nonumber
&={\normalfont{\text{P.V.}}}\int_{B} \int_{2B} \mathcal{G}_{i,j}(x-y)\om_{\varepsilon}(x-y) \Cov^{N,\delta}_{m,n}(x-y) \chi(y)H_{m,j}(y)\partial_n \Psi_i(y)\,\dd y\,\dd x\\
\nonumber
&+R_1+ R_2+R_3,
\end{align}
where
\begin{align*}
R_1&= \int_{B} \int_{2B} \mathcal{G}_{i,j}(x-y) \Cov^{N,\delta}_{m,n}(x-y)\om_{\varepsilon}(x-y)  \chi(y)H_{m,j}(y)(\partial_n \Psi_i(x)-\partial_n \Psi_i(y))\,\dd y\,\dd x\\
R_2&= \int_{B} \int_{2B} \mathcal{G}_{i,j}(x-y) \Cov^{N,\delta}_{m,n}(x-y)(1-\om_{\varepsilon}(x-y) ) \chi(y)H_{m,j}(y)\partial_n \Phi_i(x)\,\dd y\,\dd x\\
R_3&=\int_{B} \int_{2B} \mathcal{S}_{i,j}(x,y) \Cov^{N,\delta}_{m,n}(x-y) \chi(y)H_{m,j}(y)\partial_n \Phi_i(x)\,\dd y\,\dd x
\end{align*}
and for the principal value integral in \eqref{eq:decomposition_Ito_stratonovich_proof_convergence}, we used a similar convention as in \eqref{eq:helmholtz_projection_PV}.
As the following lemma shows, the remainders $R_1$, $R_2$, and $R_3$ exist as Lebesgue integrals and can be suitably controlled.

\begin{lemma}[Estimates for the remainders I]
\label{l:remainders1}
With the previous notation, there exist constants $C,\g>0$ such that, for all $\varepsilon\in (0,1]$, $\delta\in (0,1]$, $N\geq 1$, $H\in L^2(2B;\R^{d\times d})$ and $\Phi\in W^{2,\infty}(B;\R^d)$, it holds that 
$$
|R_1|+|R_2|+ |R_3|\leq C \|\Phi\|_{W^{2,\infty}(B)} \|H\|_{L^2(2B)}\big(\varepsilon+ \varepsilon^{-d/2-1} (\delta N)^{-\g}\big),
$$
where $\Psi$ is chosen as in \eqref{eq:choice_Phi_phi}.
\end{lemma}

In order not to interrupt the flow of the presentation, we postpone the proof of the above to Subsection \ref{ss:remainders_lemma}.
To proceed further, we elaborate on the first term on the right-hand side of \eqref{eq:decomposition_Ito_stratonovich_proof_convergence}.

By Fubini's theorem, it follows that 
\begin{align}
\label{eq:decomposition_Ito_stratonovich_proof_convergence1}
&{\normalfont{\text{P.V.}}}\int_{B}
\int_{2B} \mathcal{G}_{i,j}(x-y)\om_{\varepsilon}(x-y) \Cov^{N,\delta}_{m,n}(x-y) \chi(y)H_{m,j}(y)\partial_n \Psi_i(y)\,\dd y\,\dd x\\
\nonumber
&=\lim_{r\downarrow 0}\int_{B}
\int_{2B} \one_{\{|x-y|\geq r\}}\, \mathcal{G}_{i,j}(x-y)\om_{\varepsilon}(x-y) \Cov^{N,\delta}_{m,n}(x-y) \chi(y)H_{m,j}(y)\partial_n \Psi_i(y)\,\dd y\,\dd x\\
\nonumber
&=
\int_{2B}\Big({\normalfont{\text{P.V.}}}\,\int_{B} 
\mathcal{G}_{i,j}(x-y)\om_{\varepsilon}(x-y) \Cov^{N,\delta}_{m,n}(x-y)\,\dd x\Big) \chi(y)H_{m,j}(y)\partial_n \Psi_i(y)\,\dd y\\
\nonumber
&=
\int_{2B}\Big({\normalfont{\text{P.V.}}}\,\int_{B-y} 
\mathcal{G}_{i,j}(z)\om_{\varepsilon}(z) \Cov^{N,\delta}_{m,n}(z)\,\dd z\Big) \chi(y)H_{m,j}(y)\partial_n \Psi_i(y)\,\dd y\\
\nonumber
&=\sum_{k}(\theta^N_k)^2 \Big(\delta_{m,n}-\frac{k_mk_n}{|k|^2}\Big)\wh{\mathcal{G}_{i,j}}(k)\int_B H_{m,j}\partial_n \Phi_i+R_4
\end{align}
where we used the explicit value of the covariance function \eqref{eq:expression_covariance_operator}, and we set
\begin{align}
\label{eq:def_R4}
R_4 &=\sum_{k}(\theta^N_k)^2 \Big(\delta_{m,n}-\frac{k_mk_n}{|k|^2}\Big)\times\\
\nonumber
& \Big[
 \int_{2B} \chi(y)H_{m,j}(y)\partial_n \Psi_i(y)\Big({\normalfont{\text{P.V.}}}\int_{B-y} \mathcal{G}_{i,j}(x)\om_{\varepsilon}(x) e^{2\pi \i \delta k\cdot x} \,\dd x\Big)\,\dd y\\
\nonumber
&-\wh{\mathcal{G}_{i,j}}(k)\int_{B}  H_{m,j}\partial_n \Phi_i\Big],
\end{align}
where the principal value is understood in the weak $L^2$-topology as in \eqref{eq:helmholtz_projection_PV}.
For the remainder $R_4$, we have the following lemma. 

\begin{lemma}[Estimates for the remainders II]
\label{l:remainders2}
With the previous notation, there exists $C>0$ such that, for all $\varepsilon\in (0,1]$, $\delta\in (0,1]$, $N\geq 1$, $H\in L^2(2B;\R^{d\times d})$ and $\Phi\in W^{2,\infty}(B;\R^d)$, it holds that 
\begin{align*}
|R_4|
\leq C \| H\|_{L^2(2B)} \| \Phi\|_{W^{2,\infty}(B)}(\varepsilon^{-1} (\delta N)^{-1} + \sqrt{\varepsilon}) ,
\end{align*}
where $\Psi$ is chosen as in \eqref{eq:choice_Phi_phi}.
\end{lemma}

Given the above results, we can prove Lemma \ref{l:local_ito_stratonovich_corrector}.

\begin{proof}[Proof of Lemma \ref{l:local_ito_stratonovich_corrector}]
We divide the proof into three steps.

\smallskip

\emph{Step 1: It holds that 
\begin{equation}
\label{eq:step1_representation_limiting_A_proof}
\Big|\sum_{k}(\theta^N_k)^2 \Big(\delta_{n,m}- \frac{k_n k_m}{|k|^{2}}\Big)\wh{\mathcal{G}_{i,j}}(k) - A^{i,j}_{n,m} \Big| \leq \frac{C}{ N},
\end{equation}
where}
\begin{equation}
\label{eq:explicit_expression_A_local_Ito}
A^{i,j}_{n,m}\stackrel{{\rm def}}{=}\frac{2}{d^2(d+2)} \delta_{n,m}\delta_{i,j} - \frac{1}{d(d+2)}\big(\delta_{n,i}\delta_{m,j} + \delta_{n,j}\delta_{m,i}\big).
\end{equation}
\emph{Since $\delta\leq 1$, the right-hand side of \eqref{eq:step1_representation_limiting_A_proof} can be estimated by $C(\delta N)^{-1}$.}

\smallskip

Recall that $\wh{\mathcal{G}_{i,j}}(k) =(k_i k_j)/|k|^2-\delta_{i,j}/d$ by \eqref{eq:kernel_function_onRd}, and that $\theta^N$ are defined in \eqref{eq:choice_thetaN}. From the expression of $\theta^N$, it follows that $\sum_{k}(\theta^N_k)^2 \big(\delta_{n,m}- (k_n k_m)/|k|^{2}\big)\wh{\mathcal{G}_{i,j}}(k)$ is an approximation of a quotient of Riemann integrals on the annuli $\{ 1 \leq |\xi| \leq 2\}$:
\begin{align}
\label{eq:limit_Riemann_sum_annulus}
&\sum_{k}(\theta^N_k)^2 \Big(\delta_{n,m}- \frac{k_n k_m}{|k|^{2}}\Big)\Big(\frac{k_i k_j}{|k|^2} -\frac{\delta_{i,j}}{d}\Big)\\
\nonumber
&= \Big[\frac{1}{N^d}\sum_{N\leq |k| \leq 2N} \Big|\frac{N}{k}\Big|^{2a}\Big]^{-1} \Big[\frac{1}{N^d}\sum_{N\leq |k| \leq 2N} \Big|\frac{N}{k}\Big|^{2a} \Big(\delta_{n,m}- \frac{k_n k_m}{|k|^{2}}\Big)\Big(\frac{k_i k_j}{|k|^2}  -\frac{\delta_{i,j}}{d}\Big)\Big] \\
\nonumber
&=
\Big(\int_{\{ 1 \leq |\xi| \leq 2\}} \frac{1}{|\xi|^{2a}}\,\dd \xi\Big)^{-1} \int_{\{ 1 \leq |\xi| \leq 2\}} \frac{1}{|\xi|^{2a}} \Big(\delta_{n,m} -\frac{\xi_{n}\xi_{m}}{|\xi|^2} \Big)\Big(\frac{\xi_i \xi_j}{|\xi|^2}  -\frac{\delta_{i,j}}{d}\Big)\,\dd \xi+  \mathcal{O}(N^{-1})\\
\nonumber
&=\fint_{\mathbb{S}^{d-1}} (\delta_{n,m} -\sigma_{n}\sigma_{m})\Big(\sigma_i\sigma_j-\frac{\delta_{i,j}}{d}\Big)\,\dd \mathcal{H}^{d-1}(\sigma) +\mathcal{O}(N^{-1}),
\end{align}
where $\mathcal{O}(N^{-1})$ denotes the error between the left-hand side and the first term on the right-hand side.
The convergence rate $\mathcal{O}(N^{-1})$ in \eqref{eq:limit_Riemann_sum_annulus} follows from standard error estimates for Riemann sums on regular domains.

It is routine to check that the spherical integral in the last equality of \eqref{eq:limit_Riemann_sum_annulus} coincides with $A_{n,m}^{i,j}$ in \eqref{eq:explicit_expression_A_local_Ito}. This completes the proof of Step 1.

\smallskip

\emph{Step 2: Let $A^{i,j}_{n,m}$ be as in Step 1. Then, for all $\Phi\in W^{2,\infty}(B)$ such that $\Phi|_{\partial B}=0$, and traceless $H\in L^2(B;\R^{d\times d})$ satisfying $\nabla \cdot H=0$ in $\D'(B)$, it holds that} 
\begin{align*}
\sum_{i,j,n,m}A^{i,j}_{n,m}\int_{B }H_{m,j} \partial_{n} \Phi_i=  \frac{2}{d^2(d+2)}  \int_{B} H^\top : \nabla \Phi. 
\end{align*}
Since $\sum_{j,m}\int_{B} H_{m,j}\partial_m \Phi_j =\int_B H^\top: \nabla \Phi$, 
due to \eqref{eq:explicit_expression_A_local_Ito}, it remains to prove that:
\begin{align}
\label{eq:claim_step_2_ito_stat_corrector_local}
\sum_{i,j}  \int_{B }H_{j,j} \partial_{i} \Phi_i = 0, 
\qquad \text{ and }\qquad \sum_{j,m}\int_{B }H_{m,j} \partial_{j} \Phi_m = 0.
\end{align}
The first identity in \eqref{eq:claim_step_2_ito_stat_corrector_local} follows from $\Tr(H)=0$. For the second one, let $(\Phi^\ell)_{\ell}\subseteq C^{\infty}_{{\rm c}}(B;\R^d)$ be a sequence such that $\Phi^\ell\to \Phi$ as $\ell\to \infty$ in $H^1(B;\R^d)$ (this is possible since $\Phi|_{\partial B}=0$ and $\Phi\in W^{2,\infty}(B;\R^d)$). From $(\nabla \cdot H)_m= \partial_j H_{m,j}$ in $\D'(B)$, we infer
\begin{align*}
\sum_{m,j}  \int_{B }H_{m,j} \partial_{j} \Phi_m 
&=\lim_{\ell\to \infty} \sum_{m,j} 
\int_{B} H_{m,j} \partial_j \Phi^\ell_m=0
\end{align*}
as desired.

\smallskip

\emph{Step 3: Conclusion}. 
Recall that $\delta N\geq 1$ by assumption. 
Without loss of generality, we may assume that Lemma \ref{l:remainders1} holds with $\g\leq 1$. Thus, applying Lemmas \ref{l:remainders1} and \ref{l:remainders2} with the choice
$
\varepsilon= (\delta N)^{-(2\g)/(d+3)} $, we get the following estimates for the remainders
$$
\sum_{1\leq i\leq 4} |R_i|\leq C [\sqrt{\varepsilon}+ \varepsilon^{-(d+2)/2}(\delta N)^{-\g}] = C (\delta N)^{-\g/(d+3)}.
$$
Hence, the claim of Lemma \ref{l:local_ito_stratonovich_corrector} follows from the above, Steps 1 and 2, the decompositions \eqref{eq:expression_via_covariance_ito_strat}, \eqref{eq:decomposition_Ito_stratonovich_proof_convergence}-\eqref{eq:decomposition_Ito_stratonovich_proof_convergence1} and the identity $D'_d=\frac{d-1}{d^2}+ \frac{2}{d^2(d+2)}$.
\end{proof}

\subsection{Proof of Lemmas \ref{l:remainders1} and \ref{l:remainders2}}
\label{ss:remainders_lemma}
In this last subsection, we prove Lemmas \ref{l:remainders1} and \ref{l:remainders2}. Here, we use that, due to \eqref{eq:kernel_function_onRd2}, for all integers $k\geq 0$, 
\begin{equation}
\label{eq:decay_riesz_kernel}
|\nabla^k \mathcal{G}_{i,j}(x)|\leq \frac{C_k}{|x|^{d+k}} \ \text{ for all } \ x\neq 0. 
\end{equation}

\begin{proof}[Proof of Lemma \ref{l:remainders1}]
We begin by estimating the first remainder contribution $R_1$. From \eqref{eq:covariance_Linfty_bound}, the smoothness of $\chi$, and $\Psi\in W^{2,\infty}(2B;\R^d)$, it follows that 
\begin{align*}
|R_1|&\leq  C\max_{i,j,m,n} \int_{2B\times 2B} | \mathcal{G}_{i,j}(x-y)| \om_{\varepsilon}(x-y)  |H_{m,j}(y)| |
\partial_n \Psi_i(x)-\partial_n \Psi_i(y)|\,\dd y\,\dd x\\
&\leq  C\|\Psi\|_{W^{2,\infty}(2B)} \max_{m,j} \int_{2B\times 2B} |x-y|^{-d+1} \om_{\varepsilon}(x-y)  |H_{m,j}(y)| \,\dd y\,\dd x\\
&\stackrel{(i)}{\leq}  C\|\Psi\|_{W^{2,\infty}(2B)} \max_{m,j}\Big(\int_{2B} |H_{m,j}(y)|\Big)\Big( \int_{\{|z|\leq \varepsilon\}} |z|^{-d+1} \Big)\\
&\leq C \varepsilon\|\Psi\|_{W^{2,\infty}(2B)}\|H\|_{L^2(2B)}\\
&\stackrel{(ii)}{\leq} C \varepsilon\|\Phi\|_{W^{2,\infty}(B)}\|H\|_{L^2(2B)},
\end{align*}
where in $(i)$ we used $\supp\om_\varepsilon\subseteq B_{2\varepsilon}(0)$ and in $(ii)$ \eqref{eq:choice_Phi_phi}. 

\smallskip

To bound the remainders $R_2$ and $R_3$, we use Lemma \ref{l:convergence_covariance}. Let $\g>0$ be as in the latter result.
Clearly, as $\mathcal{S}_{i,j}\in C^\infty(\T^d\times \T^d)$ and $\Phi\in W^{2,\infty}(B;\R^d)$, it follows that 
\begin{align*}
|R_3|\leq \frac{C}{(\delta N)^\g }\|\Phi\|_{W^{2,\infty}(B)}\|H\|_{L^2(2B)}.
\end{align*}
It remains to estimate $R_2$. From \eqref{eq:decay_riesz_kernel} with $k\in \{0,1\}$ and Lemma \ref{l:convergence_covariance}, we have  
\begin{align*}
|R_2|
&\leq\frac{C}{(\delta N)^\g} \big(R_{2,1}+R_{2,2}\big),
\end{align*}
where
\begin{align*}
 R_{2,1}&\stackrel{{\rm def}}{=}
\max_{i,j,m,n}\big\|(x,y)\mapsto \mathcal{G}_{i,j}(x-y) (1-\om_{\varepsilon}(x-y)) \chi(y)H_{m,j}(y)\partial_n \Phi_i(x)\big\|_{L^2(B\times 2B)}\\
 R_{2,2}&\stackrel{{\rm def}}{=}\max_{i,j,m,n}\big\|(x,y)\mapsto \chi(y)H_{m,j}(y) \nabla_x [ \mathcal{G}_{i,j}(x-y) (1-\om_{\varepsilon}(x-y)) \partial_n \Phi_i(x)\big]\big\|_{L^2(B\times 2B)}.
\end{align*}
Since $\om_\varepsilon|_{B_\varepsilon(0)}=1 $, it follows from \eqref{eq:decay_riesz_kernel} and Fubini's theorem that
\begin{align*}
R_{2,1}
&\leq \max_{i,j,m,n}
\Big(\int_{B\times 2B}
\big| \mathcal{G}_{i,j}(x-y) (1-\om_{\varepsilon}(x-y)) \chi(y)H_{m,j}(y)\partial_n \Phi_i(x)\big|^2\,\dd y\,\dd x\Big)^{1/2}\\
&\leq C \|\Phi\|_{W^{1,\infty}(B)} \max_{i,j,m}\Big(\int_{(B\times 2B) \cap \{|x-y|\geq \varepsilon\}} | \mathcal{G}_{i,j}(x-y)|^2 |H_{m,j}(y)|^2\,\dd y\,\dd x\Big)^{1/2}\\
&\leq 
C \|\Phi\|_{W^{1,\infty}(B)} \Big(\int_{2B} |H|^2 \Big)^{1/2}\Big(\int_{\{|z|\geq \varepsilon\}} |z|^{-2d} \Big)^{1/2}\\
&\leq 
C\varepsilon^{-d/2} \|\Phi\|_{W^{1,\infty}(B)} \|H\|_{L^2(2B)}.
\end{align*}
Similarly, one can readily check that \eqref{eq:decay_riesz_kernel} with $k=1$ implies 
$$
|R_{2,2}|\leq 
C\varepsilon^{-d/2-1} \|\Phi\|_{W^{2,\infty}(B)} \|H\|_{L^2(2B)}.
$$
The claimed estimate of Lemma \ref{l:remainders1} follows by collecting the previous estimates.
\end{proof}

\begin{proof}[Proof of Lemma \ref{l:remainders2}]
Let $B$ be a fixed ball in $\R^d$ as in Lemma \ref{l:local_ito_stratonovich_corrector}.
We claim that there exists $C>0$ such that, for all $g\in L^2(\R^d)$, 
\begin{align}
\label{eq:estimate_G_pv_proof}
\Big| \int_{\R^d} g(y)\, {\normalfont{\text{P.V.}}} \int_{B-y} \mathcal{G}_{i,j}(x)\om_{\varepsilon}(x) e^{2\pi \i \delta k\cdot x} \,\dd x\,\dd y-\wh{\mathcal{G}_{i,j}}(k) \int_{B} g \Big|\qquad &\\
\nonumber
\leq C\|g\|_{L^2(\R^d)} \big[(\varepsilon \delta |k|)^{-1} + \sqrt{\varepsilon}\,\big]&
\end{align}
Indeed, if the above holds, then it follows from \eqref{eq:choice_Phi_phi} and \eqref{eq:def_R4} that
\begin{align*}
|R_4| 
&\leq C \Big(\big[(\varepsilon \delta N)^{-1} + \sqrt{\varepsilon}\,\big]\sum_{k}(\theta^N_k)^2 \Big|\mathrm{Id}-\frac{k\otimes k}{|k|^2}\Big|\Big)\max_{i,j,m,n} 
\|\one_{2B}\chi H_{m,j} \partial_n \Psi_i\|_{L^2(\R^d)}&\\
&\leq C \big[(\varepsilon \delta N)^{-1} + \sqrt{\varepsilon}\,\big]\|H\|_{L^2(2B)}\|\Phi\|_{W^{2,\infty}(B)},
\end{align*}
where we used $\supp\theta^N \subseteq\{N \leq |k|\leq 2N\}$ and $\|\theta^N\|_{\ell^2}=1$.

\smallskip

Thus, Lemma \ref{l:remainders2} is a consequence of \eqref{eq:estimate_G_pv_proof}. In the remaining part of the proof, we show the validity of \eqref{eq:estimate_G_pv_proof}.
Let $g\in L^2(\R^d)$ be fixed.
To prove \eqref{eq:estimate_G_pv_proof}, first, note that $\wh{\mathcal{G}_{i,j}}(k)=\wh{\mathcal{G}_{i,j}}(-\delta k)$ as $ \R^d\setminus \{0\}\ni \xi\mapsto\wh{\mathcal{G}_{i,j}}(\xi)$ is homogeneous of degree $0$ and even, see \eqref{eq:kernel_function_onRd}. Applying Fubini's theorem to the truncated integrals and using the weak $L^2$-convergence of the principal-value truncations, we may write
\begin{align*}
&\int_{\R^d} g (y)\,{\normalfont{\text{P.V.}}} \int_{B-y} \mathcal{G}_{i,j}(x)\om_{\varepsilon}(x) e^{2\pi \i \delta k\cdot x} 
\,\dd x\,\dd y\\
&= {\normalfont{\text{P.V.}}}\int_{\R^d} \mathcal{G}_{i,j}(x) \om_{\varepsilon}(x) e^{2\pi \i \delta k\cdot x}  \Big(\int_{\R^d}  g(y)\one_{B}(x+y) 
\,\dd y\Big)\, \dd x
\end{align*}
and hence,
\begin{align*}
\int_{\R^d} g (y) \, {\normalfont{\text{P.V.}}}\,\int_{B-y} \mathcal{G}_{i,j}(x)\om_{\varepsilon}(x) e^{2\pi \i \delta k\cdot x} 
\,\dd x\,\dd y
-\wh{\mathcal{G}_{i,j}}(-\delta k) \int_{B} g
=R_{4,1}+ R_{4,2},
\end{align*}
where
\begin{align*}
R_{4,1}&=	
{\normalfont{\text{P.V.}}}\int_{\R^d} \mathcal{G}_{i,j}(x) \om_{\varepsilon}(x) e^{2\pi \i \delta k\cdot x} \Big( \int_{\R^d}  g(y)[\one_{B}(x+y) -\one_{B}(y)]
\,\dd y\Big)\, \dd x\\
R_{4,2}&=
\big(\wh{\mathcal{G}_{i,j}\om_{\varepsilon}}(-\delta k)-\wh{\mathcal{G}_{i,j}}(-\delta k) \big)\int_{B} g,
\end{align*}
where $\wh{\mathcal{G}_{i,j}\om_{\varepsilon}}$ denotes the Fourier transform associated with the principal value distribution $\operatorname{P.V.}[\mathcal{G}_{i,j}\om_{\varepsilon}]$.

We estimate $R_{4,1}$ and $R_{4,2}$ in Steps 1 and 2 below, respectively.

\smallskip

\emph{Step 1: There exists $C>0$ such that, for all $g\in L^2(\R^d)$,}
$$
|R_{4,1}|\leq C \|g\|_{L^2(\R^d)} \sqrt{\varepsilon}.
$$
We begin by noticing that   
\begin{align*}
\Big|\int_{\R^d} g(y) [\one_{B}(x+y)-\one_{B}(y)] \,\dd y\Big|
&\leq \|g\|_{L^2(\R^d)}\|\one_{B}(x+y)-\one_{B}(y)\|_{L^2(\R^d)}\\
&\stackrel{(i)}{=}\|g\|_{L^2(\R^d)}\|\one_{B}(x+y)-\one_{B}(y)\|_{L^1(\R^d)}^{1/2}\\
&\stackrel{(ii)}{\leq} C\|g\|_{L^2(\R^d)}\sqrt{|x|},
\end{align*}
where in $(i)$ we used $\one_{B}(x+y)-\one_{B}(y)\in \{-1,0,1\}$, and in $(ii)$ the continuity in the $L^1$-norm of $BV$-functions, which follows by approximation \cite[Theorem 5.3]{EG_measure} and $\| \nabla \one_{B}\|_{TV}\lesssim_d 1$, where $TV$ stands for total variation.

From the above estimate, it follows that the integral in the definition of $R_{4,1}$ converges absolutely. More precisely, as $\supp\om_\varepsilon\subseteq B_{2\varepsilon}(0)$, we have 
\begin{align*}
|R_{4,1}|
\leq C \|g\|_{L^2(\R^d)}
\int_{|x|\leq 2\varepsilon} |\mathcal{G}_{i,j}(x) |\sqrt{|x|}\,\dd x 
\leq C \|g\|_{L^2(\R^d)} \sqrt{\varepsilon},
\end{align*}
where we used \eqref{eq:decay_riesz_kernel}. 

\smallskip

\emph{Step 2: There exists $C>0$ such that, for all $\xi \in \R^d\setminus\{0\}$ and $i,j\in \{1,\dots,d\}$, it holds that}
$$
\big|
\wh{\mathcal{G}_{i,j}}(\xi)- \wh{\mathcal{G}_{i,j}\om_{\varepsilon}}(\xi) \big|
\leq C \varepsilon^{-1}  |\xi|^{-1}.
$$
\emph{In particular, $|R_{4,2}|\leq C \varepsilon^{-1}(\delta N)^{-1} \|g\|_{L^2(\R^d)}$ with $C>0$ independent of $g$.}

To prove the claim of Step 2, letting $\varphi_n=\varphi(\cdot/n)$ for $n\geq 1$ and $\varphi\in C_{{\rm c}}^{\infty}(\R^d)$ such that $\varphi|_{B_1(0)}=1$ and $\varphi|_{\R^d\setminus B_2(0)}=0$, for all $k\in \{1,\dots,d\}$ and $\xi\in \R^d$ such that $\xi_k\neq 0$, we have 
\begin{align*}
&\big|\wh{\mathcal{G}_{i,j}}(\xi)- \wh{\mathcal{G}_{i,j}\om_{\varepsilon}}(\xi) \big|\\
&\qquad\quad  =
\lim_{n\to \infty} \Big| \int_{\R^d}\varphi_n(x) e^{-2\pi \i \xi  \cdot x} (1-\om_{\varepsilon}(x) )\mathcal{G}_{i,j}(x)\,\dd x\Big|\\
&\qquad \quad \stackrel{(i)}{=} 
\lim_{n\to \infty} \Big| \int_{\R^d} \varphi(\varepsilon x'/n) e^{-2\pi \i \varepsilon  \xi \cdot x'} (1-\om(x') )\mathcal{G}_{i,j}(x')\,\dd x'\Big|\\
&\qquad\quad  \lesssim \varepsilon^{-1}|\xi_k|^{-1} \Big(  \limsup_{n\to \infty} \frac{\varepsilon}{n}\Big| \int_{\R^d}  (\partial_{x'_k}\varphi)(\varepsilon x'/n) e^{-2\pi \i \varepsilon  \xi  \cdot  x'} (1-\om(x') )\mathcal{G}_{i,j}(x')\,\dd x'\Big|\\
&\qquad\quad  +\limsup_{n\to \infty} \Big| \int_{\R^d}  \varphi_n(\varepsilon x') e^{-2\pi \i \varepsilon \xi  \cdot x'}  \partial_{x_k'}[(\om(x')-1 )\mathcal{G}_{i,j}(x')]\,\dd x'\Big|\Big)\\
&\qquad\quad  \stackrel{(ii)}{\lesssim} \varepsilon^{-1}|\xi_k|^{-1}\Big( \lim_{n\to \infty} \frac{\varepsilon}{n}+1\Big)\lesssim \varepsilon^{-1}|\xi_k|^{-1},
\end{align*}
where in $(i)$ we used that $\mathcal{G}_{i,j}$ is homogeneous of order $-d$, in $(ii)$ that $\nabla [(1-\om)\mathcal{G}_{i,j}]$ is integrable over $\R^d$ as $\om|_{B_1(0)}=1$, and the fact that 
\begin{align*}
\Big| \int_{\R^d}  (\partial_{x'_k}\varphi)(\varepsilon x'/n) e^{2\pi \i \varepsilon  \xi  \cdot  x'} (1-\om(x') )\mathcal{G}_{i,j}(x')\,\dd x'\Big|
&\lesssim \int_{\{n/\varepsilon \leq |x|\leq 2n/\varepsilon\}} |\mathcal{G}_{i,j}|\\
&\eqsim_d  \int_{n/\varepsilon}^{2n/\varepsilon} \frac{\dd r}{r} \lesssim 1,
\end{align*}
as $\supp\nabla \varphi \subseteq \{1\leq |x|\leq 2\}$.
The claim of Step 2 follows from the arbitrariness of $k\in \{1,\dots,d\}$ (e.g., choosing $k\in \{1,\dots,d\}$ such that $|\xi_k|\geq |\xi|/\sqrt{d}$).
\end{proof}

\medskip

\noindent
{\bf Acknowledgements.}
The author thanks Federico Cornalba, Theresa Lange, and Mark Veraar for their helpful comments.
The author is grateful to Franco Flandoli for fruitful discussions on the local energy inequality of Subsection \ref{ss:local_energy_inequality} during a visit to Scuola Normale Superiore in January 2023.

\medskip

\noindent
{\bf Data availability.} This manuscript has no associated data.
\medskip

\noindent
{\bf Declaration – Conflict of interest.} The author has no conflict of interest.

\def\polhk#1{\setbox0=\hbox{#1}{\ooalign{\hidewidth
  \lower1.5ex\hbox{`}\hidewidth\crcr\unhbox0}}} \def\cprime{$'$}

\end{document}